\pdfoutput=1
\documentclass[twoside]{article}
\usepackage[letterpaper, left=3.5cm, right=3.5cm, top=4cm, bottom=2.5cm]{geometry}
\usepackage{graphicx}
\usepackage{amsmath,amssymb,amsthm}
\usepackage{authblk}
\usepackage[utf8]{inputenc}
\usepackage{microtype}
\usepackage{mathrsfs} 
\usepackage{enumitem}
\usepackage{calc}
\usepackage{esint}
\usepackage{fancyhdr}
\usepackage{xcolor}
\usepackage[labelfont = bf]{caption}
\usepackage{doi}
\usepackage{subfigure}
\usepackage[numbers]{natbib}
\usepackage{float}
\usepackage{stmaryrd}
\usepackage{mathtools}
\usepackage{nccmath}

\newtheorem{theorem}{Theorem}[section]
\numberwithin{equation}{section}
\newtheorem{proposition}[theorem]{Proposition}
\newtheorem{definition}[theorem]{Definition}

\newtheorem{remark}[theorem]{Remark}
\newtheorem{lemma}[theorem]{Lemma}

\newtheorem{algorithm}[theorem]{Algorithm}

\usepackage{titlesec}
\titleformat{\section}{\normalfont\scshape\centering}{\thesection.}{0.5em}{}
\titleformat*{\subsection}{\itshape}
\titleformat*{\subsubsection}{\itshape}

\providecommand{\keywords}[1]
{
	{\small\textit{Keywords:} #1}
}

\providecommand{\MSC}[1]
{
	{\small\textit{AMS MSC (2020):~~} #1}
}

\usepackage{hyperref}
\hypersetup{
	colorlinks=true,
	linkcolor=blue,
	citecolor = blue,
	filecolor=magenta, 
	urlcolor=cyan,
}

\providecommand{\jumptmp}[2]{#1\llbracket{#2}#1\rrbracket}
\providecommand{\jump}[1]{\jumptmp{}{#1}}

\AtBeginDocument{%
	\def\MR#1{}
}

\begin{document}
	\setlength{\abovedisplayskip}{5.5pt}
	\setlength{\belowdisplayskip}{5.5pt}
	\setlength{\abovedisplayshortskip}{5.5pt}
	\setlength{\belowdisplayshortskip}{5.5pt}

	\title{\vspace*{-3cm}Exact a posteriori error control for variational problems\\ via convex duality and explicit flux reconstruction}
	\author[1]{Sören Bartels\thanks{Email: \texttt{bartels@mathematik.uni-freiburg.de}}}
	\author[2]{Alex Kaltenbach\thanks{Email: \texttt{kaltenbach@math.tu-berlin.de}}}
	\date{\today\vspace*{-5mm}}
	\affil[1]{\small{Department of Applied Mathematics, University of Freiburg, Hermann--Herder--Str. 10,  79104  Freiburg}}
	\affil[2]{\small{Institute of Mathematics, Technical University of Berlin, Stra\ss e des 17.\ Juni 135\\ 10623 Berlin}}
	\maketitle

	\pagestyle{fancy}
	\fancyhf{}
	\fancyheadoffset{0cm}
	\addtolength{\headheight}{-0.25cm}
	\renewcommand{\headrulewidth}{0pt} 
	\renewcommand{\footrulewidth}{0pt}
	\fancyhead[CO]{\textsc{Exact Error Control via Convex Duality}}
	\fancyhead[CE]{\textsc{S. Bartels and A. Kaltenbach}}
	\fancyhead[R]{\thepage}
	\fancyfoot[R]{}
	
	\begin{abstract}
		A posteriori error estimates are an important tool to bound discretization errors in terms of computable quantities avoiding regularity conditions that are often difficult to establish. For non-linear and non-differentiable problems, problems involving jumping coefficients, and finite element methods using anisotropic triangulations, such estimates typically involve large factors, leading to sub-optimal error estimates. By making use of convex duality arguments, exact and explicit error representations are derived that avoid such effects.
	\end{abstract}

	
	\MSC{35Q68; 49M25; 49M29; 65K15; 65N15; 65N50.}

	\section{Introduction}\label{sec:intro}
	\thispagestyle{empty}
	\subsection{Sharp error estimation}
	
	\hspace*{5mm}The derivation of sharp a posteriori error estimates 
	has been an active area of research over several decades. 
	 Typical concepts involve the precise characterization of generic constants occurring in residual estimates (\textit{cf}.\ \cite{BaDuRo92,Ains05,VeeVer09,ErStVo10,Vohr11,CaGeRi12,ErnVoh15}), 
	 the approximation~of local problems by higher-order methods (\textit{cf}.\ \cite{MoNoSi03,CarFun99}), 
	 the usage of convex duality relations (\textit{cf}.\ \cite{NeiRep04,Verf09}),
	  and the development of post-processing procedures to obtain equilibrated fluxes
	  (\textit{cf}.\ \mbox{\cite{LucWoh04,BraSch08,Brae09,ErnVoh13,BeCaLu16,SmeVoh20}}).
	  For general discussions of various aspects of a 
	  posteriori error estimation, we refer the reader, \textit{e.g.}, to \cite{BabStr01,Brae07-book,Rep08,Ver13,ChaLeg23}.
	   Recently, fully computable error representations for several convex variational problems have been identified by deriving explicit representation formulas for solutions of dual problems in terms of non-conforming primal approximations. The concept avoids the occurrence of typical constants, applies to a large class of non-quadratic, non-differentiable, constrained, and degenerate problems, for which classical approaches lead to sub-optimal~error~control.\linebreak
	   Closely related concepts  have been
	   used in the derivation of a posteriori error estimates for mixed and non-conforming
	   methods (\textit{cf}.\ \cite{ArnBre85,Arbo95,Alon96,DDPV96,AinOde00-book}).
	   \enlargethispage{15mm}
	
	\subsection{Prager--Synge identity}
	
		\hspace*{5mm}A well-known error representation in the context of the \textit{homogeneous Poisson problem}, \textit{i.e.},
	\begin{align}\label{intro:poisson}
		\begin{aligned}
	-	\Delta u&=f&&\quad\text{ in }\Omega\,,\\
		\nabla u\cdot n&=0&&\quad\text{ on }\Gamma_N\,,\\
		u&=0&&\quad\text{ on }\Gamma_D\,,
	\end{aligned}
	\end{align}
	where $\Gamma_D,\Gamma_N\subseteq \partial\Omega$ with $\Gamma_D\dot{\cup}\Gamma_N=\partial\Omega$ and $f\in L^2(\Omega)$,  
	was pointed out in \cite{PraSyn47} and follows from the celebrated \textit{Prager--Synge identity}, \textit{i.e.}, if $u \in  W^{1,2}_D(\Omega)$ is the (weak) solution of the Poisson problem~\eqref{intro:poisson} and $z = \nabla u \in  W^2_N(\textup{div}; \Omega)$  the solution of the corresponding (Fenchel)~dual~problem, then for every $v\in W^{1,2}_D(\Omega)$ and $y\in  W^2_N(\textup{div}; \Omega)$ with  $\textup{div}\, y = -f$ a.e.\ in $\Omega$, it holds that
	\begin{align}\label{intro:pager-synge}
		\tfrac{1}{2}\| \nabla v- \nabla u\|_{2,\Omega}^2+\tfrac{1}{2}\| y- z\|_{2,\Omega}^2=\tfrac{1}{2}\| \nabla v- y\|_{2,\Omega}^2\,.
	\end{align}
	The identity \eqref{intro:pager-synge} is an immediate consequence of the $L^2$-orthogonality of  $y-z$ and $\nabla v-\nabla u$ and has the interpretation that the squared $L^2$-distance between the gradient of a primal approximation $v\in W^{1,2}_D(\Omega)$ and a dual approximation $y\in  W^2_N(\textup{div}; \Omega)$  with $\textup{div}\, y = -f$~a.e.~in~$\Omega$ 
	yields an explicit and computable way to determine the sum of  the primal~and~dual~\mbox{approximation}~errors.\pagebreak
	
	\hspace*{-5mm}A limitation of using this interpretation practically is that often merely a primal approximation is given and determining an optimal or nearly optimal dual approximation typically~is~too~\mbox{expensive}. For special problems involving two-dimensional Poisson (\textit{cf}. \cite{Brae07-book}) and obstacle problems~(\textit{cf}.~\cite{BHS08}), this difficulty has been overcome by constructing a nearly optimal discrete dual vector field via a local post-processing procedure.\vspace*{-0.5mm}
	
	\subsection{Continuous convex duality}\enlargethispage{13mm}\vspace*{-0.5mm}
	
		\hspace*{5mm}The Prager--Synge identity \eqref{intro:pager-synge} can  be generalized to a larger class of convex minimization problems following, \textit{e.g.},  \cite{Rep99,Rep08,Repin18,Rep20A,Rep20B,Rep20C,BKAFEM22,BK23ROF,repin2024posteriori}: let $\phi \colon \Omega\times \mathbb{R}^d\to \mathbb{R}\cup\{+\infty\}$~and $\psi \colon \Omega\times \mathbb{R}\to \mathbb{R}\cup\{+\infty\}$ be (Lebesgue) measurable functions such that for a.e.\ $x\hspace*{-0.1em}\in\hspace*{-0.1em} \Omega$,~the~functions 
		${\phi (x,\cdot)\colon\hspace*{-0.1em}  \mathbb{R}^d\hspace*{-0.1em}\to\hspace*{-0.1em} \mathbb{R}\hspace*{-0.1em}\cup\hspace*{-0.1em}\{+\infty\}}$~and $\psi (x,\cdot)\colon  \hspace*{-0.15em}\mathbb{R}\hspace*{-0.15em}\to \hspace*{-0.15em}\mathbb{R}\cup\{+\infty\}$  are proper, convex, and lower semi-continuous.
		Then, the minimization of the functional $I\colon \hspace*{-0.15em} W^{1,p}_D(\Omega)\hspace*{-0.15em}\to\hspace*{-0.15em}  \mathbb{R}\cup\{+\infty\}$,~for~every~${v\hspace*{-0.15em}\in \hspace*{-0.15em} W^{1,p}_D(\Omega)}$~defined~by\vspace*{-0.25mm}
	\begin{align}\label{intro:primal}
		I(v)\coloneqq  \int_\Omega{\phi(\cdot,\nabla v)\,\mathrm{d}x}+\int_\Omega{\psi(\cdot,v)\,\mathrm{d}x}\,,
	\end{align}
	denotes the \textit{(Fenchel) primal problem}. 
	A corresponding \textit{(Fenchel) dual problem} consists in the maximization of the functional $D\colon \smash{W^{p'}_N(\textup{div};\Omega)}\to \mathbb{R}\cup\{-\infty\}$, for every $y\in \smash{W^{p'}_N(\textup{div};\Omega)}$~defined~by\vspace*{-0.25mm}
	\begin{align}\label{intro:dual}
		D(y)\coloneqq - \int_\Omega{\phi^*(\cdot,y)\,\mathrm{d}x}-\int_\Omega{\psi^*(\cdot,\textup{div}\,y)\,\mathrm{d}x}\,.
	\end{align}
	Here, $\phi^* \colon \Omega\times \mathbb{R}^d\to \mathbb{R}\cup\{+\infty\}$ and $\psi^* \colon \Omega\times \mathbb{R}\to \mathbb{R}\cup\{+\infty\}$ denote the \textit{Fenchel conjugates} (with respect to the second argument) to $\phi \colon \Omega\times \mathbb{R}^d\to \mathbb{R}\cup\{+\infty\}$ and $\psi \colon \Omega\times \mathbb{R}\to \mathbb{R}\cup\{+\infty\}$, respectively.
%	 \textit{i.e.},  we have that
%	\begin{align}\label{intro:fenchel_conjugate}
%		\begin{aligned}
%		\phi ^*(x,t)&\coloneqq \sup_{s\in \mathbb{R}^d}{[t\cdot s-\phi(x,s)]}&&\quad\text{ for all }t\in\mathbb{R}^d \,,
%		\\
%			\psi ^*(x,t)&\coloneqq \sup_{s\in \mathbb{R}}{[t s-\psi(x,s)]}&&\quad\text{ for all }t\in\mathbb{R} \,.
%		\end{aligned}
%	\end{align}
	% The definitions \eqref{intro:fenchel_conjugate} give rise to
	 \textit{Fenchel--Young inequalities} (\textit{cf}.\ \eqref{eq:fenchel_young_ineq}) in combination~with~an~integration-by-parts formula imply a \textit{weak duality relation}, \textit{i.e.}, for every $v\in \smash{W^{1,p}_D(\Omega)}$~and~$y\in \smash{W^{p'}_N(\textup{div};\Omega)}$,~it~holds~that\vspace*{-0.25mm}
	\begin{align}\label{intro:weak_duality}
		\begin{aligned}
			I(v)&%= \int_\Omega{\phi(\cdot,\nabla v)\,\mathrm{d}x}+\int_\Omega{\psi(\cdot,v)\,\mathrm{d}x}
		%	\\&
			\ge \int_\Omega{\big\{\nabla v\cdot y -\phi^*(\cdot,y)\big\}\,\mathrm{d}x}+ \int_\Omega{\psi(\cdot,v)\,\mathrm{d}x}
			\\&= -\int_\Omega{\phi^*(\cdot,y)\,\mathrm{d}x}- \int_\Omega{\big\{\textup{div}\,y\,v-\psi(\cdot,v)\big\}\,\mathrm{d}x}
			%\\&
			\ge %- \int_\Omega{\phi^*(\cdot,y)\,\mathrm{d}x}-\int_\Omega{\psi(\cdot,\textup{div}\,y)\,\mathrm{d}x}\\&=
			D(y)\,. 
				\end{aligned}
	\end{align}\vspace*{-5mm}
	
	\subsection{Continuous strong duality, convex optimality relations, and flux reconstruction}\vspace*{-0.5mm}
	
	 	\hspace*{5mm}In many cases, \textit{e.g.}, if both $\phi\colon \Omega\times \mathbb{R}^d\to \mathbb{R}$ and $\psi\colon \Omega\times \mathbb{R}\to \mathbb{R}$ are Carath\'eodory mappings\footnote{A mapping $\Phi\colon \Omega\times \mathbb{R}^{\ell}\to \mathbb{R}$, $\ell\in \mathbb{N}$, is called \textit{Carath\'eodory mapping} if for a.e.\ $x\in \Omega$, the function $\Phi(x,\cdot)\colon \mathbb{R}^{\ell}\to \mathbb{R}$ is continuous, and for every $r\in \mathbb{R}^{\ell}$, the function $\Phi(\cdot,r)\colon \Omega\to \mathbb{R}$ is (Lebesgue) measurable.}
	 	and for every $v\in L^p(\Omega)$ and $y\in L^p(\Omega;\mathbb{R}^d)$, it holds that $\phi(\cdot,y),\psi(\cdot,v)\in L^1(\Omega)$ (\textit{cf}.\ \cite[Thm.\ 4.1, p.\ 59, Prop.\ 1.1, p.\ 77]{ET99}), there even holds a \textit{strong duality relation},~\textit{i.e.}, there exists minimizer $u\in W^{1,p}_D(\Omega)$ of \eqref{intro:primal}, called \textit{primal solution}, and a maximizer  $z\in W^{p'}(\textup{div};\Omega)$ of \eqref{intro:dual}, called~\textit{dual solution},  such that\vspace*{-0.5mm}
	 \begin{align}\label{intro:strong_duality}
	 	I(u) = D(z)\,.
	 \end{align}
	 The strong duality relation \eqref{intro:strong_duality} is available for a large class of convex minimization problems,~\textit{e.g.}, including \hspace*{-0.1mm}non-linear \hspace*{-0.1mm}Dirichlet \hspace*{-0.1mm}problems, \hspace*{-0.1mm}obstacle \hspace*{-0.1mm}problems, \hspace*{-0.1mm}certain \hspace*{-0.1mm}non-differentiable~\hspace*{-0.1mm}\mbox{problems},~\hspace*{-0.1mm}and degenerate minimization problems.
	%The strong duality relation \eqref{intro:strong_duality} 
It is equivalent to \textit{convex optimality relations},~\textit{i.e.},~it~holds~that\vspace*{-0.5mm}
	\begin{alignat}{2}
			z \cdot \nabla u&=\phi^*(\cdot,z)+\phi(\cdot,\nabla u)&&\quad\text{ a.e.\ in }\Omega \,,\label{intro:convex_optimality.1}\\
			\textup{div}\,z \, u&=\psi^*(\cdot,\textup{div}\, z)+\psi(\cdot, u)&&\quad\text{ a.e.\ in }\Omega\,.\label{intro:convex_optimality.2} 
	\end{alignat} 
	The convex optimality relations   \eqref{intro:convex_optimality.1}, \eqref{intro:convex_optimality.2} characterize equality in the Fenchel--Young inequalities (\textit{cf}.\ \eqref{eq:fenchel_young_ineq})  used in the derivation of the dual problem (\textit{cf}.\ \eqref{intro:weak_duality}). By the \textit{Fenchel--Young~identity} (\textit{cf}.\ \eqref{eq:fenchel_young_id}), the convex optimality relations \eqref{intro:convex_optimality.1}, \eqref{intro:convex_optimality.2} are each equivalent to the inclusions\vspace*{-0.5mm}
	\begin{alignat}{2}
		z&\in \partial \phi(\cdot,\nabla u)&&\quad\text{ a.e.\ in }\Omega \,,\label{intro:convex_optimality.3}\\
		\textup{div}\,z&\in \partial\psi(\cdot, u)&&\quad\text{ a.e.\ in }\Omega \,,\label{intro:convex_optimality.4}
	\end{alignat}
	where we denote  with $\partial \phi\colon \Omega\times \mathbb{R}^d\to 2^{\mathbb{R}^d}$ and $\partial \psi\colon \Omega\times\mathbb{R}\to 2^{\mathbb{R}}$, the corresponding sub-differentials of $\phi\colon \Omega\times \mathbb{R}^d\to \mathbb{R}\cup\{+\infty\}$ and $\psi\colon \Omega\times \mathbb{R}\to \mathbb{R}\cup\{+\infty\}$  (with respect to the~second~argument).\pagebreak
	
	\hspace*{-5mm}If  $\phi(x,\cdot)\in C^1(\mathbb{R}^d)$ for a.e.\ $x\in\Omega$ or $\psi(x,\cdot)\in C^1(\mathbb{R})$ for a.e.\ $x\in\Omega$, then %the inclusions 
	\eqref{intro:convex_optimality.3},~\eqref{intro:convex_optimality.4}~each~become
	\begin{alignat}{2}
		z&= D\phi(\cdot,\nabla u)&&\quad\text{ a.e.\ in }\Omega \,,\label{intro:convex_optimality.5}\\
		\textup{div}\,z&= D \psi(\cdot, u)&&\quad\text{ a.e.\ in }\Omega \,,\label{intro:convex_optimality.6}
	\end{alignat}
	where the identity \eqref{intro:convex_optimality.5} has an interpretation as reconstruction formula, in the sense that given a primal solution $u\in W^{1,p}_D(\Omega)$, a dual solution $z\in \smash{W^{p'}_N(\textup{div};\Omega)}$ is immediately~available~via~\eqref{intro:convex_optimality.5}.

	\subsection{Generalized Prager--Synge identity}\enlargethispage{11mm}
	
		\hspace*{5mm}The convex and concave functionals $I\colon W^{1,p}_D(\Omega)\hspace*{-0.1em}\to\hspace*{-0.1em} \mathbb{R}\cup\{+\infty\}$ and ${D\colon \hspace*{-0.1em}\smash{W^{p'}_N(\textup{div};\Omega)}\hspace*{-0.1em}\to\hspace*{-0.1em} \mathbb{R}\cup\{-\infty\}}$ give rise to the definition of the non-negative convexity measures $\rho_I^2\colon W^{1,p}_D(\Omega)\times W^{1,p}_D(\Omega)\to \mathbb{R}$~and $\rho_{-D}^2\colon W^{p'}_N(\textup{div};\Omega)\times W^{p'}_N(\textup{div};\Omega)\to \mathbb{R}$, for every $v,\widehat{v}\in W^{1,p}_D(\Omega)$ and $y,\widehat{y}\in W^{p'}_N(\textup{div};\Omega)$~defined~by 
	\begin{align}\label{intro:coercivity_measures}
		\begin{aligned}
		\rho_I^2(v, \widehat{v}) &\coloneqq I(v) - I(\widehat{v}) - \langle \delta I(\widehat{v}),v - \widehat{v}\rangle_{\smash{W^{1,p}_D(\Omega)}}\,,\\
		\rho_{-D}^2(y, \widehat{y}) &\coloneqq -D(y) + D(\widehat{y}) + \langle \delta D(\widehat{y}),y - \widehat{y}\rangle_{\smash{W^{p'}_N(\textup{div};\Omega)}}\,,
	\end{aligned}
	\end{align}
	in case that $I\colon \smash{W^{1,p}_D(\Omega)}\to \mathbb{R}\cup\{+\infty\}$ and $D\colon \smash{W^{p'}_N(\textup{div};\Omega)}\to \mathbb{R}\cup\{-\infty\}$ are Fr\'echet~\mbox{differentiable}. The quantities  \eqref{intro:coercivity_measures} for every $v,\widehat{v}\hspace*{-0.1em}\in\hspace*{-0.1em} \smash{W^{1,p}_D(\Omega)}$ and $y,\widehat{y}\hspace*{-0.1em}\in\hspace*{-0.1em} \smash{W^{p'}_N(\textup{div};\Omega)}$ measure the distance between the functionals $I$ and $-D$ and their tangents at $I(\widehat{v})$ and $-D(\widehat{y})$ evaluated~at~$v$~and~$y$,~\mbox{respectively}. More generally, the variations $\delta I$ and $\delta D$ in \eqref{intro:coercivity_measures} can be replaced by suitable sub-gradients. Since 
	$\delta I(u) = 0$ and $\delta D(z) = 0$ (or $0\in \partial I(u) $ and $0\in \partial D(z)$),   in case of a strong duality \eqref{intro:strong_duality}, we find a \textit{generalized Prager--Synge identity}, \textit{i.e.}, given a primal solution $u\in W^{1,p}_D(\Omega)$~and~dual~solution $z\in \smash{W^{p'}_N(\textup{div};\Omega)}$, for every $v\in \smash{W^{1,p}_D(\Omega)}$ with $I(v)<+\infty$ and $y\in \smash{W^{p'}_N(\textup{div};\Omega)}$ with $D(y)>-\infty$, it holds that\vspace*{-1mm}
	\begin{align}\label{intro:gen_prager_synge}
		\begin{aligned}
			\rho_{\textup{tot}}^2(v,y)&\coloneqq 	\rho_I^2(v, u) +	\rho_{-D}^2(y, z) 
			\\& =I(v)-I(u)+D(z)-D(y)
			\\& =I(v)-D(y)
			\\&\eqqcolon \eta_{\textup{gap}}^2(v,y)\,,
		\end{aligned}
	\end{align}
		\textit{i.e.}, the inaccessible \textit{total error} $\rho_{\textup{tot}}^2(v,y)$ equals the accessible \textit{primal-dual~gap~\mbox{estimator}}~$\eta_{\textup{gap}}^2(v,y)$. 
	For the Poisson problem \eqref{intro:poisson}, we have that $\phi(r) =\phi^*(r) = \frac{1}{2}\vert r\vert ^2$ for all $r\in \mathbb{R}^d$, $\psi(x,s) = -f(x)s$ for a.e.\  $x\in \Omega$ and all $s\in \mathbb{R}$, and $\psi^*(x,s) = I_{\{-f(x)\}}^{\Omega}(s)$ for a.e.\  $x\in \Omega$ and all $s\in \mathbb{R}$, where the latter indicator functional enforces the constraint $\textup{div}\, y = -f $ a.e.\ in $\Omega$. Hence, for every $\smash{v\in W^{1,p}_D(\Omega)}$  and $y\in \smash{W^{p'}_N(\textup{div};\Omega)}$ with  $\textup{div}\, y = -f$ a.e.\ in $\Omega$,~with~an~\mbox{integration-by-parts}~formula,~we~find~that
	\begin{align}\label{intro:gen_prager_synge_poisson}
		\begin{aligned}
		\eta_{\textup{gap}}^2(v,y)&%=\frac{1}{2}\int_{\Omega}{\vert \nabla v\vert^2\,\mathrm{d}x}-\int_{\Omega}{f\, v\,\mathrm{d}x}+\frac{1}{2}\int_{\Omega}{\vert y\vert^2\,\mathrm{d}x}
		%\\&
	%	=\tfrac{1}{2}\|\nabla v\|^2_{2,\Omega}+(\textup{div}\,y, v)_{\Omega}+\tfrac{1}{2}\|y\|^2_{2,\Omega}
	%	\\&
		=\tfrac{1}{2}\|\nabla v\|^2_{2,\Omega}-(y, \nabla v)_{\Omega}+\tfrac{1}{2}\|y\|^2_{2,\Omega}
		\\&=\tfrac{1}{2}\| \nabla v-y\|^2_{2,\Omega}\,.
	\end{aligned}
	\end{align}
	It is remarkable that the employed global relations lead to an error representation as an integral of a non-negative function. Hence, the right-hand side can be decomposed using partitions and thereby provides meaningful local refinement indicators. This observation can~equally~be~generalized: using the definitions \eqref{intro:primal}, \eqref{intro:dual} and an integration-by-parts formula, for every $v\in \smash{W^{1,p}_D(\Omega)}$ with $I(v)<+\infty$ and $y\in \smash{W^{p'}_N(\textup{div};\Omega)}$ with $D(y)>-\infty$, we arrive at the general representation
	\begin{align}\label{intro:eta}
		\begin{aligned}
			\eta_{\textup{gap}}^2(v,y)&=\int_{\Omega}{\big\{\phi(\cdot,\nabla v)-y\cdot \nabla v +\phi^*(\cdot,y)\big\}\,\mathrm{d}x}\\&\quad+\int_{\Omega}{\big\{\psi(\cdot, v)-\textup{div}\,y\, v +\psi^*(\cdot,\textup{div}\,y)\big\}\,\mathrm{d}x}\,.
		\end{aligned}
	\end{align}
	Both integrands on the right-hand side of  \eqref{intro:eta}, by the  Fenchel--Young inequality (\textit{cf}.\ \eqref{eq:fenchel_young_ineq}),~are point-wise \hspace*{-0.15mm}non-negative 
	\hspace*{-0.15mm}and  \hspace*{-0.15mm}vanish \hspace*{-0.15mm}if \hspace*{-0.15mm}and \hspace*{-0.15mm}only \hspace*{-0.15mm}if \hspace*{-0.15mm}convex \hspace*{-0.15mm}optimality \hspace*{-0.15mm}relations \hspace*{-0.15mm}\eqref{intro:convex_optimality.1}, \hspace*{-0.15mm}\eqref{intro:convex_optimality.2}~\hspace*{-0.15mm}are~\hspace*{-0.15mm}\mbox{satisfied}. Hence, the primal dual-gap estimator measures the validity of the convex optimality relations \eqref{intro:convex_optimality.1},~\eqref{intro:convex_optimality.2}. Note that the convex optimality relations \eqref{intro:convex_optimality.1},~\eqref{intro:convex_optimality.2} do not require any regularity of $\phi\colon \Omega\times\mathbb{R}^d\to \mathbb{R}\cup\{+\infty\}$ and $\psi\colon \Omega\times \mathbb{R}\to \mathbb{R}\cup \{+\infty\}$, which makes the primal-dual gap estimator a predestined a posteriori error estimator  for non-differentiable convex minimization problems.
	
	\subsection{Discrete reconstruction formula}
	
		\hspace*{5mm}As in the case of the Poisson problem \eqref{intro:poisson}, dual problems typically involve constraints~(\textit{e.g.,} $\textup{div}\,y=-f$ a.e.\ in $\Omega$ (\textit{cf}.\ \eqref{intro:pager-synge}))  and are, thus, significantly harder to solve than~primal~problems, which often lead to monotone operators for which efficient iterative numerical solution procedures are available. It is, therefore,
	fundamental to avoid the explicit numerical solution~of~dual~problems. 
	For a Crouzeix--Raviart approximation of the Poisson problem \eqref{intro:poisson}, seeking for $u_h^{cr}\in \mathcal{S}^{1,cr}_D(\mathcal{T}_h)$ such that for every $v_h\in \mathcal{S}^{1,cr}_D(\mathcal{T}_h)$, it holds that
	\begin{align}
		(\nabla _hu_h^{cr},\nabla _hv_h)_{\Omega}=(f_h,\Pi_hv_h)_{\Omega}\,,\label{intro:poisson_cr}
	\end{align}
	where $f_h\in \mathcal{L}^0(\mathcal{T}_h)$ is an element-wise constant approximation of  $f\in L^2(\Omega)$,  using elementary but fundamental relations between the Crouzeix--Raviart and the Raviart--Thomas~element, 
	a remarkable 
	relation, called \textit{Marini formula}, between $u_h^{cr}\hspace*{-0.15em}\in\hspace*{-0.15em} \mathcal{S}^{1,cr}_D(\mathcal{T}_h)$~and~a~\mbox{Raviart--Thomas}~solution $z_h^{rt}\in \mathcal{R}T^0_N(\mathcal{T}_h)$ of the corresponding dual formulation to \eqref{intro:poisson_cr}
	has been~established~in~\cite{Mar85,AB85,Arbo95},~\textit{i.e.},\enlargethispage{1mm}
	\begin{align}\label{intro:marini_formula}
		z_h^{rt}=\nabla_h u_h^{cr}-\frac{f_h}{d}(\textup{id}_{\mathbb{R}^d}-\Pi_h\textup{id}_{\mathbb{R}^d})\quad\text{ a.e.\ in }\Omega\,.
	\end{align} 
	In \eqref{intro:poisson_cr}, $\mathcal{S}^{1,cr}_D(\mathcal{T}_h)$ denotes the \textit{Crouzeix--Raviart finite element space},  $ \mathcal{R}T^0_N(\mathcal{T}_h)$ the \textit{Raviart--Thomas finite element space}, $\nabla_h$ the element-wise application of the gradient operator $\nabla$, and $\Pi_h $ the (local) $L^2$-projection operator onto 
	the space of element-wise constant~functions~$\mathcal{L}^0(\mathcal{T}_h)$. 
	Due to $\mathcal{R}T^0_N (\mathcal{T}_h)\subseteq W^2_N (\textup{div};\Omega)$ and $\textup{div}\, z_h^{rt} = -f_h$ a.e.\ in $\Omega$, if
	  $f=f_h\in \mathcal{L}^0(\mathcal{T}_h)$, $z_h^{rt}\in \mathcal{R}T^0_N(\mathcal{T}_h)$ given via \eqref{intro:marini_formula}  is admissible in \eqref{intro:pager-synge}. In other words, in the case~of~the~Poisson~problem~\eqref{intro:poisson}, the discrete reconstruction formula \eqref{intro:marini_formula} enables to approximate  the primal and the dual problem simultaneously using only approximation \eqref{intro:poisson_cr}. 
	
	\subsection{Discrete convex duality}
	
		\hspace*{5mm} It is possible to construct a discrete primal problem that induces the same (discrete) convex duality relations like the continuous primal problem \eqref{intro:primal}, \textit{i.e.}, a corresponding discrete dual problem, a discrete weak duality relation, the equivalence of a discrete strong duality relation to discrete convex optimality relations and, most important, a discrete analogue of the reconstruction formula \eqref{intro:convex_optimality.5} or generalization of the discrete~reconstruction~formula~\eqref{intro:marini_formula},~respectively.~To this end, 
		we need to perform three non-conforming modifications on the~primal~energy~\mbox{functional}~\eqref{intro:primal}:
		\begin{itemize}[noitemsep,topsep=2pt,leftmargin=!]
			\item  we replace the energy densities by  $\phi_h\colon \Omega\times\mathbb{R}^d\to \mathbb{R}\cup\{+\infty\}$ and $\psi_h\colon \Omega\times\mathbb{R}^d\to \mathbb{R}\cup\{+\infty\}$, which are, again, (Lebesgue) measurable functions such that for a.e.\ $x\in \Omega$, the functions $\phi_h(x,\cdot)\colon \mathbb{R}^d\to \mathbb{R}\cup\{+\infty\}$ and $\psi_h(x,\cdot)\colon\mathbb{R}\to \mathbb{R}\cup\{+\infty\}$ are proper, convex, and lower semi-continuous. In addition,  for every $r\in \mathbb{R}^d$ and $s\in \mathbb{R}$, the functions $\phi_h(\cdot,r)\colon \Omega\to \mathbb{R}\cup\{+\infty\}$ and $\psi_h(\cdot,s)\colon\Omega\to \mathbb{R}\cup\{+\infty\}$ are element-wise constant;
			\item  we incorporate a (local) $L^2$-projection $\Pi_h$ onto $\mathcal{L}^0(\mathcal{T}_h)$ into the lower-order term;
			\item we replace the weak gradient $\nabla$ with the element-wise gradient $\nabla_h$.
		\end{itemize} 
		
	With these three non-conforming modifications, the \textit{discrete primal problem} is given via the minimization of the functional $I_h^{cr}\colon \mathcal{S}^{1,cr}_D(\mathcal{T}_h)\to \mathbb{R}\cup\{+\infty\}$, for every $v_h \in  \mathcal{S}^{1,cr}_D(\mathcal{T}_h)$ defined by
	\begin{align}\label{intro:discrete_primal}
			I_h^{cr}(v_h) \coloneqq \int_{\Omega}{\phi_h(\cdot,\nabla_hv_h)\,\mathrm{d}x} + \int_{\Omega}{\psi_h(\cdot,\Pi_hv_h)\,\mathrm{d}x} \,.
	\end{align}
	A corresponding \textit{discrete dual problem} to the minimization of \eqref{intro:discrete_primal} consists in the maximization of the functional $D_h^{rt}\colon \mathcal{R}T^0_N(\mathcal{T}_h)\to \mathbb{R}\cup\{-\infty\}$,
	for every $y_h\in \mathcal{R}T^0_N(\mathcal{T}_h)$  defined by 
	\begin{align}\label{intro:discrete_dual}
			D_h^{rt}(y_h) \coloneqq- \int_{\Omega}{\phi_h^*(\cdot,\Pi _hy_h)\,\mathrm{d}x} + \int_{\Omega}{\psi_h^*(\cdot,\textup{div}\,y_h)\,\mathrm{d}x} \,.
	\end{align}
	Note that the three non-conforming modifications (more precisely, the first two modifications), in particular, ensure that the integrands in \eqref{intro:discrete_primal} and \eqref{intro:discrete_dual}, respectively, are element-wise~constant.
	This is crucial for establishing a strong duality relation between the discrete primal problem \eqref{intro:discrete_primal}  and the discrete dual problem \eqref{intro:discrete_dual}.
	As above, Fenchel--Young inequalities  imply that a \textit{discrete weak duality relation} applies, \textit{i.e.}, for every $v_h\in   \mathcal{S}^{1,cr}_D(\mathcal{T}_h)$  and $y_h\in \mathcal{R}T^0_N(\mathcal{T}_h)$,~it~holds~that
		\begin{align}\label{intro:discrete_weak_duality}
			\begin{aligned}
		I_h^{cr}(v_h)&
		\ge \int_\Omega{\big\{
			\nabla_h v_h\cdot \Pi_h y_h -\phi^*_h(\cdot,\Pi_hy_h)\big\}\,\mathrm{d}x}+ \int_\Omega{\psi_h(\cdot,\Pi_hv_h)\,\mathrm{d}x}
		\\&=  - \int_\Omega{\phi^*_h(\cdot,\Pi_hy_h)\,\mathrm{d}x}-\int_\Omega{\big\{\textup{div}\,y_hv_h-\psi_h(\cdot,v_h)\big\}\,\mathrm{d}x}\ge D_h^{rt}(y_h)\,. 
	\end{aligned}
	\end{align}
	In \eqref{intro:discrete_weak_duality}, we used a \textit{discrete integration-by-parts formula} without contributions~from~element~sides, \textit{i.e.},  for every $v_h\in \mathcal{S}^{1,cr}_D(\mathcal{T}_h)$ and $y_h\in \mathcal{R}T^0_N(\mathcal{T}_h)$, it holds that
	\begin{align}\label{intro:discrete_pi}
		(\nabla_h v_h, \Pi_h y_h)_{\Omega}=-(\textup{div}\,y_h,\Pi_hv_h)_{\Omega}\,.
	\end{align}
	The use of the (local) $L^2$-projection in \eqref{intro:discrete_pi} is optional, but
	crucial for establishing a strong duality relation between  the discrete primal problem \eqref{intro:discrete_primal}  and the discrete~dual~problem~\eqref{intro:discrete_dual}. There holds a \textit{discrete strong duality relation}, \textit{i.e.}, it holds that
	\begin{align}\label{intro:discrete_strong_duality}
		I_h^{cr}(u_h^{cr})=D_h^{rt} (z_h^{rt})\,,
	\end{align} 
	if and only if there hold \textit{discrete convex optimality relations}, \textit{i.e.}, it holds that\enlargethispage{5mm}
	\begin{alignat}{2}
			\Pi_hz_h^{rt} \cdot \nabla_h u_h^{cr}&=\phi^*_h(\cdot,\Pi_h z_h^{rt})+\phi_h(\cdot,\nabla_h u_h^{cr})&&\quad\text{ a.e.\ in }\Omega \,,\label{intro:discrete_convex_optimality.1}\\
			\textup{div}\,z_h^{rt} \, \Pi_h u_h^{cr}&=\psi^*_h(\cdot,\textup{div}\, z_h^{rt})+\psi_h(\cdot, \Pi_hu_h^{cr})&&\quad\text{ a.e.\ in }\Omega\,.\label{intro:discrete_convex_optimality.2} 
	\end{alignat} 
	By the Fenchel--Young identity (\textit{cf}.\ \eqref{eq:fenchel_young_id}), the discrete convex optimality relations \eqref{intro:discrete_convex_optimality.1}, \eqref{intro:discrete_convex_optimality.2} are each equivalent to the inclusions
	\begin{alignat}{2}
			\Pi_hz_h^{rt} &\in \partial \phi_h(\cdot,\nabla_h u_h^{cr})&&\quad\text{ a.e.\ in }\Omega \,,\label{intro:discrete_convex_optimality.3}\\
		\textup{div}\,z_h^{rt}&\in \partial \psi_h(\cdot, \Pi_hu_h^{cr})&&\quad\text{ a.e.\ in }\Omega \,.\label{intro:discrete_convex_optimality.4}
	\end{alignat}
	If  $\phi_h(x,\cdot)\hspace*{-0.15em}\in \hspace*{-0.15em} C^1(\mathbb{R}^d)$ for a.e.\ $x\hspace*{-0.15em}\in\hspace*{-0.15em}\Omega$ or $\psi_h(x,\cdot)\hspace*{-0.15em}\in\hspace*{-0.15em} C^1(\mathbb{R})$ for a.e.\ $x\hspace*{-0.15em}\in\hspace*{-0.15em}\Omega$, then~\eqref{intro:discrete_convex_optimality.3},~\eqref{intro:discrete_convex_optimality.4}~each~become
	\begin{alignat}{2}
		\Pi_hz_h^{rt}&= D\phi_h(\cdot,\nabla_h u_h^{cr})&&\quad\text{ a.e.\ in }\Omega \,,\label{intro:discrete_convex_optimality.5}\\
		\textup{div}\,z_h^{rt}&= D \psi_h(\cdot, u_h^{cr})&&\quad\text{ a.e.\ in }\Omega \,.\label{intro:discrete_convex_optimality.6}
	\end{alignat}
	Note that, different from the continuous reconstruction formula \eqref{intro:convex_optimality.5}, the discrete convex optimality relation \eqref{intro:discrete_convex_optimality.5} does not give full information about a discrete~dual~solution~${z_h^{rt}\in \mathcal{R}T^0_N(\mathcal{T}_h)}$. However, \hspace*{-0.15mm}using \hspace*{-0.15mm}the \hspace*{-0.15mm}additional \hspace*{-0.15mm}information \hspace*{-0.15mm}provided \hspace*{-0.15mm}by \hspace*{-0.15mm}the \hspace*{-0.15mm}discrete \hspace*{-0.15mm}convex~\hspace*{-0.15mm}optimality~\hspace*{-0.15mm}\mbox{relation}~\hspace*{-0.15mm}\eqref{intro:discrete_convex_optimality.6}, the surjectivity of the divergence operator $\textup{div}\colon \mathcal{R}T^0_N(\mathcal{T}_h)\to \mathcal{L}^0(\mathcal{T}_h)$ if $\Gamma_D\neq \emptyset$ (and surjectivity  of $\textup{div}\colon \mathcal{R}T^0_N(\mathcal{T}_h)\to \mathcal{L}^0(\mathcal{T}_h)/\mathbb{R}$ if $\Gamma_D= \emptyset$), and  the \textit{discrete Helmholtz decomposition}
	\begin{align}\label{intro:discrete_helmholtz}
		(\mathcal{L}^0(\mathcal{T}_h))^d=\textup{div}\,(\textup{ker}|_{\mathcal{R}T^0_N(\mathcal{T}_h)})\oplus \nabla_h(\mathcal{S}^{1,cr}_D(\mathcal{T}_h))\,,
	\end{align}
	if $\phi_h(x,\cdot)\in C^1(\mathbb{R}^d)$ and $\psi_h(x,\cdot)\in C^1(\mathbb{R})$ for a.e.\ $x\in \Omega$, it is possible to establish a discrete strong duality relation \eqref{intro:discrete_strong_duality} and a  \textit{generalized Marini formula} (\textit{cf}.\ \cite{Bar21,CL15,LLC18}),~\textit{i.e.},~it~holds~that
	\begin{align}\label{intro:gen_marini}
		z_h^{rt}= D\phi_h(\cdot,\nabla_h u_h^{cr})+\frac{D\psi_h(\cdot,\Pi_h u_h^{cr})}{d}(\textup{id}_{\mathbb{R}^d}-\Pi_h\textup{id}_{\mathbb{R}^d})\quad\text{ a.e.\ in }\Omega\,.
	\end{align}
	Similar to the equivalence of \eqref{intro:convex_optimality.4} to $u\in \partial\psi(\cdot,\textup{div}\,z)$ a.e.~in~$\Omega$, that is 
	$u= D\psi(\cdot,\textup{div}\,z)$~a.e.~in~$\Omega$ if $\psi (x,\cdot)\in C^1(\mathbb{R})$ for a.e.\ $x\in \Omega$, 
	if instead $\phi_h^*(x,\cdot)\in C^1(\mathbb{R}^d)$ and $\psi_h^*(x,\cdot)\in C^1(\mathbb{R})$ for a.e.\ $x\in \Omega$, there exists an inversion of the reconstruction formula   \eqref{intro:gen_marini} (\textit{cf}.\ \cite{BK23ROF}), \textit{i.e.}, it holds that
	\begin{align}\label{intro:gen_marini_inv}
		 	u_h^{cr} = D\psi_h^*(\cdot,\textup{div}\,z_h^{rt})+ D\phi^*_h(\cdot,\Pi_h z_h^{rt})\cdot(\textup{id}_{\mathbb{R}^d}-\Pi_h\textup{id}_{\mathbb{R}^d})
		 \quad\text{ a.e.\ in }\Omega\,.
	\end{align}
	In many cases, we have that $\psi_h = \psi$, \textit{e.g.}, if $\psi(x,s) = -f(x)s$ for a.e.\ $x\in \Omega$ and all $s\in \mathbb{R}$ with an element-wise constant function $f\in \mathcal{L}^0(\mathcal{T}_h)$, and, in this case, by Jensen’s inequality,~we~have~that $-D_h^{rt}(z_h^{rt}) \leq  -D(z_h^{rt})$. Hence, in general, an inconsistent or non-conforming discretization of \hspace*{-0.1mm}the \hspace*{-0.1mm}primal \hspace*{-0.1mm}problem \hspace*{-0.1mm}is \hspace*{-0.1mm}necessary \hspace*{-0.1mm}to \hspace*{-0.1mm}ensure \hspace*{-0.1mm}strong \hspace*{-0.1mm}discrete \hspace*{-0.1mm}duality~\hspace*{-0.1mm}and~\hspace*{-0.1mm}to~\hspace*{-0.1mm}obtain~\hspace*{-0.1mm}a~\hspace*{-0.1mm}reconstruction~\hspace*{-0.1mm}formula.
	
	\newpage
	\subsection{A priori error estimates}
	
		\hspace*{5mm}The continuous strong duality relation \eqref{intro:strong_duality} and the discrete strong duality relation \eqref{intro:discrete_strong_duality} are useful in deriving a priori error estimates for the primal problem \eqref{intro:primal} and the~dual~problem~\eqref{intro:dual}.
	Letting $v_h \coloneqq \Pi_h^{cr} u\in \mathcal{S}^{1,cr}_D(\mathcal{T}_h)$ and $y_h \coloneqq \Pi_h^{rt} z\in \mathcal{R}T^0_N(\mathcal{T}_h)$ be the Crouzeix--Raviart and  Raviart--Thomas  quasi-interpolants of a primal solution $u\hspace*{-0.1em}\in \hspace*{-0.1em}W^{1,p}_D(\Omega)$~and~a~dual~\mbox{solution}~$\smash{z\hspace*{-0.1em}\in\hspace*{-0.1em} W^{p'}_N(\textup{div};\Omega)}$, respectively, if a strong duality relation (\textit{i.e.}, \eqref{intro:strong_duality}) applies, due to the discrete weak duality relation \eqref{intro:discrete_weak_duality}, we have that
	\begin{align}\label{intro:a_priori}
			\begin{aligned} 
		e_{\textup{tot},h}^2 (v_h,y_h)&\coloneqq \rho_{I_h^{cr}}^2 (v_h,u_h^{cr})+	\rho_{-D_h^{rt}}^2 (y_h,z_h^{rt})\\&=[I_h^{cr}(v_h)-I_h^{cr}(u_h^{cr})]+[D_h^{rt}(z_h^{rt})-D_h^{rt}(y_h)]
		\\&\leq  I_h^{cr}(v_h)-D_h^{rt}(y_h)
		\\&=[I_h^{cr}(v_h)-I(u)]+[D(z)-D_h^{rt}(z_h^{rt})]\,,
			\end{aligned}
	\end{align}
	\textit{i.e.}, the primal approximation error between the discrete primal solution $u_h^{cr}\in \mathcal{S}^{1,cr}_D(\mathcal{T}_h)$ and
	the  interpolant of the primal solution $u\in W^{1,p}_D(\Omega)$ plus the dual approximation error between
	the discrete dual solution $z_h^{rt}\in \mathcal{R}T^0_N(\mathcal{T}_h)$ and the interpolant of the dual solution $z\in W^{p'}_N(\textup{div};\Omega)$ is bounded in terms of primal and dual interpolation errors.~\hspace*{-0.15mm}For~\hspace*{-0.15mm}the~\hspace*{-0.15mm}\mbox{Poisson}~\hspace*{-0.15mm}\mbox{problem}~\hspace*{-0.15mm}\eqref{intro:poisson},~\hspace*{-0.15mm}from~\hspace*{-0.15mm}\eqref{intro:a_priori}, 
	using that $\nabla_h v_h = \Pi_h \nabla u$ a.e.\ in $\Omega$ together with Jensen's inequality, $\textup{div}\,y_h=\Pi_h\textup{div}\,z$~a.e.~in~$\Omega$ together with $\textup{div}\,z\hspace*{-0.1em}=\hspace*{-0.1em}-f$ a.e.\ in $\Omega$, an integration-by-parts formula, and the~second~binomial~\mbox{formula},  
	it follows that
	\begin{align}\label{intro:a_priori_poisson}
		\begin{aligned} 
		\tfrac{1}{2}\| \nabla_h v_h -\nabla_h u_h^{cr}\|^2_{2,\Omega}+\tfrac{1}{2}\| y_h -z_h^{rt}\|^2_{2,\Omega}&\leq 
			\tfrac{1}{2}\| \nabla_h v_h\|_{2,\Omega}^2  -\tfrac{1}{2}\| \nabla u\|^2_{2,\Omega}\\&\quad-
		(f,v_h-u)_{\Omega}-	\tfrac{1}{2}\| z\|_{2,\Omega}^2  +\tfrac{1}{2}\| y_h\|_{2,\Omega}^2 
		\\&\leq (z,z-y_h)_{\Omega}-	\tfrac{1}{2}\| z\|_{2,\Omega}^2  +\tfrac{1}{2}\| y_h\|_{2,\Omega}^2 
		\\&=\tfrac{1}{2}\|z-y_h\|_{2,\Omega}^2\,.
	\end{aligned}
	\end{align}
	%Here, we used $\nabla_h v_h = \Pi_h \nabla u$ and Jensen’s inequality to show that the
%	first integral is nonpositive, and $f = -\textup{div}\,z$ and $z = \nabla u$ combined with an integration by parts to rewrite the second integral so that a binomial formula can be applied. 
	If $u \in  W^{1+s,2}(\Omega)$ (\textit{i.e.}, $z \in W^{s,2}(\Omega;\mathbb{R}^d)$), $s\in (0,1)$, 
	 the right-hand side in \eqref{intro:a_priori_poisson}~is~of~order~$\mathcal{O}(h^{2s})$. This short proof can be generalized to a large class of variational problems (\textit{cf}. \cite{Bar21,K22CR,BK22Obstacle}) and avoids the usage of Strang lemmas to control the effect of the non-conformity of the discretization. It has recently been observed that Crouzeix--Raviart discretizations lead to higher convergence rates than classical conforming methods for certain non-differentiable problems (\textit{cf}. \cite{CP20,BK22}).\enlargethispage{5mm}
	
	\subsection{Data approximation and inexact solution}

	\subsubsection{Data approximation}
	
	\hspace*{5mm}In the case of a linear lower-order term $\psi\colon \Omega\times \mathbb{R}\to \mathbb{R}$, \textit{i.e.},
	$\psi(x,s) = -f(x)s$ for a.e.\ $x\in \Omega$ and all $s\in \mathbb{R}$ for some $f\in L^{p'}(\Omega)$, the discrete dual solution $z_h^{rt}\in \mathcal{R}T^0_N(\mathcal{T}_h)$ given via \eqref{intro:gen_marini}
	is admissible in the continuous dual  problem, \textit{i.e.}, satisfies $D(z_h^{rt})>-\infty$, if and only if $\textup{div}\,y_h = -f$ a.e.\ in $\Omega$, \textit{i.e.}, if and only if 
	$f=f_h\in \mathcal{L}^0(\mathcal{T}_h)$ is element-wise constant.  If this is not the case, then we introduce a modified functional $I^{(h)}\colon W^{1,p}_D(\Omega)\to \mathbb{R}\cup\{+\infty\}$, for every~$v\in W^{1,p}_D(\Omega)$~defined~by
	\begin{align}\label{intro:data_approx_primal}
			I^{(h)}(v)\coloneqq  \int_\Omega{\phi_h(\cdot,\nabla v)\,\mathrm{d}x}+\int_\Omega{\psi_h(\cdot,\Pi_h v)\,\mathrm{d}x}\,.
	\end{align}
	Then, for a minimizer $u^{(h)}\in W^{1,p}_D(\Omega)$ of \eqref{intro:data_approx_primal}, due to $I^{(h)}(u^{(h)}) \leq  I^{(h)}(u)$, we have that 
	\begin{align*}
		\rho_I^2(u^{(h)},u)&=I(u^{(h)})-I(u)
		\\&
		\leq 
		I(u^{(h)})-I^{(h)}(u^{(h)})+I^{(h)}(u)-I(u)
		\\&\leq \|\phi(\cdot,\nabla u^{(h)})-\phi_h(\cdot,\nabla u^{(h)})\|_{1,\Omega}+\|\psi(\cdot,u^{(h)})-\psi_h(\cdot,\Pi_hu^{(h)})\|_{1,\Omega}
		\\&\quad+\|\phi_h(\cdot,\nabla u)-\phi(\cdot,\nabla u)\|_{1,\Omega}+\|\psi_h(\cdot,\Pi_h u)-\psi(\cdot,u)\|_{1,\Omega}\,.
	\end{align*}
	For autonomous higher-order term $\phi\colon \mathbb{R}^d\to \mathbb{R}\cup\{+\infty\}$ 
	and
	linear lower-order term $\psi\colon \Omega\times \mathbb{R}\to \mathbb{R}$, \textit{i.e.},
	$\psi(x,s) = -f(x)s$ for a.e.\ $x\in \Omega$ and all $s\in \mathbb{R}$ for some $f\in L^{p'}(\Omega)$,
	with discretization  $\psi_h\colon \Omega\times \mathbb{R}\to \mathbb{R}$, given via $\psi_h(x,s) = -f_h(x)s$ for a.e.\ $x\in \Omega$ and all $s\in \mathbb{R}$,~where~${f_h\coloneqq \Pi_h f\in \mathcal{L}^0(\mathcal{T}_h)}$, 
	due to $f-f_h\perp \Pi_h (u-u^{(h)})$ in $L^2(\Omega)$,   using element-wise  Poincar\'e‘s inequality, 
	%of Poincar\'e‘s inequality, 
	we find that 
	\begin{align*}
			\rho_I^2(u^{(h)},u)&=%(f-f_h,u-u^{(h)})_{\Omega}
			%\\&=
			(f-f_h,u-u^{(h)}-\Pi_h (u-u^{(h)}))_{\Omega}
			\\&\leq c_P\,\|h_{\mathcal{T}}(f-f_h)\|_{p',\Omega}\|\nabla u-\nabla u^{(h)}\|_{p,\Omega}\,.
	\end{align*}
	Then,  the arguments explained above apply to the modified functional \eqref{intro:data_approx_primal}. 
	
		\subsubsection{Inexact solution} 
		
	\hspace*{5mm}In the case $\phi_h(x,\cdot)\in C^1(\mathbb{R}^d)$ and $\psi_h(x,\cdot)\in C^1(\mathbb{R})$ for a.e.\ $x\in \Omega$, 
	we can incorporate errors resulting from the inexact iterative solution of the discrete primal problem via discrete residuals.
	More precisely, if $\widetilde{u}_h^{cr}\in  \mathcal{S}^{1,cr}_D(\mathcal{T}_h)$ is an inexact approximation of the discrete primal problem, \textit{i.e.}, quasi-minimizer of  \eqref{intro:discrete_primal},
	we represent the residual in the discrete $W^{1,2}$-semi-norm. More precisely, we choose $\widetilde{r}_h\in \mathcal{S}^{1,cr}_D(\mathcal{T}_h)$ such that~for~every~${v_h\in \mathcal{S}^{1,cr}_D(\mathcal{T}_h)}$, it holds that
	\begin{align*}
		(\nabla_h \widetilde{r}_h, \nabla_h v_h)_{\Omega}=(D\phi_h(\cdot,\nabla_h u_h^{cr}), \nabla_h v_h)_{\Omega}+(D\psi_h(\cdot, u_h^{cr}),\Pi_h v_h)_{\Omega}\,.
	\end{align*}
	Then, $\widetilde{u}_h^{cr}\in \mathcal{S}^{1,cr}_D(\mathcal{T}_h)$ is a minimizer of the %modified
	functional $\widetilde{I}_h^{cr}\colon \mathcal{S}^{1,cr}_D(\mathcal{T}_h)\to \mathbb{R}$, for every $v_h\in \mathcal{S}^{1,cr}_D(\mathcal{T}_h)$ defined by 
	\begin{align*}
		\widetilde{I}_h^{cr}(v)\coloneqq \int_\Omega{\widetilde{\phi}_h(\cdot,\nabla_h v_h)\,\mathrm{d}x}+\int_\Omega{\psi_h(\cdot, v_h)\,\mathrm{d}x}\,.
	\end{align*}
	where $\widetilde{\phi}_h\colon \Omega\times \mathbb{R}^d\to \mathbb{R}$ is defined by $\widetilde{\phi}_h(x,r)\coloneqq \phi_h(x,r)-\nabla_h\widetilde{r}_h(x)\cdot r$ for a.e.\ $x\in \Omega$~and~all~$r\in \mathbb{R}^d$.
	The identities and estimates derived above now hold with $\phi_h$ replaced with $\widetilde{\phi}_h$. 
	
	\subsection{Properties of the primal-dual gap estimator} 
	
	\hspace*{5mm}In general, the discrete primal solution $u_h^{cr}\in \mathcal{S}^{1,cr}_D(\mathcal{T}_h)$
	is not admissible in the continuous primal problem since $u_h^{cr}\notin W^{1,p}_D(\Omega)$ and, thus, cannot be inserted in the primal-dual~gap~\mbox{estimator}, while
	 the discrete
	dual solution $z_h^{rt}\in \mathcal{R}T^0_N(\mathcal{T}_h)\subseteq \smash{W^{p'}_N(\textup{div};\Omega)}$, up to data approximation terms, is admissible in the continuous dual problem. An admissible approximation $\overline{u}_h^{cr}\in W^{1,p}_D(\Omega)$, \textit{e.g.}, can be obtained via a cheap node-averaging procedure.  
	 The generalized Prager--Synge~identity~\eqref{intro:gen_prager_synge} imposes no restrictions about optimality of the arguments, so that with admissible approximations $\overline{u}_h^{cr}\in W^{1,p}_D(\Omega)$ with $I(\overline{u}_h^{cr})<+\infty$   and $\overline{z}_h^{rt}\in W^{p'}_N(\textup{div};\Omega)$ with $D(\overline{z}_h^{rt})>-\infty$,~we~still~have~that
	\begin{align}\label{intro:gen_prager_synge_inexact}
		\smash{\rho_{\textup{tot}}^2}(\overline{u}_h^{cr},\overline{z}_h^{rt})=\eta_{\textup{gap}}^2(\overline{u}_h^{cr},\overline{z}_h^{rt})\,.
	\end{align}
	The primal-dual gap error estimator has some remarkable features:
	\begin{itemize}[noitemsep,topsep=2pt,leftmargin=!]
		\item It is obtained by a simple post-processing procedure of the discrete primal problem~(\textit{cf}.~\eqref{intro:gen_marini});
		\item It is reliable and efficient with constants one (\textit{cf}.\ \eqref{intro:gen_prager_synge}); 
		\item It does not require an exact solution of the possibly non-linear primal~problem (\textit{cf}.\ \eqref{intro:gen_prager_synge_inexact}); 
		\item It is globally equivalent to residual type estimators for many model problems;
		\item Its integrands are point-wise non-negative and, thus, it is suitable for local mesh refinement;
		\item It applies to non-linear, non-differentiable, degenerate, scalar, and vectorial problems and does not require the development of a particular error analysis.
	\end{itemize}
	Besides these positive features, some relevant properties are desirable but are not yet established:
	\begin{itemize}[noitemsep,topsep=2pt,leftmargin=!]
		\item Is it possible to prove optimal convergence of adaptive methods based on the local mesh refinement indicators given by the local contributions to the primal-dual gap estimator?
		\item Can one devise a general strategy to ensure admissibility of the reconstructed flux in the dual problem?
	\end{itemize}
	The second aspect arises in model problems with non-differentiable function
	${\phi\colon\hspace*{-0.15em} \Omega\hspace*{-0.15em}\times\hspace*{-0.15em}\mathbb{R}^d\hspace*{-0.15em}\to\hspace*{-0.15em} \mathbb{R}\hspace*{-0.15em}\cup\hspace*{-0.15em}\{+\infty\}}$ and some remedies have been proposed (\textit{cf}. \cite{CP20,Bar20,BK23ROF}). The first aspect primarily relates
	to the lack of local efficiency estimates. Due to the global equivalence of 
	primal-dual gap estimators to residual type estimators, for which convergence~theories~are~available,~\mbox{convergence}~is~expected. The verification of this equivalence
	follows closely the derivation of estimators of that~type.~In~fact, the \hspace*{-0.1mm}discrete \hspace*{-0.1mm}primal \hspace*{-0.1mm}solution \hspace*{-0.1mm}$u_h^{cr}\in \mathcal{S}^{1,cr}_D(\mathcal{T}_h)$ \hspace*{-0.1mm}acts \hspace*{-0.1mm}as \hspace*{-0.1mm}a \hspace*{-0.1mm}substitute~\hspace*{-0.1mm}of~\hspace*{-0.1mm}the~\hspace*{-0.1mm}\mbox{primal}~\hspace*{-0.1mm}solution~\hspace*{-0.1mm}${u\hspace*{-0.1em}\in\hspace*{-0.1em} W^{1,p}_D(\Omega)}$. In 
	\hspace*{-0.1mm}the \hspace*{-0.1mm}case \hspace*{-0.1mm}of \hspace*{-0.1mm}the \hspace*{-0.1mm}Poisson~\hspace*{-0.1mm}problem~\hspace*{-0.1mm}\eqref{intro:poisson},~\hspace*{-0.1mm}the~\hspace*{-0.1mm}\mbox{conforming}~\hspace*{-0.1mm}$P1$-approximation~\hspace*{-0.1mm}${\smash{u_h^{p1}}\hspace*{-0.175em}\in\hspace*{-0.175em} \mathcal{S}^1_D(\mathcal{T}_h)}$,~\hspace*{-0.1mm}and~\hspace*{-0.1mm}${f\hspace*{-0.175em}=\hspace*{-0.175em}f_h}$, due to  $\Pi_hz_h^{rt}=\nabla_h u_h^{cr}$ a.e.\ in $\Omega$
	and $ \nabla \smash{u_h^{p1}}-\nabla_h u_h^{cr}\perp \frac{1}{d}f_h(\textup{id}_{\mathbb{R}^d}-\Pi_h \textup{id}_{\mathbb{R}^d})$ in $L^2(\Omega;\mathbb{R}^d)$,~we~have~that\vspace*{-0.5mm}\enlargethispage{13mm}
	\begin{align*}
	\smash{\eta^2_{\textup{gap}}}(\smash{u_h^{p1}},z_h^{rt})=\tfrac{1}{2}\| \nabla \smash{u_h^{p1}}-\nabla_h u_h^{cr}\|_{2,\Omega}^2+\tfrac{1}{2d^2}\| f_h (\textup{id}_{\mathbb{R}^d}-\Pi_h \textup{id}_{\mathbb{R}^d})\|_{2,\Omega}^2\,.
	\end{align*}
	Letting $\delta_h \coloneqq \smash{u_h^{p1}} - u_h^{cr}\in \mathcal{S}^{1,cr}_D(\mathcal{T}_h)$ and $\Pi_h^{av}\delta_h\in\mathcal{S}^1_D(\mathcal{T}_h)$ be its node-averaging quasi-interpolant, so that, due
	to $\nabla_h\delta_h\perp \nabla_h(\mathcal{S}^1_D(\mathcal{T}_h))$  in $L^2(\Omega;\mathbb{R}^d)$, via element-wise integration-by-parts,~we~obtain
	\begin{align*}
	\| \nabla \smash{u_h^{p1}}-\nabla_h u_h^{cr}\|_{2,\Omega}^2&=(\nabla_h \delta_h, \nabla_h\delta_h-\nabla_h\Pi_h^{av}\delta_h)_{\Omega}
		\\&=(\jump{\nabla \smash{u_h^{p1}}\cdot n},\{\delta_h-\Pi_h^{av}\delta_h\})_{\mathcal{S}_h}+(f_h,\delta_h-\Pi_h^{av}\delta_h)_{\Omega}
		\\&\leq c_{av}\, \eta_{\textup{res},h}^2(\smash{u_h^{p1}})\,	\| \nabla \smash{u_h^{p1}}-\nabla_h u_h^{cr}\|_{2,\Omega}\,.
	\end{align*}
	The converse estimate uses typical bubble function arguments (\textit{cf}.\ \cite{Bre15,K22CR}).\vspace*{-0.5mm}
	
	\subsection{Recent related results and open problems}\vspace*{-0.5mm}
	
	\hspace*{5mm}The concepts  described above %for a priori and a posteriori error estimation via (discrete) convex duality relations
	can be generalized to other pairs of finite element methods:\vspace*{-0.5mm}
	% Crucial is the availability of an appropriate variant of an integration-by-parts formula that is compatible with the discretizations of the primal and dual problems. 
	 
	 \begin{itemize}[noitemsep,topsep=2pt,leftmargin=!]
	 	\item In \cite{Bar20}, a discrete convex duality theory for a first-order Discontinuous~Galerkin~(DG)~method was derived. 
	 	More precisely, in \cite{Bar20}, a discrete primal problem is given via the minimization of the functional $I_h^{dg}\colon \mathcal{L}^1(\mathcal{T}_h)\to \mathbb{R}\cup\{+\infty\}$, for every $v_h\in  \mathcal{L}^1(\mathcal{T}_h)$ defined by\vspace*{-1mm}
	 	\begin{align*}
	 		I_h^{dg}(v_h)&\coloneqq \int_{\Omega}{\phi_h(\cdot,\nabla_hv_h)\,\mathrm{d}x} + \int_{\Omega}{\psi_h(\cdot,\Pi_hv_h)\,\mathrm{d}x} 
	 		\\&\quad +\sum_{S\in \mathcal{S}_h}{\tfrac{\alpha_S}{2}\|\jump{v_h}_S(x_S)\|^2_{2,S}}+\sum_{S\in \mathcal{S}_h}{\tfrac{\beta_S}{2}\| \{ v_h\}_S(x_S)\|_{2,S}^2}\,,
	 	\end{align*}
	 	and a corresponding (Fenchel) dual problem is given via the maximization of the functional $D_h^{dg}\colon \hspace*{-0.15em}\mathcal{R}T^{0,dg}(\mathcal{T}_h)\hspace*{-0.15em}\coloneqq\hspace*{-0.15em} (\mathcal{L}^0(\mathcal{T}_h))^d+(\textup{id}_{\mathbb{R}^d}-\Pi_h\textup{id}_{\mathbb{R}^d})\mathcal{L}^0(\mathcal{T}_h)\hspace*{-0.15em}\to\hspace*{-0.15em} \mathbb{R}\cup\{-\infty\}$,~for~every~${y_h\hspace*{-0.15em}\in\hspace*{-0.15em}  \mathcal{R}T^{0,dg}(\mathcal{T}_h)}$ defined by\vspace*{-1.5mm}
	 	\begin{align*}
	 		D_h^{dg}(y_h)&\coloneqq -\int_{\Omega}{\phi_h^*(\cdot,\Pi_hy_h)\,\mathrm{d}x} - \int_{\Omega}{\psi_h^*(\cdot,\textup{div}\,y_h)\,\mathrm{d}x} 
	 		\\&\quad -\sum_{S\in \mathcal{S}_h}{\tfrac{1}{2\alpha_S}\| \jump{y_h\cdot n}_S\|^2_{2,S}}-\sum_{S\in \mathcal{S}_h}{\tfrac{1}{2\beta_S}\| \{y_h\cdot n\}_S\|^2_{2,S}}\,.
	 	\end{align*} 
	 	A discrete strong duality applies, \textit{i.e.}, $I_h^{dg}(u_h^{dg})=D_h^{dg}(z_h^{dg})$ for some $u_h^{dg}\in \mathcal{L}^1(\mathcal{T}_h)$ and $z_h^{dg}\in \mathcal{R}T^{0,dg}(\mathcal{T}_h)$,  provided that the parameters $\alpha_S,\beta_S>0$, $S\in \mathcal{S}_h$, are appropriately~chosen;
	 	\item In \cite{Tran23}, a discrete convex duality theory for a Hybrid High-Order  (HHO)~method~was~derived, thus, representing the first step towards higher-order element methods.
	 \end{itemize} 
	
	\subsection{Outline of the article}\vspace*{-0.5mm}
	\hspace*{5mm}\textit{The article is organized as follows.}
	In Section  \ref{sec:preliminaries}, we introduce the employed notation and the relevant function and finite element spaces. 
	In Section \ref{sec:convex_min}, we propose a general approach for explicit a posteriori error representation based on convex duality relations.
	In Section \ref{sec:model_problems}, we apply   the general concepts of Section~\ref{sec:convex_min} to typical model problems including the non-linear Dirichlet~problem, the obstacle problem, the Signorini problem, the Rudin--Osher--Fatemi image de-noising problem, 
	a minimization problem  jumping coefficients, %the Navier--Lam\'e problem,~
	and~the~Stokes problem. In Section \ref{sec:equiv_residuals}, in the case of the non-linear Dirichlet~problem, we establish the global equivalence of the primal-dual gap estimator to a residual type estimator.
	In~Section~\ref{sec:node_avg_optimal}, we establish that the node-averaging quasi-interpolation operator locally
	preserves approximation capabilities. In Section~\ref{sec:experiments},  we review the practical relevance of the theoretical investigations of Section \ref{sec:model_problems}.

	\newpage
	\section{Preliminaries}\label{sec:preliminaries}
	
	 \subsection{Convex analysis}
	
	\hspace{5mm}For a (real) Banach space $X$, which is equipped with the norm $\|\cdot\|_X\colon X\to \mathbb{R}_{\ge 0}$, we denote its corresponding (continuous) dual space by $X^*$ equipped with the dual norm 
	$\|\cdot\|_{X^*}\colon X^*\to \mathbb{R}_{\ge 0}$, defined by $\|x^*\|_{X^*}\coloneqq \sup_{\|x\|_X\leq 1}{\langle x^*,x\rangle_X}$ for every $x^*\in X^*$, where $\langle \cdot,\cdot\rangle_X\colon X^*\times X\to \mathbb{R}$, defined by $\langle x^*,x\rangle_X\coloneqq x^*(x)$ for every $x^*\in X^*$ and $x\in X$, denotes the duality pairing.
	A functional $F\colon X\to \mathbb{R}\cup\{+\infty\}$ is called \textit{sub-differentiable} in $x\in  X$, if $ F(x)<\infty $ and if there exists $x^*\in  X^*$, called  \textit{sub-gradient}, such that for every $ y\in X $, it holds that
	\begin{align}
		\langle x^*,y-x\rangle_X\leq F(y)-F(x)\,.\label{eq:subgrad}
	\end{align} 
	The  \textit{sub-differential}  $\partial F\colon X\to  2^{X^*}$ of a functional $F\colon X\to \mathbb{R}\cup\{+\infty\}$ for every $ x\in X$~is~defined~by $\partial F(x)\coloneqq \{x^*\in X^*\mid \eqref{eq:subgrad}\text{ holds for }x^*\}$ if $F(x)<\infty$ and $\partial F(x)\coloneqq \emptyset$ else. 
	
	For a given functional $F\colon X\to \mathbb{R}\cup\{\pm\infty\}$, we denote its corresponding \textit{(Fenchel)~conjugate}~by $F^*\colon X^*\to \mathbb{R}\cup\{\pm\infty\}$, which for every $x^*\in X^*$ is defined by 
	\begin{align}
		F^*(x^*)\coloneqq \sup_{x\in X}{\langle x^*,x\rangle_X-F(x)}\,.\label{def:fenchel}
	\end{align}
	If $F\colon X\to \mathbb{R}\cup\{+\infty\}$ is a proper, convex, and lower semi-continuous functional, then also~its~(Fen-chel) conjugate $F^*\colon X^*\to\mathbb{R}\cup\{+\infty\}$ is a proper, convex, and lower semi-continuous~functional (\textit{cf}.\  \cite[p.\ 17]{ET99}). 
	Furthermore, for every $x^*\in X^*$ and $x\in X$ such that 
	$ F^*(x^*)+F(x)$~is~well-defined, \textit{i.e.},  critical cancellations $\infty-\infty$ do not occur, the \textit{Fenchel--Young inequality}
	\begin{align}
		\langle x^*,x\rangle_X\leq F^*(x^*)+F(x)\label{eq:fenchel_young_ineq}
	\end{align}
	applies. 
	In particular, 
	for every $x^*\in X^*$ and $x\in X$, it holds the \textit{Fenchel--Young identity}
	\begin{align}
		x^*\in \partial F(x)\quad\Leftrightarrow \quad	\langle x^*,x\rangle_X= F^*(x^*)+F(x)\,.\label{eq:fenchel_young_id}
	\end{align}
	\hspace{5mm}The following convexity measures for functionals play an important role in the derivation~of an explicit a posteriori error representation for convex minimization problems in Section \ref{sec:convex_min}; for further information, we refer the reader to  \cite{bregman67,NSV00,OBGXY05,bartels15}.
	
	\begin{definition}[Br\`egman distance and symmetric Br\`egman distance]\label{def:convexity_measure}
		Let $X$ be a (real) Banach space and $F\colon X\to \mathbb{R}\cup\{+\infty\}$ proper, \textit{i.e.}, $\textup{dom}(F)\coloneqq \{x\in X\mid F(x)<+\infty\}\neq \emptyset$.
		\begin{itemize}[noitemsep,topsep=2pt,leftmargin=!,labelwidth=\widthof{(ii)}]
			\item[(i)] The \textup{Br\`egman distance} $\sigma^2_F\colon \hspace*{-0.15em}
			\textup{dom}(F)\times X\hspace*{-0.15em}\to\hspace*{-0.15em} [0,+\infty]$ for every $x\hspace*{-0.15em}\in\hspace*{-0.15em} \textup{dom}(F)$~and~$y\hspace*{-0.15em}\in\hspace*{-0.15em} X$~is~\mbox{defined}~by
			\begin{align*}
				\sigma^2_F(y,x)\coloneqq F(y)-F(x)-\sup_{x^*\in \partial F(x)}{\langle x^*,y-x\rangle_X}\,,
			\end{align*}
			where we use the convention $\sup(\emptyset)\coloneqq-\infty$.
			\item[(ii)] The \hspace*{-0.15mm}\textup{symmetric \hspace*{-0.15mm}Br\`egman \hspace*{-0.15mm}distance} \hspace*{-0.15mm}$\sigma^2_{F,s}\colon \hspace*{-0.15em}
			\textup{dom}(F)^2\hspace*{-0.15em}\to\hspace*{-0.15em}  [0,+\infty]$ \hspace*{-0.15mm}for \hspace*{-0.15mm}every \hspace*{-0.15mm}$x,y\hspace*{-0.15em}\in \hspace*{-0.15em}\textup{dom}(F)$~\hspace*{-0.15mm}is~\hspace*{-0.15mm}\mbox{defined}~\hspace*{-0.15mm}by
			\begin{align*}
				\sigma_{F,s}^2(y,x)\coloneqq\sigma_F^2(y,x)+\sigma_F^2(x,y)=\inf_{x^*\in \partial F(x);y^*\in (\partial F)(y)}{\langle x^*-y^*,x-y\rangle_X}\,,
			\end{align*}
			where we use the convention $\inf(\emptyset)\coloneqq +\infty$.
		\end{itemize}
	\end{definition}
	
	\begin{definition}[Optimal convexity measure at a minimizer]\label{def:convexity_measure_optimal}
		Let $X$ be a (real) Banach~space~and $F\colon X\to \mathbb{R}\cup\{+\infty\}$ proper. Moreover, let $x\in X$ be minimal for $F\colon X\to \mathbb{R}\cup\{+\infty\}$.~Then,~the \textup{optimal convexity measure} $\rho^2_F\colon 
		X^2\to [0,+\infty]$ \textup{at} $x\in X$ for every $y\in X$~is~defined~by
		\begin{align*}
			\rho^2_F(y,x)\coloneqq F(y)-F(x)\ge 0 \,.
		\end{align*}
	\end{definition}
	
	\begin{remark}\label{rem:convexity_measure_optimal}
		Let $X$ be a (real) Banach space and $F\colon X\to \mathbb{R}\cup\{+\infty\}$ proper.~Moreover,~let $x\in X$ be minimal for $F\colon X\to \mathbb{R}\cup\{+\infty\}$. Then, due to $0\in \partial F(x)$, for every $y\in \textup{dom}(F)$,~it~holds~that
		\begin{align*}
			\sigma^2_F(y,x)\le  \rho^2_F(y,x)\,.
		\end{align*}
	\end{remark}
	
	  \subsection{Function spaces}\enlargethispage{11mm}
	
	\hspace{5mm}Throughout the article, unless otherwise specified, we denote by ${\Omega \subseteq \mathbb{R}^d}$, $d \in \mathbb{N}$, a bounded simplicial Lipschitz domain, whose (topological) boundary is disjointly divided into a  Dirichlet part $\Gamma_D$ and a Neumann part $\Gamma_N$, \textit{i.e.}, $\Gamma_D,\Gamma_N\subseteq \partial\Omega$ and 
	$\Gamma_D\dot{\cup}\Gamma_N=\partial\Omega $. We assume that either $\vert\Gamma_D\vert  >0$ or $\Gamma_D=\emptyset $ as well as $\vert\Gamma_N\vert  >0$ or $\Gamma_N=\emptyset $.

	For  $\ell\in \mathbb{N}$ and $p\in [1,+\infty)$, we employ the notation
	\begin{align*}
		\begin{aligned}
			U_{\ell}^p(\Omega)&\coloneqq  \big\{v\in L^p(\Omega;\mathbb{R}^{\ell}) \mid \nabla v\in L^p(\Omega;\mathbb{R}^{\ell\times d})\big\}\,,\\
			Z_{\ell}^p(\Omega)&\coloneqq  \big\{y\in L^{p'}(\Omega;\mathbb{R}^{\ell\times d}) \mid \textup{div}\,y\in L^{p'}(\Omega;\mathbb{R}^{\ell})\big\}\,,
		\end{aligned}
	\end{align*} 
	where the divergence needs to be understood row-wise , \textit{i.e.}, if $y=(y_{ij})_{i\in\{1,\ldots,\ell\},j\in\{1,\ldots,d\}} \in Z_{\ell}^p(\Omega)$, then 
	$(\textup{div}\,y)_i\coloneqq \smash{\sum_{j=1}^d{\partial_j y_{ij}}}$ for all $i=1,\ldots,\ell$.
	In the special case $\ell=1$,
	we employ the standard notation  $L^p(\Omega) \coloneqq L^p(\Omega;\mathbb{R}^1)$, $W^{1,p}(\Omega)\coloneqq U^p_1(\Omega)$, and $W^{p'}(\textup{div};\Omega)\coloneqq Z^p_1(\Omega)$. 
	
	For $\ell\in \mathbb{N}$, a (Lebesgue) measurable set $M\subseteq \mathbb{R}^d$, $d\in \mathbb{N}$, and   (Lebesgue) measurable functions, vector or tensor fields $u,v\colon M\to \mathbb{R}^{\ell}$, we employ the inner product
	\begin{align*}
		(u,v)_{M}\coloneqq \int_{M}{u\odot v\,\mathrm{d}x}\,,
	\end{align*}
	whenever the right-hand side is well-defined, where $\odot\colon\mathbb{R}^{\ell}\times \mathbb{R}^{\ell}\to \mathbb{R}$ either denotes scalar multiplication, the Euclidean inner product, or the Frobenius inner product.
	 For $\ell\in \mathbb{N}$,~${p\in [1,+\infty]}$, and a (Lebesgue) measurable set $M\subseteq \mathbb{R}^n$, $n\in \mathbb{N}$, we employ the notation $\|\cdot\|_{p,M}\coloneqq \|\cdot\|_{L^p(M;\mathbb{R}^{\ell})}$.
	
	Denote by $\textup{tr}(\cdot)\colon \hspace*{-0.1em}U_{\ell}^p(\Omega)\hspace*{-0.1em}\to\hspace*{-0.1em} W^{1-\smash{\frac{1}{p}},p}(\partial\Omega;\mathbb{R}^{\ell})$ the trace  and by $\textup{tr}(\cdot)n\colon\hspace*{-0.1em} Z_{\ell}^p(\Omega)\hspace*{-0.1em}\to\hspace*{-0.1em} W^{-\smash{\frac{1}{p'}},p'}(\partial\Omega;\mathbb{R}^{\ell})$\footnote{Here, $W^{-\smash{\frac{1}{p'}},p'}(\gamma;\mathbb{R}^{\ell})\coloneqq  (W^{\smash{1-\smash{\frac{1}{p}},p}}(\gamma;\mathbb{R}^{\ell}))^*$ for all $\gamma\in \{\Gamma_N,\partial\Omega\}$.} the normal trace operator. Then,  for every $v\in U_{\ell}^p(\Omega)$ and $y\in Z_{\ell}^p(\Omega)$, it holds that%the integration-by-parts formula 
	\begin{align}\label{eq:pi_cont}
		(\nabla v,y)_{\Omega}+(v,\textup{div}\,y )_{\Omega}=\langle \textup{tr}(y)n,\textup{tr}(v)\rangle_{\partial\Omega}\,,
	\end{align}
	where we abbreviate $\langle \textup{tr}(y)n,\textup{tr}(v)\rangle_{\gamma}\coloneqq \langle \textup{tr}(y)n,\textup{tr}(v)\rangle_{W^{\smash{1-\smash{\frac{1}{p}},p}}(\gamma;\mathbb{R}^{\ell})}$ for all $y\in W^{-\smash{\frac{1}{p'}},p'}(\gamma;\mathbb{R}^{\ell})$, $v\in W^{\smash{1-\frac{1}{p},p}}(\gamma;\mathbb{R}^{\ell})$, and $\gamma\in \{\Gamma_N,\partial\Omega\}$.
	Then, for $\ell\in \mathbb{N}$ and $p\in [1,+\infty]$, we employ~the~notation
	\begin{align*}
		\begin{aligned}
			U_{\ell,D}^p(\Omega)&\coloneqq  \big\{v\in U_{\ell}^p(\Omega) \mid \textup{tr}(v)=0\textup{ a.e.\ on }\Gamma_D\big\}\,,\\
			Z_{\ell,N}^p(\Omega)&\coloneqq  \big\{y\in Z_{\ell}^p(\Omega)\mid \langle\textup{tr}(y)n,v\rangle_{\partial\Omega} =0\text{ for all }\smash{v\in U_{\ell,D}^p(\Omega)}\big\}\,.
		\end{aligned}
	\end{align*}
	In what follows, we omit writing both $\textup{tr}(\cdot)$ and $\textup{tr}(\cdot)n$ in this context. For $\ell\in \mathbb{N}$ and $p\in [1,+\infty)$, 
	we employ the notation $U_{\ell,0}^p(\Omega)\coloneqq U_{\ell,D}^p(\Omega)$ if $\Gamma_D=\partial\Omega$ as well as  $Z_{\ell,0}^p(\Omega)\coloneqq Z_{\ell,N}^p(\Omega)$~if~$\Gamma_N=\partial \Omega$.
	In the special case $\ell\hspace*{-0.1em}=\hspace*{-0.1em}1$,
	we employ the standard~notation~${\smash{W^{1,p}_D(\Omega)}\hspace*{-0.1em}\coloneqq\hspace*{-0.1em} U_{1,D}^p(\Omega)}$,~${\smash{W^{1,p}_0(\Omega)}\hspace*{-0.1em}\coloneqq\hspace*{-0.1em} U_{1,0}^p(\Omega)}$, ${W^{p'}_N(\textup{div};\Omega)\coloneqq Z_{1,N}^p(\Omega)}$, and  ${W^{p'}_0(\textup{div};\Omega)\coloneqq Z_{1,0}^p(\Omega)}$.

	\subsection{Triangulations}
	
	\hspace{5mm}Throughout the entire paper, we denote by $\{\mathcal{T}_h\}_{h>0}$, a family of  
	triangulations of $\Omega\hspace*{-0.1em}\subseteq\hspace*{-0.1em} \mathbb{R}^d$,~${d\hspace*{-0.1em}\in\hspace*{-0.1em} \mathbb{N}}$.
	Here,~${h>0}$~refers to the \textit{average mesh-size}, \textit{i.e.}, we set $h\coloneqq (\vert\Omega\vert/\textup{card}(\mathcal{N}_h))^{1/d}$. Moreover, 
	 we set $h_T\coloneqq  \textup{diam}(T)$ for all $T\in \mathcal{T}_h$. 
	For every element $T \in \mathcal{T}_h$,
	we denote by $\rho_T>0$, the supremum of diameters of~inscribed~balls. We assume that there exists a constant $\omega_0>0$, independent of $h>0$, such that $\max_{T\in \mathcal{T}_h}{h_T}{\rho_T^{-1}}\le
	\omega_0$. The smallest such constant is called the \textit{chunkiness}~of~$\{\mathcal{T}_h\}_{h>0}$. We~define  
	\begin{align*}
		\mathcal{S}_h&\coloneqq \mathcal{S}_h^{i}\cup \mathcal{S}_h^{\partial\Omega}\,,\\
		\mathcal{S}_h^{i}&\coloneqq  \{T\cap T'\mid T,T'\in\mathcal{T}_h\,,\text{dim}_{\mathscr{H}}(T\cap T')=d-1\}\,,\\ 
		\mathcal{S}_h^{\partial\Omega}&\coloneqq\{T\cap \partial\Omega\mid T\in \mathcal{T}_h\,,\text{dim}_{\mathscr{H}}(T\cap \partial\Omega)=d-1\}\,,\\
		\mathcal{S}_h^\gamma&\coloneqq\{S\in \mathcal{S}_h\mid \textup{int}(S)\subseteq \gamma\}\text{ for } \gamma\in \{\Gamma_D,\Gamma_N\}\,,
	\end{align*}
	where for every  $M\subseteq \mathbb{R}^d$, we denote by $\text{dim}_{\mathscr{H}}(M)\coloneqq\inf\{d'\geq 0\mid \mathscr{H}^{d'}(M)=0\}$, the Hausdorff dimension. 
	The set $\mathcal{N}_h$  contains the vertices of $\mathcal{T}_h$.

	For $k\in \mathbb{N}\cup\{0\}$ and $T\in \mathcal{T}_h$, let $\mathbb{P}^k(T)$ denote the set of polynomials of maximal~degree~$k$~on~$T$. Then, for $k\in \mathbb{N}\cup\{0\}$, the sets of element-wise polynomial functions and  continuous element-wise polynomial functions, respectively, are defined by
	\begin{align*}
		\mathcal{L}^k(\mathcal{T}_h)&\coloneqq  \big\{v_h\in L^\infty(\Omega)\mid v_h|_T\in\mathbb{P}^k(T)\text{ for all }T\in \mathcal{T}_h\big\}\,,\\
		\mathcal{S}^k(\mathcal{T}_h)&\coloneqq  	\mathcal{L}^k(\mathcal{T}_h)\cap C^0(\overline{\Omega})\,.
	\end{align*}
	In addition, we set $\mathcal{S}^k_D(\mathcal{T}_h)\coloneqq \mathcal{S}^k(\mathcal{T}_h)\cap W^{1,p}_D(\Omega)$ and $\mathcal{S}^k_0(\mathcal{T}_h)\coloneqq \mathcal{S}^k(\mathcal{T}_h)\cap W^{1,p}_0(\Omega)$. 
	The (local) $L^2$-projection $\Pi_h\colon \hspace*{-0.1em}L^1(\Omega;\mathbb{R}^{\ell})\hspace*{-0.1em}\to\hspace*{-0.1em} (\mathcal{L}^0(\mathcal{T}_h))^{\ell}$ onto element-wise constant functions, vector~or~tensor~fields, respectively, for every 
	$v\in L^1(\Omega;\mathbb{R}^{\ell}) $ is defined by 
	\begin{align*}
		\Pi_h v|_T\coloneqq \fint_T{v\,\mathrm{d}x}\quad\text{ for all }T\in \mathcal{T}_h\,.
	\end{align*}
	The \hspace*{-0.1mm}element-wise \hspace*{-0.1mm}gradient 
	\hspace*{-0.1mm}$\nabla_h \colon \hspace*{-0.15em}(\mathcal{L}^1(\mathcal{T}_h))^{\ell}\hspace*{-0.15em}\to\hspace*{-0.15em} (\mathcal{L}^0(\mathcal{T}_h))^{\ell\times d}$ \hspace*{-0.1mm}is \hspace*{-0.1mm}given \hspace*{-0.1mm}via \hspace*{-0.1mm}the \hspace*{-0.1mm}element-wise~\hspace*{-0.1mm}\mbox{application}~\hspace*{-0.1mm}of the gradient operator, \textit{i.e.}, for every $v_h\in  (\mathcal{L}^1(\mathcal{T}_h))^{\ell}$, we~have~that~${\nabla_h v_h|_T\hspace*{-0.1em}\coloneqq \hspace*{-0.1em} \nabla(v_h|_T)}$~for~all~${T\hspace*{-0.1em}\in\hspace*{-0.1em} \mathcal{T}_h}$.
	
	Moreover, for $m\in \mathbb{N}\cup\{0\}$ and $S\in \mathcal{S}_h$, let $\mathbb{P}^m(S)$ denote the set of polynomials of maximal degree $m$ on $S$. Then, for $m\in \mathbb{N}\cup\{0\}$ and $\mathcal{M}_h\in \{\mathcal{S}_h,\mathcal{S}^{i}_h,\mathcal{S}^{\partial\Omega}_h,\mathcal{S}^{\Gamma_D}_h,\mathcal{S}^{\Gamma_N}_h\}$, the set of side-wise polynomial functions  is defined by
	\begin{align*}
		\mathcal{L}^m(\mathcal{M}_h)\coloneqq  \big\{v_h\in L^\infty(\cup\mathcal{M}_h)\mid v_h|_S\in\mathbb{P}^m(S)\text{ for all }S\in \mathcal{M}_h\big\}\,.
	\end{align*} 
	The (local) $L^2$-projection $\pi_h\colon L^1(\cup\mathcal{S}_h;\mathbb{R}^{\ell})\to (\mathcal{L}^0(\mathcal{S}_h))^{\ell}$ onto side-wise constant functions, vector, or tensor fields, respectively,  for every 
	$v\in L^1(\cup\mathcal{S}_h;\mathbb{R}^{\ell}) $ is defined by 
	\begin{align*}
		\pi_h v|_S\coloneqq \fint_S{v\,\mathrm{d}s}\quad\text{ for all }S\in \mathcal{S}_h\,.
	\end{align*}

	\subsubsection{Crouzeix--Raviart element}\enlargethispage{13mm}\vspace*{-0.5mm}
	
	\qquad The \textit{Crouzeix--Raviart finite element space} (\textit{cf}.\ \cite{CR73}) is defined as the space of element-wise affine functions that are continuous in the barycenters of inner element sides, \textit{i.e.},\footnote{Here, for every inner side $S\in\mathcal{S}_h^{i}$, the jump is defined by $\jump{v_h}_S\coloneqq v_h|_{T_+}-v_h|_{T_-}$ on $S$, where $T_+, T_-\in \mathcal{T}_h$ satisfy $\partial T_+\cap\partial T_-=S$, and for every boundary side $S\in\mathcal{S}_h\cap\partial \Omega$, the jump is defined by $\jump{v_h}_S\coloneqq v_h|_T$ on $S$, where $T\in \mathcal{T}_h$ satisfies $S\subseteq \partial T$.}
	\begin{align*}\mathcal{S}^{1,cr}(\mathcal{T}_h)\coloneqq \big\{v_h\in \mathcal{L}^1(\mathcal{T}_h)\mid   \pi_h \jump{v_h}_S=0\text{ in }S\text{ for all }S\in \mathcal{S}_h^{i}\big\}\,.
	\end{align*}
	The Crouzeix--Raviart finite element space with homogeneous Dirichlet boundary condition on $\Gamma_D$  is defined as the space of 
	Crouzeix--Raviart finite element functions that vanish in the barycenters of boundary element sides that belong to $\Gamma_D$, \textit{i.e.},
	\begin{align*}
	\mathcal{S}^{1,cr}_D(\mathcal{T}_h)\coloneqq \big\{v_h\in\smash{\mathcal{S}^{1,cr}(\mathcal{T}_h)}\mid  \pi_h \jump{v_h}_S=0\text{ in }S\text{ for all }S\in \mathcal{S}_h^{\Gamma_D}\big\}\,.
	\end{align*} 
	We employ the notation $\mathcal{S}^{1,cr}_0(\mathcal{T}_h)=\mathcal{S}^{1,cr}_D(\mathcal{T}_h)$ if $\Gamma_D=\partial\Omega$. 
	The functions $\varphi_S\in \smash{\mathcal{S}^{1,cr}(\mathcal{T}_h)}$,~${S\in \mathcal{S}_h}$, that satisfy the Kronecker property $\varphi_S(x_{S'})=\delta_{S,S'}$ for all $S,S'\in \mathcal{S}_h$, form a basis of $\smash{\mathcal{S}^{1,cr}(\mathcal{T}_h)}$. Then, 
	the functions 	 $\varphi_S\in \smash{\mathcal{S}^{1,cr}_D(\mathcal{T}_h)}$, $S\in \mathcal{S}_h\setminus\mathcal{S}_h^{\Gamma_D}$, form a basis of $\smash{\smash{\mathcal{S}^{1,cr}_D(\mathcal{T}_h)}}$.  
	For $\ell\in \mathbb{N}$,~we~employ the notation
	\begin{align*}
		\begin{aligned}
			U^{cr}_{\ell}(\mathcal{T}_h)\coloneqq (\mathcal{S}^{1,cr}(\mathcal{T}_h))^{\ell}\,,\qquad
			 U^{cr}_{\ell,D}(\mathcal{T}_h)\coloneqq  (	\mathcal{S}^{1,cr}_D(\mathcal{T}_h))^{\ell}\,,\qquad U^{cr}_{\ell,0}(\mathcal{T}_h)\coloneqq  (	\mathcal{S}^{1,cr}_0(\mathcal{T}_h))^{\ell}\,.
		\end{aligned}
	\end{align*} 
	The (Fortin) quasi-interpolation operator $\Pi_h^{cr}\colon \hspace*{-0.15em}U_{\ell}^p(\Omega)\hspace*{-0.15em}\to\hspace*{-0.15em} U^{cr}_{\ell}(\mathcal{T}_h)$, for every $v\hspace*{-0.15em}\in\hspace*{-0.15em} U_{\ell}^p(\Omega)$~is~\mbox{defined}~by
	\begin{align}
		\Pi_h^{cr}v\coloneqq \sum_{S\in \mathcal{S}_h}{v_S\,\varphi_S}\,,\quad\text{ where } v_S\coloneqq \fint_S{v\,\textup{d}s}\text{ for all }S\in \mathcal{S}_h\,,\label{CR-interpolant}
	\end{align}
	preserves averages of gradients and moments (on sides), \textit{i.e.}, for every $v\in U_{\ell}^p(\Omega)$, it holds that
	\begin{alignat}{2}
		\nabla_h\Pi_h^{cr}v&=\Pi_h\nabla v&&\quad \text{ in  }(\mathcal{L}^0(\mathcal{T}_h))^{\ell\times d}\,,\label{eq:grad_preservation}\\
		\pi_h\Pi_h^{cr}v&=\pi_h v&&\quad \text{ in }(\mathcal{L}^0(\mathcal{S}_h))^{\ell}\,.\label{eq:trace_preservation}
	\end{alignat}
	In particular, from \eqref{eq:trace_preservation}, it follows that $\Pi_h^{cr}(U^p_{\ell,D}(\Omega))\subseteq U^{cr}_{\ell,D}(\mathcal{T}_h)$ and $\Pi_h^{cr}(U^p_{\ell,0}(\Omega))\subseteq U^{cr}_{\ell,0}(\mathcal{T}_h)$.

	\subsubsection{Raviart--Thomas element}
	
	\qquad The \textit{(lowest order) Raviart--Thomas finite element space} (\textit{cf}.\ \cite{RT75}) is defined as the space of element-wise affine vector fields that have continuous constant normal components on inner elements sides, \textit{i.e.},\footnote{For every inner side $S\hspace{-0.1em}\in\hspace{-0.1em}\mathcal{S}_h^{i}$, the normal jump is defined by $\jump{y_h n}_S\hspace{-0.1em}\coloneqq \hspace{-0.1em}\smash{y_h|_{T_+} n_{T_+}+y_h|_{T_-} n_{T_-}}$ on $S$, where $T_+, T_-\hspace{-0.1em}\in\hspace{-0.1em} \mathcal{T}_h$~satisfy~$\smash{\partial T_+\cap\partial T_-\hspace{-0.1em}=\hspace{-0.1em}S}$, and for every $T\in \mathcal{T}_h$, $\smash{n_T\colon\partial T\to \mathbb{S}^{d-1}}$ denotes the outward unit normal vector field~to~$ T$, 
		and for every boundary side $\smash{S\in\mathcal{S}_h\cap\partial \Omega}$, the normal jump is defined by $\smash{\jump{y_h n}_S\coloneqq \smash{y_h|_T n}}$ on $S$, where $T\in \mathcal{T}_h$ satisfies $S\subseteq \partial T$.}
	\begin{align*}
		\mathcal{R}T^0(\mathcal{T}_h)\coloneqq \big\{y_h\in (\mathcal{L}^1(\mathcal{T}_h))^d\mid &\,\smash{y_h|_T\cdot n_T=\textup{const}\text{ on }\partial T\text{ for all }T\in \mathcal{T}_h\,,}\\ 
		&\smash{	\jump{y_h\cdot n}_S=0\text{ on }S\text{ for all }S\in \mathcal{S}_h^{i}\big\}\,.}
	\end{align*}
	The Raviart--Thomas finite element space with homogeneous slip boundary condition on $\Gamma_N$ is defined as the space of Raviart--Thomas vector fields whose normal components~vanish~on~$\Gamma_N$,~\textit{i.e.},
	\begin{align*}
		\mathcal{R}T^{0}_N(\mathcal{T}_h)\coloneqq \big\{y_h\in	\mathcal{R}T^0(\mathcal{T}_h)\mid y_h\cdot n=0\text{ on }\Gamma_N\big\}\,.
	\end{align*} 
	We employ the notation $\mathcal{R}T^{0}_0(\mathcal{T}_h)\hspace*{-0.1em}\coloneqq\hspace*{-0.1em}	\mathcal{R}T^{0}_N(\mathcal{T}_h)$ if  $\Gamma_N\hspace*{-0.1em}=\hspace*{-0.1em}\partial\Omega$.
	The vector fields $\psi_S\hspace*{-0.1em}\in\hspace*{-0.1em} \mathcal{R}T^0(\mathcal{T}_h)$,~${S\hspace*{-0.1em}\in\hspace*{-0.1em} \mathcal{S}_h}$, that satisfy the Kronecker property $\psi_S|_{S'}\cdot n_{S'}\hspace*{-0.1em}=\hspace*{-0.1em}\delta_{S,S'}$ on $S'$ for all $S'\hspace*{-0.1em}\in\hspace*{-0.1em} \mathcal{S}_h$, where~$n_S$~for~all~${S\hspace*{-0.1em}\in\hspace*{-0.1em} \mathcal{S}_h}$~is the unit normal vector on $S$ pointing from $T_-$ to $T_+$ if $T_+\cap T_-=S\in \mathcal{S}_h$, form~a~basis~of~$\mathcal{R}T^0(\mathcal{T}_h)$.
	Then, the vector fields $\psi_S\hspace*{-0.15em}\in\hspace*{-0.15em} \smash{\mathcal{R}T^{0}_N(\mathcal{T}_h)}$, $S\hspace*{-0.15em}\in\hspace*{-0.15em} \mathcal{S}_h\setminus\Gamma_N$, form~a~basis~of~$\mathcal{R}T^{0}_N(\mathcal{T}_h)$. For $\ell\hspace*{-0.15em}\in\hspace*{-0.15em} \mathbb{N}$,~we~\mbox{employ} the notations
	\begin{align*}
		\begin{aligned}
			Z^{rt}_{\ell}(\mathcal{T}_h)&\coloneqq  \big\{y=(y_{ij})_{i\in\{1,\ldots,\ell\},j\in\{1,\ldots,d\}}\mid (y_{ij})_{j\in\{1,\ldots,d\}}\in \mathcal{R}T^0(\mathcal{T}_h)\text{ for all }i= 1,\ldots,\ell\big\}\,,\\
			Z^{rt}_{\ell,N}(\mathcal{T}_h)&\coloneqq  \big\{y=(y_{ij})_{i\in\{1,\ldots,\ell\},j\in\{1,\ldots,d\}}\mid (y_{ij})_{j\in\{1,\ldots,d\}}\in \mathcal{R}T^0_N(\mathcal{T}_h)\text{ for all }i= 1,\ldots,\ell\big\}\,,\\
			Z^{rt}_{\ell,0}(\mathcal{T}_h)&\coloneqq  \big\{y=(y_{ij})_{i\in\{1,\ldots,\ell\},j\in\{1,\ldots,d\}}\mid (y_{ij})_{j\in\{1,\ldots,d\}}\in \mathcal{R}T^0_0(\mathcal{T}_h)\text{ for all }i= 1,\ldots,\ell\big\}\,.
		\end{aligned}
	\end{align*} 
	The (Fortin) quasi-interpolation operator $\Pi_h^{rt}\colon \hspace*{-0.15em}W^{1,1}(\Omega;\mathbb{R}^{\ell\times d})\hspace*{-0.15em}\to\hspace*{-0.15em} Z^{rt}_{\ell}(\mathcal{T}_h)$,~for~every~$y\hspace*{-0.15em}\in\hspace*{-0.15em} W^{1,1}(\Omega;\mathbb{R}^{\ell\times d})$ is defined by
	\begin{align}
		\Pi_h^{rt} y\coloneqq \sum_{S\in \mathcal{S}_h}{y_S\,\psi_S}\,,\quad\text{ where } y_S\coloneqq \fint_S{yn_S\,\textup{d}s}\text{ for all }S\in \mathcal{S}_h\,,\label{RT-interpolant}
	\end{align}
	preserves averages of divergences and normal traces, \textit{i.e.}, for every $y\in W^{1,1}(\Omega;\mathbb{R}^{\ell\times d})$,~it~holds~that
	\begin{alignat}{2}
		\textup{div}\,\Pi_h^{rt}y&=\Pi_h\textup{div}\,y&&\quad \text{ in }(\mathcal{L}^0(\mathcal{T}_h))^{\ell}\,,\label{eq:div_preservation}\\
		\Pi_h^{rt}y n&=\pi_hy n&&\quad \text{ in }(\mathcal{L}^0(\mathcal{S}_h))^{\ell}\,.\label{eq:normal_trace_preservation}
	\end{alignat}
	In particular, from \eqref{eq:normal_trace_preservation}, it follows that $\Pi_h^{rt}(Z^p_{\ell,N}(\Omega))\subseteq Z^{rt}_{\ell,N}(\mathcal{T}_h)$ and $\Pi_h^{rt}(Z^p_{\ell,0}(\Omega))\subseteq Z^{rt}_{\ell,0}(\mathcal{T}_h)$.
	
	\subsubsection{Discrete integration-by-parts formula}
	
\hspace{5mm}For every $v_h\in U^{cr}_{\ell}(\mathcal{T}_h)$ and $y_h\in Z^{rt}_{\ell}(\mathcal{T}_h)$, there holds the \textit{discrete integration-by-parts
	formula}
\begin{align}
	(\nabla_hv_h,\Pi_h y_h)_{\Omega}+(\Pi_h v_h,\,\textup{div}\,y_h)_{\Omega}=(\pi_h v_h,y_h n)_{\partial\Omega}\,.\label{eq:pi}
\end{align}
	which follows from the fact that for every $y_h\in Z^{rt}_{\ell}(\mathcal{T}_h)$, it holds that  $y_h|_Tn_T=\textrm{const}$~on~$\partial T$~for~all $T\in \mathcal{T}_h$ and	$\jump{y_h n}_S=0$ on $S$ for all $S\in \mathcal{S}_h^{i}$, and for every ${v_h\in U^{cr}_{\ell}(\mathcal{T}_h)}$, it holds that $\pi_h\jump{v_h}_S=0$ for all $S\in \mathcal{S}_h^{i}$.
In \cite{CL15,LLC18,CP20,Bar20,Bar21,BKAFEM22}, the discrete integration-by-parts formula \eqref{eq:pi} formed a cornerstone in the \mbox{derivation} of a discrete convex duality theory and, as such,  plays a central 
role in the derivation of the results presented below.
Appealing to \cite[Sec.~2.4]{BW21} and \cite[Subsec.~2.3.3]{BK23ROF}, there hold the \textit{discrete Helmholtz decompositions}
\begin{align}
	(\mathcal{L}^0(\mathcal{T}_h))^{\ell\times d}&=\textup{ker}(\textup{div}|_{\smash{Z^{rt}_{\ell,N}(\mathcal{T}_h)}})\oplus \nabla_{\!h}( U^{cr}_{\ell,D}(\mathcal{T}_h))
	\,,\label{eq:decomposition.2}\\
	(\mathcal{L}^0(\mathcal{T}_h))^{\ell}&=\textup{ker}(\nabla_{\!h}|_{\smash{ U^{cr}_{\ell,D}(\mathcal{T}_h)}})\oplus \textup{div}\,(Z^{rt}_{\ell,N}(\mathcal{T}_h)) \,.\label{eq:decomposition.1}
\end{align}

\newpage
	
	\section{Exact a posteriori error estimation for convex minimization problems} \label{sec:convex_min}

\subsection{Continuous convex duality} 

\hspace{5mm}\textit{Primal problem.} Let $\phi\colon \Omega\times \mathbb{R}^{\ell \times d}\to \mathbb{R}\cup\{+\infty\}$ and $\psi\colon \Omega\times\mathbb{R}^{\ell} \to \mathbb{R}\cup\{+\infty\}$ be measurable functions such that for a.e.\ $x\in \Omega$,~the~functions $\phi(x,\cdot)\colon \mathbb{R}^{\ell \times d}\to \mathbb{R}\cup\{+\infty\}$ and $\psi(x,\cdot)\colon \mathbb{R}^{\ell}\to \mathbb{R}\cup\{+\infty\}$ are proper, convex, and lower semi-continuous such that for every $y\in L^p(\Omega;\mathbb{R}^{\ell\times d})$ and $v\in L^p(\Omega;\mathbb{R}^{\ell})$, the following integrals exist and are finite or infinity, \textit{i.e.}, 
\begin{align*}
	\int_{\Omega}{\phi(\cdot,y)\,\textup{d}x}, \int_{\Omega}{\psi(\cdot,v)\,\textup{d}x}\in \mathbb{R}\cup\{+\infty\}\,.
\end{align*}
Moreover,  let $g\in W^{\smash{-\frac{1}{p'}},p'}(\Gamma_N;\mathbb{R}^{\ell})$ be given Neumann boundary data and  let $u_D\in W^{1-\frac{1}{p},p}(\partial\Omega)$ be given Dirichlet boundary data.
We examine the  minimization of the functional $I\colon U_{\ell}^p(\Omega)\to \mathbb{R}\cup\{+\infty\}$, for every $v\in U_{\ell}^p(\Omega)$ defined by 
\begin{align}
	I(v)\coloneqq \int_{\Omega}{\phi(\cdot,\nabla v)\,\textup{d}x}+\int_{\Omega}{\psi(\cdot,v)\,\textup{d}x}-\langle g, v\rangle_{\Gamma_N}+\smash{I_{\{u_D\}}^{\Gamma_D}}(v)\,,\label{primal}
\end{align}
where $\smash{I_{\{u_D\}}^{\Gamma_D}}\colon \smash{W^{\smash{1-\frac{1}{p}},p}(\Gamma_D;\mathbb{R}^{\ell})}\to \mathbb{R}\cup\{+\infty\}$ for every $\widehat{v}\in  \smash{W^{\smash{1-\frac{1}{p}},p}(\Gamma_D;\mathbb{R}^{\ell})}$ is defined by 
\begin{align*}
	\smash{I_{\{u_D\}}^{\Gamma_D}}(\widehat{v})\coloneqq \begin{cases}
			0&\text{ if }\widehat{v}=u_D\text{ a.e.\ on }\Gamma_D\,,\\
			+\infty&\text{ else}\,.
	\end{cases}
\end{align*}
In what follows, we refer to the minimization of $I \colon U_{\ell}^p(\Omega) \to\mathbb{R} \cup \{+\infty\}$ as the \textit{primal problem}.

\textit{Dual problem.}  Let $\phi^*\colon \Omega\times \mathbb{R}^{\ell \times d}\to \mathbb{R}\cup\{+\infty\}$ and $\psi^*\colon \Omega\times\mathbb{R}^{\ell}\to \mathbb{R}\cup\{+\infty\}$ be the Fenchel conjugates to $\phi\colon \Omega\times \mathbb{R}^{\ell\times d}\to \mathbb{R}\cup\{+\infty\}$ and $\psi\colon \Omega\times\mathbb{R}^{\ell} \to \mathbb{R}\cup\{+\infty\}$, respectively, with respect to the second argument.
% and assume that for every $y\in L^{p'}(\Omega;\mathbb{R}^{\ell})$ and $v\in L^{p'}(\Omega)$, the following integrals exist and are finite or infinity, \textit{i.e.},\vspace*{-0.5mm}
%\begin{align*}
%	 \int_{\Omega}{\phi^*(\cdot,y)\,\textup{d}x}, \int_{\Omega}{\psi^*(\cdot,v)\,\textup{d}x}\in \mathbb{R}\cup\{+\infty\}\,.
%\end{align*}
Then, a \textit{(Fenchel) dual problem} to the minimization of \eqref{primal} is given via the maximization of the functional $D\colon  Z_{\ell}^p(\Omega)\to \mathbb{R} \cup \{ -\infty \}$, for every $y\in  Z_{\ell}^p(\Omega)$ defined by 
\begin{align}
	D(y)\coloneqq -\int_{\Omega}{\phi^*( \cdot, y)\,\textup{d}x}-\int_{\Omega}{\psi^*(\cdot,\textup{div}\,y)\,\textup{d}x}+\langle y n ,u_D\rangle_{\partial\Omega}- \langle g,u_D\rangle_{\Gamma_N} -\smash{I_{\{g\}}^{\Gamma_N}}(y n)\,,\label{dual}
\end{align}
where   $\smash{I_{\{g\}}^{\Gamma_N}}\colon \smash{W^{\smash{-\frac{1}{p'}},p'}(\Gamma_N;\mathbb{R}^{\ell})}\to \mathbb{R}\cup\{+\infty\}$ for every $\widehat{y}\in\smash{W^{\smash{-\frac{1}{p'}},p'}(\Gamma_N;\mathbb{R}^{\ell})}$ is defined by 
\begin{align*}
		\smash{I_{\{g\}}^{\Gamma_N}}(\widehat{y})\coloneqq \begin{cases}
		0&\text{ if }\langle \widehat{y}, v\rangle_{\partial \Omega} = \langle g,v\rangle_{\Gamma_N}\text{ for all }v\in U_{\ell,D}^p(\Omega)\,,\\
		+\infty&\text{ else}\,.
	\end{cases}
\end{align*}
We always assume that $\phi\colon \Omega\times\mathbb{R}^{\ell \times d}\to \mathbb{R}\cup\{+\infty\}$ and $\psi\colon \Omega\times \mathbb{R}^{\ell}\to \mathbb{R}\cup\{+\infty\}$ are such that \eqref{primal} admits at least one minimizer $u\in U_{\ell}^p(\Omega)$, called  \textit{primal solution},
and that \eqref{dual} admits at least one maximizer  $z\in Z_{\ell}^p(\Omega)$, called \textit{dual solution}.
The derivation of a weak duality relation between of \eqref{primal} and \eqref{dual} can be found in the proof of the following results that also establishes the equivalence of a strong duality relation and convex optimality relations. 

\begin{proposition}[Strong duality and convex duality relations]\label{prop:duality} The following statements apply:
	\begin{itemize}[noitemsep,topsep=2pt,leftmargin=!,labelwidth=\widthof{(ii)}]
		\item[(i)]  A \textup{weak duality relation} applies, \textit{i.e.}, 
		\begin{align}
			\inf_{v\in U_{\ell}^p(\Omega)}{I(v)}
			\ge \sup_{y\in Z_{\ell}^p(\Omega)}{D(y)}\,.\label{eq:weak_duality}
		\end{align}
		\item[(ii)]  A  \textup{strong duality relation} applies, \textit{i.e.}, 
		\begin{align}
			I(u) = D(z)\,,\label{eq:strong_duality}
		\end{align}
		if and only if the \textup{convex optimality relations} apply, \textit{i.e.}, 
		\begin{alignat}{2} 
			z: \nabla u&=\phi^*(\cdot,z)+\phi(\cdot,\nabla u)&&\quad\text{ a.e.\ in }\Omega\,,\label{eq:optimality_relations.1}\\
			\textup{div}\,z\cdot u& =\psi^*(\cdot,	\textup{div}\,z)+\psi(\cdot, u)&&\quad\text{ a.e.\ in }\Omega\,.\label{eq:optimality_relations.2}
		\end{alignat}
	\end{itemize}
\end{proposition}

\begin{proof}
	\textit{ad (i).} For every $v\in U_{\ell}^p(\Omega)$ and $y\in Z_{\ell}^p(\Omega)$, by the Fenchel--Young inequality  (\textit{cf}.\  \eqref{eq:fenchel_young_ineq}), we have that 
	\begin{align}\label{prop:duality.1}
		\begin{aligned}
			\phi(\cdot,\nabla v)&\geq y: \nabla v- \phi^*(\cdot,y)&&\quad\text{ a.e.\ in }\Omega\,,\\
			\psi(\cdot,v)&\geq	\textup{div}\,y\cdot v- \psi^*(\cdot,	\textup{div}\,y)&&\quad\text{ a.e.\ in }\Omega\,.
		\end{aligned}
	\end{align}
	Adding the two inequalities in \eqref{prop:duality.1}, using the integration-by-parts formula \eqref{eq:pi_cont}, for every $v\in U_{\ell}^p(\Omega)$ with $I_{\{u_D\}}^{\Gamma_D}(v)=0$  and $y\in Z_{\ell}^p(\Omega)$ with $I_{\{g\}}^{\Gamma_N}(yn)=0$, we find that 
	\begin{align*}
		I(v)&\ge \int_{\Omega}{\phi(\cdot,\nabla v)\,\mathrm{d}x}+\int_{\Omega}{\psi(\cdot,v)\,\mathrm{d}x}-\langle g,v\rangle_{\Gamma_N}
		\\&\ge 
		-	\int_{\Omega}{\phi^*(\cdot,y)\,\mathrm{d}x}-\int_{\Omega}{\psi^*(\cdot,	\textup{div}\,y)\,\mathrm{d}x}+\langle y n,u_D\rangle_{\partial\Omega}-\langle g,u_D\rangle_{\Gamma_N}\\&
		=D(y)\,,
	\end{align*}
	where we also used that for every $v\in U_{\ell}^p(\Omega)$ with $I_{\{u_D\}}^{\Gamma_D}(v)=0$  and $y\in Z_{\ell}^p(\Omega)$ with $I_{\{g\}}^{\Gamma_N}(yn)=0$, it holds that
	\begin{align}\label{eq:boundary}
		\begin{aligned}
		\langle yn,v\rangle_{\partial\Omega}&=\langle yn,u_D\rangle_{\partial\Omega}+\langle yn,v-u_D\rangle_{\partial\Omega}
		\\&=\langle yn,u_D\rangle_{\partial\Omega}-\langle g,u_D\rangle_{\Gamma_N}+\langle g,v\rangle_{\Gamma_N}\,.
	\end{aligned}
	\end{align}
	On the other hand, for every $v\in U_{\ell}^p(\Omega)$ such that $\smash{I_{\{u_D\}}^{\Gamma_D}}(v)=+\infty$, we have that $I(v)=+\infty$, and 
	for every $y\in Z_{\ell}^p(\Omega)$  such that $\smash{I_{\{g\}}^{\Gamma_N}}(y n) =+\infty$, we have that
	$D(z)=-\infty$.  
	
	\textit{ad (ii).} The strong duality relation \eqref{eq:strong_duality} is equivalent to 
	\begin{align*}
	&	\int_{\Omega}{\big\{\phi^*(\cdot,z)-z: \nabla u+\phi(\cdot,\nabla u)\big\}\,\mathrm{d}x}\\&\quad+\int_{\Omega}{\big\{\psi^*(\cdot,	\textup{div}\,z)-\textup{div}\,z\cdot u+\psi(\cdot, u)\big\}\,\mathrm{d}x}=0\,.
	\end{align*}
	Therefore, due to \eqref{prop:duality.1}, we find that the strong duality relation \eqref{eq:strong_duality} is equivalent to the convex optimality relations \eqref{eq:optimality_relations.1}, \eqref{eq:optimality_relations.2}.
\end{proof}

\begin{remark}[Equivalent convex optimality relations ]\label{rem:duality}
		\begin{itemize}[noitemsep,topsep=2pt,leftmargin=!,labelwidth=\widthof{(iii)}]
			\item[(i)] If $\phi(x,\cdot)\hspace*{-0.15em}\in \hspace*{-0.15em} C^1(\mathbb{R}^{\ell\times d})$~for~a.e.~${x\hspace*{-0.15em}\in\hspace*{-0.15em} \Omega}$, 
			%then, 
		%	\hspace*{-0.1mm}
			by \hspace*{-0.1mm}the \hspace*{-0.1mm}Fenchel--Young \hspace*{-0.1mm}identity \hspace*{-0.1mm}(\textit{cf}.\ \hspace*{-0.1mm}\eqref{eq:fenchel_young_id}), \hspace*{-0.1mm}the \hspace*{-0.1mm}convex \hspace*{-0.1mm}optimality~\hspace*{-0.1mm}relation~\hspace*{-0.1mm}\eqref{eq:optimality_relations.1}~\hspace*{-0.1mm}is~\hspace*{-0.1mm}\mbox{equivalent}~\hspace*{-0.1mm}to 
			\begin{align}
				z= D\phi(\cdot,\nabla u)\quad\textup{ a.e.\ in }\Omega\,;\label{eq:optimality_relations.3}
			\end{align}
			\item[(ii)] If $\phi^*(x,\cdot)\in C^1(\mathbb{R}^{\ell\times d})$ for a.e.\ $x\in \Omega$, by the Fenchel--Young identity (\textit{cf}.\ \eqref{eq:fenchel_young_id}), the convex optimality relation \eqref{eq:optimality_relations.1} is equivalent to 
			\begin{align}
				\nabla u= D\phi^*(\cdot,z)\quad\textup{ a.e.\ in }\Omega\,;\label{eq:optimality_relations.4}
			\end{align}
			
			\item[(iii)] If  
			$\psi(x,\cdot)\in C^1(\mathbb{R}^{\ell})$ for a.e.\ $x\in \Omega$, 
			then, by the Fenchel--Young identity (\textit{cf}.\ \eqref{eq:fenchel_young_id}),  the convex optimality relation \eqref{eq:optimality_relations.2} is equivalent to 
			\begin{align}
				\textup{div}\,z=D\psi(\cdot, u)\quad\textup{ a.e.\ in }\Omega\,;\label{eq:optimality_relations.5}
			\end{align}
			\item[(iv)] If $\psi^*(x,\cdot)\in C^1(\mathbb{R}^{\ell})$ for a.e.\ $x\in \Omega$, 
			then, by the Fenchel--Young identity (\textit{cf}.\ \eqref{eq:fenchel_young_id}),  the convex optimality relation \eqref{eq:optimality_relations.2} is equivalent to 
			\begin{align}
				u=D\psi^*(\cdot, \textup{div}\,z)\quad\textup{ a.e.\ in }\Omega\,.\label{eq:optimality_relations.6}
			\end{align}
		\end{itemize}
\end{remark}

The convex duality relations \eqref{eq:optimality_relations.1}, \eqref{eq:optimality_relations.2} motivate introducing the \textit{primal-dual~gap~estimator} $\smash{\eta^2_{\textup{gap}}}\colon U_{\ell}^p(\Omega)\times Z_{\ell}^p(\Omega)\to [0,+\infty]$, for every 
$(v,y)^\top \in U_{\ell}^p(\Omega)\times Z_{\ell}^p(\Omega)$ defined by
\begin{align}
	\smash{\eta^2_{\textup{gap}}}(v,y)\coloneqq I(v)-D(y)\,.\label{def:eta}
\end{align}
Note that the sign of the estimator \eqref{def:eta} is a consequence of the weak duality relation \eqref{eq:weak_duality}.\enlargethispage{10mm}

Together with the optimal convexity measures (\textit{cf}.\ Definition \ref{def:convexity_measure_optimal}) $\rho_I^2\colon U_{\ell}^p(\Omega)\to [0,+\infty]$ of \eqref{primal} at a primal solution $u\in U_{\ell}^p(\Omega)$ and $\rho_{-D}^2\colon Z_{\ell}^p(\Omega)\to [0,+\infty]$ of the negative of \eqref{dual} at a dual solution $z\in Z_{\ell}^p(\Omega)$, we arrive at the following \textit{generalized Prager--Synge identity}.

\begin{theorem}[Generalized Prager--Synge identity]\label{thm:main}
	If the strong duality relation \eqref{eq:strong_duality} applies, then the following statements apply:
	\begin{itemize}[noitemsep,topsep=2pt,leftmargin=!,labelwidth=\widthof{(ii)}]
		\item[(i)] For every $v\in U_{\ell}^p(\Omega)$ and $y\in  Z_{\ell}^p(\Omega)$, for the \textup{primal-dual total error}, we  have that 
		\begin{align}\label{eq:representation.1} 
				\smash{\rho_{\textup{tot}}^2}(v,y)\coloneqq\rho^2_I(v,u)+\rho^2_{-D}(y,z)=\smash{\eta^2_{\textup{gap}}}(v,y)\,. 
		\end{align} 
		
		\item[(ii)]  For every $v\in U_{\ell}^p(\Omega)$ with $v=u_D$ a.e.\ in $\Gamma_D$ and $y\in  Z_{\ell}^p(\Omega)$ with $I_{\{g\}}^{\Gamma_N}(yn) =0$, we have that 
		\begin{align}\label{eq:representation} 
			\begin{aligned}
				\smash{\eta^2_{\textup{gap}}}(v,y)
				&=\int_{\Omega}{\big\{\phi(\cdot,\nabla v)-\nabla v:  y+\phi^*(\cdot,y)\big\}\,\mathrm{d}x}\\&\quad +\int_{\Omega}{\big\{\psi(\cdot, v)- v\cdot\mathrm{div}\,y+\psi^*(\cdot,\mathrm{div}\,y)\big\}\,\mathrm{d}x}\,.
			\end{aligned} 
		\end{align} 
	\end{itemize}
\end{theorem}

\begin{proof}\let\qed\relax 
	\textit{ad (i).} Due to  \eqref{eq:strong_duality},  Definition \ref{def:convexity_measure_optimal}, and \eqref{def:eta}, 
	for every $v\in U_{\ell}^p(\Omega)$ and $y\in Z_{\ell}^p(\Omega)$, we have that 
	\begin{align*}
	\smash{\rho_{\textup{tot}}^2}(v,y)=	\rho^2_I(v,u)+\rho^2_{-D}(y,z)=I(v)-I(u)+D(z)-D(y)=\smash{\eta^2_{\textup{gap}}}(v,y)\,.
	\end{align*}
	
	 \textit{ad (ii).} The identitiy \eqref{eq:representation} follows from  \eqref{primal}, \eqref{dual}, and the integration-by-parts~\mbox{formula}~\eqref{eq:pi_cont} together with \eqref{eq:boundary}
%	 and that 
%	 for every $v\hspace*{-0.1em}\in\hspace*{-0.1em} U_{\ell}^p(\Omega)$ with $v\hspace*{-0.1em}=\hspace*{-0.1em}u_D$ a.e.\ in $\Gamma_D$ and $y\hspace*{-0.1em}\in\hspace*{-0.1em}  Z_{\ell}^p(\Omega)$~with~${I_{\{g\}}^{\Gamma_N}(yn) =0}$,~it~holds~that
%	 \begin{align*}
%	 	\langle yn,v\rangle_{\partial\Omega}&=\langle yn,u_D\rangle_{\partial\Omega}+\langle yn,v-u_D\rangle_{\partial\Omega}
%	 	\\&=\langle yn,u_D\rangle_{\Gamma_D}-\langle g,u_D\rangle_{\Gamma_N}+\langle g,v\rangle_{\Gamma_N}\,.\tag*{$\qedsymbol$}
%	 \end{align*}
\end{proof}

\begin{remark}
	\begin{itemize}[noitemsep,topsep=2pt,leftmargin=!,labelwidth=\widthof{(ii)}]
		\item[(i)] By the Fenchel--Young inequality (\textit{cf}.\ \eqref{eq:fenchel_young_ineq}), the integrands in the representation \eqref{eq:representation} are non-negative, \textit{i.e.}, for every $v\in U_{\ell}^p(\Omega)$ and $y\in  Z_{\ell}^p(\Omega)$, we~have~that
		\begin{align*}
			\begin{aligned}
		\phi(\cdot,\nabla v)-\nabla v:  y+\phi^*(\cdot,y)&\ge 0&&\quad\text{ a.e.\ in }\Omega\,,\\
		\psi(\cdot, v)- v\cdot\mathrm{div}\,y+\psi^*(\cdot,\mathrm{div}\,y)&\ge 0&&\quad \text{ a.e.\ in }\Omega\,.
	\end{aligned}
		\end{align*}
		 and, thus, are suitable as local refinement indicators in an adaptive mesh refinement~procedure. Apart from that, due to Proposition \ref{prop:duality}(ii), for every $v\in U_{\ell}^p(\Omega)$ with $v=u_D$ a.e.\ in $\Gamma_D$ and $y\in  Z_{\ell}^p(\Omega)$ with $I_{\{g\}}^{\Gamma_N}(yn) =0$, 
		 we have that
		 \begin{align*}
		 	\begin{aligned}
		 	\phi(\cdot,\nabla v)-\nabla v:  y+\phi^*(\cdot,y)&= 0&&\quad\text{ a.e.\ in }\Omega\,,\\
		 	\psi(\cdot, v)- v\cdot\mathrm{div}\,y+\psi^*(\cdot,\mathrm{div}\,y)&=0&&\quad \text{ a.e.\ in }\Omega\,,
		 		\end{aligned}
		 \end{align*}
		 if and only if %$I(v)=D(y)$, 
		 \begin{align*}
		 	I(v)=D(y)\,,
		 \end{align*}
		 \textit{i.e.}, if $v\in U_{\ell}^p(\Omega)$ is minimal for \eqref{primal} and $y\in Z_{\ell}^p(\Omega)$ is maximal for \eqref{dual}.~In~other~words, given the strong duality relation \eqref{eq:strong_duality}, the primal-dual gap estimator  measures how well approximations $v\in U_{\ell}^p(\Omega)$ and $y\in  Z_{\ell}^p(\Omega)$ satisfy the convex optimality~relations~\eqref{eq:optimality_relations.1},~\eqref{eq:optimality_relations.2}. This is an advantage compared to residual type a posteriori error estimators which traditionally measure how well an approximation $v\in U_{\ell}^p(\Omega)$  satisfies the strong formulation of optimality conditions and, consequently, employ classical derivatives of the energy densities. The convex optimality relations \eqref{eq:optimality_relations.1},~\eqref{eq:optimality_relations.2}, however, do not require any regularity~of~the~energy~\mbox{densities}. This makes the primal-dual gap estimator a predestined a posteriori~error~estimator, in~particular, for non-differentiable convex minimization problems.
		
		\item[(ii)] Due to Remark \ref{rem:convexity_measure_optimal}, from Theorem \ref{thm:main}(i), for every $v\in  U_{\ell}^p(\Omega)$ and~$y\in  Z_{\ell}^p(\Omega)$,~it~\mbox{follows}~that
		\begin{align*}
			\sigma_I^2(v,u)+\sigma_{-D}^2(y,z)\leq \smash{\eta^2_{\textup{gap}}}(v,y)\,.
		\end{align*} 
	\end{itemize}
\end{remark}

%Since the dual problem to the minimization of the negative of \eqref{dual}, in turn, consists in the maximization of the negative of \eqref{primal},
%the roles of the primal problem and the dual problem may be interchanged. An advantage of Theorem \ref{thm:main} consists in the fact that it yields reliable and efficient a posteriori error estimators for both the primal problem and the dual problem.\enlargethispage{10mm}
%
%
%\begin{remark}
%	Theorem \ref{thm:main} shows that for every $y\in Z_{\ell}^p(\Omega)$, the estimator $	\smash{\eta^2_{I,y}}\coloneqq  ({v\mapsto \smash{\eta^2_{\textup{gap}}}(v,y)})\colon $ $U_{\ell}^p(\Omega)\to [0,+\infty]$ 
%	satisfies 
%	\begin{align}\label{eq:a_posteriori_primal}
%		\smash{\smash{\rho_{\textup{tot}}^2}(v,y)=\eta^2_{I,y}(v)}\quad\text{ for all }v\in U_{\ell}^p(\Omega)\,,
%	\end{align}
%	and for each $v\in U_{\ell}^p(\Omega)$, the estimator $	\smash{\eta^2_{-D,v}}\coloneqq  (y\mapsto \smash{\eta^2_{\textup{gap}}}(v,y))\colon Z_{\ell}^p(\Omega)\to [0,+\infty]$~\mbox{satisfies}
%	\begin{align}\label{eq:a_posteriori_dual}
%			\smash{\smash{\rho_{\textup{tot}}^2}(v,y)=\eta^2_{-D,v}(y)}\quad \text{ for all }y\in Z_{\ell}^p(\Omega) \,.
%	\end{align}
%\end{remark}

For the a posteriori error estimators \eqref{eq:representation}  for being numerically practicable, it is necessary to have a 
computationally cheap way to obtain sufficiently accurate approximations of the dual solution %for \eqref{eq:a_posteriori_primal})
 and/or of the primal solution,  respectively.
%(for \eqref{eq:a_posteriori_dual}), respectively. 
In Section~\ref{sec:discrete_duality}, resorting to (discrete) convex duality relations between a non-conforming Crouzeix--Raviart approximation of the primal problem and a  Raviart--Thomas approximation of the dual problem, we arrive at discrete reconstruction formulas, called \textit{generalized Marini formula}~(\textit{cf}.~\cite{Mar85,Bar21}).

\subsection{Discrete convex minimization problem and discrete convex duality}\label{sec:discrete_duality}\enlargethispage{10mm}

\hspace{5mm}\textit{Discrete primal problem.} 	Let $\phi_h\colon \Omega\times\mathbb{R}^{\ell\times d}\to \mathbb{R}\cup\{+\infty\}$  and $\psi_h\colon \Omega\times\mathbb{R}^{\ell}\to \mathbb{R}\cup\{+\infty\}$ be approximations of $ \phi\colon \Omega\times\mathbb{R}^{\ell \times d}\to \mathbb{R}\cup\{+\infty\}$  and $ \psi\colon \Omega\times\mathbb{R}^{\ell} \to \mathbb{R}\cup\{+\infty\}$, respectively, such that for a.e.\ $x\in \Omega$, the functions $\phi_h(x,\cdot)\colon  \mathbb{R}^{\ell \times d}\to \mathbb{R}\cup\{+\infty\}$  and $\psi_h(x,\cdot)\colon \mathbb{R}^{\ell}\to \mathbb{R}\cup\{+\infty\}$ are proper, convex,  and lower semi-continuous and
 $\phi_h(\cdot,r),\psi_h(\cdot,s)\in \mathcal{L}^0(\mathcal{T}_h)$ for all $r\in \mathbb{R}^d$ and $s\in\mathbb{R}$. 
Moreover, let $g_h\in (\mathcal{L}^0(\mathcal{S}_h^{\Gamma_N}))^{\ell }$ and $u_D^h\in (\mathcal{L}^0(\mathcal{S}_h^{\Gamma_D}))^{\ell}$ be approximations of the Neumann boundary data $g\in W^{-\smash{\frac{1}{p'}},p'}(\Gamma_N;\mathbb{R}^{\ell})$ and the Dirichlet boundary~data~${u_D\in W^{\smash{1-\frac{1}{p}},p}(\partial\Omega;\mathbb{R}^{\ell})}$,~respectively.
We examine the minimization of the functional $I_h^{cr}\colon U^{cr}_{\ell}(\mathcal{T}_h)\to \mathbb{R}\cup\{+\infty\}$,~for~every~$v_h\in U^{cr}_{\ell}(\mathcal{T}_h)$ defined by
\begin{align}
	I_h^{cr}(v_h)\coloneqq \int_{\Omega}{\phi_h(\cdot,\nabla_h v_h)\,\textup{d}x}+\int_{\Omega}{\psi_h(\cdot,\Pi_h v_h)\,\textup{d}x}-(g_h,\pi_h v_h)_{\Gamma_N}+I_{\{u_D^h\}}^{\Gamma_D}(\pi_hv_h)\,,\label{discrete_primal}
\end{align}
where $I_{\{u_D^h\}}^{\Gamma_D}\colon (\mathcal{L}^0(\mathcal{S}_h^{\Gamma_D}))^{\ell}\to \mathbb{R}\cup\{+\infty\}$ for every $\widehat{v}_h\in (\mathcal{L}^0(\mathcal{S}_h^{\Gamma_D}))^{\ell}$ is defined by
\begin{align*}
	I_{\{u_D^h\}}^{\Gamma_D}(\widehat{v}_h)\coloneqq \begin{cases}
		0&\text{ if }\widehat{v}_h=u_D^h\text{ a.e.\ on }\Gamma_D\,,\\
		+\infty&\text{ else}\,.
	\end{cases}
\end{align*}
In what follows, we refer to the minimization of $I_h^{cr}\colon U^{cr}_{\ell}(\mathcal{T}_h)\to \mathbb{R}\cup\{+\infty\}$ as the \textit{discrete primal problem}.

\textit{Discrete dual problem.} A corresponding \textit{discrete (Fenchel) dual problem} to the minimization~of~\eqref{discrete_primal} 
is found to be given via the maximization of the~functional~${D_h^{rt}\colon \hspace*{-0.1em}Z^{rt}_{\ell}(\mathcal{T}_h)\hspace*{-0.1em}\to\hspace*{-0.1em} \mathbb{R}\hspace*{-0.1em}\cup\hspace*{-0.1em}\{-\infty\}}$, for every $y_h\in Z^{rt}_{\ell}(\mathcal{T}_h)$ defined by
\begin{align}
	D_h^{rt}(y_h)\coloneqq-\int_{\Omega}{\phi_h^*(\cdot,\Pi_h y_h)\,\textup{d}x}-\int_{\Omega}{\psi_h^*(\cdot,\textup{div}\,y_h)\,\textup{d}x}+(y_h n,u_D^h)_{\Gamma_D}-I_{\{g_h\}}^{\Gamma_N}(y_hn)\,,\label{discrete_dual}
\end{align} 
where $I_{\{g_h\}}^{\Gamma_N}\colon (\mathcal{L}^0(\mathcal{S}_h^{\Gamma_N}))^{\ell}\to \mathbb{R}\cup\{+\infty\}$ for every $\widehat{y}_h\in (\mathcal{L}^0(\mathcal{S}_h^{\Gamma_N}))^{\ell}$ is defined by
\begin{align*}
	I_{\{g_h\}}^{\Gamma_N}(\widehat{y}_h)\coloneqq \begin{cases}
		0&\text{ if }\widehat{y}_h =g_h\text{ a.e.\ on }\Gamma_N\,,\\
		+\infty&\text{ else}\,.
	\end{cases}
\end{align*}

We will always assume that $\phi_h\colon \Omega\times\mathbb{R}^{\ell \times d}\to \mathbb{R}\cup\{+\infty\}$ and $\psi_h\colon \Omega\times \mathbb{R}^{\ell}\to \mathbb{R}\cup\{+\infty\}$ are such that \eqref{discrete_primal} admits at least one minimizer $u_h^{cr}\in U^{cr}_{\ell}(\mathcal{T}_h)$, called \textit{discrete primal solution},
 and that \eqref{discrete_dual} admits at least one maximizer  $z_h^{rt}\in Z^{rt}_{\ell}(\mathcal{T}_h)$, called  \textit{discrete dual solution}.
The derivation of the discrete dual problem \eqref{discrete_dual} can be found in the proof of the following proposition that also establishes the equivalence of a discrete strong duality relation and discrete convex optimality relations.

\begin{proposition}[Strong duality and convex duality relations]\label{prop:discrete_duality} The following statements apply:
	\begin{itemize}[noitemsep,topsep=2pt,leftmargin=!,labelwidth=\widthof{(ii)}]
		\item[(i)]  A \textup{discrete weak duality relation} applies, \textit{i.e.},
		\begin{align}
			\inf_{v_h\in U^{cr}_{\ell}(\mathcal{T}_h)}{I_h^{cr}(v_h)}
			\ge \sup_{y_h\in Z^{rt}_{\ell}(\mathcal{T}_h)}{D_h^{rt}(y_h)}\,.\label{eq:discrete_weak_duality}
		\end{align}
		\item[(ii)]  A  \textup{discrete strong duality relation} applies, \textit{i.e.},
		\begin{align}
			I_h^{cr}(u_h^{cr}) = D_h^{rt}(z_h^{rt})\,,\label{eq:discrete_strong_duality}
		\end{align}
		if and only if  \textup{discrete convex optimality relations} apply, \textit{i.e.},
		\begin{alignat}{2} 
					\Pi_h z_h^{rt}: \nabla_h u_h^{cr}&=\phi^*_h(\cdot,\Pi_hz_h^{rt})+\phi_h(\cdot,\nabla_h u_h^{cr})&&\quad\text{ a.e. in }\Omega\,,\label{eq:discrete_optimality_relations.1}\\
					\textup{div}\,z_h^{rt}\cdot\Pi_hu_h^{cr}& =\psi_h^*(\cdot,	\textup{div}\,z_h^{rt})+\psi_h(\cdot,\Pi_hu_h^{cr})&&\quad\text{ a.e. in }\Omega\,.
					\label{eq:discrete_optimality_relations.2}
		\end{alignat}
	\end{itemize}
\end{proposition}

\begin{proof}
	
	\textit{ad (i).} For every $v_h\hspace*{-0.1em}\in\hspace*{-0.1em} U^{cr}_{\ell}(\mathcal{T}_h)$ and $y_h\hspace*{-0.1em}\in \hspace*{-0.1em} Z^{rt}_{\ell}(\mathcal{T}_h)$, by the Fenchel--Young~\mbox{inequality}~(\textit{cf}.~\eqref{eq:fenchel_young_ineq}), we have that 
	\begin{align}\label{prop:discrete_duality.1}
		\begin{aligned}
		\phi_h(\cdot,\nabla_h v_h)&\geq 	\Pi_h y_h: \nabla_h v_h- \phi^*_h(\cdot,\Pi_hy_h)&&\quad\text{ a.e. in }\Omega\,,\\
		\psi_h(\cdot,\Pi_hv_h)&\geq 	\textup{div}\,y_h\cdot\Pi_hv_h- \psi_h^*(\cdot,	\textup{div}\,y_h)&&\quad\text{ a.e. in }\Omega\,.
		\end{aligned}
	\end{align}
	\newpage 
	\hspace*{-5.5mm}Adding the two inequalities in \eqref{prop:discrete_duality.1}, using the discrete integration-by-parts formula \eqref{eq:pi}, for every $v_h\hspace*{-0.15em}\in \hspace*{-0.15em} U^{cr}_{\ell}(\mathcal{T}_h)$ with $\pi_hv_h\hspace*{-0.15em}=\hspace*{-0.15em}u^h_D$ a.e.\ on $\Gamma_D$ and $y_h\hspace*{-0.15em}\in\hspace*{-0.15em} Z^{rt}_{\ell}(\mathcal{T}_h)$ with $y_hn \hspace*{-0.15em}=\hspace*{-0.15em}g_h$~a.e.~on~$\Gamma_N$,~we~find~that 
	\begin{align*}
	I_h^{cr}(v_h)&=\int_{\Omega}{\phi_h(\cdot,\nabla_h v_h)\,\mathrm{d}x}+\int_{\Omega}{\psi_h(\cdot,\Pi_hv_h)\,\mathrm{d}x}-(g_h,\pi_h v_h)_{\Gamma_N}
	\\&\ge 
	-	\int_{\Omega}{\phi^*_h(\cdot,\Pi_hy_h)\,\mathrm{d}x}-\int_{\Omega}{\psi_h^*(\cdot,	\textup{div}\,y_h)\,\mathrm{d}x}+(y_h n,u_D^h)_{\Gamma_D}=D_h^{rt}(y_h)\,.
	\end{align*}
	On the other hand, for every $v_h\hspace*{-0.15em}\in\hspace*{-0.15em} U^{cr}_{\ell}(\mathcal{T}_h)$ such that $\smash{I_{\{u_D^h\}}^{\Gamma_D}(\pi_hv_h)}\hspace*{-0.15em}=\hspace*{-0.15em}+\infty$, we~have~that~${I_h^{cr}(v_h)\hspace*{-0.15em}=\hspace*{-0.15em}+\infty}$, and 
	for every $y_h\in Z^{rt}_{\ell}(\mathcal{T}_h)$ such that $\smash{I_{\{g_h\}}^{\Gamma_N}(y_h n) }=+\infty$, we have that
	$D_h^{rt}(z_h^{rt})=-\infty$.
	
	\textit{ad (ii).} The discrete strong duality relation \eqref{eq:discrete_strong_duality} is equivalent to 
	\begin{align*}
		&\int_{\Omega}{\big\{\phi^*_h(\cdot,\Pi_h z_h^{rt})-\Pi_h z_h^{rt}: \nabla_h u_h^{cr}+\phi_h(\cdot,\nabla_h u_h^{cr})\big\}\,\mathrm{d}x}\\&\quad+\int_{\Omega}{\big\{\psi^*_h(\cdot,	\textup{div}\,z_h^{rt})-\textup{div}\,z_h^{rt}\cdot\Pi_h u_h^{cr}+\psi_h(\cdot, \Pi_h u_h^{cr})\big\}\,\mathrm{d}x}=0\,.
	\end{align*}
	Therefore, due to \eqref{prop:discrete_duality.1}, we find that the discrete strong duality relation \eqref{eq:discrete_strong_duality} is equivalent to the discrete convex optimality relations \eqref{eq:discrete_optimality_relations.1}, \eqref{eq:discrete_optimality_relations.2}.
\end{proof}

\begin{remark}[Equivalent discrete convex optimality relations]\label{rem:discrete_duality}
	\begin{itemize}[noitemsep,topsep=2pt,leftmargin=!,labelwidth=\widthof{(iii)}]
		\item[(i)] If $\phi_h(x,\cdot)\in C^1(\mathbb{R}^{\ell \times d})$ for a.e.\ $x\in \Omega$,  by the Fenchel--Young identity (\textit{cf}.\ \eqref{eq:fenchel_young_id}), the discrete convex optimality relation \eqref{eq:discrete_optimality_relations.1} is equivalent to 
		\begin{align}
				\Pi_h z_h^{rt}=D\phi_h(\cdot,\nabla_h u_h^{cr})\quad\text{ a.e.\ in }\Omega\,;\label{eq:discrete_optimality_relations.3}
		\end{align}
		\item[(ii)]  If 
		$\phi^*_h(x,\cdot)\in  C^1(\mathbb{R}^{\ell \times d})$ for a.e.\ $x\in \Omega$,   by the Fenchel--Young identity (\textit{cf}.\  \eqref{eq:fenchel_young_id}),  the discrete convex optimality relation \eqref{eq:discrete_optimality_relations.2} is equivalent to 
		\begin{align}
			\nabla_h u_h^{cr}=D\phi^*_h(\cdot,\Pi_h z_h^{rt})\quad\text{ a.e.\ in }\Omega\,;\label{eq:discrete_optimality_relations.4}
		\end{align}
		\item[(iii)] If  
		$\psi_h(x,\cdot)\in C^1(\mathbb{R}^{\ell })$ for a.e.\ $x\in \Omega$, 
		 by the Fenchel--Young identity (\textit{cf}.\ \eqref{eq:fenchel_young_id}), the discrete convex optimality relation  \eqref{eq:discrete_optimality_relations.2} is equivalent to 
		\begin{align}
			\textup{div}\,z_h^{rt}=D\psi_h(\cdot,\Pi_hu_h^{cr})\quad\text{ a.e.\ in }\Omega\,;\label{eq:discrete_optimality_relations.5}
		\end{align}
		 \item[(iv)]  If  $\psi_h^*(x,\cdot)\in C^1(\mathbb{R}^{\ell})$ for a.e.\ $x\in \Omega$,  by the Fenchel--Young identity (\textit{cf}.\  \eqref{eq:fenchel_young_id}), the discrete convex optimality relation \eqref{eq:discrete_optimality_relations.2} is equivalent to 
		\begin{align}
			\Pi_hu_h^{cr}=D\psi_h^*(\cdot,\textup{div}\,z_h^{rt})\quad\text{ a.e.\ in }\Omega\,.\label{eq:discrete_optimality_relations.6}
		\end{align}
	\end{itemize}
\end{remark}

The relations \eqref{eq:discrete_optimality_relations.3}--\eqref{eq:discrete_optimality_relations.6} motivate the following discrete reconstruction formulas for a discrete dual solution $z_h^{rt}\in Z^{rt}_{\ell}(\mathcal{T}_h)$ from a discrete primal solution  $u_h^{cr}\in U^{cr}_{\ell}(\mathcal{T}_h)$~and~vice~versa, called \textit{generalized Marini formulas} (\textit{cf}.\  \cite{Mar85,Bar21}).

\begin{proposition}[Generalized Marini formulas]\label{prop:gen_marini}
	The following statements apply:
	\begin{itemize}[noitemsep,topsep=2pt,leftmargin=!,labelwidth=\widthof{(ii)}]
		\item[(i)] If $\phi_h(x,\cdot)\in  C^1(\mathbb{R}^{\ell \times d})$ and $\psi_h(x,\cdot)\in C^1(\mathbb{R}^{\ell})$ for a.e.\ $x\in \Omega$, then, given a minimizer $u_h^{cr}\in U^{cr}_{\ell}(\mathcal{T}_h)$ of \eqref{discrete_primal},
		a maximizer $z_h^{rt}\in Z^{rt}_{\ell}(\mathcal{T}_h)$ of \eqref{discrete_dual} is given by
		\begin{align}
			z_h^{rt}= D\phi_h(\cdot,\nabla_h u_h^{cr})+\frac{D\psi_h(\cdot, \Pi_hu_h^{cr})}{d}\otimes(\textup{id}_{\mathbb{R}^d}-\Pi_h\textup{id}_{\mathbb{R}^d})\quad\text{ a.e.\ in }\Omega\,,\label{eq:reconstruction_formula.1}
		\end{align}
		and a discrete strong duality relation, \textit{i.e.}, \eqref{eq:discrete_strong_duality},  applies.
		\item[(ii)] If $\phi^*_h(x,\cdot)\in  C^1(\mathbb{R}^{\ell \times d})$ and $\psi_h^*(x,\cdot)\in C^1(\mathbb{R}^{\ell})$ for a.e.\ $x\in \Omega$, then, given a maximizer $z_h^{rt}\in Z^{rt}_{\ell}(\mathcal{T}_h)$ of \eqref{discrete_dual}, a minimizer $u_h^{cr}\in U^{cr}_{\ell}(\mathcal{T}_h)$ of \eqref{discrete_primal} is given by
		\begin{align}
			u_h^{cr} = D\psi_h^*(\cdot,\textup{div}\,z_h^{rt})+ D\phi^*_h(\cdot,\Pi_h z_h^{rt})(\textup{id}_{\mathbb{R}^d}-\Pi_h\textup{id}_{\mathbb{R}^d})
			\quad\text{ a.e.\ in }\Omega\,,\label{eq:reconstruction_formula.2}
		\end{align}
		and a discrete strong duality relation, \textit{i.e.}, \eqref{eq:discrete_strong_duality}, applies.
	\end{itemize}        
\end{proposition}

\begin{proof}\let\qed\relax 
	\textit{ad (i).} By definition, it holds that $ z_h^{rt}\in (\mathcal{L}^1(\mathcal{T}_h))^{\ell\times d}$ and the discrete convex optimality relation \eqref{eq:discrete_optimality_relations.3} is satisfied.
	Since $u_h^{rt}\in U^{cr}_{\ell}(\mathcal{T}_h)$ is minimal for \eqref{discrete_primal} as well as $\phi_h(x,\cdot)\in  C^1(\mathbb{R}^{\ell\times d})$ for a.e.\ $x\in \Omega$ and $\psi_h(x,\cdot)\in C^1(\mathbb{R}^{\ell})$ for a.e.\ $x\in \Omega$, for every $v_h\in U^{cr}_{\ell,D}(\mathcal{T}_h)$, we have that\vspace*{-0.25mm}
	\begin{align}\label{prop:gen_marini0.1}
		\smash{(D\phi_h(\cdot,\nabla_h u_h^{cr}),\nabla_h v_h)_{\Omega}+(D\psi_h(\cdot,\Pi_h u_h^{cr}),\Pi_h v_h)_{\Omega}-(g_h,\pi_h v_h)_{\Gamma_N}=0\,.}
	\end{align}
	%In particular, \eqref{prop:gen_marini0.1} implies that $D\phi_h(\cdot,\nabla_h u_h^{cr})\in (\nabla_h(\textup{ker}(\Pi_h|_{\smash{U^{cr}_{\ell,0}(\mathcal{T}_h)}})))^\perp$. Due~to~\mbox{\cite[Thm.~3.1]{BW21}}, it holds that
	%$(\nabla_h(\textup{ker}(\Pi_h|_{\smash{U^{cr}_{\ell,0}(\mathcal{T}_h)}})))^\perp=\Pi_h(Z^{rt}_{\ell}(\mathcal{T}_h))$. 
	By the surjectivity of $\textup{div}\colon Z^{rt}_{\ell}(\mathcal{T}_h)\to (\mathcal{L}^0(\mathcal{T}_h))^{\ell}$,  there exists~$y_h\in Z^{rt}_{\ell}(\mathcal{T}_h)$~such~that\vspace*{-0.25mm}
	\begin{align}\label{prop:gen_marini0.2}
	\smash{	\textup{div}\, y_h= D\psi_h(\cdot,\Pi_h u_h^{cr})\quad\text{ a.e.\ in }\Omega\,.}
	\end{align}
	As a result,  resorting to the discrete integration-by-parts~formula~\eqref{eq:pi} and to \eqref{prop:gen_marini0.2}, \eqref{prop:gen_marini0.1},~and \eqref{eq:discrete_optimality_relations.3}, for every $v_h\in U^{cr}_{\ell,0}(\mathcal{T}_h)$, we~find~that\vspace*{-0.25mm}
	\begin{align}\label{prop:gen_marini0.3}
		\begin{aligned}
		\smash{(z_h^{rt}-y_h,\nabla_hv_h)_{\Omega}  
			=(D\phi_h(\cdot,\nabla_h u_h^{cr}),\nabla_h v_h)_{\Omega}+(D\psi_h(\cdot,\Pi_h u_h^{cr}),\Pi_h v_h)_{\Omega}=0}\,. 
		\end{aligned}
	\end{align}
%	In other words, for every $v_h\in (\mathcal{S}^{1,cr}_0(\mathcal{T}_h))^{\ell} $, we have that
%	\begin{align}\label{prop:gen_marini0.3}
%		\begin{aligned}
%			(y_h- z_h^{rt},\nabla_hv_h)_{\Omega} = 	(\Pi_hy_h-\Pi_h z_h^{rt},\nabla_hv_h)_{\Omega} =0\,.
%		\end{aligned}
%	\end{align}
	On the other hand, we have that $\textup{div}\,((z_h^{rt}\hspace*{-0.1em}-\hspace*{-0.1em}y_h)|_T)\hspace*{-0.15em}=\hspace*{-0.15em}0$ in $T$ for all $T\in \mathcal{T}_h$, \textit{i.e.}, ${z_h^{rt}\hspace*{-0.15em}-\hspace*{-0.15em}y_h\hspace*{-0.15em}\in \hspace*{-0.15em}(\mathcal{L}^0(\mathcal{T}_h))^{\ell \times d}}$.
	Thus, \eqref{prop:gen_marini0.3} in conjunction with \eqref{eq:decomposition.2} implies that
	$z_h^{rt}-y_h\in (\nabla_h(U^{cr}_{\ell,0}(\mathcal{T}_h)))^{\perp}=\textup{ker}(\textup{div} |_{\smash{Z^{rt}_{\ell}(\mathcal{T}_h)}})$. As a result, due to $y_h\in Z^{rt}_{\ell}(\mathcal{T}_h)$, we conclude that $z_h^{rt}\in Z^{rt}_{\ell}(\mathcal{T}_h)$ with\vspace*{-0.25mm}
	\begin{align}\label{prop:gen_marini0.4}
		\begin{aligned}
		\Pi_h z_h^{rt}	&=D\phi_h(\cdot,\nabla_h u_h^{cr})&&\quad\text{ a.e.\ in }\Omega\,,\\[-0.5mm]
		\textup{div}\,z_h^{rt}	&=D\psi_h(\cdot,\Pi_hu_h^{cr})&&\quad\text{ a.e.\ in }\Omega\,,\\[-0.5mm]
		z_h^{rt}n&=g_h&&\quad\text{ a.e.\ on }\Gamma_N\,.
		\end{aligned}
	\end{align}
	By the Fenchel--Young identity (\textit{cf}.\ \eqref{eq:fenchel_young_id}), \eqref{prop:gen_marini0.4}$_{1,2}$ are equivalent to\vspace*{-0.25mm}
	\begin{align}\label{prop:gen_marini0.5}
		\begin{aligned}
			\Pi_h z_h^{rt}: \nabla_h u_h^{cr}&=\phi^*_h(\cdot,\Pi_hz_h^{rt})+\phi_h(\cdot,\nabla_h u_h^{cr})&&\quad\text{ a.e. in }\Omega\,,\\
			\textup{div}\,z_h^{rt}\cdot\Pi_hu_h^{cr}& =\psi_h^*(\cdot,	\textup{div}\,z_h^{rt})+\psi_h(\cdot,\Pi_hu_h^{cr})&&\quad\text{ a.e. in }\Omega\,.
		\end{aligned}
	\end{align}
	Adding \eqref{prop:gen_marini0.5}$_1$ and \eqref{prop:gen_marini0.5}$_2$, subsequently, integrating with respect to $x\in \Omega$, using the discrete integration-by-parts formula \eqref{eq:pi}, and using the definitions \eqref{discrete_primal}~and~\eqref{discrete_dual}, we arrive at $I_h^{cr}(u_h^{cr})=D_h^{rt}(z_h^{rt})$, 
	which, by the discrete weak duality relation \eqref{eq:discrete_weak_duality}, implies that $z_h^{rt}\in Z^{rt}_{\ell}(\mathcal{T}_h)$ is maximal for \eqref{discrete_dual}.\enlargethispage{13mm}
	
	\textit{ad (ii).} By definition, it holds that $ u_h^{cr}\in (\mathcal{L}^1(\mathcal{T}_h))^{\ell}$ and the discrete convex optimality relations \eqref{eq:discrete_optimality_relations.4}, \eqref{eq:discrete_optimality_relations.6} are satisfied.
	Since $z_h^{rt}\in Z^{rt}_{\ell}(\mathcal{T}_h)$ is maximal for \eqref{discrete_dual} as well as $\phi^*_h(x,\cdot)\in  C^1(\mathbb{R}^{\ell \times d})$ for a.e.\ $x\in \Omega$ and $\psi_h^*(x,\cdot)\in C^1(\mathbb{R}^{\ell})$ for a.e.\ $x\in \Omega$, for every $y_h\in Z^{rt}_{\ell,N}(\mathcal{T}_h)$,~we~have~that\vspace*{-0.5mm}
	\begin{align}\label{prop:gen_marini.1}
		\smash{(D\phi^*_h(\cdot,\Pi_h z_h^{rt}),\Pi_hy_h)_{\Omega}+(D\psi_h^*(\cdot,\textup{div}\,z_h^{rt}),\textup{div}\,y_h)_{\Omega}+(y_h n,u_D^h)_{\Gamma_D}=0\,.}
	\end{align}
	In particular, \eqref{prop:gen_marini.1} implies that $D\phi^*_h(\cdot,\Pi_h z_h^{rt})\in  (\textup{ker}(\textup{div}|_{\smash{Z^{rt}_{\ell,0}(\mathcal{T}_h)}}))^\perp$.
	Due to \eqref{eq:decomposition.2},~it~holds~that
	$(\textup{ker}(\textup{div}|_{\smash{Z^{rt}_{\ell,0}(\mathcal{T}_h)}}))^\perp=\nabla_h (U^{cr}_{\ell}(\mathcal{T}_h))$. Therefore, there exists 
	$v_h\in U^{cr}_{\ell}(\mathcal{T}_h)$~such~that\vspace*{-0.25mm} 
	\begin{align}\label{prop:gen_marini.2}
		\smash{\nabla_h  v_h= D\phi^*_h(\cdot,\Pi_h z_h^{rt})\quad\text{ a.e.\ in }\Omega\,.}
	\end{align}
	As a result,  resorting to the discrete integration-by-parts~formula~\eqref{eq:pi} and to \eqref{prop:gen_marini.2}, \eqref{prop:gen_marini.1}, and \eqref{eq:discrete_optimality_relations.6}, for every $y_h\in Z^{rt}_{\ell,0}(\mathcal{T}_h)$, we~find~that\vspace*{-0.25mm}
	\begin{align}\label{prop:gen_marini.2.1}
		\begin{aligned}
				\smash{( u_h^{cr}-v_h,\textup{div}\,y_h)_{\Omega} 
			= (D\phi^*_h(\cdot,\Pi_h z_h^{rt}),\Pi_hy_h)_{\Omega}+(D\psi_h^*(\cdot,\textup{div}\,z_h^{rt}),\textup{div}\,y_h)_{\Omega}=0\,.} 
		\end{aligned}
	\end{align} 
	On the other hand, we have that $\nabla_h(u_h^{cr}-v_h)=0$ a.e.\ in $\Omega$, \textit{i.e.}, $u_h^{cr}-v_h\in (\mathcal{L}^0(\mathcal{T}_h))^{\ell}$.
	Thus, \eqref{prop:gen_marini.2.1} in conjunction with \eqref{eq:decomposition.1} implies that
	$u_h^{cr}-v_h\in (\textup{div}\,(Z^{rt}_{\ell,0}(\mathcal{T}_h)))^{\perp}=\textup{ker}(\nabla_h |_{U^{cr}_{\ell}(\mathcal{T}_h)})$.\linebreak As a result, due to $v_h\in U^{cr}_{\ell}(\mathcal{T}_h)$, we conclude that $u_h^{cr}\in U^{cr}_{\ell}(\mathcal{T}_h)$ with\vspace*{-0.25mm}
	\begin{align}\label{prop:gen_marini.3}
		\begin{aligned}
			\nabla_h u_h^{cr}&=D\phi^*_h(\cdot,\Pi_h z_h^{rt})&&\quad\text{ a.e.\ in }\Omega\,,\\[-0.5mm]
			\Pi_hu_h^{cr}&=D\psi_h^*(\cdot,\textup{div}\,z_h^{rt})&&\quad\text{ a.e.\ in }\Omega\,,\\[-0.5mm]
			\pi_hu_h^{cr}&=u_D^h&&\quad\text{ a.e.\ on }\Gamma_D\,.
		\end{aligned}
	\end{align}
	By \hspace*{-0.1mm}the \hspace*{-0.1mm}Fenchel--Young \hspace*{-0.1mm}identity \hspace*{-0.1mm}(\textit{cf}.\ \hspace*{-0.1mm}\eqref{eq:fenchel_young_id}), \hspace*{-0.1mm}\eqref{prop:gen_marini.3}$_{1,2}$ \hspace*{-0.1mm}is \hspace*{-0.1mm}equivalent \hspace*{-0.1mm}to \hspace*{-0.1mm}\eqref{prop:gen_marini0.5}. 
	Adding~\hspace*{-0.1mm}\eqref{prop:gen_marini0.5}$_1$~and \eqref{prop:gen_marini0.5}$_2$, subsequently, integrating with respect to $x\in \Omega$, using the discrete integration-by-parts formula \eqref{eq:pi}, and using the definitions \eqref{discrete_primal}~and~\eqref{discrete_dual}, we arrive at $I_h^{cr}(u_h^{cr})=D_h^{rt}(z_h^{rt})$, 
	which, by  the discrete weak duality relation \eqref{eq:discrete_weak_duality}, implies that $u_h^{cr}\in U^{cr}_{\ell}(\mathcal{T}_h)$~is~\mbox{minimal}~for~\eqref{discrete_primal}.~$\qedsymbol$
\end{proof}\newpage

		\section{Model problems}\label{sec:model_problems}\vspace*{-0.75mm}
	
	\hspace*{4mm}In this section, we illustrate the procedure described above by addressing~model~\mbox{problems}.\vspace*{-0.75mm}\enlargethispage{15mm}
	
	\subsection{Non-linear Dirichlet problem}\label{subsec:nl_dirichlet}\vspace*{-0.75mm}
	
	\subsubsection{Continuous problem}\vspace*{-0.75mm}
	
	\hspace*{5mm}A class of variational problems that
	includes the non-linear Laplace operator (\textit{cf}.\ \cite{CL15,LLC18,K22CR,balci2023error}) involves a function $\phi\colon \Omega\times\mathbb{R}^d\to \mathbb{R}$ that satisfies 
	\begin{itemize}[noitemsep,topsep=2pt,leftmargin=!,labelwidth=\widthof{(A.2)},font=\itshape]
		\item[(A.1)]\hypertarget{A.1}{} $\phi\colon \Omega\times\mathbb{R}^d\to \mathbb{R}$ is a Carath\'eodory mapping and $\phi(x,\cdot)\colon\mathbb{R}^d\to \mathbb{R}$ is convex for a.e.\ $x\in \Omega$;
		\item[(A.2)]\hypertarget{A.2}{} There exist constants $\alpha_m,\alpha_M>0$, functions $\beta_m,\beta_M\in L^1(\Omega)$, and a variable exponent $p\in L^\infty(\Omega)$ with $p^-\coloneqq \textup{ess\,sup}_{x\in \Omega}{p(x)}>1$ such that for a.e.\ $x\in \Omega$ and $r\in \mathbb{R}^d$,~it~holds~that\vspace*{-0.75mm}
		\begin{align*}
			\beta_m(x)+\alpha_m\,\vert r\vert^{p(x)}\leq \phi(x,r)\leq \beta_M(x)+\alpha_M\,\vert r\vert^{p(x)}\,;
		\end{align*}
	\end{itemize}
	and a function  $\psi\colon \Omega\times\mathbb{R}\to \mathbb{R}$ defined by $\psi(x,s)\coloneqq -f(x)s$ for a.e.\ $x\in \Omega$ and all $s\in \mathbb{R}$,~where~$f\in L^{p'(\cdot)}(\Omega)$\footnote{$L^{p'(\cdot)}(\Omega)\coloneqq \{ v\in L^1(\Omega)\mid \vert v\vert^{p'(\cdot)}\in L^1(\Omega)\}$, $W^{1,p(\cdot)}(\Omega)\coloneqq \{v\in W^{1,1}(\Omega)\mid \vert v\vert^{p(\cdot)}, \vert \nabla v\vert^{p(\cdot)}\in L^1(\Omega)\}$, $W^{p'(\cdot)}(\textup{div};\Omega)\coloneqq \{y\in W^1(\textup{div};\Omega)\mid \vert y\vert^{p'(\cdot)}, \vert \textup{div}\,y\vert^{p'(\cdot)}\in L^1(\Omega)\}$.} and $p'(x)\coloneqq \frac{p(x)}{p(x)-1}$ for a.e.\ $x\in \Omega$. Then, for $g \in  L^{(p^-)'}(\Gamma_N )$ and $u_D\in W^{\smash{1-\frac{1}{p^-}},p^-}(\partial\Omega)$, the \textit{non-linear Dirichlet problem} is given via the optimality condition of the minimization of the functional $I\colon W^{1,p(\cdot)}(\Omega)\to \mathbb{R}\cup\{+\infty\}$,~for~every~${v\in W^{1,p(\cdot)}(\Omega)}$~defined~by\vspace*{-0.75mm}
	\begin{align}\label{eq:nl_dirichlet_primal}
		I(v)\coloneqq \int_{\Omega}{\phi(\cdot,\nabla v)\,\mathrm{d}x}-(f, v)_{\Omega}-( g, v)_{\Gamma_N}+I_{\{u_D\}}^{\Gamma_D}(v)\,,
	\end{align} 
	Then, a corresponding (Fenchel) dual problem is given via the maximization of
	the functional $D\colon W^{\smash{p'(\cdot)}}(\textup{div};\Omega)\to \mathbb{R}\cup\{-\infty\}$, for every $y\in W^{\smash{p'(\cdot)}}(\textup{div};\Omega)$  defined by\vspace*{-0.75mm}
	\begin{align}\label{eq:nl_dirichlet_dual}
	D(y)\coloneqq  -\int_{\Omega}{\phi^*(\cdot,y)\,\mathrm{d}x}+\langle y\cdot n, u_D\rangle_{\partial\Omega}-(g, u_D)_{\Gamma_N}-I_{\{-f\}}^{\Omega}(\textup{div}\,y)-I_{\{g\}}^{\Gamma_N}(y\cdot n) \,.
	\end{align}
	As a result, given the two functionals \eqref{eq:nl_dirichlet_primal} and \eqref{eq:nl_dirichlet_dual}, the corresponding 
	primal-dual gap estimator\linebreak $\eta^2_{\textup{gap}}\colon W^{1,p(\cdot)}(\Omega)\times W^{\smash{p'(\cdot)}}(\textup{div};\Omega)\to [0,+\infty]$,~for~every~${v\in W^{1,p(\cdot)}(\Omega)}$ with $v=u_D$ a.e.\ on $\Gamma_D$ and $y\in W^{\smash{p'(\cdot)}}(\textup{div};\Omega)$ with  
	$\textup{div}\, y= -f$ a.e.\ in $\Omega$ and $I_{\{g\}}^{\Gamma_N}(y\cdot n)=0$, is given via\vspace*{-0.75mm}
	\begin{align}
		\eta_{\textup{gap}}^2(v,y)\coloneqq \int_{\Omega}{\big\{\phi(\cdot,\nabla v)-\nabla v\cdot y+\phi^*(\cdot,y)\big\}\,\mathrm{d}x}\,.\label{eq:nl_dirichlet_eta}
	\end{align}
	The integrand of \eqref{eq:nl_dirichlet_eta}, 
	by the  Fenchel--Young inequality (\textit{cf}.\ \eqref{eq:fenchel_young_ineq}), is point-wise non-negative~and, by the  Fenchel--Young identity (\textit{cf}.\ \eqref{eq:fenchel_young_id}), vanishes if and only if $y \in  \partial\phi(\cdot,\nabla v)$ a.e.\ in $\Omega$.\vspace*{-0.5mm}

	\subsubsection{Discrete  problem}\vspace*{-0.75mm}
	
	\hspace*{5mm}Let  $\phi_h\colon \Omega\times \mathbb{R}^d\to \mathbb{R}$ be an approximation $\phi\colon \Omega\times \mathbb{R}^d\to \mathbb{R}$ satisfying (\hyperlink{A.1}{A.1})
	and $\phi_h(\cdot,r)\in \mathcal{L}^0(\mathcal{T}_h)$ for all $r\in \mathbb{R}^d$. Moreover, let $f_h\in \mathcal{L}^0(\mathcal{T}_h)$, $g_h\in \mathcal{L}^0(\mathcal{S}_h^{\Gamma_N})$, and $u_D^h\in \mathcal{L}^0(\mathcal{S}_h^{\Gamma_D})$ be  approximations of $f\in L^{p'(\cdot)}(\Omega)$, $g\in L^{\smash{(p^-)'}}(\Gamma_N)$, and $u_D|_{\Gamma_D}\in W^{\smash{1-\frac{1}{p^-}},p^-}(\Gamma_D)$.~Then, 
	the  \textit{discrete non-linear Dirichlet problem} is given via the minimization of the functional $I_h^{cr}\colon \mathcal{S}^{1,cr}(\mathcal{T}_h)\to \mathbb{R}\cup\{+\infty\}$, for every $v_h\in  \mathcal{S}^{1,cr}(\mathcal{T}_h)$ defined by\vspace*{-0.75mm}
	\begin{align}\label{eq:nl_dirichlet_discrete_primal}
		I_h^{cr}(v_h)\coloneqq \int_{\Omega}{\phi_h(\cdot,\nabla_h v_h)\,\mathrm{d}x}-(f_h, \Pi_hv_h)_{\Omega}-(g_h, \pi_h v_h)_{\Gamma_N}+I_{\{u_D^h\}}^{\Gamma_D}(\pi_hv_h)\,.
	\end{align}
	Then, a corresponding (Fenchel) dual problem is given via the maximization of 
	the functional $D_h^{rt}\colon \mathcal{R}T^0(\mathcal{T}_h)\to \mathbb{R}\cup\{-\infty\}$, for every $y_h\in \mathcal{R}T^0(\mathcal{T}_h)$ defined by\vspace*{-0.75mm}
	\begin{align}\label{eq:nl_dirichlet_discrete_dual}
		D_h^{rt}(y_h)\coloneqq  -\int_{\Omega}{\phi^*_h(\cdot,\Pi_h y_h)\,\mathrm{d}x}+(y_h\cdot n, u_D^h)_{\Gamma_D}-I_{\{-f_h\}}^{\Omega}(\textup{div}\,y_h)-I_{\{g_h\}}^{\Gamma_N}(y_h\cdot n) \,.
	\end{align} 
	If $\phi_h(x,\cdot)\in C^1(\mathbb{R}^d)$  for a.e.\ $x\in\Omega$, then
	given a discrete primal solution $u_h^{cr}\in \mathcal{S}^{1,cr}(\mathcal{T}_h)$, a discrete dual solution $z_h^{rt}\in \mathcal{R}T^0(\mathcal{T}_h)$ is immediately available via the generalized Marini formula\vspace*{-0.75mm}
	\begin{align}\label{eq:nl_dirichlet_marini}
		z_h^{rt} =D\phi_h(\cdot,\nabla_h u_h^{cr})-\frac{f_h}{d}(\textup{id}_{\mathbb{R}^d}-\Pi_h\textup{id}_{\mathbb{R}^d})\quad\text{ a.e.\  in }\Omega\,.
	\end{align} 
	\newpage
	
	\subsection{Obstacle problem}\label{subsec:obstacle}
	
	\subsubsection{Continuous problem}\enlargethispage{10mm}
	
	\hspace*{5mm}Non-differentiable lower-order contributions arise in formulating an obstacle problem as a variation problem (\textit{cf}.\ \cite{friedmann,caffarelli}): for an external force $f\hspace*{-0.15em}\in\hspace*{-0.15em} L^2(\Omega)$, Dirichlet boundary~data~${u_D\hspace*{-0.15em}\in\hspace*{-0.15em} W^{\frac{1}{2},2}(\partial\Omega)}$, and an obstacle $\chi\in W^{1,2}(\Omega)$ with $\chi\leq u_D$ a.e.\ on $\Gamma_D$, the \textit{obstacle problem} is given via   the minimization of the functional $I\colon W^{1,2}(\Omega)\to \mathbb{R}\cup\{+\infty\}$, for every $v \in W^{1,2}(\Omega)$ defined by 
	\begin{align}\label{eq:obstacle_primal}
		I(v) \coloneqq \tfrac{1}{2}\|\nabla v\|_{2,\Omega}^2-(f, v)_{\Omega}+I_+^{\Omega}(v-\chi)+I_{\{u_D\}}^{\Gamma_D}(v)\,,
	\end{align}
	where the indicator functional $I_+^{\Omega}\colon L^2(\Omega)\to \mathbb{R}\cup\{+\infty\}$ for every $\widehat{v}\in L^2(\Omega)$ is defined by
	\begin{align*}
		I_+^{\Omega}(\widehat{v})\coloneqq \begin{cases}
			0&\text{ if }\widehat{v}\ge 0\text{ a.e.\ in }\Omega\,,\\
			+\infty&\text{ else}\,.
		\end{cases}
	\end{align*}
	 Then, a corresponding (Fenchel) dual problem is given via the maximization of the functional $D\colon L^2(\Omega;\mathbb{R}^d)\to \mathbb{R}\cup\{-\infty\}$,  for every $y\in W^2_N(\textup{div};\Omega)$ defined by
	\begin{align}\label{eq:obstacle_dual}
		D(y)\coloneqq  -\tfrac{1}{2}\| y\|_{2,\Omega}^2-(\textup{div}\,y+f,\chi)_{\Omega}+\langle y\cdot n,u_D\rangle_{\partial\Omega}-I_-^{\Omega}(f+\textup{div}\,y)\,,
	\end{align}
	where the indicator functional $I_-^{\Omega}\colon L^2(\Omega)\to \mathbb{R}\cup\{+\infty\}$ for every $\widehat{y}\in L^2(\Omega)$ is defined by
	\begin{align*}
		I_-^{\Omega}(\widehat{y})\coloneqq \begin{cases}
			0&\text{ if }\widehat{y}\le 0\text{ a.e.\ in }\Omega\,,\\
			+\infty&\text{ else}\,.
		\end{cases}
	\end{align*}
	The dual problem, in general,  is ill-posed in $W^2_N(\textup{div};\Omega)$. A maximizing vector field $z\in  L^2(\Omega;\mathbb{R}^d)$ with  distributional divergence defines a non-negative distribution~${\lambda \coloneqq -\textup{div}\, y - f\in (W^{1,2}_D(\Omega))^*}$.
	As a result, given the two functionals \eqref{eq:obstacle_primal} and \eqref{eq:obstacle_dual}, the corresponding primal-dual gap estimator \linebreak$\smash{\eta^2_{\textup{gap}}}\colon W^{1,2}(\Omega)\times W^2_N(\textup{div};\Omega)\hspace*{-0.15em}\to\hspace*{-0.15em}[0,+\infty]$, for every $v\hspace*{-0.15em}\in\hspace*{-0.15em} W^{1,2}(\Omega)$ with $v\hspace*{-0.15em}=\hspace*{-0.15em}u_D$ a.e.\ on $\Gamma_D$~and~${v\hspace*{-0.15em}\ge\hspace*{-0.15em} \chi}$~a.e.\ in $\Omega$ and $y\in W^2_N(\textup{div};\Omega)$  and $\textup{div}\,y\le -f$ a.e.\ in $\Omega$, using the~\mbox{integration-by-parts}~\mbox{formula}~\eqref{eq:pi_cont}, is given via
	\begin{align}\label{eq:obstacle_eta}
		\smash{\eta^2_{\textup{gap}}}(v,y)=\tfrac{1}{2}\|\nabla v-y\|_{2,\Omega}^2+(-\textup{div}\,y-f,v-\chi)_{\Omega}\,.
	\end{align}
	The first part of \eqref{eq:obstacle_eta} measures a violation of the optimality condition $y=\nabla v$ a.e.\ in $\Omega$ and  
	the second part of \eqref{eq:obstacle_eta} measures a violation of the complementarity property $(\textup{div}\,y+f, v-\chi)_{\Omega} = 0$.
	
	\subsubsection{Discrete problem}
	
	\hspace*{5mm}Let   $f_h\in  \mathcal{L}^0(\mathcal{T}_h)$,  $u_D^h\in  \mathcal{L}^0(\mathcal{S}_h^{\Gamma_D})$, and $\chi_h\in \mathcal{L}^0(\mathcal{T}_h)$ be~approximations~of~$f\in L^2(\Omega)$,~$u_D|_{\Gamma_D}\in W^{\frac{1}{2},2}(\Gamma_D)$, and $\chi\in W^{1,2}(\Omega)$. Then, the \textit{discrete obstacle problem} is given via the minimization of the functional
	$I_h^{cr}\colon \mathcal{S}^{1,cr}(\mathcal{T}_h)\to \mathbb{R}\cup\{+\infty\}$, for every $v_h \in  \mathcal{S}^{1,cr}(\mathcal{T}_h)$~defined~by 
	\begin{align}\label{eq:obstacle_discrete_primal}
		I_h^{cr}(v_h)
		\coloneqq \tfrac{1}{2}\|\nabla_h v_h\|_{2,\Omega}^2-(f_h,\Pi_h v_h)_{\Omega} +I_+^{\Omega}(\Pi_hv_h-\chi_h)+I_{\{u_D^h\}}^{\Gamma_D}(\pi_hv_h)\,.
	\end{align}
	Then, a corresponding (Fenchel) dual problem is given via the maximization of the functional $D_h^{rt}\colon \mathcal{R}T^0_N(\mathcal{T}_h)\to \mathbb{R}\cup\{-\infty\}$,  for every $y_h\in \mathcal{R}T^0_N(\mathcal{T}_h)$ defined by
	\begin{align}\label{eq:obstacle_discrete_dual}
		\begin{aligned} 
		 D_h^{rt}(y_h)&\coloneqq  -\tfrac{1}{2}\| \Pi_hy_h\|_{2,\Omega}^2-(\textup{div}\,y_h+f_h,\chi_h)_{\Omega}+(y_h\cdot n,u_D^h)_{\Gamma_D}-I_-^{\Omega}(f_h+\textup{div}\,y_h)\,.
		\end{aligned}
	\end{align}
	Given a discrete primal solution $u_h^{cr}\in \mathcal{S}^{1,cr}(\mathcal{T}_h)$ and a  Lagrange multiplier~${\smash{\overline{\lambda}}_h^{cr}\in \Pi_h( \mathcal{S}^{1,cr}_D(\mathcal{T}_h))}$, for every $v_h \in \mathcal{S}^{1,cr}_D(\mathcal{T}_h)$ satisfying  
	\begin{align}\label{eq:obstacle_multiplier}
		(\smash{\overline{\lambda}}_h^{cr}, \Pi_h v_h)_{\Omega}=(f_h, \Pi_h v_h)_{\Omega}-(\nabla_h u_h^{cr},\nabla_h v_h)_{\Omega}\,,
	\end{align}
	proceeding as in the proof Proposition \ref{prop:gen_marini}(i), 
	a discrete dual solution $z_h^{rt}\in \mathcal{R}T^0(\mathcal{T}_h)$ is immediately available via the generalized~Marini~formula
	\begin{align}\label{eq:obstacle_marini}
		z_h^{rt}=\nabla_h u_h^{cr}-\frac{f_h-\smash{\overline{\lambda}}_h^{cr}}{d}(\textup{id}_{\mathbb{R}^d}-\Pi_h\textup{id}_{\mathbb{R}^d})\quad\text{ a.e.\  in }\Omega\,.
	\end{align}  
	
	\newpage
	\subsection{Scalar Signorini problem}\label{subsec:signorini}

	\subsubsection{Continuous problem}\enlargethispage{15mm}
	
	\hspace*{5mm}A scalar variant of elastic contact problems leads to a variational problem with an inequality constraint on a subset $\Gamma_C$ of the boundary on which penetration~of~an~obstacle~is~\mbox{prevented}~(\textit{cf}.~\cite{signorini}): 
	for $\Gamma_D,\Gamma_N,\Gamma_C\subseteq \partial\Omega$ with $\Gamma_D\dot{\cup}\Gamma_N\dot{\cup}\Gamma_C=\partial\Omega$,
	$f\in L^2(\Omega)$, $g\in W^{-\smash{\frac{1}{2}},2}(\Gamma_N)$, $u_D\in W^{\frac{1}{2},2}(\partial\Omega)$, and $\chi\in W^{1,2}(\Omega)$ with $\chi=u_D$ a.e.\ on $\Gamma_D$,  the \textit{scalar Signorini problem} is given via the minimization of  the functional $I\colon W^{1,2}(\Omega)\to \mathbb{R}\cup\{+\infty\}$, for every $v \in W^{1,2}(\Omega)$ defined by
	\begin{align}\label{eq:signorini_primal}
			I(v)\coloneqq \tfrac{1}{2}\| \nabla v\|_{2,\Omega}^2-(f, v)_{\Omega}-\langle g, v\rangle_{\Gamma_N}+I_+^{\Gamma_C}(v-\chi)+I_{\{u_D\}}^{\Gamma_D}(v)\,,
	\end{align}
	where the indicator functional $I_+^{\Gamma_C}\colon W^{\frac{1}{2},2}(\Gamma_C)\to \mathbb{R}\cup\{+\infty\}$ for every $\widehat{v}\in W^{\frac{1}{2},2}(\Gamma_C)$~is~defined~by
	\begin{align*}
			I_+^{\Gamma_C}(\widehat{v})\coloneqq
			\begin{cases}
				0&\text{ if }\widehat{v}\ge 0\text{ a.e.\ on }\Gamma_C\,,\\
				+\infty&\text{ else}\,.
			\end{cases}
	\end{align*}
	Then, a corresponding (Fenchel) dual problem is given via the maximization~of~the~\mbox{functional} $D\colon  W^2(\textup{div};\Omega)\to \mathbb{R}\cup\{-\infty\}$, for every $y\in W^2(\textup{div};\Omega)$ defined by
	\begin{align}\label{eq:signorini_dual}
		\begin{aligned} 
		D(y)&\coloneqq  -\tfrac{1}{2}\| y\|_{2,\Omega}^2+\langle y\cdot n,\chi\rangle_{\partial\Omega}-\langle g,\chi\rangle_{\Gamma_N}
		-\smash{I_{\{-f\}}^{\Omega}}(\textup{div}\,y)\\&\quad-\smash{I_{\{g\}}^{\Gamma_N}}(y\cdot n)-\smash{I_+^{\Gamma_C}}(y\cdot n)\,,
	\end{aligned}
	\end{align} 
	where the indicator functional $I_-^{\Omega}\colon \hspace*{-0.1em}W^{-\frac{1}{2},2}(\partial\Omega)\hspace*{-0.1em}\to\hspace*{-0.1em} \mathbb{R}\cup\{+\infty\}$ for every $\widehat{y}\hspace*{-0.1em}\in\hspace*{-0.1em} W^{-\frac{1}{2},2}(\partial\Omega)$~is~defined~by
	\begin{align*}
		I_+^{\Gamma_C}(\widehat{y})\coloneqq \begin{cases}
			0&\text{ if }\langle \widehat{y},v\rangle_{\partial\Omega}\hspace*{-0.1em}\ge\hspace*{-0.1em} 0\text{ for all }v\hspace*{-0.1em}\in\hspace*{-0.1em} W^{1,2}_D(\Omega)\text{ with }v\hspace*{-0.1em}=\hspace*{-0.1em}0\text{ a.e.\ on }\Gamma_N\text{ and }v\hspace*{-0.1em}\ge\hspace*{-0.1em} 0\text{ a.e.\ on }\Gamma_C\,,\\
			+\infty&\text{ else}\,.
		\end{cases}
	\end{align*}
	As a result, given the two functionals \eqref{eq:signorini_primal} and \eqref{eq:signorini_dual},
	the corresponding primal-dual~gap~estimator $\eta^2_{\textup{gap}}\colon W^{1,2}(\Omega)\times W^2(\textup{div};\Omega)\to [0,+\infty]$, for every $v\in W^{1,2}(\Omega)$
	with $v=u_D$ a.e.\ on $\Gamma_D$ and $v\ge \chi$ a.e.\ on $\Gamma_C$ and $y\in W^2(\textup{div};\Omega)$ with $I_{\{g\}}^{\Gamma_N}(y\cdot n)=0$ and $\smash{I_+^{\Gamma_C}}(y\cdot n)=0$, using the integration-by-parts formula \eqref{eq:pi_cont}, is given via 
	\begin{align}\label{eq:signorini_eta}
		\smash{\eta^2_{\textup{gap}}}(v,y)=\tfrac{1}{2}\|\nabla v-y\|_{2,\Omega}^2+\langle y\cdot n,v-\chi\rangle_{\partial\Omega}-\langle g,v-\chi\rangle_{\Gamma_N}\,.
	\end{align}
	The first part of \eqref{eq:signorini_eta} measures a  violation of the optimality relation $y=\nabla v$ a.e.\ in $\Omega$  
	and the second part of \eqref{eq:signorini_eta}  measures a violation of the  complementarity property $\langle y\cdot n,v-\chi\rangle_{\partial\Omega}=\langle g,v-\chi\rangle_{\Gamma_N}$.\vspace*{-1mm}

	\subsubsection{Discrete problem}

	\hspace*{5mm}Let \hspace*{-0.1mm}$f_h\hspace*{-0.175em}\in\hspace*{-0.175em} \mathcal{L}^0(\mathcal{T}_h)$, \hspace*{-0.1mm}$g_h\hspace*{-0.175em}\in\hspace*{-0.175em} \mathcal{L}^0(\mathcal{S}_h^{\Gamma_N})$, \hspace*{-0.1mm}$u_D^h\hspace*{-0.175em}\in \hspace*{-0.175em}\mathcal{L}^0(\mathcal{S}_h^{\Gamma_D})$, and $\chi_h\hspace*{-0.175em}\in\hspace*{-0.175em} \mathcal{L}^0(\mathcal{S}_h^{\Gamma_C})$ \hspace*{-0.1mm}be \hspace*{-0.1mm}approximations~\hspace*{-0.1mm}of~\hspace*{-0.1mm}${f\hspace*{-0.175em}\in\hspace*{-0.175em} L^2(\Omega)}$, $g\in W^{-\frac{1}{2},2}(\Gamma_N)$, $u_D|_{\Gamma_D}\in W^{\frac{1}{2},2}(\Gamma_D)$, and $\chi|_{\Gamma_C}\in W^{\frac{1}{2},2}(\Gamma_C)$.
	 Then,~the~\textit{discrete~scalar~Signorini problem} is given via the minimization of $I_h^{cr}\colon \mathcal{S}^{1,cr}(\mathcal{T}_h)\to \mathbb{R}\cup\{+\infty\}$, for every $v_h\in \mathcal{S}^{1,cr}(\mathcal{T}_h)$ defined by
	\begin{align}
		\begin{aligned}
			I_h^{cr}(v_h)&\coloneqq   \tfrac{1}{2}\|\nabla_hv_h\|_{2,\Omega}^2-( f_h,\Pi_hv_h)_{\Omega}-(g_h,\pi_h v_h)_{\Gamma_N}\\&\quad+I_{\{u_D^h\}}^{\Gamma_D}(\pi_h v_h)+\smash{I_+^{\Gamma_C}}(\pi_h v_h-\chi_h)\,,
		\end{aligned}
		\label{eq:signorini_discrete_primal}
	\end{align}
	Then, a corresponding (Fenchel)  dual problem is given via the maximization~of~the~\mbox{functional} $D_h^{rt}\colon \mathcal{R}T^0(\mathcal{T}_h)\to \mathbb{R}\cup\{-\infty\}$, for every $y_h\in \mathcal{R}T^0(\mathcal{T}_h)$ defined by 
	\begin{align}
		\begin{aligned}
			D_h^{rt}(y_h)&\coloneqq -\tfrac{1}{2}\| \Pi_hy_h\|_{2,\Omega}^2+(y_h\cdot n,\pi_h u_D^h)_{\Gamma_D\cup\Gamma_C}\\&\quad-
			I_{\{-f_h\}}^{\Omega}(\textup{div}\,y_h)-I_{\{g_h\}}^{\Gamma_N}(y_h\cdot n)-\smash{I_+^{\Gamma_C}}(y_h\cdot n)\,.
		\end{aligned}\label{eq:signorini_discrete_dual}
	\end{align}
	Appealing to Proposition \ref{prop:gen_marini}(i), given a discrete primal solution $u_h^{cr}\in \mathcal{S}^{1,cr}(\mathcal{T}_h)$, a discrete dual solution $z_h^{rt}\in  \mathcal{R}T^0(\mathcal{T}_h)$ is immediately available via the generalized Marini formula 
	\begin{align*}
		z_h^{rt}=\nabla_h u_h^{cr}-\frac{f_h}{d}(\textup{id}_{\mathbb{R}^d}-\Pi_h\textup{id}_{\mathbb{R}^d})\quad\text{ a.e.\  in }\Omega\,.
	\end{align*} 
%	For an element-wise affine obstacle $\chi\in \mathcal{S}^1(\mathcal{T}_h)$,
%	an admissible approximation $\overline{u}_h^{cr}\in W^{1,2}(\Omega)$, \textit{e.g.}, can be obtained via node-averaging combined with setting $\overline{u}_h^{cr}(\nu)\coloneqq \max\{\chi(\nu),(\Pi_h^{av}u_h^{cr})(\nu)\}$ for all boundary nodes $\nu\in \mathcal{N}_h$ belonging to the contact boundary $\Gamma_C$.
	
	\newpage
	\subsection{Rudin--Osher--Fatemi image de-noising model}\label{subsec:rof}\vspace*{-0.5mm}

	\subsubsection{Continuous problem}\vspace*{-0.5mm}\enlargethispage{10mm}
	
	\hspace*{5mm}A model problem that requires the usage of spaces of functions with bounded variation is the 
	\textit{Rudin--Osher--Fatemi image de-noising model} (\textit{cf}.\ \cite{ROF92,Bar15}): for a noisy image $g\in L^2(\Omega)$ and a fidelity parameter $\alpha>0$, it is given via the minimization of the functional ${I\colon BV(\Omega)\cap L^2(\Omega)\to \mathbb{R}}$, for every $v\in BV(\Omega)\cap L^2(\Omega)$  defined by 
	\begin{align}\label{eq:tv_primal}
	\smash{	I(v)\coloneqq \vert \textup{D}v\vert(\Omega)+\tfrac{\alpha}{2}\|v-g\|_{2,\Omega}^2\,.}
	\end{align}
	Here, $\vert \textup{D}(\cdot)\vert(\Omega)\colon \smash{L^1_{\textup{loc}}(\Omega)}\to \mathbb{R}\cup\{+\infty\}$, for every $v\in \smash{L^1_{\textup{loc}}(\Omega)}$ defined by 
	\begin{align*}
		\vert \textup{D}v\vert(\Omega)\coloneqq \sup_{\varphi\in C^\infty_c(\Omega;\mathbb{R}^d)\,:\,\|\varphi\|_{\infty,\Omega}\leq 1}{(v,\textup{div}\,\varphi)_{\Omega}}\,,
	\end{align*}
	denotes the total variation functional and $BV(\Omega)\coloneqq\{v\in L^1(\Omega)\mid \vert \textup{D}v\vert(\Omega)<+\infty \}$ the space of functions with bounded variation.
	The total variation functional  can be seen as an extension~of~the~semi-norm
	in $W^{1,1}(\Omega)$ and allows for discontinuous minimizers. For every $ v\in L^2(\Omega)$,
	it can be characterized via the Fenchel duality
	\begin{align}\label{eq:tv_hfdiv}
		\vert \textup{D}v\vert(\Omega)=\sup_{y\in W^2_0(\textup{div};\Omega)}{-(v,\textup{div}\,y)_{\Omega}-\smash{I_{K_1(0)}^{\Omega}}(y)}\,,
	\end{align}
	where \hspace*{-0.1mm}the \hspace*{-0.1mm}indicator \hspace*{-0.1mm}functional \hspace*{-0.1mm}$\smash{I_{K_1(0)}^{\Omega}}\colon \hspace*{-0.15em} W^2_0(\textup{div};\Omega)\hspace*{-0.15em}\to \hspace*{-0.15em}\mathbb{R}\cup\{+\infty\}$ \hspace*{-0.1mm}for \hspace*{-0.1mm}every \hspace*{-0.1mm}$\widehat{y}\hspace*{-0.15em}\in  \hspace*{-0.15em} W^2_0(\textup{div};\Omega)$~\hspace*{-0.1mm}is~\hspace*{-0.1mm}\mbox{defined}~\hspace*{-0.1mm}by
	\begin{align*}
		\smash{I_{K_1(0)}^{\Omega}}(\widehat{y})\coloneqq \begin{cases}
			0&\text{ if }\vert \widehat{y}\vert \leq 1\text{ a.e.\ in }\Omega\,,\\
			+\infty&\text{ else}\,.
		\end{cases}
	\end{align*}
 	By means of the relation \eqref{eq:tv_hfdiv}, one finds that  (\textit{cf}.\ \cite{HK04}) the (Fenchel)~(pre-)dual~problem~is~given via the maximization of $D\colon W^2_0(\textup{div};\Omega)\to \mathbb{R}\cup\{-\infty\}$, for every $y\in W^2_0(\textup{div};\Omega)$ defined~by
	\begin{align}\label{eq:tv_dual}
	\smash{D(y)\coloneqq 
	 -\smash{I_{K_1(0)}^{\Omega}}(y)-\tfrac{1}{2\alpha}\|\textup{div}\,y+\alpha\,g\|_{2,\Omega}^2+\tfrac{\alpha}{2}\|g\|_{2,\Omega}^2\,}. 
	\end{align}
	As a result, given the two functionals \eqref{eq:tv_primal} and \eqref{eq:tv_dual}, 
	the corresponding primal-dual~gap~estimator $\eta_{\textup{gap}}^2\colon BV(\Omega)\cap L^2(\Omega)\times W^2_0(\textup{div};\Omega)\to[0,+\infty] $, for every 
	$v\in BV(\Omega)\cap L^2(\Omega)$ and $y\in W^2_0(\textup{div};\Omega)$ with $\vert y\vert \leq 1$ a.e.\ in~$\Omega$, is given via 
	\begin{align}\label{eq:tv_eta}
		\smash{\eta_{\textup{gap}}^2(v,y)=\vert \textup{D}v\vert(\Omega)+(v,\textup{div}\, y)_{\Omega}+\tfrac{1}{2\alpha}\|\textup{div}\,y+\alpha\,(v-g)\|_{2,\Omega}^2\,.}
	\end{align}
	The \hspace*{-0.1mm}first \hspace*{-0.1mm}two \hspace*{-0.1mm}parts \hspace*{-0.1mm}of \hspace*{-0.1mm}\eqref{eq:tv_eta}  \hspace*{-0.1mm}measure  \hspace*{-0.1mm}a \hspace*{-0.1mm}violation \hspace*{-0.1mm}of \hspace*{-0.1mm}the \hspace*{-0.1mm}optimality \hspace*{-0.1mm}condition \hspace*{-0.1mm}${
	\vert \textup{D}v\vert(\Omega)=-(v,\textup{div}\, y)_{\Omega}}$ 
	and the third part of \eqref{eq:tv_eta}  measures a violation of the optimality condition $\textup{div}\, y = \alpha(v-g)$ a.e.\ in $\Omega$.\vspace*{-1mm}
	
	\subsubsection{Discrete problem}\vspace*{-0.5mm}
	
	\hspace*{5mm}Let $\phi_h\in C^1(\mathbb{R}^d)$  
	 and  $g_h\in \mathcal{L}^0(\mathcal{T}_h)$ be approximations of the Euclidean length $\vert \cdot\vert$~and~$g\in L^2(\Omega)$. Then, the \textit{discrete Rudin--Osher--Fatemi image de-noising model} is given via the minimization of the functional $I_h^{cr}\colon \mathcal{S}^{1,cr}(\mathcal{T}_h)\to \mathbb{R}$, for every $v_h\in \mathcal{S}^{1,cr}(\mathcal{T}_h)$ defined by 
	\begin{align}\label{eq:tv_discrete_primal}
		I_h^{cr}(v_h)=\int_{\Omega}{\phi_h(\nabla_h v_h)\,\mathrm{d}x}+\tfrac{\alpha}{2}\|\Pi_h v_h-g_h\|_{2,\Omega}^2\,.
	\end{align} 
	Then, a corresponding (Fenchel) dual problem is given via the maximization~of~the~functional $D_h^{rt}\colon \mathcal{R}T^0_0(\mathcal{T}_h)\to \mathbb{R}\cup\{-\infty\}$, for every $y_h\in \mathcal{R}T^0_0(\mathcal{T}_h)$ defined by 
	\begin{align}\label{eq:tv_discrete_dual}
		D_h^{rt}(y_h)\coloneqq -\int_{\Omega}{\phi_h^*(\Pi_h y_h)\,\mathrm{d}x}-\tfrac{1}{2\alpha}\|\textup{div}\,y_h+\alpha\,g_h\|_{2,\Omega}^2+\tfrac{\alpha}{2}\|g_h\|_{2,\Omega}^2\,.
	\end{align}
	Appealing to Proposition \ref{prop:gen_marini}(i), given a discrete primal solution $u_h^{cr}\in \mathcal{S}^{1,cr}(\mathcal{T}_h)$, a discrete dual solution $z_h^{rt}\in  \mathcal{R}T^0_0(\mathcal{T}_h)$ is immediately available via the generalized Marini formula 
	\begin{align}\label{eq:rof_marini}
		z_h^{rt}=D\phi_h(\nabla_h u_h^{cr})+\frac{\alpha (\Pi_hu_h^{cr}-g_h)}{d}(\textup{id}_{\mathbb{R}^d}-\Pi_h\textup{id}_{\mathbb{R}^d})\quad\text{ a.e.\  in }\Omega\,.
	\end{align} 
	
	\newpage
	\subsection{Jumping coefficients}\vspace*{-0.5mm}\label{subsec:jumping}\enlargethispage{12mm}
	
	\subsubsection{Continuous problem}\vspace*{-0.5mm}
	\hspace*{5mm}The derivation of sharp a posteriori error estimates is particularly challenging if a partial differential equation involves coefficients whose minimal and maximal values are not comparable:
	for a right-hand side $f\in L^2(\Omega)$, Neumann boundary data $g\in W^{-\frac{1}{2},2}(\Gamma_N)$, and Dirichlet boundary data $u_D\in W^{\frac{1}{2}, 2}(\partial\Omega)$, a
	related model problem is given via the minimization~of~the~functional $I\colon W^{1,2}(\Omega)\to\mathbb{R}\cup\{+\infty\}$, for every $v\in W^{1,2}(\Omega)$ defined by\vspace*{-0.5mm}
	\begin{align}\label{eq:jumping_primal}
		I(v)\coloneqq \tfrac{1}{2}\|A^{\frac{1}{2}}(\cdot) \nabla v\|_{2,\Omega}^2 -(f,v)_{\Omega}-\langle g,v\rangle_{\Gamma_N}+I_{\{u_D\}}^{\Gamma_D}(v)\,,
	\end{align}
	where $A \colon \Omega \to \mathbb{R}^{d\times d}$ is a tensor-valued mapping having the following properties:
	\begin{itemize}[noitemsep,topsep=2pt,leftmargin=!,labelwidth=\widthof{(A.2)},font=\itshape]
		\item[(B.1)]\hypertarget{B.1}{} $A \colon \Omega \to \mathbb{R}^{d\times d}$ is (Lebesgue) measurable;
		\item[(B.2)]\hypertarget{B.2}{} For a.e.\ $x\in \Omega$, the tensor $A(x)\in \mathbb{R}^{d\times d}$ is symmetric and positive definite;
		\item[(B.3)]\hypertarget{B.3}{} There exist constants $\alpha_m,\alpha_M>0$ such that for every $r\in \mathbb{R}^d$ and a.e.\  $x\in \Omega$, it holds that\vspace*{-0.5mm}
		\begin{align*}
			\alpha_m\vert r\vert^2\leq \vert A^{\frac{1}{2}}(x)r\vert^2=A(x)r\cdot r \leq \alpha_M\vert r\vert^2\,.
		\end{align*}
	\end{itemize}
	Note that, due to (\hyperlink{B.2}{B.2}), for a.e.\ $x\in \Omega$, the tensor $A(x)\in \mathbb{R}^{d\times d}$ admits a root $A(x)^{\frac{1}{2}}\in \mathbb{R}^{d\times d}$.
	The tensor-valued mapping $A^{\frac{1}{2}} \colon \Omega \to \mathbb{R}^{d\times d}$ in \eqref{eq:jumping_primal} is defined by $A^{\frac{1}{2}}(x)\coloneqq A(x)^{\frac{1}{2}}$~for~a.e.~$x\in \Omega$.
	Since for a.e.\ $x\in \Omega$, the tensor $A(x)^{\frac{1}{2}}\in \mathbb{R}^{d\times d}$ is symmetric and positive definite,~it~is~invertible.
	The tensor-valued mapping $A^{-\frac{1}{2}} \colon \Omega \to \mathbb{R}^{d\times d}$  is defined by $A^{-\frac{1}{2}}(x)\coloneqq A(x)^{-\frac{1}{2}}$~for~a.e.~$x\in \Omega$.
	If $\phi\colon \Omega\times \mathbb{R}^d\to \mathbb{R}$ is defined by 	$\phi(x,r)\coloneqq\frac{1}{2}\vert A^{\frac{1}{2}}(x)r\vert^2$ for a.e.\ $x\in \Omega$ and all $r\in \mathbb{R}^d$, then $\phi^*\colon \Omega\times \mathbb{R}^d\to \mathbb{R}$ for a.e.\ $x\in \Omega$ and every $s\in \mathbb{R}^d$ is found to be given via\vspace*{-0.5mm}
	\begin{align*}
		 \phi^*(x,s)=\vert A^{-\frac{1}{2}}(x)s\vert^2\,.
	\end{align*}
	Then, a corresponding (Fenchel) dual problem is given via the maximization of the functional $D\colon W^2(\textup{div};\Omega)\to \mathbb{R}\cup\{-\infty\}$,
	 for every  $ y\in W^2(\textup{div};\Omega)$ defined by 
	\begin{align}\label{eq:jumping_dual}
	D(y)\coloneqq -\tfrac{1}{2}\|A^{-\frac{1}{2}}(\cdot) y\|_{2,\Omega}^2+\langle y\cdot n,u_D\rangle_{\partial\Omega}-\langle g,u_D\rangle_{\Gamma_N}-I_{\{-f\}}^{\Omega}(\textup{div}\,y)-I_{\{g\}}^{\Gamma_N}(y\cdot n)\,.
	\end{align}
	As \hspace*{-0.1mm}a \hspace*{-0.1mm}result, \hspace*{-0.1mm}given \hspace*{-0.1mm}the \hspace*{-0.1mm}two \hspace*{-0.1mm}functionals \hspace*{-0.1mm}\eqref{eq:jumping_primal} \hspace*{-0.1mm}and \hspace*{-0.1mm}\eqref{eq:jumping_dual}, \hspace*{-0.1mm}the \hspace*{-0.1mm}corresponding \hspace*{-0.1mm}primal-dual~\hspace*{-0.1mm}gap~\hspace*{-0.1mm}\mbox{estimator} $\eta_{\textup{gap}}^2\colon  W^{1,2}(\Omega)\times W^2(\textup{div};\Omega)\to [0,+\infty]$,
	for every $v\in W^{1,2}(\Omega)$ with $v=u_D$ a.e.\ on~$\Gamma_D$~and~$ y\in W^2(\textup{div};\Omega)$ with $I_{\{g\}}^{\Gamma_N}(y\cdot n)=0$ and $\textup{div}\, y=-f$ a.e.\ in $\Omega$, using the integration-by-parts formula~\eqref{eq:pi_cont},  is given via\vspace*{-0.5mm}
	\begin{align}\label{eq:jumping_eta}
		\smash{\eta^2_{\textup{gap}}}(v,y)=\tfrac{1}{2}\|A^{\frac{1}{2}}(\cdot)\nabla v-A^{-\frac{1}{2}}(\cdot)y\|_{2,\Omega}^2\,,
	\end{align}
	measuring the violation of the optimality condition $y=A(\cdot)\nabla v$ a.e.\ in $\Omega$.
	
	\subsubsection{Discrete problem}\vspace*{-0.5mm}
	\hspace*{5mm}Let $A_h\in (\mathcal{L}^0(\mathcal{T}_h))^{d\times d}$ be an  approximation of 
	$A\colon \Omega\to \mathbb{R}^{d\times d}$ satisfying (\hyperlink{B.2}{B.2}). Moreover, let $f_h\in \mathcal{L}^0(\mathcal{T}_h)$, $g_h\in \mathcal{L}^0(\mathcal{S}_h^{\Gamma_N})$, and $u_D^h\in \mathcal{L}^0(\mathcal{S}_h^{\Gamma_D})$ be approximations of $f\in L^2(\Omega)$, $g\in W^{-\frac{1}{2},2}(\Gamma_N)$, and $u_D|_{\Gamma_D}\in W^{\frac{1}{2},2}(\Gamma_D)$. Then, the discrete problem is given via the minimization of the functional $I_h^{cr}\colon \mathcal{S}^{1,cr}(\mathcal{T}_h)\to \mathbb{R}\cup\{+\infty\}$, for every $v_h\in \mathcal{S}^{1,cr}(\mathcal{T}_h)$ defined by\vspace*{-0.5mm}
	\begin{align}
		\label{eq:jumping_discrete_primal}
		I_h^{cr}(v_h)\coloneqq \tfrac{1}{2}\|A_h^{\frac{1}{2}}(\cdot) \nabla_h v_h\|_{2,\Omega}^2 -(f_h,\Pi_hv_h)_{\Omega}-(g_h,\pi_h v_h)_{\Gamma_N}+I_{\{u_D^h\}}^{\Gamma_D}(\pi_h v_h)\,.
	\end{align}
	Then, a corresponding (Fenchel) dual problem is given via the maximization of the functional $D_h^{rt}\colon \mathcal{R}T^0(\mathcal{T}_h)\to \mathbb{R}\cup\{-\infty\}$,
	for every  $ y_h\in \mathcal{R}T^0(\mathcal{T}_h)$ defined by\vspace*{-0.5mm}
	\begin{align}\label{eq:jumping_discrete_dual}
		D_h^{rt}(y_h)\coloneqq -\tfrac{1}{2}\|A_h^{-\frac{1}{2}}(\cdot) y_h\|_{2,\Omega}^2+( y_h\cdot n,u_D^h)_{\Gamma_D}-I_{\{-f_h\}}^{\Omega}(\textup{div}\,y_h)-I_{\{g_h\}}^{\Gamma_N}(y_h\cdot n)\,.
	\end{align}
	Appealing to Proposition \ref{prop:gen_marini}(i),  given a discrete primal solution $u_h^{cr}\in \mathcal{S}^{1,cr}(\mathcal{T}_h)$, a discrete dual solution $z_h^{rt}\in  \mathcal{R}T^0(\mathcal{T}_h)$ is immediately available via the generalized Marini formula\vspace*{-0.5mm}
	\begin{align}\label{eq:jumping_marini}
		z_h^{rt}=A_h(\cdot)\nabla_h u_h^{cr}-\frac{f_h}{d}(\textup{id}_{\mathbb{R}^d}-\Pi_h\textup{id}_{\mathbb{R}^d})\quad\text{ a.e.\  in }\Omega\,.
	\end{align}

	\newpage
	\subsection{Elasto-plastic torsion}\label{subsec:elasto}
	
	\subsubsection{Continuous problem}

	\hspace*{5mm}The \textit{elasto-plastic torsion problem} models the torsion of an infinitely long elasto-plastic cylinder of a cross section $\Omega$ and plasticity yield $r>0$ (\textit{cf}.\ \cite{daniele2014,CH22}), where we, for simplicity,~set~$r=1$:
	for a right-hand side $f\in L^2(\Omega)$, Neumann boundary data $g\in W^{-\frac{1}{2},2}(\Gamma_N)$, and Dirichlet boundary data $u_D\in W^{\frac{1}{2},2}(\partial\Omega)$, it  is given via the minimization of the functional $I\colon W^{1,2}(\Omega)\to\mathbb{R}\cup\{+\infty\}$, for every $v\in W^{1,2}(\Omega)$ defined by\enlargethispage{8mm}
	\begin{align}\label{eq:elasto_primal}
		I(v)\coloneqq \tfrac{1}{2}\|\nabla v\|_{2,\Omega}^2+I_{K_1(0)}(\nabla v)  -(f,v)_{\Omega}-\langle g,v\rangle_{\Gamma_N}+I_{\{u_D\}}^{\Gamma_D}(v)\,.
	\end{align}
	If $\phi\colon \mathbb{R}^d\to \mathbb{R} \cup\{+\infty\}$
	 is defined by 	$\phi(r)\coloneqq\frac{1}{2}\vert r\vert^2+I_{K_1(0)}(r)$ for all $r\in \mathbb{R}^d$, then $\phi^*\colon  \mathbb{R}^d\to \mathbb{R}$~and every $s\in \mathbb{R}^d$ is found to be given via 
	\begin{align*}
		\phi^*(s)=\begin{cases}
			\tfrac{1}{2}\vert  s\vert ^2&\text{ if }\vert s\vert \leq 1\,,\\
			\vert  s\vert -\tfrac{1}{2}&\text{ else }\,.
		\end{cases} 
	\end{align*}
	Then, a corresponding (Fenchel) dual problem is given via the maximization of the functional $D\colon W^1(\textup{div};\Omega)\to \mathbb{R}\cup\{-\infty\}$,
	for every  $ y\in W^1(\textup{div};\Omega)$ defined by 
	\begin{align}\label{eq:elasto_dual}
		D(y)\coloneqq -\int_\Omega{\phi^*(y)\,\mathrm{d}x} +\langle y\cdot n,u_D\rangle_{\partial\Omega}-\langle g,u_D\rangle_{\Gamma_N}-I_{\{-f\}}^{\Omega}(\textup{div}\,y)-I_{\{g\}}^{\Gamma_N}(y\cdot n)\,.
	\end{align}
	The best of the authors knowledge, establishing the well-posedness the dual problem and a strong duality relation are still open problems.
	As \hspace*{-0.1mm}a \hspace*{-0.1mm}result, \hspace*{-0.1mm}given \hspace*{-0.1mm}the \hspace*{-0.1mm}two \hspace*{-0.1mm}functionals~\hspace*{-0.1mm}\eqref{eq:jumping_primal}~\hspace*{-0.1mm}and~\hspace*{-0.1mm}\eqref{eq:jumping_dual}, \hspace*{-0.1mm}the \hspace*{-0.1mm}corresponding \hspace*{-0.1mm}primal-dual~\hspace*{-0.1mm}gap~\hspace*{-0.1mm}\mbox{estimator} $\eta_{\textup{gap}}^2\colon  W^{1,2}(\Omega)\times W^1(\textup{div};\Omega)\to  [0,+\infty]$,
	for every $v\hspace*{-0.1em}\in\hspace*{-0.1em} W^{1,2}(\Omega)$ with $\vert \nabla v\vert\leq 1$ a.e.\ in $\Omega$ and $v= u_D$ a.e.\ on $\Gamma_D$ and $y\in W^1(\textup{div};\Omega)$ with $I_{\{g\}}^{\Gamma_N}(y\cdot n)=0$ and $\textup{div}\, y=-f$ a.e.\ in $\Omega$, using the integration-by-parts formula~\eqref{eq:pi_cont},  is given via\vspace*{-0.5mm}
	\begin{align}\label{eq:elasto_eta}
		\smash{\eta^2_{\textup{gap}}}(v,y)&=
		\tfrac{1}{2}\|\nabla v\|_{2,\Omega}^2-(\nabla v,y)_\Omega+\int_\Omega{\phi^*(y)\,\mathrm{d}x} \,,
	\end{align}
	measuring the violation of the optimality condition $\frac{1}{2}\vert \nabla v\vert^2-\nabla v\cdot y+\phi^*(y)=0$ a.e.\ in $\Omega$.
	
	\subsubsection{Discrete problem}\vspace*{-0.5mm}
	\hspace*{5mm}Let $\phi_h\colon \hspace*{-0.1em}\mathbb{R}^d\hspace*{-0.1em}\to\hspace*{-0.1em} \mathbb{R}\cup\{+\infty\}$ 
	  be an  approximation of $I_{K_1(0)}\colon \mathbb{R}^d\hspace*{-0.1em}\to\hspace*{-0.1em} \mathbb{R}\cup\{+\infty\}$~such~that~${\phi_h^*\hspace*{-0.1em}\in\hspace*{-0.1em} C^1(\mathbb{R}^d)}$. Moreover, let $f_h\in \mathcal{L}^0(\mathcal{T}_h)$, $g_h\in \mathcal{L}^0(\mathcal{S}_h^{\Gamma_N})$, and $u_D^h\in \mathcal{L}^0(\mathcal{S}_h^{\Gamma_D})$ be approximations~of~${f\in L^2(\Omega)}$, $g\in W^{-\frac{1}{2},2}(\Gamma_N)$, and $u_D|_{\Gamma_D}\in W^{\frac{1}{2},2}(\Gamma_D)$. Then, the discrete problem is given via the minimization of the functional $I_h^{cr}\colon \mathcal{S}^{1,cr}(\mathcal{T}_h)\to \mathbb{R}\cup\{+\infty\}$, for every $v_h\in \mathcal{S}^{1,cr}(\mathcal{T}_h)$ defined by 
	\begin{align}
		\label{eq:elasto_discrete_primal}
		I_h^{cr}(v_h)\coloneqq \tfrac{1}{2}\|\nabla_h v_h\|_{2,\Omega}^2+\int_{\Omega}{\phi_h(\nabla_h v_h)\,\mathrm{d}x} -(f_h,\Pi_hv_h)_{\Omega}-(g_h,\pi_h v_h)_{\Gamma_N}+I_{\{u_D^h\}}^{\Gamma_D}(\pi_h v_h)\,.
	\end{align}
	Then, a corresponding (Fenchel) dual problem is given via the maximization of the functional $D_h^{rt}\colon \mathcal{R}T^0(\mathcal{T}_h)\to \mathbb{R}\cup\{-\infty\}$,
	for every  $ y_h\in \mathcal{R}T^0(\mathcal{T}_h)$ defined by 
	\begin{align}\label{eq:elasto_discrete_dual}
		D_h^{rt}(y_h)\coloneqq -\int_\Omega{\phi^*_h(\Pi_hy_h)\,\mathrm{d}x}  +( y_h\cdot n,u_D^h)_{\Gamma_D}-I_{\{-f_h\}}^{\Omega}(\textup{div}\,y_h)-I_{\{g_h\}}^{\Gamma_N}(y_h\cdot n)\,.
	\end{align}
	Proceeding as in the proof Proposition \ref{prop:gen_marini}(ii),  given a discrete dual solution $z_h^{rt}\in  \mathcal{R}T^0(\mathcal{T}_h)$, a discrete primal solution $u_h^{cr}\in \mathcal{S}^{1,cr}(\mathcal{T}_h)$ is immediately available via the generalized~Marini~\mbox{formula}
	\begin{align}\label{eq:elasto_marini}
		u_h^{cr}=\lambda_h^{rt}+D\phi^*_h(\Pi_h z_h^{rt})\cdot (\textup{id}_{\mathbb{R}^d}-\Pi_h\textup{id}_{\mathbb{R}^d})\quad\text{ a.e.\  in }\Omega\,,
	\end{align} 
	where $\lambda_h^{rt}\in \mathcal{L}^0(\mathcal{T}_h)$ is such that for every $(y_h,\mu_h)^\top\in \mathcal{R}T^0(\mathcal{T}_h)\times \mathcal{L}^0(\mathcal{T}_h)$, it holds that
	\begin{align*}
		(D\phi_h^*(\Pi_hz_h^{rt}),\Pi_hy_h)_{\Omega}+(\lambda_h^{rt},\textup{div}\,y_h)_{\Omega}&=0\,,\\
		(\textup{div}\,z_h^{rt},\mu_h)_{\Omega}+(f_h,\mu_h)_{\Omega}&=0\,.
	\end{align*}
	
	\newpage

	\newpage
	\subsection{Stokes' problem}\vspace*{-0.5mm}\label{subsec:stokes}
	
	\subsubsection{Continuous problem}\vspace*{-0.5mm}\enlargethispage{12mm}
	
	\hspace*{5mm}Stokes' problem can be formulated as a minimization problem over divergence-free velocity fields (\textit{cf}.\ \cite{Boffi2008}): for an external force $f\in L^2(\Omega;\mathbb{R}^d)$, Neumann boundary data  $g\in W^{-\frac{1}{2},2}(\Gamma_N;\mathbb{R}^d)$, and Dirichlet boundary data $u_D\in W^{\frac{1}{2},2}(\partial\Omega;\mathbb{R}^d)$, the \textit{Stokes minimization problem} is defined via the minimization of the functional  $I\colon U_d^2(\Omega)\to \mathbb{R}\cup\{+\infty\}$, for every $v\in U_d^2(\Omega)$~defined~by 
	\begin{align}\label{eq:stokes_primal}
		I(v)\coloneqq \tfrac{1}{2}\| \nabla v\|_{2,\Omega}^2+I_{\{0\}}^{\Omega}(\textup{tr}\,\nabla v)-(f, v)_{\Omega}-\langle g,v\rangle_{\Gamma_N}+I_{\{u_D\}}^{\Gamma_D}(v)\,.
	\end{align}
	where the indicator functional $I_{\{0\}}^{\Omega}\colon L^2(\Omega)\to \mathbb{R}\cup\{+\infty\}$ for every $\widehat{v}\in L^2(\Omega)$ is defined by\vspace*{-0.5mm}
	\begin{align*}
		I_{\{0\}}^{\Omega}(\widehat{v})\coloneqq \begin{cases}
			0&\text{ if }\widehat{v}=0\text{ a.e.\ in }\Omega\,,\\
			+\infty&\text{ else}\,.
		\end{cases}
	\end{align*}
	If the function $\phi\colon \mathbb{R}^{d\times d}\to \mathbb{R}\cup\{+\infty\}$ is defined by 	$\phi(R)\coloneqq\frac{1}{2}\vert R\vert^2-I_{\{0\}}^{\Omega}(\textup{tr}\,R)$ for all ${R\in\mathbb{R}^{d\times d}}$, then $\phi^*\colon   \mathbb{R}^{d\times d}\to \mathbb{R}\cup\{+\infty\}$ for every $S\in \mathbb{R}^{d\times d}$ is found to be given via
	\begin{align*}
		\phi^*(S)=\tfrac{1}{2}\vert \textup{dev}\,S\vert^2\,.
	\end{align*}
	Then, a corresponding (Fenchel) dual problem is given via the maximization of the functional $D\colon Z_d^2(\Omega)\to \mathbb{R}\cup\{-\infty\}$, for every $y\in Z_d^2(\Omega)$ defined by\vspace*{-0.5mm}
	\begin{align}\label{eq:stokes_dual}
		D(y)\coloneqq -\tfrac{1}{2}\|\textup{dev}\, 	y\|_{2,\Omega}^2+\langle yn,u_D\rangle_{\partial\Omega}-\langle g,u_D\rangle_{\Gamma_N}-I_{\{-f\}}^{\Omega}(\textup{div}\,y)-I_{\{g\}}^{\Gamma_N}(yn)\,.
	\end{align}
	As a result, given the two functionals \eqref{eq:stokes_primal} and \eqref{eq:stokes_dual}, 
	the corresponding primal-dual gap estimator $\eta_{\textup{gap}}^2\colon U_d^2(\Omega)\times Z_d^2(\Omega)\to \mathbb{R}$, for every 
	$v\in U_d^2(\Omega)$ with $v=u_D$ a.e.\ on~$\Gamma_D$ and $\textup{tr}\,\nabla v=0$ a.e.\ in $\Omega$ and $y\in Z_d^2(\Omega)$ with $I_{\{g\}}^{\Gamma_N}(yn)=0$ and $\textup{div}\,y=-f$ a.e.\ in $\Omega$, 
	using the integration-by-parts formula \eqref{eq:pi_cont} and that $(\textup{dev}\,y,\nabla v)_{\Omega}=(y,\nabla v)_{\Omega}$, 
	is~given~via
	\begin{align*}
		\eta_{\textup{gap}}^2(v,y)=\tfrac{1}{2}\|\nabla v-y \|_{2,\Omega}^2\,,
	\end{align*}
	measuring the violation of the optimality condition $y=\nabla v$ a.e.\ in $\Omega$.

	\subsubsection{Discrete problem}
	
	\hspace*{5mm}Let $f_h\hspace*{-0.12em}\in \hspace*{-0.12em}(\mathcal{L}^0(\mathcal{T}_h))^d$, $g_h\hspace*{-0.12em}\in\hspace*{-0.12em} (\mathcal{L}^0(\mathcal{S}_h^{\Gamma_N}))^d$, and $u_D^h\hspace*{-0.12em}\in\hspace*{-0.12em} (\mathcal{L}^0(\mathcal{S}_h^{\Gamma_D}))^d$ be~\mbox{approximations}~of~${f\hspace*{-0.12em}\in\hspace*{-0.12em} L^2(\Omega;\mathbb{R}^d)}$, $g\in W^{-\frac{1}{2},2}(\Gamma_N;\mathbb{R}^d)$, and $u_D|_{\Gamma_D}\in W^{\frac{1}{2},2}(\Gamma_D;\mathbb{R}^d)$. Then, the \textit{discrete Stokes minimization problem} is given via the minimization of 
	the functional $I_h^{cr}\colon U_d^{cr}(\mathcal{T}_h)\to \mathbb{R}\cup\{+\infty\}$, for every $v_h\in U_d^{cr}(\mathcal{T}_h)$ defined by 
	\begin{align*}
		I_h^{cr}(v_h)\coloneqq \tfrac{1}{2}\| \nabla_h v_h\|_{2,\Omega}^2+I_{\{0\}}^{\Omega}(\textup{tr}\,\nabla_h v_h)-(f_h,\Pi_h v_h)_{\Omega}-(g_h,\pi_hv_h)_{\Gamma_N}+I_{\{u_D^h\}}^{\Omega}(\pi_hv_h)\,.
	\end{align*}
	Then, a corresponding (Fenchel) dual problem is given via the maximization of the functional $D_h^{rt}\colon Z_d^{rt}(\mathcal{T}_h)\to \mathbb{R}\cup\{-\infty\}$, for every $y_h\in Z_d^{rt}(\mathcal{T}_h)$ defined by
	\begin{align*}
		D_h^{rt}(y_h)\coloneqq -\tfrac{1}{2}\|\Pi_h \textup{dev}\, 	y_h\|_{2,\Omega}^2+
		(y_hn, u_D^h)_{\Gamma_D} -I_{\{-f_h\}}^{\Omega}(\textup{div}\,y_h)-I_{\{g_h\}}^{\Gamma_N}(y_hn)\,.
	\end{align*} 
	Appealing to Proposition \ref{prop:gen_marini}(i), given a discrete primal solution $u_h^{cr}\in U_d^{cr}(\mathcal{T}_h)$, a discrete dual solution $z_h^{rt}\in Z_d^{rt}(\mathcal{T}_h)$ is immediately available  via the generalized Marini formula\vspace*{-1mm}
	\begin{align*}
		z_h^{rt}=\nabla_h u_h^{cr}-\frac{1}{d} f_h\otimes(\textup{id}_{\mathbb{R}^d}-\Pi_h\textup{id}_{\mathbb{R}^d})\quad\text{ a.e.\  in }\Omega\,.
	\end{align*}
	Note that, different from the previous model problem, an admissible approximation $\overline{u}_h^{cr}\in U_d^2(\Omega)$,
	\textit{i.e.}, $\overline{u}_h^{cr}=u_D$ a.e.\ on $\Gamma_D$ and $\textup{tr}\,\nabla\overline{u}_h^{cr}=0$ a.e.\ in $\Omega$, 
	cannot be obtained via simple node-averaging since, in general, we have that $\vert \{\textup{tr}\,\nabla \Pi_h^{av}u_h^{cr}\neq 0\}\vert>0$, although, by construction, we have that
	$\textup{tr}\,\nabla_h u_h^{cr}=0$ a.e.\ in $\Omega$. Instead an approximation $\overline{u}_h^{cr}\in U_d^2(\Omega)$ can be obtained via node-averaging  combined with a local divergence-correction procedure (\textit{cf}.\ \cite{zanotti_stokes}), in which one solves local discrete Stokes problems in  finite element spaces with higher polynomial degree. Since these local problems can be solved in parallel, the overall cost of the  local divergence-correction~procedure~is~moderate.
	\newpage

	\section{Equivalence to residual type error estimators}\label{sec:equiv_residuals}
	
	\hspace*{5mm}In the case of the \textit{$\varphi$-Dirichlet problem}, \textit{i.e.}, if we have that $\phi\coloneqq\varphi\circ \vert \cdot\vert\in C^0(\mathbb{R}^d)\cap C^2(\mathbb{R}^d\setminus\{0\})$, where $\varphi\colon \mathbb{R}_{\ge 0}\to \mathbb{R}_{\ge 0}$ is an $N$-function (\textit{cf}.\  Appendix \ref{subsec:auxiliary}) that satisfies the following conditions (\textit{cf}.\  \cite[Assumption 1]{DK08}):
	\begin{itemize}[noitemsep,topsep=2pt,leftmargin=!,labelwidth=\widthof{(C.4)},font=\itshape]
		\item[(C.1)]\hypertarget{C.1}{}  $\varphi$ satisfies $\Delta_2$-condition (\textit{i.e.}, $\Delta_2(\varphi)<\infty$) and the $\nabla_2$-condition (\textit{i.e.}, $\Delta_2(\varphi^*)<\infty$);
		\item[(C.2)]\hypertarget{C.2}{}  $\varphi\in C^2(0,\infty)$ and uniformly with respect to $t\ge 0$, it holds that\footnote{Here, we employ the notation $f\sim g $ for two (Lebesgue) measurable functions $f,g\colon\Omega\to \mathbb{R}$, if there exists a constant $c>0$ such that $c^{-1}f\leq g\leq cf$ almost everywhere in $\Omega$.} 
		\begin{align*}
			\varphi'(t)\sim t\,\varphi''(t)\,;
		\end{align*}
	\end{itemize}
	and if we have that $\psi(x,\cdot)\coloneqq (s\mapsto -f(x)s)\hspace*{-0.05em}\in\hspace*{-0.05em} C^1(\mathbb{R})$ for a.e.\ $x\in \Omega$, assuming that ${f\hspace*{-0.05em}=\hspace*{-0.05em}f_h\hspace*{-0.05em}\in\hspace*{-0.05em} \mathcal{L}^0(\mathcal{T}_h)}$ and $g=g_h=0$,
	we can
	relate the primal-dual gap estimator to the  residual type~\mbox{estimator}~in~\cite{DK08}, which, in fact, coincides with the standard residual type estimator for the Poisson~problem~\eqref{intro:poisson} (\textit{i.e.}, \hspace*{-0.1mm}$\varphi(t)\hspace*{-0.175em}\coloneqq\hspace*{-0.175em}\frac{1}{2} t^2$ \hspace*{-0.1mm}for \hspace*{-0.1mm}ll \hspace*{-0.1mm}$t\hspace*{-0.175em}\ge\hspace*{-0.175em} 0$).
 \hspace*{-0.1mm}If \hspace*{-0.1mm}$u_h^{p1}\hspace*{-0.175em}\in\hspace*{-0.175em} \mathcal{S}^1_D(\mathcal{T}_h)$ \hspace*{-0.1mm}is \hspace*{-0.1mm}the \hspace*{-0.1mm}unique~\hspace*{-0.1mm}minimizer~\hspace*{-0.1mm}of~\hspace*{-0.1mm}${I_h^{p1}\hspace*{-0.175em}\coloneqq\hspace*{-0.175em} \smash{I|_{\smash{\mathcal{S}^1_D(\mathcal{T}_h)}}}\colon \hspace*{-0.175em}\mathcal{S}^1_D(\mathcal{T}_h)\hspace*{-0.175em}\to\hspace*{-0.175em} \mathbb{R}}$, the residual type~\mbox{estimator} denotes the quantity
	\begin{align}
		\begin{aligned}
			\eta^2_{\textup{res},h}(u_h^{p1})\coloneqq \sum_{T\in \mathcal{T}_h}{\eta_{\textup{res},T}^2(u_h^{p1})}\,,
			\end{aligned}\label{thm:residual.1.1}
	\end{align}
	where, if $h_S\coloneqq \textup{diam}(S)$ for all $S\in \mathcal{S}_h$, for every $T\in \mathcal{T}_h$ and $S\in \mathcal{S}_h^{i}$ with $S\subseteq \partial T$,
	\begin{align}
		\begin{aligned}
			\eta_{\textup{res},T}^2(u_h^{p1})&\coloneqq \eta_{E,T}^2(u_h^{p1})+\sum_{S\in \mathcal{S}_h^{i}\,:\,S\subseteq \partial T}{\eta_{J,S}^2(u_h^{p1})}\,,\\
			\eta_{E,T}^2(u_h^{p1})&\coloneqq \|(\varphi_{\vert \nabla \smash{u_h^{p1}}\vert})^*(h_{\mathcal{T}}\vert f_h\vert)\|_{1,T}\,,\\
			\eta_{J,S}^2(u_h^{p1})&\coloneqq \big\|h_S^{\frac{1}{2}} \jump{F(\nabla u_h^{p1})}_S\big\|_{2,S}^2\,.\end{aligned}\label{thm:residual.1.2}
	\end{align}
	In \eqref{thm:residual.1.2},  for every $a\hspace*{-0.1em}\ge\hspace*{-0.1em} 0$, the function $(\varphi_a)^*\colon\hspace*{-0.1em}\mathbb{R}_{\ge 0}\hspace*{-0.1em}\to \hspace*{-0.1em}\mathbb{R}_{\ge 0}$ is the Fenchel conjugate of ${\varphi_a\colon\hspace*{-0.1em}\mathbb{R}_{\ge 0}\hspace*{-0.1em}\to\hspace*{-0.1em} \mathbb{R}_{\ge 0}}$ (\textit{cf}.\ Appendix~\ref{subsec:auxiliary}) and the function $F\colon \mathbb{R}^d\to \mathbb{R}^d$ for every $r\in \mathbb{R}^d$ is defined by
	\begin{align*}
		F(r)\coloneqq \sqrt{\frac{\varphi'(\vert r\vert )}{\vert r\vert}}r\,.
	\end{align*}
	In \cite[Lem.\ 8 \& Cor.\ 11]{DK08}, it has been shown that the error estimator \eqref{thm:residual.1.1} is reliable and efficient with respect to the primal approximation error, \textit{i.e.}, there exist constants $ c_{\textup{rel}},c_{\textup{eff}}>0$ such that
	\begin{align}
	c_{\textup{rel}}\,\|F(\nabla u_h^{p1})-F(\nabla u)\|_{2,\Omega}^2 \leq \eta_{\textup{res},h}^2(u_h^{p1})\leq c_{\textup{eff}}\,\|F(\nabla u_h^{p1})-F(\nabla u)\|_{2,\Omega}^2\,.\label{eq:rel-eff}
	\end{align}
	
	Generalizing the procedure in  \cite{BreS08,Gud10A,Gud10B,BKAFEM22} and resorting to  properties of the node-averaging quasi-interpolation operator $\Pi_h^{av}\colon\mathcal{S}^{1,cr}_D(\mathcal{T}_h)\to  \mathcal{S}^1_D(\mathcal{T}_h)$ (\textit{cf}.\ Appendix \ref{subsec:node_average}), we are able to establish the global equivalence of the primal-dual gap estimator \eqref{eq:nl_dirichlet_eta} in the case $v= \smash{u_h^{p1}}\in \mathcal{S}^1_D(\mathcal{T}_h)$ and $y=z_h^{rt}\in  \mathcal{R}T^0_N(\mathcal{T}_h)$ to the residual type estimator  \eqref{thm:residual.1.1}.\enlargethispage{10mm}
	
	%\if0
	\begin{theorem}\label{thm:equivalences} Let   $\phi=\varphi\circ \vert \cdot\vert\in C^0(\mathbb{R}^d)\cap C^2(\mathbb{R}^d\setminus\{0\})$, where $\varphi\colon \mathbb{R}_{\ge 0}\to \mathbb{R}_{\ge 0}$ is an $N$-function satisfying (\hyperlink{C.1}{C.1}), (\hyperlink{C.2}{C.2}),   $\psi(x,\cdot)\coloneqq (s\mapsto -f(x)s)\in C^1(\mathbb{R})$ for a.e.\ $x\in \Omega$~for~$f=f_h\in \mathcal{L}^0(\mathcal{T}_h)$, and let $g=g_h=0$.  Then, it holds that
		\begin{align} \eta_{\textup{res},h}^2(\smash{u_h^{p1}}) \sim \smash{\eta^2_{\textup{gap}}}(\smash{u_h^{p1}},z_h^{rt})=\int_{\Omega}{\big\{\phi(\nabla \smash{u_h^{p1}})-\nabla \smash{u_h^{p1}}\cdot z_h^{rt}+\phi^*(z_h^{rt})\big\}\,\mathrm{d}x } \,,
			\label{eq:equivalence}
		\end{align}
		where the equivalence $\sim$ depends only on $\omega_0$, $\Delta_2(\varphi)$, and $\nabla_2(\varphi)$.
	\end{theorem}

	\begin{remark}
		\label{rem:equivalences}
		Theorem \ref{thm:equivalences} extends the results in \cite{DK08} by the aspect that the residual type estimator \eqref{thm:residual.1.1} is not only equivalent to the primal approximation (\textit{i.e.}, to $\rho^2_I(u_h^{p1},u)$), but to the primal approximation plus the dual approximation error  (\textit{i.e.}, to $\rho_{\textup{tot}}^2(u_h^{p1},z_h^{rt})\coloneqq\rho^2_I(u_h^{p1},u)+\rho^2_{-D}(z_h^{rt},z)$). In~other~words, the residual type estimator \eqref{thm:residual.1.1} also provides control of the approximation error of the Raviart--Thomas approximation of the dual problem. 
	\end{remark}
	
	\begin{proof}[Proof of  (Theorem \ref{thm:equivalences}).]\let\qed\relax
		
	\textit{ad  $\eta_{\textup{gap}}^2(\smash{u_h^{p1}},z_h^{rt}) \leq c\, \eta_{\textup{res},h}^2(\smash{u_h^{p1}})$.} Using the discrete optimality~\mbox{relation} \eqref{eq:discrete_optimality_relations.1} (which is equivalent to \eqref{eq:discrete_optimality_relations.3}), we find that
	\begin{align}\label{thm:residual.1}
		\left.\begin{aligned}
		\phi(\nabla u_h^{p1})&-\nabla u_h^{p1}\cdot \Pi_h z_h^{rt}+\phi^*(z_h^{rt})\\&= \phi(\nabla u_h^{p1})-D\phi(\nabla_h u_h^{cr})\cdot (\nabla u_h^{p1}-\nabla_h u_h^{cr})+\phi(\nabla_hu_h^{cr})
		\\&\quad+	\phi^*(z_h^{rt})-\phi^*(\Pi_h z_h^{rt})\,.\end{aligned}\right\}\quad\text{ a.e.\ in }\Omega\,.
	\end{align}
	On the other hand, by the convextity of $\phi,\phi^*\in C^1(\mathbb{R}^d)$, we have that
	\begin{align}\label{thm:residual.2}
		\begin{aligned}
			-\phi(\nabla_hu_h^{cr})&\le -\phi(\nabla u_h^{p1})+D\phi(\nabla u_h^{p1})\cdot (\nabla u_h^{p1}-\nabla_h u_h^{cr})&&\quad\text{ a.e.\ in }\Omega\,,\\
			-\phi^*(\Pi_h z_h^{rt})&\le -\phi^*( z_h^{rt})+D\phi^*( z_h^{rt})\cdot ( z_h^{rt}-\Pi_h z_h^{rt})&&\quad\text{ a.e.\ in }\Omega\,,
		\end{aligned}
	\end{align}
	Therefore, using \eqref{thm:residual.1} and \eqref{thm:residual.2} together with $D\phi^*(\Pi_h z_h^{rt})\hspace*{-0.2em}\perp \hspace*{-0.2em} (z_h^{rt}-\Pi_h z_h^{rt})$ in $L^2(\Omega;\mathbb{R}^d)$,~we~find~that
	\begin{align}
		\label{thm:residual.3}
		\begin{aligned}
		\smash{\eta^2_{\textup{gap}}}(\smash{u_h^{p1}},z_h^{rt})&\leq (D\phi(\nabla u_h^{p1})-D\phi(\nabla_h u_h^{cr}),\nabla u_h^{p1}-\nabla_h u_h^{cr})_\Omega\\
		&\quad+(D\phi^*(z_h^{rt})-D\phi^*(\Pi_h z_h^{rt}),z_h^{rt}-\Pi_h z_h^{rt})_\Omega
		\\&\eqqcolon I_h^1+I_h^2\,.
			\end{aligned}
	\end{align}
	So, it is enough to estimate $I_h^1$ and $I_h^2$:
	
	\textit{ad $I_h^1$.} 
	Abbreviating $e_h\coloneqq \smash{u_h^{p1}}-u_h^{cr}\in \mathcal{S}^{1,cr}_D(\mathcal{T}_h)$ and using  Galerkin orthogonality of the continuous and the discrete primal problem, we find that
	\begin{align}\begin{aligned}\label{thm:residual.4}
		I_h^1&= 
			(D\phi(\nabla \smash{u_h^{p1}}),\nabla_h( e_h-  \Pi_h^{av} e_h) )_\Omega	+(f_h, \Pi_h^{av} e_h-e_h )_\Omega
		%	\\&\quad+(D\phi(\nabla \smash{u_h^{p1}})-D\phi(\nabla u) ,\nabla  \Pi_h^{av} e_h)_\Omega
			\\&=\vcentcolon I_h^{1,1}+I_h^{1,2}
			%+I_h^{1,3}
			\,.
		\end{aligned}
	\end{align} 
	Let us next estimate $\smash{I_h^{1,1}}$ and $\smash{I_h^{1,2}}$:%, and $\smash{I_h^{1,3}}$:
	
	\textit{ad $I_h^{1,1}$.}
	Using that $\jump{D\phi(\nabla \smash{u_h^{p1}})\hspace*{-0.1em}\cdot \hspace*{-0.1em}n( e_h\hspace*{-0.1em}-\hspace*{-0.1em}\Pi_h^{av}e_h)}_S\hspace*{-0.15em}=\hspace*{-0.15em}\jump{D\phi(\nabla \smash{u_h^{p1}})\hspace*{-0.1em}\cdot \hspace*{-0.05em}n}_S\{e_h\hspace*{-0.05em}-\hspace*{-0.05em} \Pi_h^{av} e_h\}_S +\{D\phi(\nabla \smash{u_h^{p1}})\cdot n\}_S \jump{e_h-\Pi_h^{av} e_h}_S$ on $S$, $\pi_h\jump{e_h- \Pi_h^{av} e_h}_S=0$ and $\{D\phi(\nabla \smash{u_h^{p1}})\}_S=\textup{const}$ on $S$ for all ${S\in \mathcal{S}_h^{i}}$, an element-wise \hspace*{-0.1mm}integration-by-parts, \hspace*{-0.1mm}a \hspace*{-0.1mm}discrete \hspace*{-0.1mm}trace \hspace*{-0.1mm}inequality \hspace*{-0.1mm}(\textit{cf}.\ \hspace*{-0.1mm}\cite[\hspace*{-0.5mm}Lem.\ \hspace*{-0.5mm}12.8]{EG21}),~\hspace*{-0.1mm}and~\hspace*{-0.1mm}\mbox{Proposition} \ref{cor:n-function}, 
	denoting for every $S\in \mathcal{S}_h$ by $\omega_S\coloneqq \bigcup\{T\in \mathcal{T}_h\mid S\subseteq \partial T\}$ the side patch and for every $T\in \mathcal{T}_h$ by $\omega_T\coloneqq \bigcup\{T'\in \mathcal{T}_h\mid T\cap T'\neq \emptyset\}$  the element patch,
	we find that\enlargethispage{10mm}
	\begin{align}\label{thm:residual.5}
		\begin{aligned}
			I_h^1&=(\jump{D\phi(\nabla \smash{u_h^{p1}})\cdot n},\{e_h- \textcolor{black}{ \Pi_h^{av}} e_h\})_{\mathcal{S}_h^{i}} 
			\\&\leq 
			\sum_{S\in \mathcal{S}_h^{i}}{\vert \jump{D\phi(\nabla \smash{u_h^{p1}})\cdot n}_S\vert \|\{e_h- \Pi_h^{av} e_h\}_S\|_{1,S}} 
			\\&\leq c\,\sum_{S\in \mathcal{S}_h^{i}}{\vert \jump{D\phi(\nabla \smash{u_h^{p1}})\cdot n}_S\vert h_S^{-1}\|e_h- \Pi_h^{av} e_h\|_{1,\omega_S}} \\&
			\leq c\,\sum_{S\in \mathcal{S}_h^{i}}{  \sum_{T\in \mathcal{T}_h\,:\,T\subseteq \omega_S}{\|\vert \jump{D\phi(\nabla \smash{u_h^{p1}})\cdot n}_S\vert\nabla_h e_h\|_{1,\omega_T}}}\,.
		\end{aligned}	
	\end{align}
	Then, for every $T\in \mathcal{T}_h$, using in $\omega_T$, the $\varepsilon$-Young inequality (\textit{cf}.\  \eqref{eq:eps-young}) for  $\smash{\varphi_{\vert\nabla \smash{u_h^{p1}}(T)\vert}\colon \mathbb{R}_{\ge 0}\to \mathbb{R}_{\ge 0}}$ 
	and $(\varphi_{\vert \nabla \smash{u_h^{p1}}(T)\vert})^*(\vert \jump{D\phi(\nabla \smash{u_h^{p1}})\cdot n}_S\vert)\hspace*{-0.1em}\sim \hspace*{-0.1em}\vert \jump{F(\nabla \smash{u_h^{p1}})}_S\vert^2$~for~all~$ S\hspace*{-0.1em}\in\hspace*{-0.1em}\mathcal{S}_h^{i}$ with $S\hspace*{-0.1em}\subseteq\hspace*{-0.1em} \partial  T$ (\textit{cf}.\ \mbox{\cite[Cor.~6]{DK08}}), where we write $\nabla \smash{u_h^{p1}}(T)$ to indicate that the shift on $\omega_T$ depends only on the value of $\nabla \smash{u_h^{p1}}$~on~$T$,  from \eqref{thm:residual.5}, for every $\varepsilon>0$, we deduce~that
	\begin{align}\label{thm:residual.6}
		\begin{aligned}
			I_h^{1,1}&\leq c\,\sum_{S\in \mathcal{S}_h^{i}}{  \sum_{T\in \mathcal{T}_h\,:\,T\subseteq \omega_S}{
					c_\varepsilon\,\|(\varphi_{\vert \nabla \smash{u_h^{p1}}(T)\vert})^*(\vert \jump{D\phi(\nabla \smash{u_h^{p1}})\cdot n}_S\vert)\|_{1,\omega_T}}}\\&\quad+
			\varepsilon\,c\,\sum_{S\in \mathcal{S}_h^{i}}{  \sum_{T\in \mathcal{T}_h\,:\,T\subseteq \omega_S}{\|\varphi_{\vert\nabla \smash{u_h^{p1}}(T)\vert}(\vert \nabla_ he_h\vert)\|_{1,\omega_T}}}
			\\
			&
			\leq c_\varepsilon\,\sum_{S\in \mathcal{S}_h^{i}}{\eta_{J,S}^2(u_h^{p1})}+ \varepsilon\,c\, \sum_{T\in \mathcal{T}_h}{ \|\varphi_{\vert\nabla \smash{u_h^{p1}}(T)\vert}(\vert \nabla_ he_h\vert)\|_{1,\omega_T}}\,.
		\end{aligned}
	\end{align}
	
	\textit{ad $I_h^{1,2}$.}
	Using the $\varepsilon$-Young inequality (\textit{cf}.\ \eqref{eq:eps-young})  and Proposition \ref{cor:n-function}, for every ${\varepsilon\hspace*{-0.1em}>\hspace*{-0.1em}0}$,~we~\mbox{obtain} 
	\begin{align}\label{thm:residual.7}
		\begin{aligned}
			I_h^{1,2}&\leq c_\varepsilon\,\sum_{T\in \mathcal{T}_h}{\|(\varphi_{\vert \nabla \smash{u_h^{p1}}\vert})^*(h_T\vert f_h\vert)\|_{1,T}}\\&\quad+\varepsilon\,\sum_{T\in \mathcal{T}_h}{\|\varphi_{\vert \nabla \smash{u_h^{p1}}\vert }(h_T^{-1}\vert e_h- \Pi_h^{av}e_h\vert )\|_{1,T}}\\&\leq 
			c_\varepsilon\,\sum_{T\in \mathcal{T}_h}{\eta_{E,T}^2(u_h^{p1})}+ \varepsilon\,c\,\sum_{T\in \mathcal{T}_h}{\|\varphi_{\vert \nabla \smash{u_h^{p1}}(T)\vert }(\vert \nabla_he_h\vert )\|_{1,\omega_T}}\,.
		\end{aligned} 
	\end{align}
	
%	\textit{ad $I_h^{1,3}$.}
%	The $\varepsilon$-Young inequality (\textit{cf}.\ \eqref{eq:eps-young}), the equivalence $(\varphi_{\vert\nabla \smash{u_h^{p1}}\vert })^*(\vert D\phi(\nabla \smash{u_h^{p1}})-D\phi(\nabla u)\vert )$ $\sim \vert F(\nabla \smash{u_h^{p1}})- F(\nabla u)\vert^2$ (\textit{cf}.\ \cite[Cor.\ 6]{DK08}), and Proposition \ref{cor:n-function}, for every ${\varepsilon>0}$, yield that\enlargethispage{11.5mm}
%	\begin{align}\label{thm:residual.8}
%		\begin{aligned}
%			I_h^{1,3}&\leq c_\varepsilon\,\sum_{T\in \mathcal{T}_h}{\|(\varphi_{\vert\nabla \smash{u_h^{p1}}\vert })^*(\vert D\phi(\nabla \smash{u_h^{p1}})-D\phi(\nabla u)\vert )\|_{1,T}}\\&\quad+\varepsilon\,\sum_{T\in \mathcal{T}_h}{\|\varphi_{\vert\nabla \smash{u_h^{p1}}\vert }(\vert \nabla \Pi_h^{av}e_h\vert )\|_{1,T}}\\&\leq 
%			c_\varepsilon\,\|F(\nabla \smash{u_h^{p1}})-F(\nabla u)\|_{2,\Omega}^2 
%			+ \varepsilon\,c\sum_{T\in \mathcal{T}_h}{\|\varphi_{\vert\nabla \smash{u_h^{p1}}(T)\vert }(\vert \nabla_h e_h\vert )\|_{1,\omega_T}}\,.
%		\end{aligned}
%	\end{align}
	\hspace*{-5mm}Combining \eqref{thm:residual.6} and \eqref{thm:residual.7} in \eqref{thm:residual.4}, using \eqref{eq:rel-eff} in doing so,  for every $\varepsilon>0$,  we conclude that
	\begin{align}\label{thm:residual.9}
			I_h^1\leq c_\varepsilon\,\eta_{\textup{res},h}^2(\smash{u_h^{p1}})+\varepsilon\,c\sum_{T\in \mathcal{T}_h}{\|\varphi_{\vert\nabla \smash{u_h^{p1}}(T)\vert }(\vert \nabla_h e_h\vert )\|_{1,\omega_T}} \,.
	\end{align}
	 Proceeding as in \cite[p.\ 9 \& 10]{DK08}, we find that
	\begin{align}\label{thm:residual.10}
		\sum_{T\in \mathcal{T}_h}{\|\varphi_{\vert \nabla \smash{u_h^{p1}}(T)\vert }(\vert \nabla_he_h\vert )\|_{1,\omega_T}}\leq c\, \sum_{T\in \mathcal{T}_h}{\|\varphi_{\vert \nabla \smash{u_h^{p1}}\vert }(\vert \nabla_he_h\vert )\|_{1,\omega_T}}+c\,\sum_{S\in \mathcal{S}_h^{i}}{\eta_{J,S}^2(u_h^{p1})}\,.
	\end{align}
	Therefore, using \eqref{thm:residual.10} in \eqref{thm:residual.9} together with the equivalence chain 
	\begin{align}\label{eq:second_equiv}
		\left.\begin{aligned}
				\varphi_{\vert \nabla \smash{u_h^{p1}}\vert }(\vert \nabla_he_h\vert )&\sim \vert F(\nabla \smash{u_h^{p1}})-F(\nabla_h u_h^{cr})\vert^2\\&\sim
			(D\phi(\nabla \smash{u_h^{p1}})-D\phi(\nabla_h u_h^{cr}))\cdot(\nabla \smash{u_h^{p1}}-\nabla_h u_h^{cr})
		\end{aligned}\right\}\quad\text{ a.e.\ in }\Omega\,,
	\end{align}
%	$\varphi_{\vert \nabla \smash{u_h^{p1}}\vert }(\vert \nabla_he_h\vert )\sim \vert F(\nabla \smash{u_h^{p1}})-F(\nabla_h u_h^{cr})\vert^2\sim
%	(D\phi(\nabla \smash{u_h^{p1}})-D\phi(\nabla_h u_h^{cr}))\cdot(\nabla \smash{u_h^{p1}}-\nabla_h u_h^{cr})$ a.e.\ in $\Omega$, 
	for every $\varepsilon>0$,  we arrive at
	\begin{align}\label{thm:residual.11} 
	I_h^1 \leq c_\varepsilon\,\eta_{\textup{res},h}^2(\smash{u_h^{p1}})+\varepsilon\,c\,I_h^1\,. 
	\end{align}
	Resorting the reconstruction formula \eqref{eq:nl_dirichlet_marini} and \cite[Lem.\ 3]{DK08}, with $F^*\colon \mathbb{R}^d\to \mathbb{R}^d$ for every $r\in \mathbb{R}^d$ defined by\vspace*{-1mm}
	\begin{align*}
			F^*(r)\coloneqq \sqrt{\frac{(\varphi^*)'(\vert r\vert )}{\vert r\vert}}r\,,
	\end{align*}
	and the equivalence  
	\begin{align*}
		\vert F^*(z_h^{rt})-F^*(\Pi_hz_h^{rt})\vert^2\sim
		(D\phi^*(z_h^{rt})-D\phi^*(\Pi_hz_h^{rt}))\cdot(z_h^{rt}-\Pi_hz_h^{rt})\quad\text{ a.e.\ in }\Omega\,,
	\end{align*}
	%$\vert F^*(z_h^{rt})-F^*(\Pi_hz_h^{rt})\vert^2\sim 	(D\phi^*(z_h^{rt})-D\phi^*(\Pi_hz_h^{rt}))\cdot(z_h^{rt}-\Pi_hz_h^{rt})$~a.e.~in~$\Omega$, 
	a change of shift (\textit{cf}.\ \cite[Cor.\ 28]{DK08}), and the second equivalence  in \eqref{eq:second_equiv}, for every ${\varepsilon\hspace*{-0.1em}>\hspace*{-0.1em}0}$,~we~find~that
	\begin{align}\label{thm:residual.12} 
		\begin{aligned} 
	I_h^2&\leq c\,
		\|(\varphi_{\vert \nabla_hu_h^{cr}\vert})^*(h_T\vert f_h\vert )\|_{1,\Omega}
		\\&
		\leq c_\varepsilon\,\eta_{\textup{res},h}^2(\smash{u_h^{p1}})+\varepsilon\, c\, I_h^1 \,. 
			\end{aligned}
	\end{align} 
	For $\varepsilon>0$ sufficiently small, from \eqref{thm:residual.11} and \eqref{thm:residual.12} in \eqref{thm:residual.3}, we conclude that
	\begin{align}
		\eta_{\textup{gap}}^2(\smash{u_h^{p1}},z_h^{rt})
		\leq c\,\eta_{\textup{res},h}^2(\smash{u_h^{p1}})\,. \label{thm:residual.13} 
	\end{align}	
	
	\textit{ad $\eta_{\textup{res},h}^2(\smash{u_h^{p1}})\leq c\,\eta_{\textup{gap}}^2(\smash{u_h^{p1}},z_h^{rt})$.}
	From Theorem \ref{thm:main}(i), \eqref{thm:residual.13}, and \eqref{eq:rel-eff} together with the equivalence 
	\begin{align*}
		\rho^2_I(\smash{u_h^{p1}},u)\sim  \|F(\nabla \smash{u_h^{p1}})-F(\nabla u)\|_{2,\Omega}^2\,,
	\end{align*}
	we conclude that
	\begin{align*}
		\eta_{\textup{res},h}^2(\smash{u_h^{p1}})&\leq c\,\rho^2_I(\smash{u_h^{p1}},u)
		\\&\leq  c\,\rho^2_{\textup{tot}}(\smash{u_h^{p1}},z_h^{rt})
		\\&\leq c\, \eta_{\textup{gap}}^2(\smash{u_h^{p1}},z_h^{rt})\,.\tag*{$\qedsymbol$}
	\end{align*} 
	\end{proof}
	\newpage

		\section{Quasi-optimality of node-averaging   operator}\label{sec:node_avg_optimal}
	
	\hspace*{5mm}Since, in general, for a discrete primal solution, we have that $u_h^{cr}\notin U_{\ell}^p(\Omega)$, it is necessary to post-process the discrete primal solution to obtain an admissible  approximation~$\overline{u}_h^{cr}\in U_{\ell}^p(\Omega)$. Here, it is convenient to enforce admissibility via the node-averaging quasi-interpolation operator $\Pi_h^{av}\colon \hspace*{-0.175em}\mathcal{S}^{1,cr}_D(\mathcal{T}_h)\hspace*{-0.175em}\to\hspace*{-0.175em} \mathcal{S}^1_D(\mathcal{T}_h)$ (\textit{cf}.\ Appendix \ref{subsec:node_average}), which satisfies  the following~local~\mbox{best-approximation} result.\enlargethispage{3mm}
	
	\begin{lemma}\label{lem:best-approx-inv}
		Let   $\phi=\varphi\circ \vert \cdot\vert\in C^0(\mathbb{R}^d)\cap C^2(\mathbb{R}^d\setminus\{0\})$, where $\varphi\colon \mathbb{R}_{\ge 0}\to \mathbb{R}_{\ge 0}$ is an $N$-function satisfying (\hyperlink{C.1}{C.1}), (\hyperlink{C.2}{C.2}).  Then,  there exists a constant $c>0$ depending only on $\omega_0$, $\Delta_2(\varphi)$, and $\nabla_2(\varphi)$, such that  for every $v_h\in \mathcal{S}^{1,cr}_D(\mathcal{T}_h)$ and $T\in \mathcal{T}_h$, it holds that
		\begin{align*}
			\|F(\nabla_hv_h)-F(\nabla \Pi_h^{av} v_h)\|_{2,T}^2&\leq c\,\inf_{v\in W^{1,1}_D(\Omega)\,:\,F(\nabla v)\in L^2(\Omega;\mathbb{R}^d)}{\|F(\nabla_hv_h)-F(\nabla v)\|_{2,\omega_T}^2}
			\\&\quad+c\,\big\|\smash{h_{\mathcal{S}}^{1/2}}\jump{F(\nabla_h v_h)}\big\|_{2,\mathcal{S}_h^{i}(T)}^2\,,
		\end{align*} 
		where $\mathcal{S}_h^{i}(T)\coloneqq \{S\in \mathcal{S}^{i}_h\mid S\cap T\neq \emptyset\}$ for all $T\in \mathcal{T}_h$ and 
		$h_{\mathcal{S}}|_S\coloneqq h_S$ for all $S\in \mathcal{S}_h$.
	\end{lemma}
	
	\begin{proof}
		Follows along the lines of the proof of \cite[Lem.\ 3.8]{K22CR} up to minor adjustments.
	\end{proof}
	
	Using Lemma \ref{lem:best-approx-inv}, in turn, we can deduce that the local distance of a node-averaged Crouzeix--Raviart function to a Sobolev function on an element is bounded by the local distance of the same Crouzeix--Raviart function to the same Sobolev function on an element patch plus an additive term quantifying the local fractional higher regularity of the Sobolev function. This justifies the usage of the node-averaging quasi-interpolation operator in local mesh~refinement~procedures.
	
	In order to express the fractional
	regularity of functions, we make use of Nikolski\u{\i} spaces. For
	given $p \in  [1, \infty)$, $\beta\in (0,1]$, an open set $G\subseteq \mathbb{R}^d$, $d\in \mathbb{N}$, and $v\in
	L^p(G)$, the  \textit{Nikolski\u{\i} semi-norm} is
	defined by
	\begin{align*}
		[v]_{N^{\beta,p}(G)}\coloneqq \sup_{h\in \mathbb{R}^d\setminus\{0\}}{\vert h\vert^{-\beta}\bigg(\int_{G\cap (G-h)}{\vert v(\cdot + h)-v\vert^p\,\mathrm{d}x}\bigg)^{\frac{1}{p}}}<\infty\,.
	\end{align*}
	Then, for $p \in  [1, \infty)$ and $\beta\in (0,1]$, the \textit{Nikolski\u{\i} space} is defined by
	\begin{align*}
		N^{\beta,p}(G)\coloneqq \big\{ v\in L^p(G)\mid [v]_{N^{\beta,p}(G)}<\infty\big\}\,,
	\end{align*}
	and the \textit{Nikolski\u{\i} norm} $\|\cdot\|_{N^{\beta,p}(G)}\coloneqq \| \cdot\|_{p,G}+[\cdot]_{N^{\beta,p}(G)}$ turns $N^{\beta,p}(G)$  into a Banach space.\enlargethispage{5mm}
	
	\begin{proposition}\label{prop:best-approx}
		Let   $\phi=\varphi\circ \vert \cdot\vert\in C^0(\mathbb{R}^d)\cap C^2(\mathbb{R}^d\setminus\{0\})$, where $\varphi\colon \mathbb{R}_{\ge 0}\to \mathbb{R}_{\ge 0}$ is an $N$-function satisfying (\hyperlink{C.1}{C.1}), (\hyperlink{C.2}{C.2}). Then, there exists a constant $c\hspace*{-0.15em}>\hspace*{-0.15em}0$ depending~only~on~$\omega_0$,~$\Delta_2(\varphi)$,~and~$\nabla_2(\varphi)$, such that
		for every $v_h\in\mathcal{S}^{1,cr}_D(\mathcal{T}_h)$, $v\in W^{1,1}_D(\Omega)$ with $F(\nabla v)\in L^2(\Omega;\mathbb{R}^d)$,~and~$T\in\mathcal{T}_h$,~it~holds~that 
		\begin{align}\label{prop:best-approx.1}
			\begin{aligned} 
				\|F(\nabla \Pi_h^{av} v_h)-F(\nabla v)\|_{2,T}^2&\leq c\,\|F(\nabla_h v_h)-F(\nabla v)\|_{2,\omega_T}^2+c\,\inf_{r\in \mathbb{R}^d}{\|F(\nabla v)-F(r)\|_{2,\omega_T}^2}\,,
			\end{aligned} 
		\end{align}
		and if, in addition, $F(\nabla v)\in N^{\beta,2}(\textup{int}(\omega_T);\mathbb{R}^d)$ with $\beta\in (0,1]$, then it holds that
		\begin{align}\label{prop:best-approx.2}
			\begin{aligned}  
					\|F(\nabla \Pi_h^{av} v_h)-F(\nabla v)\|_{2,T}^2&\leq c\,\|F(\nabla_h v_h)-F(\nabla v)\|_{2,\omega_T}^2 +c\,h_T^{2\beta}\,[F(\nabla v)]_{N^{\beta,2}(\textup{int}(\omega_T))}^2\,. 
			\end{aligned} 
		\end{align}
	\end{proposition}
	
	\begin{proof}
		\textit{ad \eqref{prop:best-approx.1}.}
	For every $v_h\in\mathcal{S}^{1,cr}_D(\mathcal{T}_h)$, $v\in W^{1,1}_D(\Omega)$~with~${F(\nabla v)\in L^2(\Omega;\mathbb{R}^d)}$, and $T\in\mathcal{T}_h$, it holds that 
		\begin{align}\label{rem:best-approxP1CR.4} 
			\begin{aligned}
				\big\|\smash{h_{\mathcal{S}}^{1/2}}\jump{F(\nabla_h v_h)}\big\|_{2,\mathcal{S}_h^{i}(T)}^2&=\inf_{r\in \mathbb{R}^d}{\big\|\smash{h_{\mathcal{S}}^{1/2}}\jump{F(\nabla_h v_h)-F(r)}\big\|_{2,\mathcal{S}_h^{i}(T)}^2}\\&\leq c\,\inf_{r\in \mathbb{R}^d}{
					\|F(\nabla_h v_h)-F(r)\|_{2,\omega_T}^2}
				\\&\leq 
				c\,\|F(\nabla_h v_h)-F(\nabla v)\|_{2,\omega_T}^2
				+c\,\inf_{r\in \mathbb{R}^d}{\|F(\nabla v)-F(r)\|_{2,\omega_T}^2}\,,
			\end{aligned}
		\end{align}
		so that, using Lemma \ref{lem:best-approx-inv}, we conclude that \eqref{prop:best-approx.1} applies. 

		\textit{ad \eqref{prop:best-approx.2}.}
		If, in addition, $F(\nabla v)\in N^{\beta,2}(\textup{int}(\omega_T);\mathbb{R}^d)$ with $\beta\in (0,1]$, choosing $r=\fint_{\omega_T}{\nabla v\,\mathrm{d}x}$ for every $T\in \mathcal{T}_h$ in \eqref{prop:best-approx.1}, by \cite[ineqs. (4.6), (4.7)]{breit-lars-etal}, we find that \eqref{prop:best-approx.2} applies. 
	\end{proof} 
 
	\section{Numerical experiments}\label{sec:experiments}
	
	\hspace*{5mm}In this section, we review the practical relevance of the theoretical investigations of Section~\ref{sec:model_problems}.
	 In doing so, we restrict to scalar model problems of Section~\ref{sec:model_problems} (\textit{i.e.}, we restrict~to~the~case~${\ell\hspace*{-0.1em}=\hspace*{-0.1em}1}$).~The vectorial model problems of Section~\ref{sec:model_problems} (\textit{cf}. Subsections \ref{subsec:navier-lame}, \ref{subsec:stokes}) will be experimentally investigated in forthcoming articles. All experiments were conducted employing the finite element software package \texttt{FEniCS}  (version 2019.1.0, \textit{cf}.\  \cite{LW10}). All graphics were generated using~the~\texttt{Matplotlib} library (version 3.5.1, \textit{cf}.\ \cite{Hun07}) and the \texttt{Vedo} library (version 2023.4.4, \textit{cf}.\ \cite{vedo}). In the following, all convergence rates are to be understood in terms of the squared primal-dual gap estimator~$\eta_{\textup{gap}}^2$.% and not $\eta_{\textup{gap}}$.
	
	\subsection{Implementation details regarding the adaptive mesh refinement procedure}\vspace*{-0.5mm}
	 
	\hspace*{5mm}The experiments are based on the following generic  \textit{adaptive algorithm} (\textit{cf}.\ \cite{BabStr01,BanRan03,Ver13}):\vspace*{-0.5mm}
	
		\begin{algorithm}[AFEM]\label{alg:afem}
		Let $\varepsilon_{\textup{STOP}}>0$, $\theta\in (0,1)$, and  $\mathcal{T}_0$ an initial  triangulation of $\Omega$. Then, for every $k\in \mathbb{N}\cup \{0\}$:
		\begin{description}[noitemsep,topsep=1pt,labelwidth=\widthof{\textit{('Estimate')}},leftmargin=!,font=\normalfont\itshape]
			\item[('Solve')]\hypertarget{Solve}{}
			Compute a discrete primal solution $u_k^{cr}\hspace*{-0.1em}\coloneqq \hspace*{-0.1em} u_{h_k}^{cr}\hspace*{-0.1em}\in\hspace*{-0.1em} \smash{\mathcal{S}^{1,cr}(\mathcal{T}_k)}$ (\textit{i.e.}, a minimizer~of~\eqref{discrete_primal}) and a discrete dual solution $z_k^{rt}\coloneqq z_{h_k}^{rt}\in \smash{\mathcal{R}T^0(\mathcal{T}_k)}$ (\textit{i.e.}, a maximizer of  \eqref{discrete_dual}).
			Post-process $u_k^{cr}\in \smash{\mathcal{S}^{1,cr}(\mathcal{T}_k)}$ and $z_k^{rt}\in \smash{\mathcal{R}T^0(\mathcal{T}_k)}$
			to  obtain admissible approximations $\overline{u}_k^{cr}\in W^{1,p}(\Omega)$ with $I(\overline{u}_k^{cr})<\infty$  and  $\overline{z}_k^{rt}\in W^{p'}(\textup{div};\Omega)$ with  $D(\overline{z}_k^{rt})>-\infty$;
			\item[('Estimate')]\hypertarget{Estimate}{} Compute the resulting local refinement primal-dual indicators $\smash{\{\eta^2_{\textup{gap},T}(\overline{u}_k^{cr},\overline{z}_k^{rt})\}_{T\in \mathcal{T}_k}}$. If $\smash{\eta^2_{\textup{gap}}}(\overline{u}_k^{cr},\overline{z}_k^{rt})\leq \varepsilon_{\textup{STOP}}$, then \textup{STOP}; otherwise, continue with step (\hyperlink{Mark}{'Mark'});
			\item[('Mark')]\hypertarget{Mark}{}  Choose a minimal (in terms of cardinality) subset $\mathcal{M}_k\subseteq\mathcal{T}_k$ such that\vspace*{-0.5mm} %the following bulk-criterion is satisfied:
			\begin{align*}
				\sum_{T\in \mathcal{M}_k}{\eta_{\textup{gap},T}^2(\overline{u}_k^{cr},\overline{z}_k^{rt})}\ge \theta^2\sum_{T\in \mathcal{T}_k}{\eta_{\textup{gap},T}^2(\overline{u}_k^{cr},\overline{z}_k^{rt})}\,;
			\end{align*}
			\item[('Refine')]\hypertarget{Refine}{} Perform a (minimal) conforming refinement of $\mathcal{T}_k$ to obtain $\mathcal{T}_{k+1}$~such~that~each element $T\in \mathcal{M}_k$  is `refined' in $\mathcal{T}_{k+1}$.  
			Increase~$k\mapsto k+1$~and~continue~with~step~(\hyperlink{Solve}{'Solve'}).
		\end{description}
	\end{algorithm}
	
	\begin{remark}
		\begin{description}[noitemsep,topsep=1pt,labelwidth=\widthof{\textit{(iii)}},leftmargin=!,font=\normalfont\itshape]
			\item[(i)] If not otherwise specified, we employ the parameter $\theta=\smash{\frac{1}{2}}$ in step (\hyperlink{Estimate}{'Mark'}).
			\item[(ii)] To find the  set $\mathcal{M}_k\subseteq \mathcal{T}_k$ in step (\hyperlink{Mark}{'Mark'}), we resort to the D\"orfler marking~strategy~(\textit{cf}.~\cite{Doe96}).
			\item[(iii)] The  (minimal) conforming refinement of $\mathcal{T}_k$ with respect to  $\mathcal{M}_k$ in step (\hyperlink{Refine}{'Refine'}) is obtained by deploying the \textit{red}-\textit{green}-\textit{blue}-refinement algorithm (\textit{cf}.\ \cite{Ver13}) for $d=2$ and by deploying the Plaza--Carey refinement algorithm (\textit{cf}.\ \cite{plaza}) for $d=3$.
		\end{description}
	\end{remark}

	\subsection{$p(\cdot)$-Dirichlet problem}
	
	 \hspace*{5mm}In this subsection,  
	 we review the  theoretical findings of  Subsection \ref{subsec:nl_dirichlet}.

	 \subsubsection{Implementation details regarding the optimization procedure}\vspace*{-0.5mm}\enlargethispage{15mm}
	
	\hspace*{5mm}Before we present our numerical experiments, we briefly outline  implementation details regarding the optimization procedure.\vspace*{-0.5mm}
	
	\begin{remark}
		\begin{description}[noitemsep,topsep=1pt,labelwidth=\widthof{\textit{(iii)}},leftmargin=!,font=\normalfont\itshape]
			\item[(i)] The discrete primal solution $u_k^{cr}\in \smash{\mathcal{S}^{1,cr}(\mathcal{T}_k)}$ (\textit{i.e.}, minimizer of \eqref{eq:nl_dirichlet_discrete_primal})~in~step (\hyperlink{Solve}{'Solve'}) is computed using the Newton line search algorithm of \texttt{\textup{PETSc}} (version~3.17.3,~\textit{cf}.~\cite{PETSc19}) (with an absolute tolerance of $\tau_{\textup{abs}}= 1.0\times 10^{-8}$ and a relative~tolerance~of~$\tau_{\textup{rel}}= 1.0\times 10^{-10} $) to the corresponding discrete Euler--Lagrange equations. The linear system emerging in each Newton step is solved using the sparse direct~solver~from~\textup{\texttt{MUMPS}} (version 5.5.0, \textit{cf}.\ \cite{mumps});
			\item[(ii)] The reconstruction of the discrete dual solution $z_k^{rt}\in \smash{\mathcal{R}T^0(\mathcal{T}_k)}$ (\textit{i.e.}, maximizer of \eqref{eq:nl_dirichlet_discrete_dual}) in step (\hyperlink{Solve}{'Solve'}) is based on the generalized Marini formula \eqref{eq:nl_dirichlet_marini};
			\item[(iii)] As \hspace*{-0.1mm}conforming \hspace*{-0.1mm}approximations \hspace*{-0.1mm}in \hspace*{-0.1mm}step \hspace*{-0.1mm}(\hyperlink{Solve}{'Solve'}), \hspace*{-0.1mm}we \hspace*{-0.1mm}employ \hspace*{-0.1mm}${\overline{u}_k^{cr}\hspace*{-0.175em}\coloneqq\hspace*{-0.175em}\Pi^{av}_{h_k} u_k^{cr}\hspace*{-0.175em}\in\hspace*{-0.175em} \mathcal{S}^1_D(\mathcal{T}_k)\hspace*{-0.175em}\subseteq\hspace*{-0.175em} W^{1,p(\cdot)}_D(\Omega)}$ in the case $u_D=0$ and $\overline{z}_k^{rt}=z_k^{rt}\in \mathcal{R}T^0(\mathcal{T}_k) \subseteq W^{p'(\cdot)}(\textup{div};\Omega)$ in any case.
			\item[(iv)] The local refinement indicators $\smash{\{\eta^2_{\textup{gap},T}(\overline{u}_k^{cr},z_k^{rt})\}_{T\in \mathcal{T}_k}}\subseteq \mathbb{R}_{\ge 0}$, for every $T\in \mathcal{T}_h$,~are~given~via\vspace*{-0.5mm}
			\begin{align*}
				\eta^2_{\textup{gap},T}(\overline{u}_k^{cr},z_k^{rt})\coloneqq \int_T{\big\{\phi(\cdot,\nabla \overline{u}_k^{cr})-\nabla\overline{u}_k^{cr}\cdot z_k^{rt}+\phi^*(\cdot,z_k^{rt})\big\}\,\mathrm{d}x}\,,
			\end{align*}
			which follows from restricting $\eta^2_{\textup{gap}}(\overline{u}_k^{cr},z_k^{rt})$ (\textit{cf}. \eqref{eq:nl_dirichlet_eta}) to each element $T\in \mathcal{T}_h$.
		\end{description}
	\end{remark}
	
	\subsubsection{Example with corner singularity on $L$-shape domain}\enlargethispage{2mm}
	
	\hspace*{5mm}For our numerical experiments, we choose $\Omega\coloneqq \left(-1,1\right)^2 \setminus ([0,1]\times [-1,0])$, $ \Gamma_D \coloneqq \partial\Omega$, $ \Gamma_N \coloneqq \emptyset$, $u_D=0\in W^{1-\smash{\frac{1}{p^-}},p^-}(\Gamma_D)$, $p\in C^\infty(\overline{\Omega})$, for every $x\in \overline{\Omega}$ defined by
	\begin{align*}
		p(x)\coloneqq p^-+\tfrac{1}{2}\vert x\vert^2\,,
	\end{align*} 
	where $p^-\hspace*{-0.1em}\in\hspace*{-0.1em} \{1.5,2\}$,  $\phi\colon\hspace*{-0.1em} \Omega\times \mathbb{R}^2\hspace*{-0.1em}\to\hspace*{-0.1em} \mathbb{R}$ (satisfying (\hyperlink{A.1}{A.1}), (\hyperlink{A.1}{A.2})),
	for every  $x\hspace*{-0.1em}\in \hspace*{-0.1em}\overline{\Omega}$ and $r\hspace*{-0.1em}\in\hspace*{-0.1em} \mathbb{R}^2$~defined~by 
	\begin{align*}
		\phi(x,r)\coloneqq \tfrac{1}{p(x)}\vert r\vert ^{p(x)}\,,
	\end{align*}
	and as manufactured primal solution $u\in W^{1,p(\cdot)}_0(\Omega)$ (\textit{i.e.}, minimizer of \eqref{eq:nl_dirichlet_primal}), in polar coordinates, for every $(r,\varphi)^\top\in (0,\infty)\times (0,2\pi)$ defined~by
	\begin{align*}
			u(r, \varphi) \coloneqq (1-r^2\cos^2(\varphi))(1-r^2\sin^2(\varphi))r^{\sigma(r)}\sin(\tfrac{2}{3} \varphi)\,,
	\end{align*}
	where $\sigma\in C^\infty(0,+\infty)$ with $\sigma(r)\in (0,1)$ for all $r\in  (0,1)$, \textit{i.e.}, we choose $f\hspace*{-0.1em}\in\hspace*{-0.1em} L^{p'(\cdot)}(\Omega)$ accordingly.
	
	As approximations, for $k=0,\ldots, 20$, we employ $\phi_k\coloneqq \phi_{h_k}\colon \Omega\times \mathbb{R}^2\to \mathbb{R}$, for a.e.\ $x\in \Omega$ and every $r\in \mathbb{R}^2$ defined by 
	\begin{align*}
		\phi_{h_k}(x,r)&\coloneqq \tfrac{1}{p_k(x)-1}(h_k^2+\vert r\vert )^{p_k(x)-1}\vert r\vert \\&\quad-\tfrac{1}{p_k(x)(p_k(x)-1)}((h_k^2+\vert r\vert)^{p_k(x)}-h_k^{2p_k(x)})\,,
	\end{align*}
	where $p_k\coloneqq p_{h_k}\in \mathcal{L}^0(\mathcal{T}_h)$ is defined by  
	\begin{align*}
		p_{h_k}|_T\coloneqq p(x_T)\quad\text{ for all }T\in \mathcal{T}_k\,,
	\end{align*}
	where $x_T\hspace*{-0.1em}\coloneqq \hspace*{-0.1em}\frac{1}{3}\sum_{\nu\in \mathcal{N}_h\,:\,\nu \in T}{\nu}$ is the barycenter of $T$ for all $T\hspace*{-0.1em}\in\hspace*{-0.1em} \mathcal{T}_k$,  and $f_k\hspace*{-0.1em}\coloneqq\hspace*{-0.1em} f_{h_k}\hspace*{-0.1em}\coloneqq \hspace*{-0.1em}\Pi_{h_k} f\hspace*{-0.1em}\in\hspace*{-0.1em} \mathcal{L}^0(\mathcal{T}_h)$.
	
	For every $p^-\in \{1.5,2\}$, the function $\sigma\in C^\infty(0,+\infty)$ for every $r\in (0,\infty)$ is defined by  
	\begin{align*}
		\sigma(r) \coloneqq 1.01-\tfrac{1}{p^-+r^2}\,,
	\end{align*}
	which just yields the fractional regularity
	\begin{align}
		F(\cdot,\nabla u) \in\mathcal{N}^{\frac{1}{2},2}(\Omega;\mathbb{R}^2)\,,\label{eq:frac_reg}
	\end{align}
	where $F\colon \Omega\times \mathbb{R}^2\to \mathbb{R}^2$ is defined by $F(x,r)\coloneqq \vert r\vert^{\smash{\frac{p(x)-2}{2}}}r$ for a.e.\ $x\in \Omega$ and all $r\in \mathbb{R}^2$.
	
	\hspace{-1mm}In \hspace{-0.1mm}the \hspace{-0.1mm}case \hspace{-0.1mm}of \hspace{-0.1mm}uniform \hspace{-0.1mm}mesh \hspace{-0.1mm}refinement \hspace{-0.1mm}(\textit{i.e.}, \hspace{-0.1mm}$\theta\hspace*{-0.15em}=\hspace*{-0.15em}1$ \hspace{-0.1mm}in \hspace{-0.1mm}Algorithm \hspace{-0.1mm}\ref{alg:afem}), \hspace{-0.1mm}the~\hspace{-0.1mm}fractional~\hspace{-0.1mm}\mbox{regularity}~\hspace{-0.1mm}\eqref{eq:frac_reg} let us expect the reduced convergence rate~$\smash{h_k \sim N_k^{\smash{-\frac{1}{2}}}}$,~${k=0,\ldots,20}$, where $N_k\coloneqq  \textup{dim}(\mathcal{S}^{1,cr}_0(\mathcal{T}_k))$,  for the alternative total error quantity
	\begin{align*}
		\left.\begin{aligned}
				\tilde{\rho}_{\textup{tot}}^2(\overline{u}_k^{cr},z_k^{rt})&\coloneqq \|F(\cdot,\nabla\overline{u}_k^{cr})-F(\cdot,\nabla u)\|_{2,\Omega}^2\\&\quad+\|F^*(\cdot,z_k^{rt})-F^*(\cdot,z)\|_{2,\Omega}^2
		\end{aligned}\right\} \,,\quad k=0,\ldots,20\,,
	\end{align*}
	where $F^*\colon \Omega\times \mathbb{R}^2\to \mathbb{R}^2$ is defined by $F^*(x,r)\coloneqq \vert r\vert^{\smash{\frac{p'(x)-2}{2}}}r$ for a.e.\ $x\in \Omega$~and~all~$r\in \mathbb{R}^2$, which is equivalent to the primal-dual total error $\rho^2_{\textup{tot}}(\overline{u}_k^{cr},z_k^{rt})$, $k=0,\ldots,20$, \textit{i.e.}, there exists a constant $c_{p(\cdot)}>0$, depending only on $p^-\coloneqq \text{ess\,inf}_{x\in \Omega}{p(x)}$ and $p^+\coloneqq \text{ess\,sup}_{x\in \Omega}{p(x)}$, such that
	\begin{align*}
				c_{p(\cdot)}^{-1}\,\rho_{\textup{tot}}^2(\overline{u}_k^{cr},z_k^{rt})\leq	\tilde{\rho}_{\textup{tot}}^2(\overline{u}_k^{cr},z_k^{rt})\leq c_{p(\cdot)}\,\rho_{\textup{tot}}^2(\overline{u}_k^{cr},z_k^{rt})\,,\quad k=0,\ldots,20\,.
	\end{align*}

	The coarsest triangulation $\mathcal{T}_0$ of Figure \ref{fig:pxDirichlet_Triang} consists of 96 elements and 65 vertices.~In~Figure~\ref{fig:pxDirichlet_Rate}, for $p^-\in \{1.5,2\}$, one finds that uniform~mesh~refinement
	(\textit{i.e.}, $\theta=1$ in Algorithm \ref{alg:afem})
	 yields the expected reduced convergence rate $\smash{h_k \sim N_k^{\smash{-\frac{1}{2}}}}$, $k=0,\dots,4$, while adaptive mesh refinement (\textit{i.e.}, $\theta=\frac{1}{2}$ in Algorithm \ref{alg:afem}) yields the optimal convergence rate $\smash{h_k^2\sim N_k^{-1}}$, $k=0,\dots,20$. In~\mbox{particular}, for~every $p^-\in \{1.5,2\}$ and $k=0,\dots,20$, when using adaptive mesh refinement, and $ k=0,\dots,4$, when using uniform mesh refinement, the primal-dual gap  estimator $\eta^2_{\textup{gap}}(\overline{u}_k^{cr},z_k^{rt})$ is reliable and efficient with respect to $\tilde{\rho}_{\textup{tot}}^2(\overline{u}_k^{cr},z_k^{rt})$, $k=0,\dots,20$.%, although it is  an upper bound only up to a constant. This is  due to the unspecified~constant~$c_{p(\cdot)}>0$. 

	\begin{figure}[H]
		\centering
		\includegraphics[width=14.5cm]{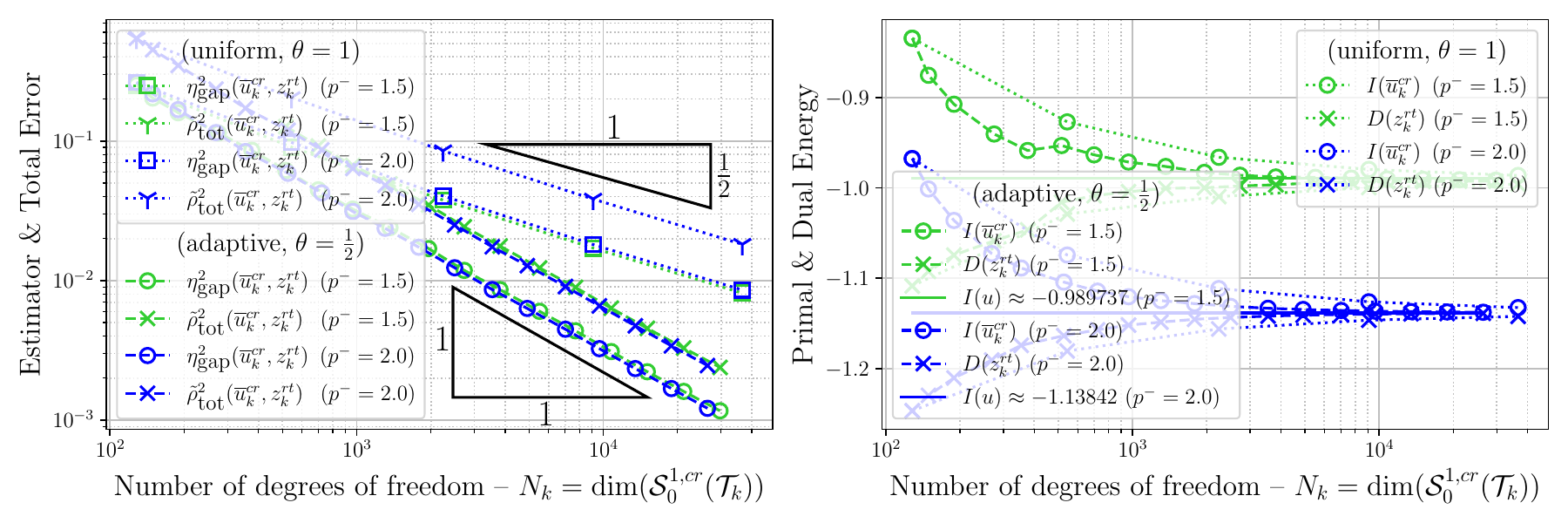}
		\caption{\hspace*{-0.15mm}LEFT: \hspace*{-0.15mm}primal-dual \hspace*{-0.15mm}gap \hspace*{-0.15mm}estimator \hspace*{-0.15mm}$\smash{\eta^2_{\textup{gap}}}(\overline{u}_k^{cr},z_k^{rt})$ \hspace*{-0.15mm}and \hspace*{-0.15mm}alternative~\hspace*{-0.15mm}total~\hspace*{-0.15mm}error~\hspace*{-0.15mm}$\tilde{\rho}^2_{\textup{tot}}(\overline{u}_k^{cr},z_k^{rt})$; RIGHT: primal energy  $I(\overline{u}_k^{cr})$  and dual energy $D(z_k^{rt})$;
			each for $p^-\in\{1.5,2\}$  and $k=0,\dots,20$, when using adaptive mesh refinement (\textit{i.e.}, $\theta=\frac{1}{2}$ in Algorithm \ref{alg:afem}), and for $k=0,\dots, 4$, when uniform mesh refinement (\textit{i.e.}, $\theta\hspace*{-0.1em}=\hspace*{-0.1em}1$ in Algorithm \ref{alg:afem}),  in the \mbox{$p(\cdot)$-Dirichlet}~\mbox{problem}.}\vspace*{-1mm}
		\label{fig:pxDirichlet_Rate}
	\end{figure}
	
	In Figure \ref{fig:pxDirichlet_Triang}, for $p^-\in \{1.5,2\}$, one finds that
	Algorithm \ref{alg:afem} refines towards the origin, where the gradient of the primal solution $u\in W^{1,p(\cdot)}_0(\Omega)$ has its singularity. 
	More precisely, Figure \ref{fig:pxDirichlet_Triang} displays the triangulations $\mathcal{T}_k$, $k\in \{0,10,20\}$, generated by Algorithm \ref{alg:afem} in the case $p^-=2$.
	
	This behavior can also be seen in Figure \ref{fig:pxDirichlet_Solution}, in which the discrete primal~solution~${u_{10}^{cr}\hspace*{-0.1em}\in\hspace*{-0.1em} \mathcal{S}^{1,cr}_0(\mathcal{T}_{10})}$, the node-averaged discrete primal solution $\Pi_{h_{10}}^{av}u_{10}^{cr}\in \mathcal{S}^1_0(\mathcal{T}_{10})$, and the discrete dual solution $z_{10}^{rt}\in \mathcal{R}T^0(\mathcal{T}_{10})$ are plotted, each in the case $p^-=2$.\vspace*{-1mm}\enlargethispage{10mm}
	
	\begin{figure}[H]
		\centering
		\includegraphics[width=14.5cm]{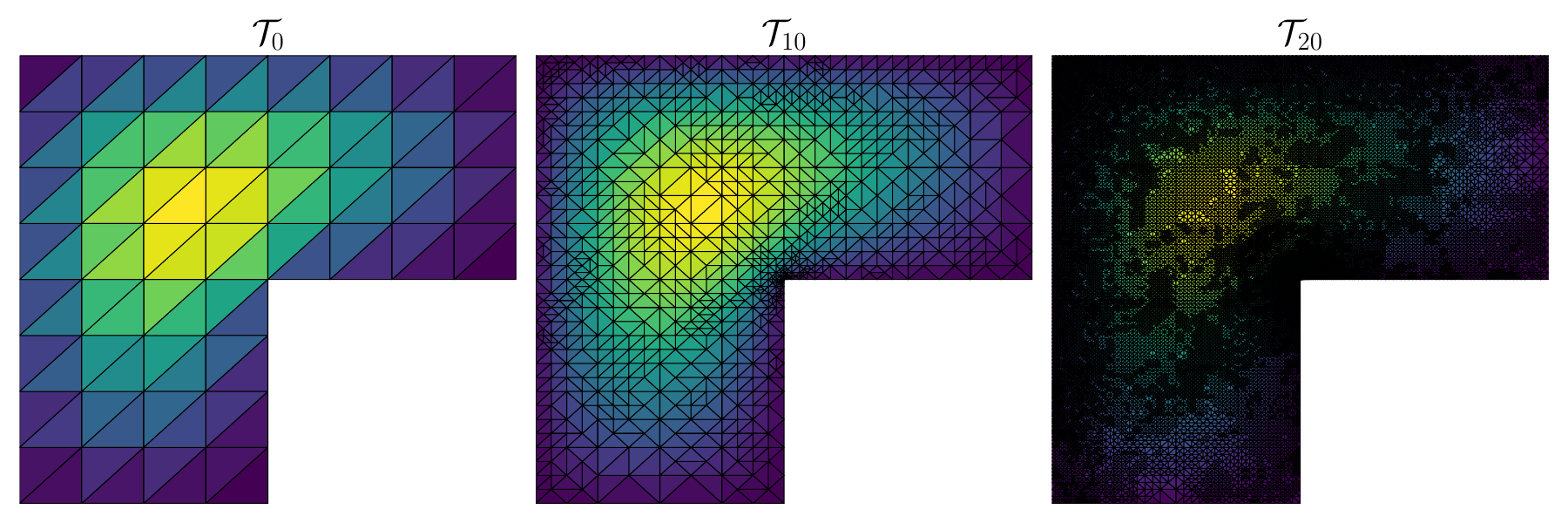}
		\caption{Initial triangulation $\mathcal{T}_0$ and 
			adaptively refined triangulations $\mathcal{T}_k$,~${k\hspace*{-0.15em}\in\hspace*{-0.15em} \{10,20\}}$,~\mbox{generated} by  Algorithm \ref{alg:afem} for $\theta=\frac{1}{2}$, each in the case $p^-=2$  in the $p(\cdot)$-Dirichlet~problem.}
		\label{fig:pxDirichlet_Triang}
	\end{figure}\vspace*{-5mm}
	
	\begin{figure}[H]
		\centering
		\includegraphics[width=14.5cm]{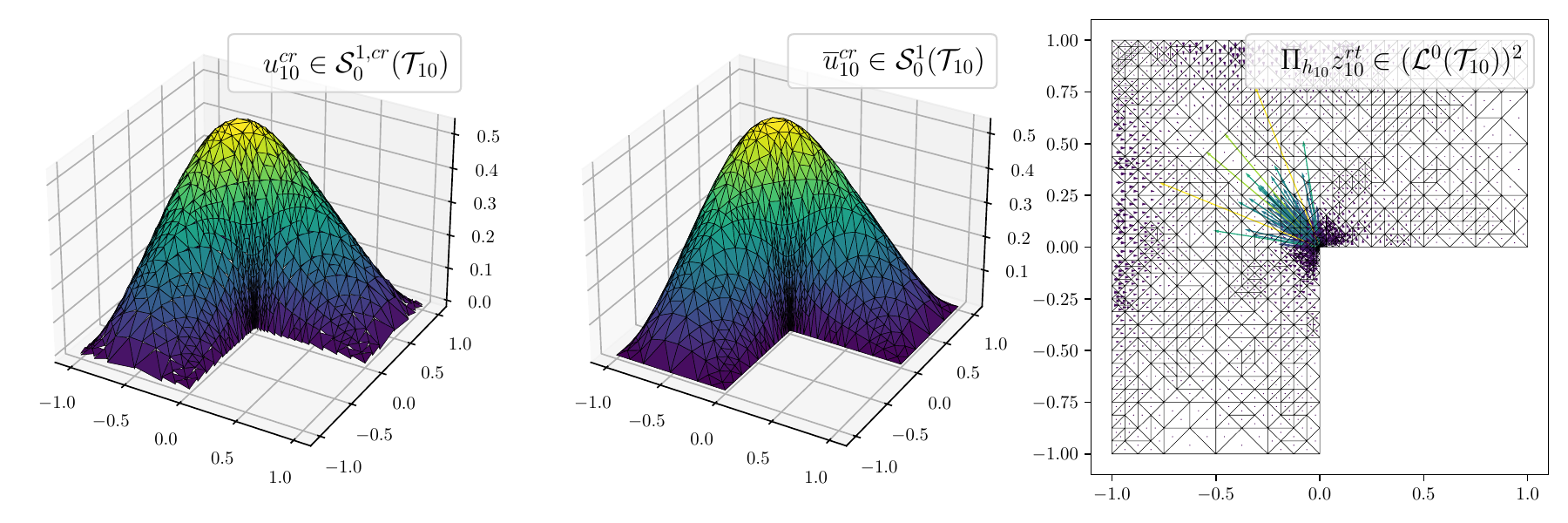}
			\caption{LEFT: discrete primal solution $u_{10}^{cr}\in \mathcal{S}^{1,cr}_0(\mathcal{T}_{10})$; MIDDLE: node-averaged discrete primal solution $\overline{u}_{10}^{cr}\in \mathcal{S}^1_0(\mathcal{T}_{10})$;  RIGHT: (local) $L^2$-projection (onto $(\mathcal{L}^0(\mathcal{T}_{10}))^2$) of discrete dual solution $z_{10}^{rt}\in \mathcal{R}T^0(\mathcal{T}_{10})$, each in the the case $p^-=2$ in the $p(\cdot)$-Dirichlet~problem.}
		\label{fig:pxDirichlet_Solution}
	\end{figure}

	\newpage
	\subsection{Obstacle problem}\label{subsec:num_obstacle}\vspace*{-0.5mm}

	\hspace*{5mm}In this subsection,  
	we review the  theoretical findings of  Subsection \ref{subsec:obstacle}.\vspace*{-0.5mm}

	\subsubsection{Implementation details regarding the optimization procedure}\vspace*{-0.5mm}
	
	\hspace*{5mm}Before we  present our numerical experiments, again, we briefly outline  implementation details regarding the optimization procedure.\vspace*{-0.5mm}
		
	\begin{remark}
		\begin{description}[noitemsep,topsep=1pt,labelwidth=\widthof{\textit{(iii)}},leftmargin=!,font=\normalfont\itshape]
			\item[(i)] The discrete primal solution $u_k^{cr}\in\mathcal{S}^{1,cr}(\mathcal{T}_k)$ (\textit{i.e.}, minimizer of \eqref{eq:obstacle_discrete_primal}) and the discrete Lagrange multiplier $\smash{\overline{\lambda}}_k^{cr}\in \Pi_{h_k}(\mathcal{S}^{1,cr}_D(\mathcal{T}_k))$ (\textit{i.e.}, solution of \eqref{eq:obstacle_multiplier})~in~step~(\hyperlink{Solve}{'Solve'}) are computed using a primal-dual active set strategy interpreted as a locally super-linear converging semi-smooth Newton method (\textit{cf}.\ \cite[Alg.\ 6.1  with  $\alpha=1$]{BK22Obstacle}). Since only a finite number of active sets are possible, the algorithm terminates after a finite number of iterations at $(u_k^{cr},\smash{\overline{\lambda}}_k^{cr})^\top\in\mathcal{S}^{1,cr}(\mathcal{T}_k)\times \Pi_{h_k}(\mathcal{S}^{1,cr}_D(\mathcal{T}_k))$ . The linear system emerging in each semi-smooth Newton step is solved using the sparse direct solver of \texttt{\textup{SciPy}} (version~1.8.1,~\textit{cf}.~\cite{scipy});
			\item[(ii)] The reconstruction of the discrete dual solution $z_k^{rt}\in \smash{\mathcal{R}T^0(\mathcal{T}_k)}$ (\textit{i.e.}, maximizer of \eqref{eq:obstacle_discrete_dual})  in step (\hyperlink{Solve}{'Solve'}) is based on the generalized Marini formula \eqref{eq:obstacle_marini};
			\item[(iii)] As conforming approximations in step (\hyperlink{Solve}{'Solve'}),~we~employ~${\overline{u}_k^{cr}\hspace*{-0.1em}\coloneqq \hspace*{-0.1em}\max\{\Pi^{av}_{h_k} u_k^{cr},\chi\}\hspace*{-0.1em}\in\hspace*{-0.1em} W^{1,2}_D(\Omega)}$ in the case $u_D=0$ and $\overline{z}_k^{rt}=z_k^{rt}\in \mathcal{R}T^0(\mathcal{T}_k) \subseteq W^2(\textup{div};\Omega)$ in any case. 
			\item[(iv)] The local refinement indicators $\smash{\{\eta^2_{\textup{gap},T}(\overline{u}_k^{cr},z_k^{rt})\}_{T\in \mathcal{T}_k}}\subseteq \mathbb{R}_{\ge 0}$, for every $T\in \mathcal{T}_h$,~are~given~via
			\begin{align*}
				\eta^2_{\textup{gap},T}(\overline{u}_k^{cr},z_k^{rt})&\coloneqq \tfrac{1}{2}\|\nabla \overline{u}_k^{cr}-z_k^{rt}\|_{2,T}^2+ (-\textup{div}\,z_k^{rt}-f,\overline{u}_k^{cr}-\chi)_T\,,
%				\eta^2_{\textup{gap},A,T}(\overline{u}_k^{cr},z_k^{rt})&\coloneqq \tfrac{1}{2}\|\nabla \overline{u}_k^{cr}-z_k^{rt}\|_{2,T}^2\,,\\
%				\eta^2_{\textup{gap},B,T}(\overline{u}_k^{cr},z_k^{rt})&\coloneqq (-\textup{div}\,z_k^{rt}-f,\overline{u}_k^{cr}-\chi)_T\,,
			\end{align*}
			which follows from restricting $\eta^2_{\textup{gap}}(\overline{u}_k^{cr},z_k^{rt})$ (\textit{cf}. \eqref{eq:obstacle_eta}) to each element $T\in \mathcal{T}_h$.
		\end{description}
	\end{remark}

	\subsubsection{Example with unknown exact solution}\enlargethispage{14mm}\vspace*{-0.5mm}
	
	\qquad For our numerical experiments, we choose $\Omega\coloneqq (-\frac{3}{2},\frac{3}{2})^2 $, $ \Gamma_D \coloneqq \partial\Omega$, $ \Gamma_N \coloneqq \emptyset$, $f=0\in L^2(\Omega)$, $u_D=0\in W^{\frac{1}{2},2}(\Gamma_D)$, and $\chi\in W^{1,2}_0(\Omega)$ (\textit{cf}.\ Figure \ref{fig:Obstacle_Chi}), for every $x=(x_1,x_2)^\top\in \Omega$ defined by 
	\begin{align*}
		\chi(x_1,x_2)\coloneqq \max\{ 0, \min\{ \min\{ x_1 + 1, \tfrac{1}{2}, 1 - x_1 \}, \min\{ x_2 + 1, \tfrac{1}{2}, 1 - x_2 \} \} \}\,.
	\end{align*}
	
	As approximation, for $k=0,\ldots,25$, we employ $\chi_k\coloneqq  \chi_{h_k}\coloneqq\Pi_{h_k}\chi \in \mathcal{L}^0(\mathcal{T}_k)$ (\textit{cf}.\ Figure~\ref{fig:Obstacle_Chi}).
	The primal solution $u\in W^{1,2}_0(\Omega)$ (\textit{i.e.}, minimizer of \eqref{eq:obstacle_primal}) is not known and cannot be expected to satisfy $u\in W^{2,2}(\Omega)$ inasmuch as $\chi\notin W^{2,2}(\Omega)$.
	In consequence,  uniform mesh refinement (\textit{i.e.}, $\theta=1$ in Algorithm \ref{alg:afem})  is expected to yield a reduced convergence rate compared to the optimal  convergence rate $h_k^2\sim N_k^{-1}$, $k=0,\ldots,25$, where ${N_k\coloneqq \textup{dim}(\mathcal{S}^{1,cr}_0(\mathcal{T}_k))+\textup{dim}(\mathcal{L}^0(\mathcal{T}_k))}$.\vspace*{-1mm}
	
	\begin{figure}[H]
		\centering
		\includegraphics[width=14.5cm]{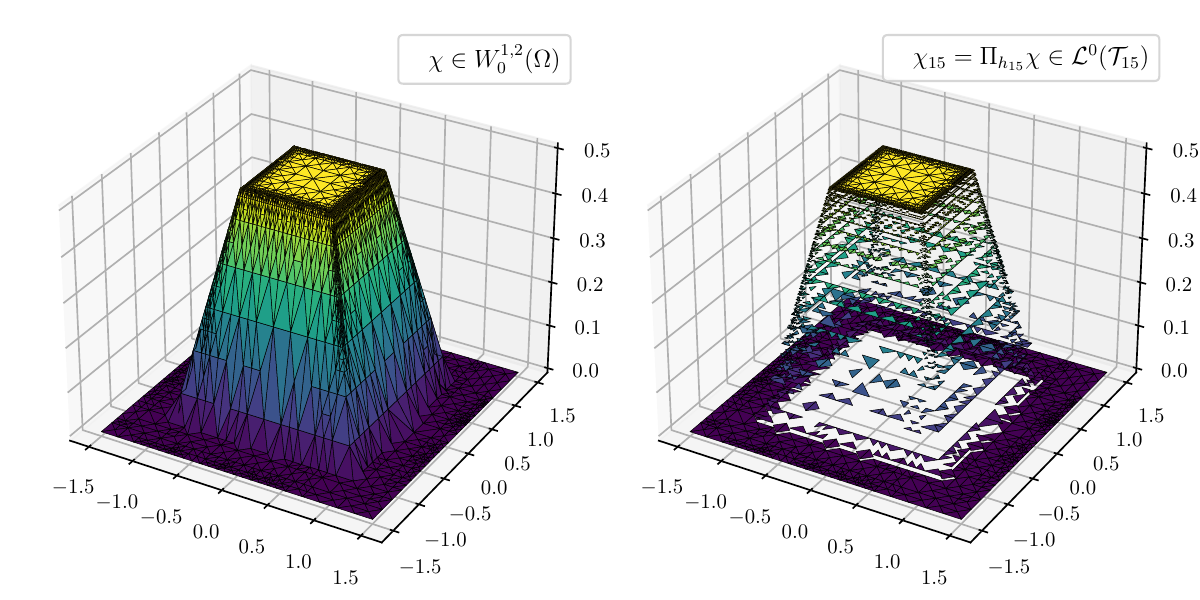}\vspace*{-1mm}
		\caption{LEFT: nodal interpolant of the obstacle $\chi\in W^{1,2}_0(\Omega)$; RIGHT: (local) $L^2$-projection (onto $\mathcal{L}^0(\mathcal{T}_h)$) $\chi_{15}\coloneqq \Pi_{h_{15}}\chi \in \mathcal{L}^0(\mathcal{T}_{15})$ of the obstacle $\chi\in W^{1,2}_0(\Omega)$.}
		\label{fig:Obstacle_Chi}
	\end{figure}
	
	\newpage The coarsest triangulation $\mathcal{T}_0$ in Figure \ref{fig:Obstacle_Triang}   consists of 64 elements and 41 vertices. 
	More precisely, Figure~\ref{fig:Obstacle_Triang} displays
	\hspace*{-0.15mm}the \hspace*{-0.15mm}triangulations \hspace*{-0.15mm}$\mathcal{T}_k$, \hspace*{-0.15mm}$k\hspace*{-0.15em}\in\hspace*{-0.15em} \{0,5,10,15,20,25\}$, \hspace*{-0.15mm}generated \hspace*{-0.15mm}by~\hspace*{-0.15mm}\mbox{Algorithm}~\ref{alg:afem}~\hspace*{-0.15mm}for~\hspace*{-0.15mm}${\theta\hspace*{-0.15em}=\hspace*{-0.15em}\frac{1}{2}}$.
   The discrete contact zones $\mathcal{C}_k^{cr}\hspace*{-0.1em}\coloneqq\hspace*{-0.1em}\{\Pi_{h_k}u_k^{cr}\hspace*{-0.1em}=\hspace*{-0.1em}\chi_k\}\hspace*{-0.1em}=\hspace*{-0.1em}\{\smash{\overline{\lambda}}_k^{cr}\hspace*{-0.1em}<\hspace*{-0.1em}0\}$, $k\hspace*{-0.1em}\in\hspace*{-0.1em} \{0,5,10,15,20,25\}$,~are~plotted in white in Figure \ref{fig:Obstacle_Triang} while their complements are  shaded.

	 Algorithm \ref{alg:afem}  refines the triangulations towards the contact zone $\mathcal{C}\coloneqq \{u=\chi\}$~(\textit{cf}.~\mbox{Figure}~\ref{fig:Obstacle_Triang}). The discrete contact zones $\mathcal{C}_k^{cr}$, $k\in\{0,\dots,25\}$, reduce to $\mathcal{C}$.  
	This  can also be~observed~in~Figure~\ref{fig:solution_pyramid}, where the discrete primal solution $u_{15}^{cr}\in \mathcal{S}^{1,cr}_0(\mathcal{T}_{15})$, the node-averaged discrete primal solution $\Pi_{h_{15}}^{av}u_{15}^{cr}\in \mathcal{S}^1_0(\mathcal{T}_{15})$, the discrete~Lagrange~multiplier~$\smash{\overline{\lambda}}_{15}^{cr}\in \Pi_{h_{15}}(\mathcal{S}^{1,cr}_0(\mathcal{T}_{15}))$, and the discrete dual solution $z_{15}^{rt}\in \mathcal{R}T^0(\mathcal{T}_{15})$ are plotted. 
	In Figure \ref{fig:Obstacle_Rate}, one finds that uniform mesh refinement (\textit{i.e.}, $\theta=1$ in Algorithm \ref{alg:afem}) yields the 
	  expected reduced convergence rate $\smash{h_k \sim N_k^{\smash{-\frac{1}{2}}}}$, $k=0,\dots,4$, while adaptive mesh refinement (\textit{i.e.}, $\theta=\frac{1}{2}$ in Algorithm \ref{alg:afem}) 
	yields  the optimal convergence rate $\smash{h_k^2\sim N_k^{-1}}$, $k=0,\dots,25$.\vspace*{-3mm}\enlargethispage{15mm}

		\begin{figure}[H]
		\centering
		\hspace*{-1mm}\includegraphics[width=14.8cm]{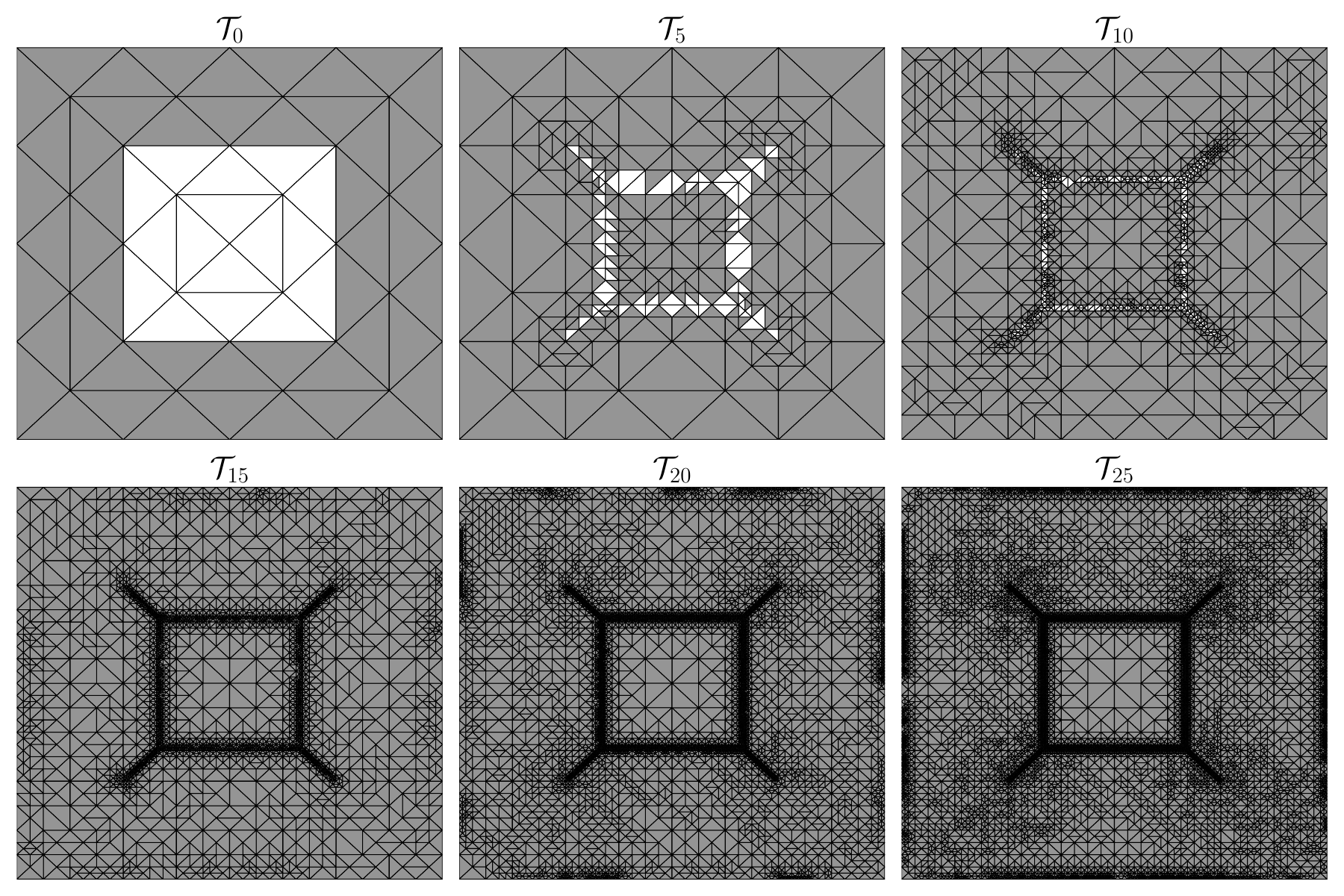}\vspace*{-2mm}
		\caption{Adaptively refined triangulations $\mathcal{T}_k$, $k\in\{0,5,10,15,20,25\}$, with discrete~contact zones $\mathcal{C}_k^{cr}$, $k\in\{0,5,10,15,20,25\}$, shown in white in the obstacle problem.}
		\label{fig:Obstacle_Triang}
	\end{figure}\vspace*{-5mm}

	 \begin{figure}[H]
		\centering
		\includegraphics[width=14.5cm]{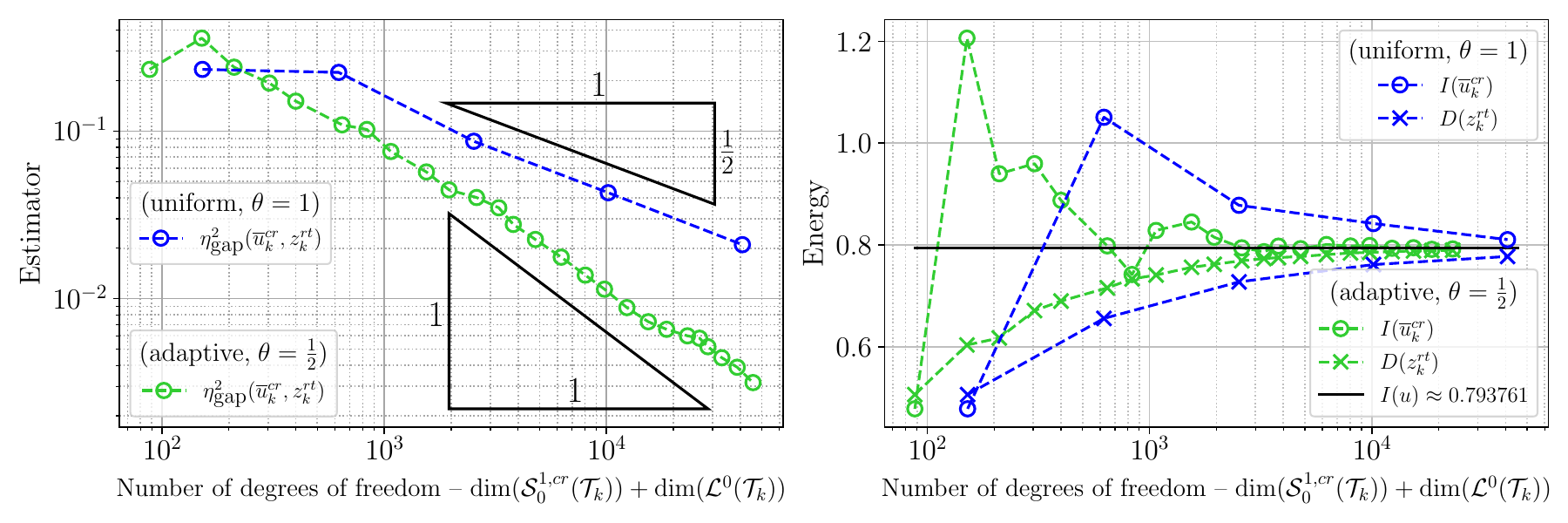}\vspace*{-1mm}
		\caption{LEFT: primal-dual gap estimator $\smash{\eta^2_{\textup{gap}}}(\overline{u}_k^{cr},z_k^{rt})$; RIGHT: primal energy  $I(\overline{u}_k^{cr})$, dual energy $D(z_k^{rt})$, and primal energy $I(u)$ approximated via Aitken’s $\delta^2$-process (\textit{cf}. \cite{Ait26});
			each for $k=0,\dots,25$, when using adaptive mesh refinement  (\textit{i.e.}, $\theta=\frac{1}{2}$ in Algorithm \ref{alg:afem}),  and for $k=0,\dots, 4 $, when uniform mesh refinement (\textit{i.e.}, $\theta=1$ in Algorithm \ref{alg:afem}),~in~the~obstacle~problem.}
		\label{fig:Obstacle_Rate}
	\end{figure}

	\hphantom{.}
	\vfill
	\begin{figure}[H]
		\centering
		\hspace*{-2mm}\includegraphics[width=15cm]{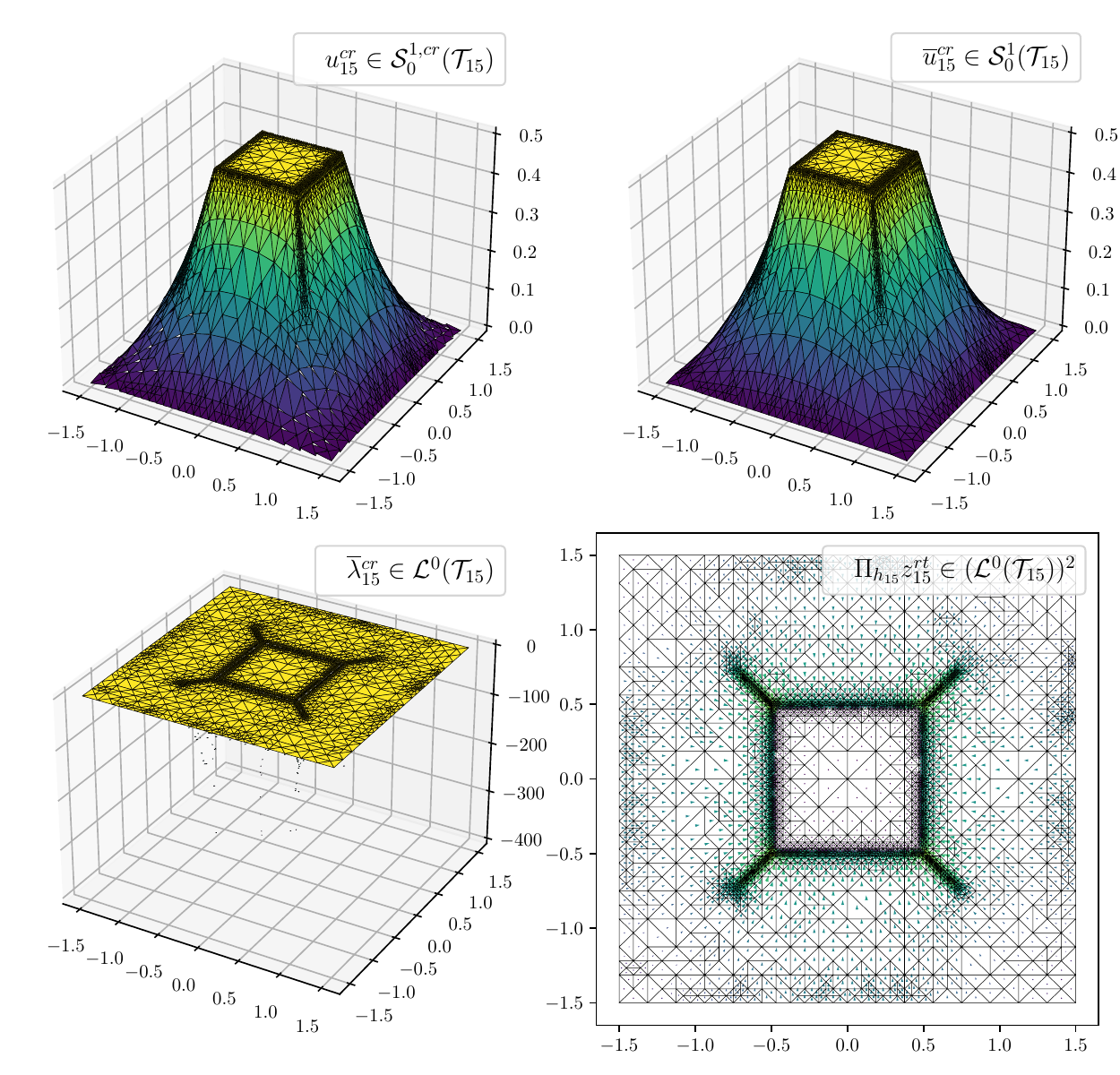}
		\caption{UPPER LEFT: discrete primal solution $u_{15}^{cr}\in \mathcal{S}^{1,cr}_0(\mathcal{T}_{15})$; UPPER RIGHT: node-averaged  discrete primal solution $\Pi_{h_{15}}^{av}u_{15}^{cr}\in \mathcal{S}^1_0(\mathcal{T}_{15})$; LOWER LEFT: discrete Lagrange multiplier
			$\smash{\overline{\lambda}}_{15}^{cr}\hspace*{-0.1em}\in\hspace*{-0.1em} \mathcal{L}^0(\mathcal{T}_{15})$; LOWER RIGHT: (local) $L^2$-projection of the discrete~dual~solution~${z_{15}^{rt}\hspace*{-0.1em}\in\hspace*{-0.1em} \mathcal{R}T^0(\mathcal{T}_{15})}$; each on the triangulation $\mathcal{T}_{15}$ in the obstacle problem.}
		\label{fig:solution_pyramid}
	\end{figure} 
	\vfill
	\newpage
	\subsection{Rudin--Osher--Fatemi (ROF) image denoising problem}\label{subsec:num_rof}\enlargethispage{11mm}\vspace*{-0.5mm}
	
		\hspace*{5mm}In this subsection,  
	we review the  theoretical findings of  Subsection \ref{subsec:rof}. To compare approximations to an exact solution, we impose homogeneous Dirichlet boundary~conditions~on~$\Gamma_D=\partial\Omega$, even though, then, a corresponding existence theory is difficult to establish, in general. However, the set-up~derived~in~\mbox{Subsection}~\ref{subsec:rof} carries over verbatimly with $\Gamma_N=\emptyset$~provided~that~the~existence of a minimizer  is a priori  guaranteed.\vspace*{-0.5mm}
	
		\subsubsection{Implementation details regarding the optimization procedure}\vspace*{-0.5mm}
		
	 \hspace*{5mm}Before we present our numerical experiments, again, we briefly outline  implementation details regarding the optimization procedure.\vspace*{-0.5mm}
	
	\begin{remark}
		\begin{description}[noitemsep,topsep=1pt,labelwidth=\widthof{\textit{(iii)}},leftmargin=!,font=\normalfont\itshape]
			\item[(i)] The discrete  primal solution $u_k^{cr}\in\mathcal{S}^{1,cr}_0(\mathcal{T}_k)$ in step (\hyperlink{Solve}{'Solve'}) is computed using 
			a semi-implicit discretized $L^2$-gradient flow (\textit{cf}.\ \cite[Alg.\ 5.1]{BK23ROF}) for fixed~step-size~${\tau=1}$, stopping criterion $\varepsilon_{stop}^{h_k}\coloneqq \frac{h_k}{\sqrt{20}}$, and initial condition $u_k^0=0\in \mathcal{S}_0^{1,cr}(\mathcal{T}_k)$. Appealing to \cite[Prop.\ 5.2(ii)]{BK23ROF}, \cite[Alg.\ 5.1]{BK23ROF} is unconditionally strongly stable, so that employing the fixed step-size $\tau=1$ is a canonical choice. 
			The stopping~criterion~$\smash{\varepsilon_{stop}^{h_k}\coloneqq \frac{h_k}{\sqrt{20}}}$ ensures (\textit{cf}.\  the argumentation below \cite[Alg.\ 5.1]{BK23ROF}) that the final iterate $u_{h_k}^{i^*}\hspace*{-0.1em}\in\hspace*{-0.1em} \mathcal{S}^{1,cr}_0(\mathcal{T}_k)$,~${i^*\hspace*{-0.1em}\in \hspace*{-0.1em}\mathbb{N}}$,~is~a~\mbox{sufficiently} accurate approximation  (in $L^2(\Omega)$) of the discrete primal solution $u_k^{cr}\in \mathcal{S}^{1,cr}_0(\mathcal{T}_k)$, in the sense 
			that its accuracy does not violate the best possible~convergence~rate~(\textit{cf}.~\mbox{\cite[Rem.~5.6]{BK23ROF}}). The \hspace*{-0.1mm}linear \hspace*{-0.1mm}systems \hspace*{-0.1mm}emerging \hspace*{-0.1mm}in \hspace*{-0.1mm}each \hspace*{-0.1mm}gradient \hspace*{-0.1mm}descent \hspace*{-0.1mm}step \hspace*{-0.1mm}is \hspace*{-0.1mm}solved~\hspace*{-0.1mm}using~\hspace*{-0.1mm}the~\hspace*{-0.1mm}conjugate~\hspace*{-0.1mm}\mbox{gradient} \hspace*{-0.1mm}method \hspace*{-0.1mm}of \hspace*{-0.1mm}\textup{\texttt{PETSc}} \hspace*{-0.1mm}(version~3.17.3, \hspace*{-0.1mm}\textit{cf}.\ \hspace*{-0.1mm}\cite{PETSc19})~\hspace*{-0.1mm}preconditioned~\hspace*{-0.1mm}with~\hspace*{-0.1mm}an~\hspace*{-0.1mm}\mbox{incomplete}~\hspace*{-0.1mm}LU~\hspace*{-0.1mm}\mbox{factorization}.
			\item[(ii)] The reconstruction of the discrete dual solution $z_k^{rt}\in \mathcal{R}T^0(\mathcal{T}_k)$ (\textit{i.e.}, maximizer of \eqref{eq:tv_discrete_dual}) in step (\hyperlink{Solve}{'Solve'}) 
			is based on the generalized Marini formula \eqref{eq:rof_marini}.
			\item[(iii)] As conforming approximations in step (\hyperlink{Solve}{'Solve'}), we employ $\overline{u}_k^{cr}\coloneqq	\Pi_{h_k}^{\partial\Omega}  u_k^{cr} \in \mathcal{S}^{1,cr}_0(\mathcal{T}_k)$ with $\overline{u}_k^{cr}=0$ a.e.\ on $\partial\Omega$, 
			where the operator $\Pi_{h_k}^{\partial\Omega} \colon \mathcal{S}^{1,cr}(\mathcal{T}_k)\to \mathcal{S}^{1,cr}_0(\mathcal{T}_k)$ for every $v_{h_k}\in\mathcal{S}^{1,cr}(\mathcal{T}_k) $ is defined by\vspace*{-0.5mm}
			\begin{align*}
				\Pi_{h_k}^{\partial\Omega} v_{h_k}\coloneqq \sum_{S\in \mathcal{S}_{h_k}^{i}\,:\,S\cap \partial\Omega=\emptyset}{v_{h_k}(x_S)\,\varphi_S}\,,
			\end{align*}
			where $x_S\coloneqq \frac{1}{d}\sum_{\nu\in \mathcal{N}_h\,:\, \nu \in S}{\nu}$ denotes the barycenter of $S$ for all $S\in \mathcal{S}_h$,
			and we employ $\overline{z}_k^{rt}\in  \mathcal{R}T^0(\mathcal{T}_k)$ with $\vert \overline{z}_k^{rt}\vert \leq 1$ a.e.\ in $\Omega$, defined by\vspace*{-0.5mm}
				\begin{align}\label{eq:projection_step}
				\overline{z}_k^{rt}\coloneqq \frac{z_k^{rt}}{\max\{1,\|z_k^{rt}\|_{\infty,\Omega}\}} \,.
			\end{align}
			Note that the post-processing $\overline{u}_k^{cr}\coloneqq	\Pi_{h_k}^{\partial\Omega}  u_k^{cr}$ is only due to the imposed~homogeneous~Dirichlet boundary condition. In the case $\Gamma_D\hspace{-0.1em}=\hspace{-0.1em}\emptyset$,  the choice $\overline{u}_k^{cr}\hspace{-0.1em}\coloneqq \hspace{-0.1em}u_k^{cr}\hspace{-0.1em}\in\hspace{-0.1em} \mathcal{S}^{1,cr}(\mathcal{T}_k)$ is always~admissible.
			
			\item[(iv)] The local refinement indicators $\smash{\{\eta^2_{\textup{gap},T}(\overline{u}_k^{cr},\overline{z}_k^{rt})\}_{T\in \mathcal{T}_k}}\subseteq \mathbb{R}_{\ge 0}$, for every $T\in \mathcal{T}_h$,~are~given~via
			\begin{align*}
				\eta^2_{\textup{gap},T}(\overline{u}_k^{cr},z_k^{rt})&\coloneqq \eta^2_{\textup{gap},A,T}(\overline{u}_k^{cr},\overline{z}_k^{rt})+ \eta^2_{\textup{gap},B,T}(\overline{u}_k^{cr},\overline{z}_k^{rt})\,,\\
				\eta^2_{\textup{gap},A,T}(\overline{u}_k^{cr},\overline{z}_k^{rt})&\coloneqq \|\nabla_h \overline{u}_k^{cr}\|_{1,T}-(\nabla_h \overline{u}_k^{cr},\Pi_h \overline{z}_k^{rt})_T+\sum_{S\in \mathcal{S}_h\,:\,S\subseteq \partial T}{\|\jump{\overline{u}_k^{cr}}_S\|_{1,S}}\,,\\[-2.5mm]
				\eta^2_{\textup{gap},B,T}(\overline{u}_k^{cr},\overline{z}_k^{rt})&\coloneqq \tfrac{1}{2\alpha }\|\textup{div}\,\overline{z}_k^{rt}+\alpha\,(\overline{u}_k^{cr}-g)\|_{2,T}^2\,,
			\end{align*}
			where we used in \eqref{eq:tv_eta} the discrete integration-by-parts formula \eqref{eq:pi} and the discrete representation of the total variation  (\textit{cf}. \cite{discrete_tv}), \textit{i.e.}, for every $v_h\in \mathcal{L}^1(\mathcal{T}_h)$, it holds that
			\begin{align*}
				\vert \textup{D}v_h\vert (\Omega)=\|\nabla_h v_h\|_{1,\Omega}+\sum_{S\in \mathcal{S}_h}{\|\jump{v_h}_S\|_{1,S}}\,,
			\end{align*}
			to arrive at an alternative representation for \eqref{eq:tv_eta}, \text{i.e.}, for every $v_h\in \mathcal{S}^{1,cr}(\mathcal{T}_h)$ and $y_h\in \mathcal{R}T^0_0(\mathcal{T}_h)$ with $\vert y_h\vert \leq 1$ a.e.\ in $\Omega$, we have that
			\begin{align*}
				\eta^2_{\textup{gap}}(v_h,y_h)&=\|\nabla_h v_h\|_{1,\Omega}-(\nabla_h v_h,\Pi_h y_h)_\Omega+\sum_{S\in \mathcal{S}_h}{\|\jump{v_h}_S\|_{1,S}}\\[-2.5mm]&\quad +\tfrac{1}{2\alpha }\|\textup{div}\,y_h+\alpha\,(v_h-g)\|_{2,T}^2\,.
			\end{align*}
		\end{description}
	\end{remark}
	
	\newpage
	\subsection{Example with Lipschitz continuous dual solution}\enlargethispage{16mm}\vspace*{-0.5mm}
	
	\hspace{5mm}We examine an example from \cite{BTW21}. In this example, we let $\Omega\coloneqq (-1,1)^d$, $d\in \{2,3\}$,~${\Gamma_D\coloneqq \partial\Omega}$,  $r\coloneqq \smash{\frac{1}{2}}$, $\alpha \coloneqq 10$, and $g\coloneqq \chi_{B_r^d(0)}\in BV(\Omega)\cap L^\infty(\Omega)$. Then,~the~primal~solution~$u\in BV(\Omega)\cap L^\infty(\Omega)$ and a dual solution $z\in W^2(\textup{div};\Omega)\cap L^\infty(\Omega;\mathbb{R}^d)$, for a.e.\ $x\in \Omega$ are~given~via\vspace*{-1mm}
	\begin{align}\label{one_disk_primal_solution}
		u(x)&\coloneqq (1-\tfrac{d}{\alpha r})\,g(x)\,,\qquad
		z(x)\coloneqq\begin{cases}
			-\tfrac{x}{r}&\textup{ if }\vert x\vert < r\,,\\
			-\tfrac{rx}{\vert x\vert^d}&\textup{ if }\vert x\vert \geq r\,.
		\end{cases}
	\end{align}

	As approximations, for  $k=0,\ldots,25$, we employ  $\smash{\phi_k\coloneqq\phi_{h_k}\in C^1(\mathbb{R}^d)}$, for every $\smash{r\in \mathbb{R}^d}$~defined by $\smash{\phi_{h_k}(r)\coloneqq (1-h_k^2)\,(h_k^4+\vert r\vert^2)^{\frac{1}{2}}}$, 
	and $\smash{g_k\coloneqq g_{h_k}\coloneqq \Pi_{h_k} g\in \mathcal{L}^0(\mathcal{T}_k)}$.
	
	%Note that $z\in W^{1,\infty}(\Omega;\mathbb{R}^d)$, so that, appealing to \cite{CP20,Bar21}, uniform mesh refinement~is~expected to yield the reduced convergence rate $h_k\sim N_k^{\smash{-\frac{1}{2}}}$, $k=0,\ldots,25$.

	\textit{2D Case.}
	The coarsest triangulation $\mathcal{T}_0$ of Figure \ref{fig:ROF_Triang} consists of $32$ elements and $25$ vertices.
	More~precisely, Figure \ref{fig:ROF_Triang} displays
	the triangulations $\mathcal{T}_k$, $k\in \{0,15,25\}$, generated~by~Algorithm~\ref{alg:afem}.
	A refinement towards $\partial B_r^2(0)$, \textit{i.e.}, the jump set $J_u$ of the  primal~solution~${u\hspace*{-0.1em}\in\hspace*{-0.1em} BV(\Omega)\cap L^\infty(\Omega)}$\linebreak (\textit{cf}.\ \eqref{one_disk_primal_solution}) is reported.
	This behavior is also seen in Figure \ref{fig:ROF_Solution}, in which the discrete primal solution $u_{15}^{cr}\in \mathcal{S}^{1,cr}_0(\mathcal{T}_{15})$, the (local)
	$L^2$-projection (onto $\mathcal{L}^0(\mathcal{T}_{15})$) ${\Pi_{h_{15}} u_{15}^{cr}\in \mathcal{L}^0(\mathcal{T}_{15})}$, 
	the (local) $L^2$-projections (onto $\mathcal{L}^1(\mathcal{T}_{15})$) of 
	the modulus of the  discrete dual solution $z_{15}^{rt}\in \mathcal{R}T^0(\mathcal{T}_{15})$, and the (local) $L^2$-projections (onto $\mathcal{L}^1(\mathcal{T}_{15})$) of 
	the modulus of the unit-length scaled  discrete dual solution $\overline{z}_{15}^{rt}\in \mathcal{R}T^0(\mathcal{T}_{15})$ (\textit{cf}. \eqref{eq:projection_step}) are plotted. In 
	Figure \ref{fig:ROF_Rate}, one sees that uniform~mesh~refinement (\textit{i.e.}, $\theta=1$ in Algorithm \ref{alg:afem}) yields the expected reduced 
	 the convergence rate $h_k\sim N_k^{\smash{-\frac{1}{2}}}$, $k=0,\ldots,25$,   (predicted by \cite{CP20,Bar21}), while 
	 adaptive mesh refinement (\textit{i.e.}, $\theta=\frac{1}{2}$ in Algorithm~\ref{alg:afem}) yields the optimal convergence rate $h_k^2\sim N_k^{\smash{-1}}$,~$k=0,\ldots,25$.
	In addition, Figure \ref{fig:ROF_Rate} indicates the primal-dual gap estimator is reliable and efficient~respect~to the alternative total error quantity
	\begin{align}\label{eq:reduced_rho}
		\tilde{\rho}^2_{\textup{tot}}(\overline{u}_k^{cr},\overline{z}_k^{rt})\coloneqq \tfrac{\alpha}{2}\|\overline{u}_k^{cr}-u\|^2_{2,\Omega}+\tfrac{1}{2\alpha}\|\textup{div}\,\overline{z}_k^{rt}-\textup{div}\,z\|^2_{2,\Omega}\,,\quad k=0,\ldots,25\,,
	\end{align}
	which is  a lower bound for the total error $\smash{\rho_{\textup{tot}}^2(\overline{u}_k^{cr},\overline{z}_k^{rt})=\eta^2_{\textup{gap}}(\overline{u}_k^{cr},\overline{z}_k^{rt})}$, $k=0,\ldots,25$.\vspace*{-2.5mm}%, (\textit{cf}.\  Remark \ref{rmk:examples} (iv)).

	\begin{figure}[H]
		\centering 
		\hspace*{-1mm}\includegraphics[width=14.5cm]{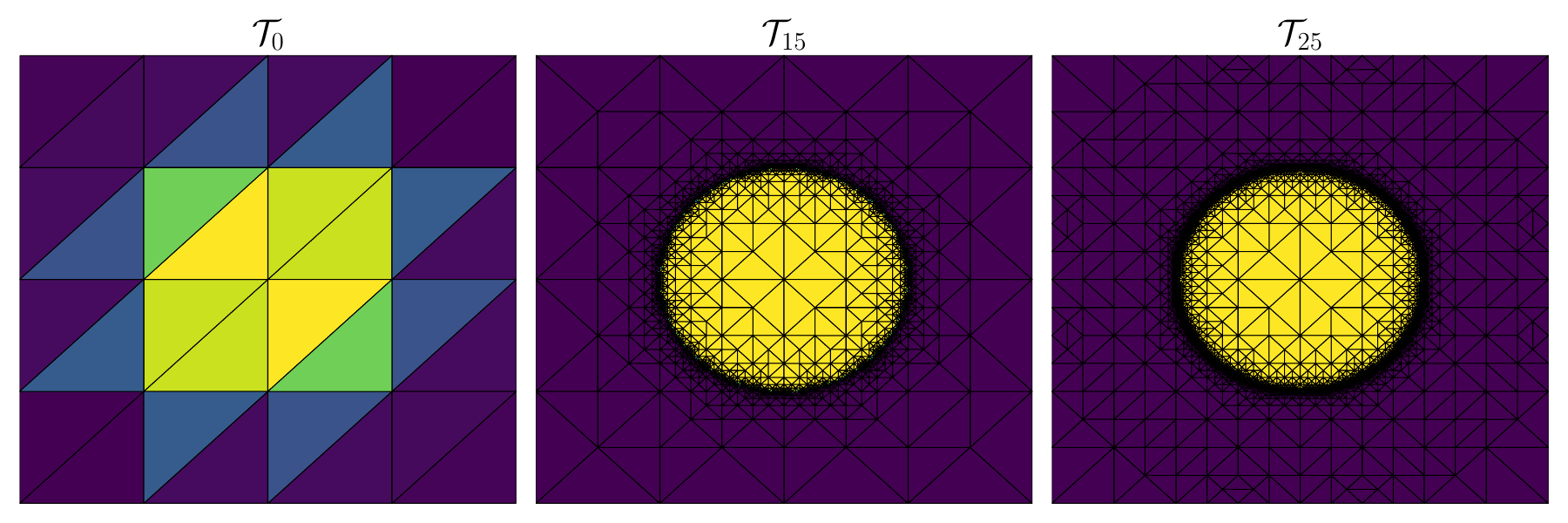}\vspace*{-2mm}
		\caption{Initial triangulation $\mathcal{T}_0$ and 
			adaptively refined triangulations $\mathcal{T}_k$,~$k\hspace*{-0.15em}\in\hspace*{-0.15em} \{15,25\}$,~generated by  Algorithm \ref{alg:afem} for $\theta=\frac{1}{2}$ in the Rudin--Osher--Fatemi image de-noising problem.}
		\label{fig:ROF_Triang} 
	\end{figure}\vspace*{-6mm}
	
	\begin{figure}[H]
		\centering
		\hspace*{-1mm}\includegraphics[width=14.5cm]{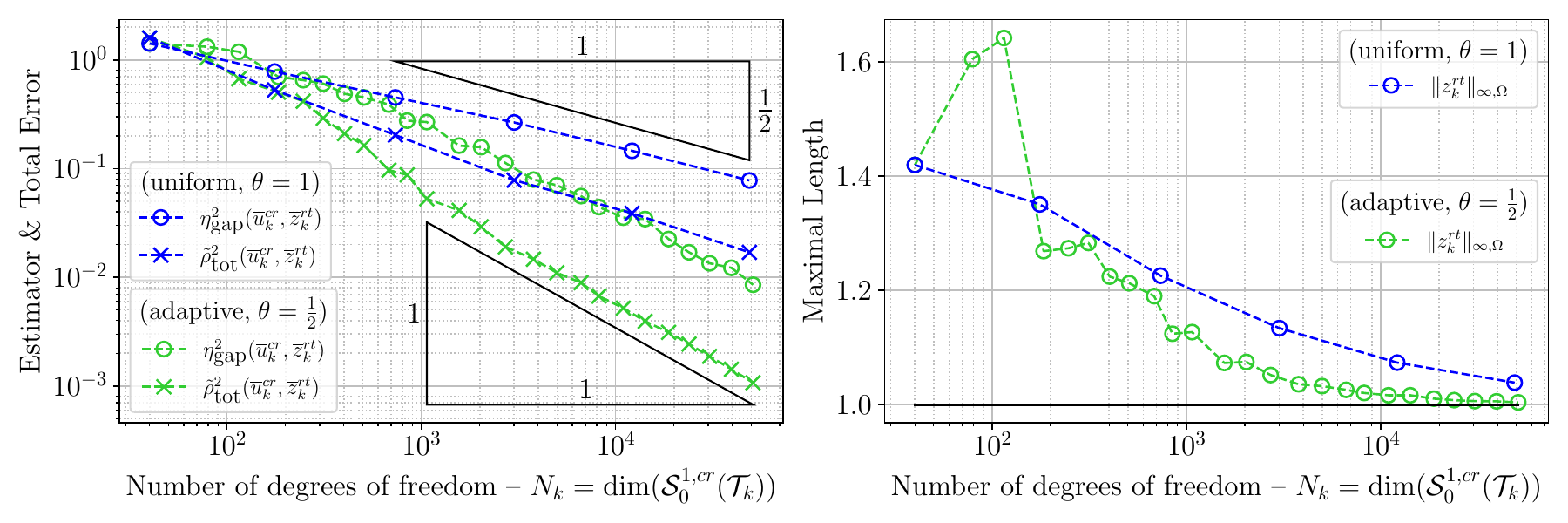}\vspace{-1.5mm}
		\caption{LEFT: \hspace*{-0.1mm}primal-dual \hspace*{-0.1mm}gap \hspace*{-0.1mm}estimator \hspace*{-0.1mm}$\eta^2_{\textup{gap}}(\overline{u}_k^{cr},\overline{z}_k^{rt})$ \hspace*{-0.1mm}and \hspace*{-0.1mm}alternative~\hspace*{-0.1mm}total~\hspace*{-0.1mm}error~\hspace*{-0.1mm}$\tilde{\rho}^2_{\textup{tot}}(\overline{u}_k^{cr},\overline{z}_k^{rt})$; RIGHT: maximal length of discrete dual solution $\|z_k^{rt}\|_{\infty,\Omega}$;
		each for $k=0,\dots,25$,~when~using~ad-aptive \hspace*{-0.15mm}mesh \hspace*{-0.15mm}refinement \hspace*{-0.15mm}(\textit{i.e.}, \hspace*{-0.15mm}$\theta\hspace*{-0.15em}=\hspace*{-0.15em}\frac{1}{2}$ \hspace*{-0.15mm}in \hspace*{-0.15mm}Algorithm \hspace*{-0.15mm}\ref{alg:afem}), \hspace*{-0.15mm}and \hspace*{-0.15mm}for \hspace*{-0.15mm}$k=0,\dots, 5$,~\hspace*{-0.15mm}when~\hspace*{-0.15mm}using~\hspace*{-0.15mm}uniform~\hspace*{-0.15mm}mesh refinement (\textit{i.e.}, $\theta\hspace*{-0.15em}=\hspace*{-0.15em}1$ in Algorithm \ref{alg:afem}), in the Rudin--Osher--Fatemi~image~\mbox{de-noising}~problem.} 
		\label{fig:ROF_Rate}
	\end{figure}
	
	\hphantom{.}
	\vfill
	\begin{figure}[H]
		\centering
		\hspace*{-1mm}\includegraphics[width=15cm]{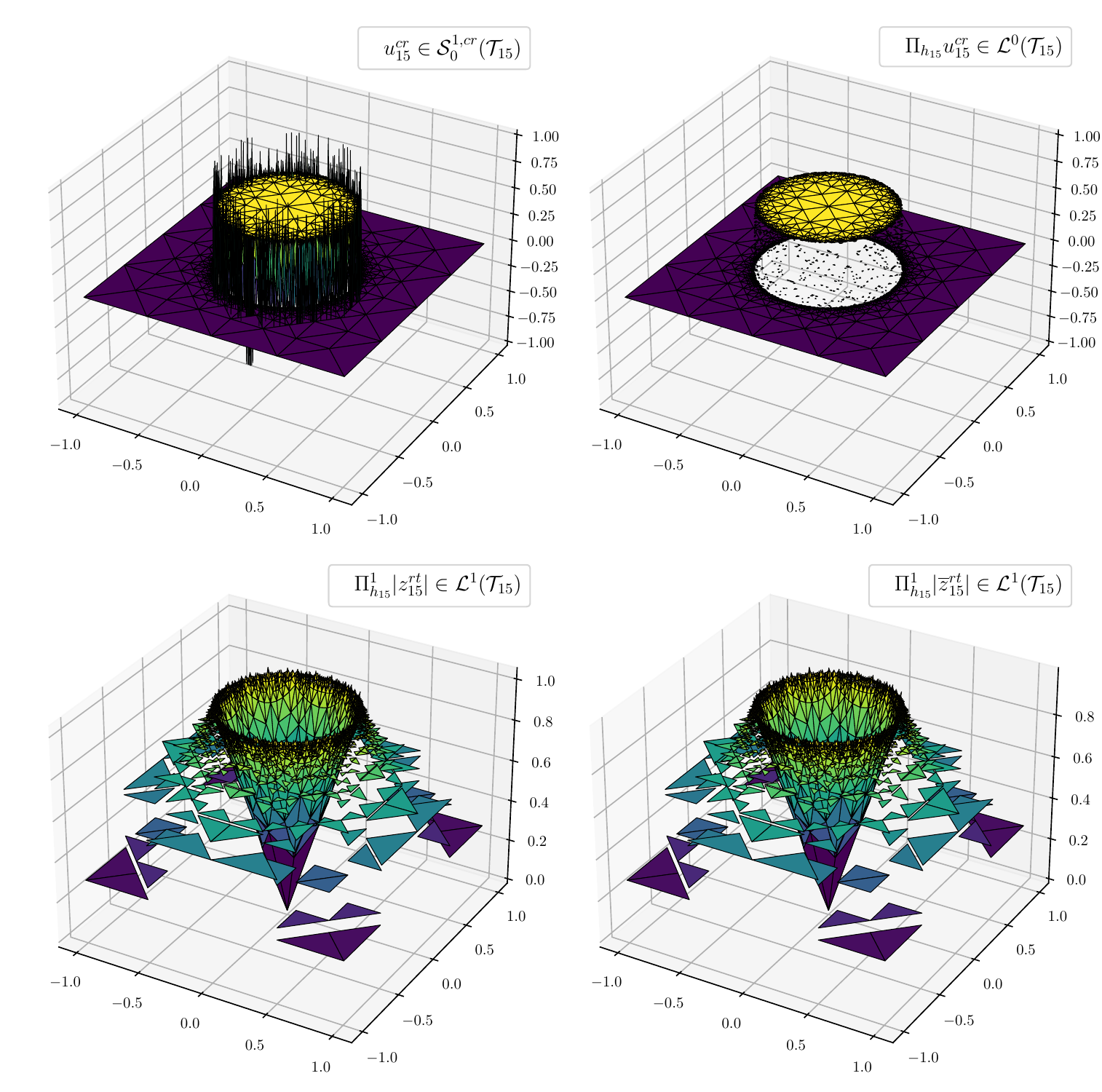}
		\caption{UPPER LEFT: discrete primal solution $u_{15}^{cr}\in \mathcal{S}^{1,\textit{\textrm{cr}}}_0(\mathcal{T}_{15})$, UPPER RIGHT: (local) $L^2$-projection of discrete primal solution $\Pi_{h_{15}}u_{15}^{cr}\in  \mathcal{L}^0(\mathcal{T}_{15})$; LOWER LEFT: (local) $L^2$-projection (onto $\mathcal{L}^1(\mathcal{T}_{15})$) of discrete  dual solution $\Pi_{h_{15}}^1\vert z_{15}^{rt}\vert \in \mathcal{L}^1(\mathcal{T}_{15})$; LOWER RIGHT:  (local) $L^2$-projection (onto \hspace*{-0.1mm}$\mathcal{L}^1(\mathcal{T}_{15})$)  \hspace*{-0.1mm}of \hspace*{-0.1mm}unit-length \hspace*{-0.1mm}scaled \hspace*{-0.1mm}discrete~\hspace*{-0.1mm}dual~\hspace*{-0.1mm}\mbox{solution}~\hspace*{-0.1mm}${\Pi_{h_{15}}^1\vert \overline{z}_{15}^{rt}\vert\hspace*{-0.1em}\in\hspace*{-0.1em} \mathcal{L}^1(\mathcal{T}_{15})}$~\hspace*{-0.1mm}(\textit{cf}.~\hspace*{-0.1mm}\eqref{eq:projection_step});
		each on the triangulation $\mathcal{T}_{15}$ in the Rudin--Osher--Fatemi image de-noising problem.}
		\label{fig:ROF_Solution}
	\end{figure}
	\vfill
	\newpage

	\textit{3D Case.} The initial triangulation  $\mathcal{T}_0$  of Algorithm~\ref{alg:afem} consists of $384$ elements~and~$125$~vertices.
	Algorithm \ref{alg:afem} refines the triangulations towards $\partial B_r^3(0)$, \textit{i.e.}, the jump set $J_u$ of the exact solution $u\in BV(\Omega)\cap L^\infty(\Omega)$ (\textit{cf}.\  \eqref{one_disk_primal_solution}),  which can be  observed
	in Figure \ref{fig:OneDisk3D_solution}, in which, more concretely, for $k\in \{0,5,9\}$, the discrete primal solution $u_k^{cr}\in \mathcal{S}^{1,cr}_0(\mathcal{T}_k)$ and 
	the (local) $L^2$-projection (onto $\mathcal{L}^1(\mathcal{T}_k)$) of 
	the modulus of the discrete dual solution $z_k^{rt}\in \mathcal{R}T^0(\mathcal{T}_k)$ are plotted. 
	In 
	Figure \ref{fig:OneDisk3D_rate}, one finds that uniform mesh refinement (\textit{i.e.}, $\theta=1$ in Algorithm \ref{alg:afem}) yields 
	the expected reduced convergence rate $h_k\sim N_k^{\smash{-\frac{1}{3}}}$, $k=0,\ldots,3$, (predicted~by~\cite{CP20,Bar21}),
	while adaptive mesh refinement (\textit{i.e.}, $\theta=\frac{1}{2}$ in Algorithm \ref{alg:afem})  yields the improved
	 convergence rate $h_k^2\sim N_k^{\smash{-\frac{2}{3}}}$,~$k=0,\ldots,9$.\vspace*{-2.5mm}
	
	    \begin{figure}[H]
		\centering
		\includegraphics[width=4.75cm]{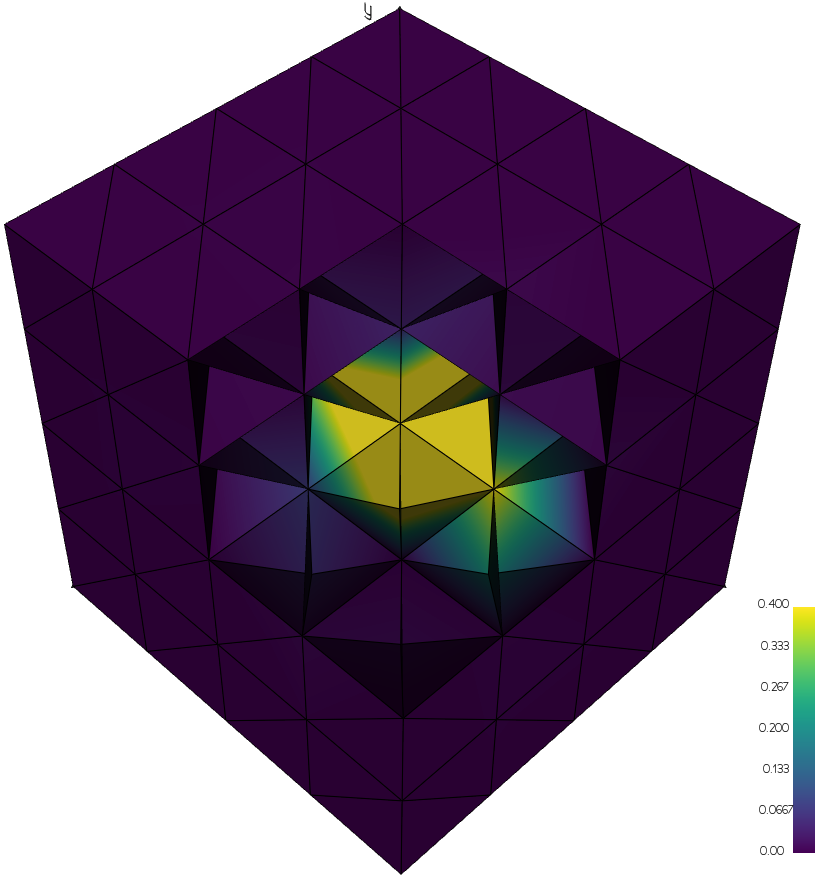}\includegraphics[width=4.75cm]{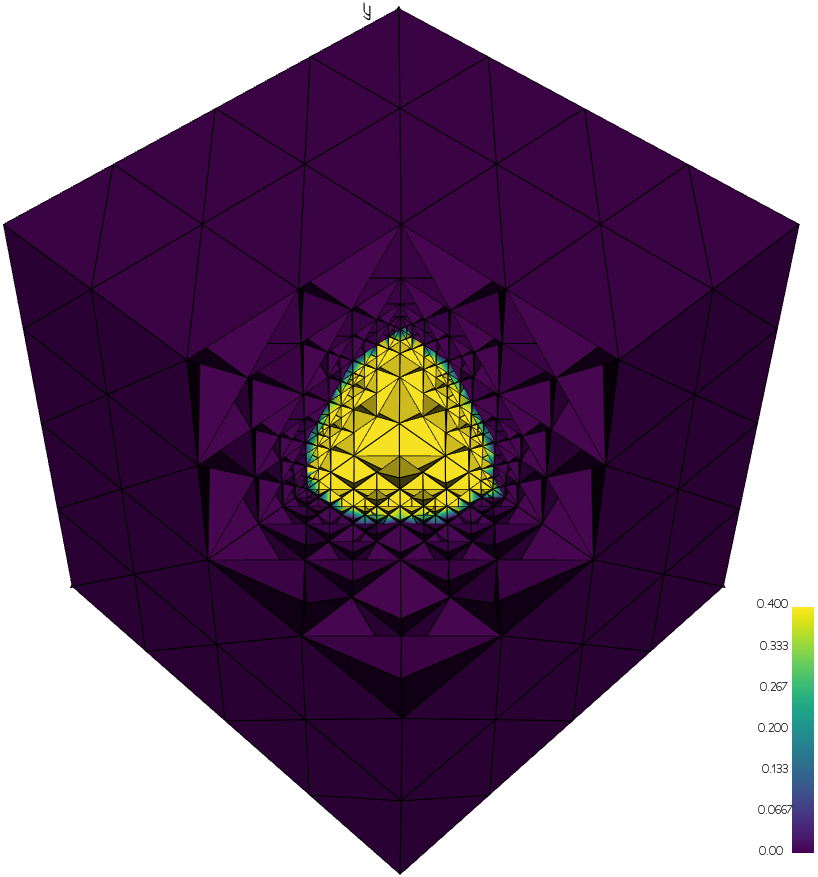}\includegraphics[width=4.75cm]{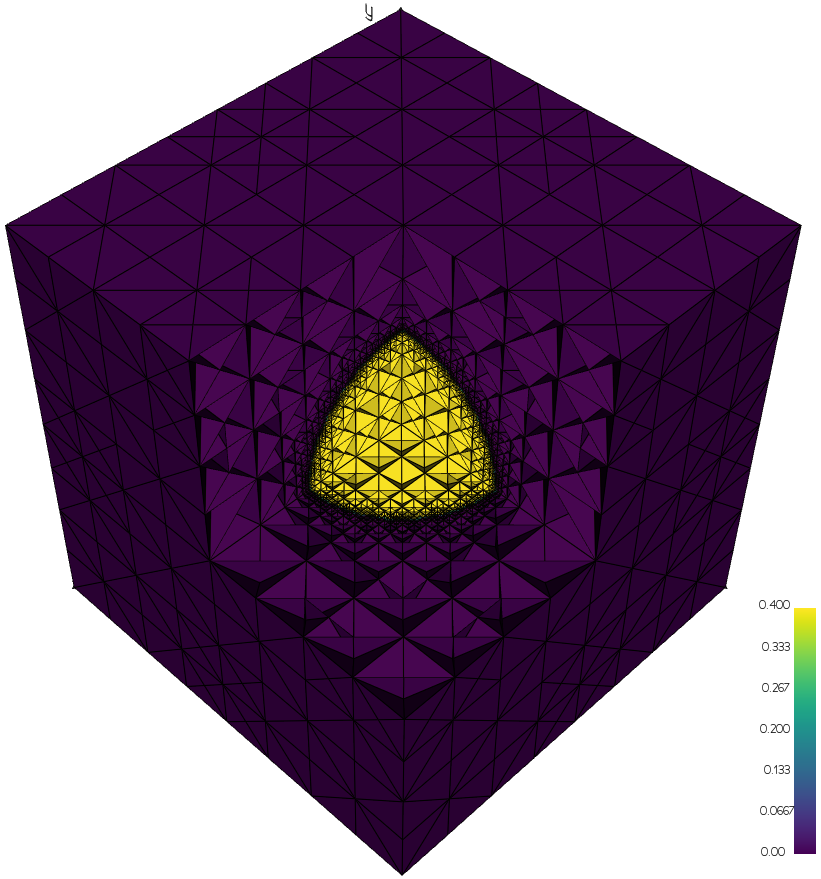}
		\includegraphics[width=4.75cm]{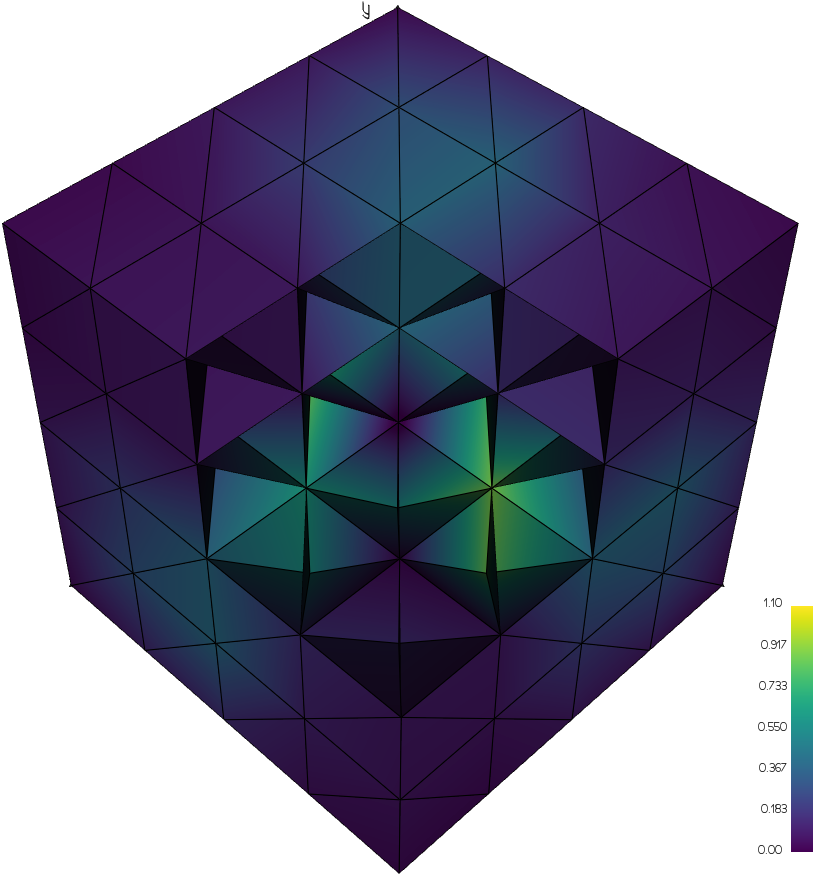}\includegraphics[width=4.75cm]{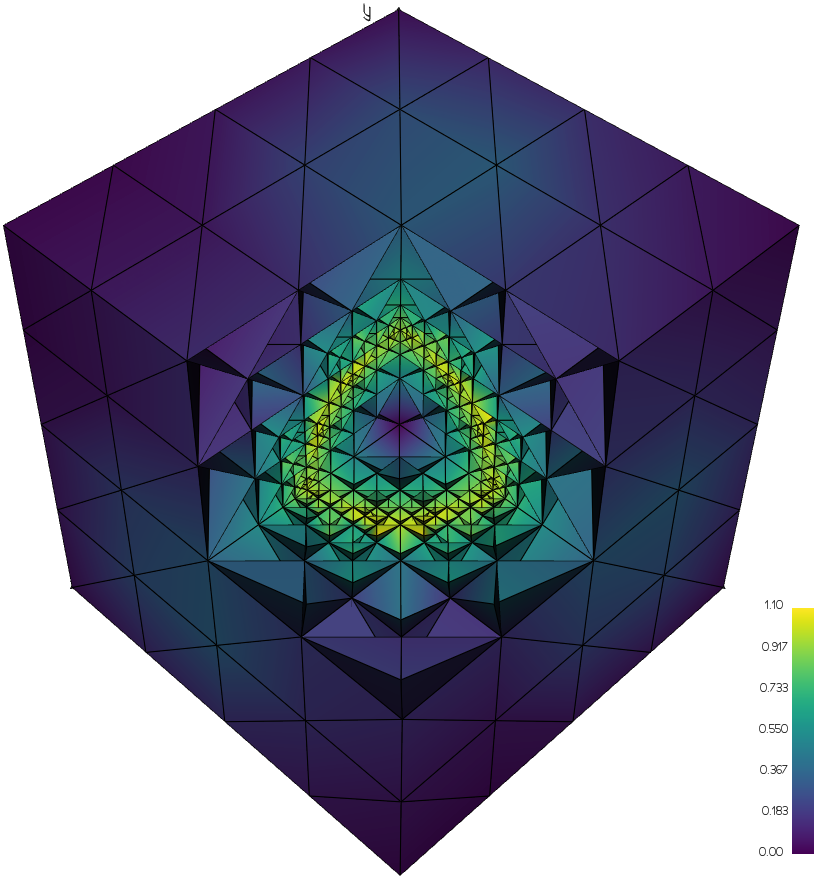}\includegraphics[width=4.75cm]{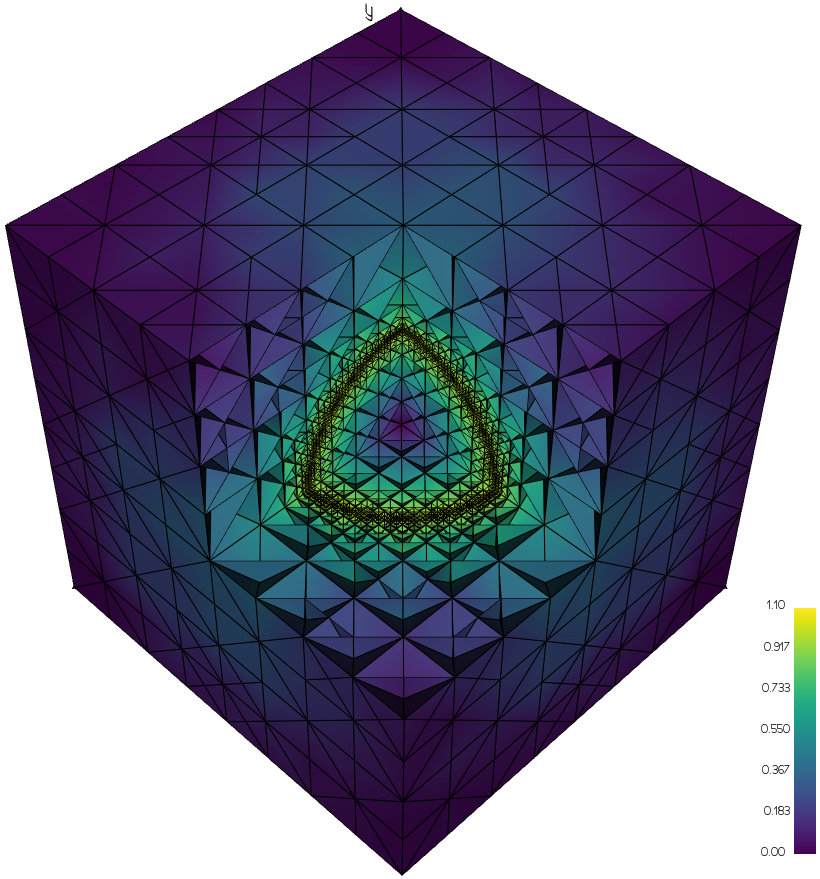}
		\vspace{-1mm}
		\caption{\hspace*{-0.15mm}TOP:  \hspace*{-0.15mm}discrete  \hspace*{-0.15mm}primal  \hspace*{-0.15mm}solutions  \hspace*{-0.15mm}$u_k^{cr}\hspace*{-0.15em}\in\hspace*{-0.15em}\mathcal{S}^{1,cr}_0(\mathcal{T}_k)$, $k\hspace*{-0.15em}\in\hspace*{-0.15em} \{0,5,9\}$;  \hspace*{-0.15mm}BOTTOM:~(\mbox{local})~$L^2$- projection \hspace*{-0.15mm}(onto \hspace*{-0.15mm}$\mathcal{L}^1(\mathcal{T}_k)$) \hspace*{-0.15mm}of \hspace*{-0.15mm}the \hspace*{-0.15mm}moduli \hspace*{-0.15mm}of \hspace*{-0.15mm}the  \hspace*{-0.15mm}discrete~\hspace*{-0.15mm}dual~\hspace*{-0.15mm}solutions~\hspace*{-0.15mm}${\Pi_{h_k}^1\vert z_k^{rt}\vert\hspace*{-0.15em} \in\hspace*{-0.15em} \mathcal{L}^1(\mathcal{T}_k)}$, \hspace*{-0.15mm}${k\hspace*{-0.15em}\in\hspace*{-0.15em} \{0,5,9\}}$; each on the triangulation $\mathcal{T}_k$,  $k\in \{0,5,9\}$, in the Rudin--Osher--Fatemi image~\mbox{de-noising}~problem.}
		\label{fig:OneDisk3D_solution}
	\end{figure}\vspace{-5mm}\enlargethispage{15mm}
	
	\begin{figure}[H]
		\centering
		\hspace*{-2mm}\includegraphics[width=14.5cm]{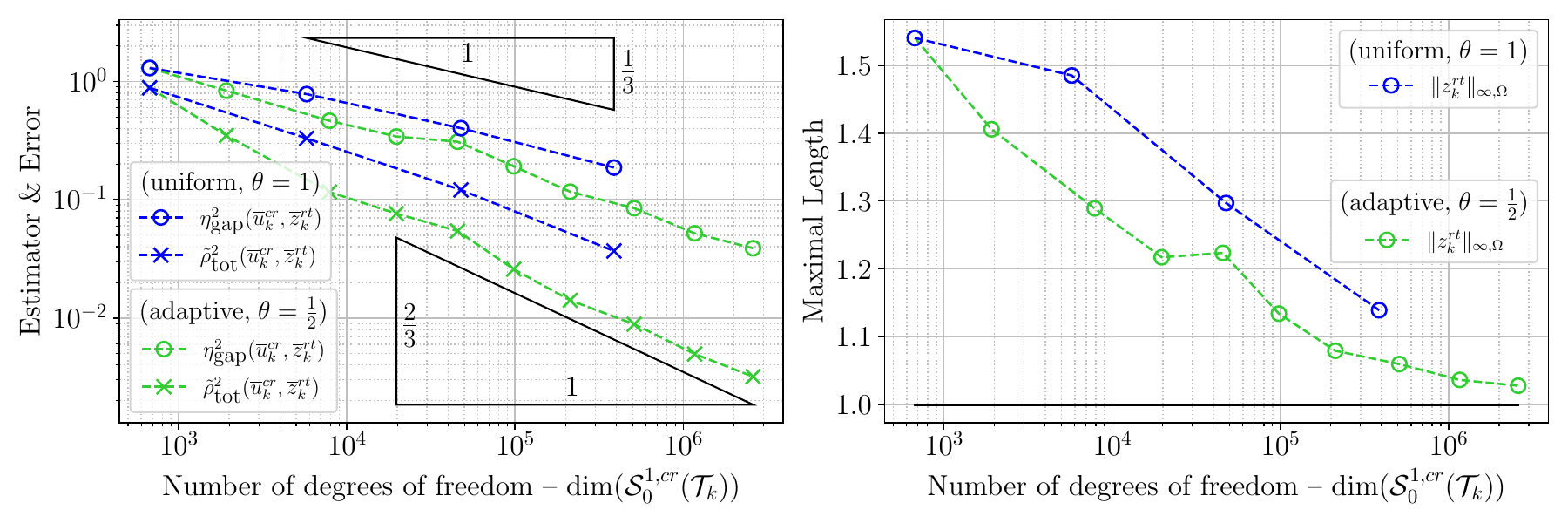}\vspace{-2.5mm}
		\caption{\hspace*{-1mm}LEFT: \hspace*{-0.15mm}primal-dual \hspace*{-0.15mm}gap \hspace*{-0.15mm}estimator \hspace*{-0.15mm}$\eta^2_{\textup{gap}}(\overline{u}_k^{cr},\overline{z}_k^{rt})$ \hspace*{-0.15mm}and \hspace*{-0.15mm}alternative~\hspace*{-0.15mm}total~\hspace*{-0.15mm}error~\hspace*{-0.15mm}$\tilde{\rho}^2_{\textup{tot}}(\overline{u}_k^{cr},\overline{z}_k^{rt})$; RIGHT: maximal length of discrete dual solution $\|z_k^{rt}\|_{\infty,\Omega}$;
		each for $k=0,\dots,9$,~when~using~ad-aptive \hspace*{-0.15mm}mesh \hspace*{-0.15mm}refinement \hspace*{-0.15mm}(\textit{i.e.}, \hspace*{-0.15mm}$\theta\hspace*{-0.15em}=\hspace*{-0.15em}\frac{1}{2}$ \hspace*{-0.15mm}in \hspace*{-0.15mm}Algorithm \hspace*{-0.15mm}\ref{alg:afem}), \hspace*{-0.15mm}and \hspace*{-0.15mm}for \hspace*{-0.15mm}$k=0,\dots, 3$,~\hspace*{-0.15mm}when~\hspace*{-0.15mm}using~\hspace*{-0.15mm}uniform~\hspace*{-0.15mm}mesh refinement (\textit{i.e.}, $\theta\hspace*{-0.15em}=\hspace*{-0.15em}1$ in Algorithm \ref{alg:afem}), in the Rudin--Osher--Fatemi~image~\mbox{de-noising}~problem.} 
		\label{fig:OneDisk3D_rate}
	\end{figure}\newpage

	\subsection{Jumping coefficients}\label{num:jumping}\vspace*{-0.5mm}

	\hspace*{5mm}In this subsection,  
	we review the  theoretical findings of  Subsection \ref{subsec:jumping}.\vspace*{-0.5mm}
	
	\subsubsection{Implementation details regarding the optimization procedure}\vspace*{-0.5mm}
	
	\hspace*{5mm}Before we present our numerical experiments, again, we briefly outline  implementation details regarding the optimization procedure.\vspace*{-0.5mm}\enlargethispage{10mm}
	
	\begin{remark}
		\begin{description}[noitemsep,topsep=1pt,labelwidth=\widthof{\textit{(iii)}},leftmargin=!,font=\normalfont\itshape]
			\item[(i)] The discrete primal solution $u_k^{cr}\in \smash{\mathcal{S}^{1,cr}(\mathcal{T}_k)}$  (\textit{i.e.}, minimizer of \eqref{eq:jumping_discrete_primal})~in~step (\hyperlink{Solve}{'Solve'}) is computed  using the sparse direct solver from \textup{\textsf{MUMPS}} (version 5.5.0, \textit{cf}.\ \cite{mumps}) applied to the corresponding discrete Euler--Lagrange equation;
			\item[(ii)] The reconstruction of the discrete dual solution $z_k^{rt}\in \smash{\mathcal{R}T^0(\mathcal{T}_k)}$ (\textit{i.e.}, maximizer of \eqref{eq:jumping_discrete_dual})  in step (\hyperlink{Solve}{'Solve'}) is based on the generalized Marini formula \eqref{eq:jumping_marini};
			\item[(iii)] As conforming approximations in (\hyperlink{Solve}{'Solve'}), we employ $\overline{u}_k^{cr}\hspace*{-0.175em}\coloneqq\hspace*{-0.175em}\Pi^{av}_{h_k} u_k^{cr}\hspace*{-0.175em}\in\hspace*{-0.175em} \mathcal{S}^1_D(\mathcal{T}_k)\subseteq  W^{1,2}_D(\Omega)$ in the case $u_D=0$ and $\overline{z}_k^{rt}=z_k^{rt}\in \mathcal{R}T^0(\mathcal{T}_k) \subseteq W^2(\textup{div};\Omega)$ in any case.
			\item[(iv)] The local refinement indicators $\smash{\{\eta^2_{\textup{gap},T}(\overline{u}_k^{cr},z_k^{rt})\}_{T\in \mathcal{T}_k}}\subseteq \mathbb{R}_{\ge 0}$, for every $T\in \mathcal{T}_h$,~are~given~via
			\begin{align*}
				\eta^2_{\textup{gap},T}(\overline{u}_k^{cr},z_k^{rt})\coloneqq \tfrac{1}{2}\|A^{\frac{1}{2}}(\cdot) \nabla \overline{u}_k^{cr}-A^{-\frac{1}{2}}(\cdot)z_k^{rt}\|_{2,T}^2\,,
			\end{align*}
			which follows from restricting $\eta^2_{\textup{gap}}(\overline{u}_k^{cr},z_k^{rt})$ (\textit{cf}. \eqref{eq:jumping_eta}) to each element $T\in \mathcal{T}_h$.
		\end{description}
	\end{remark}
	
	\subsubsection{Example with unknown exact solution}\vspace*{-0.5mm}
	
	\hspace*{5mm}For our numerical experiments, we choose $\Omega\hspace*{-0.15em}\coloneqq \hspace*{-0.15em}\left(-1,1\right)^2$, $ \Gamma_D \hspace*{-0.15em}\coloneqq \hspace*{-0.15em}\partial\Omega$,~$ \Gamma_N\hspace*{-0.15em} \coloneqq\hspace*{-0.15em} \emptyset$,~${u_D\hspace*{-0.15em}=\hspace*{-0.15em}0\in W^{\frac{1}{2},2}(\Gamma_D)}$, $f=1\in L^2(\Omega)$, and 
	as jumping coefficient matrix $ A^\varepsilon \colon \Omega\to\mathbb{R}^{2\times 2}$,~for~every~$\varepsilon\in \{16,32,64\}$ and $x\in \Omega$ defined by
	\begin{align}\label{def:A_eps}
		A^\varepsilon(x)\coloneqq 
		\begin{cases}
			\varepsilon \textup{I}_{2\times 2}&\text{ if }\vert x+\mathrm{e}_1\vert< 1\,,\\
			\frac{1}{\varepsilon}\textup{I}_{2\times 2}&\text{ else}\,,
		\end{cases}
	\end{align}
	where $\mathrm{e}_1\coloneqq (1,0)^\top\in \mathbb{R}^2$.
	
	As approximation, for $k=0,\ldots,40$ and $\varepsilon\in \{16,32,64\}$, we employ $A_k^{\varepsilon}\coloneqq A_{h_k}^{\varepsilon}\coloneqq \Pi_{h_k}^{\varepsilon} A^{\varepsilon}\in (\mathcal{L}^0(\mathcal{T}_k^{\varepsilon}))^{2\times 2}$.
	For every $\varepsilon\in \{16,32,64\}$, the primal solution $u^{\varepsilon}\in W^{1,2}_0(\Omega)$ (\textit{i.e.} minimizer~of~\eqref{eq:jumping_primal} with $A=A^\varepsilon$) is not known and cannot be expected to satisfy $u^\varepsilon\hspace*{-0.1em}\in\hspace*{-0.1em} W^{2,2}(\Omega)$~since~${A^{\varepsilon}\hspace*{-0.1em}\notin \hspace*{-0.1em} C^0(\overline{\Omega};\mathbb{R}^{2\times 2})}$.
	As a consequence, uniform mesh refinement (\textit{i.e.}, $\theta=1$ in Algorithm \ref{alg:afem})  is expected to~yield~a reduced convergence rate compared to the optimal~convergence rate $(h_k^{\varepsilon})^2\sim (N_k^{\varepsilon})^{-1}$, where $N_k^{\varepsilon}\coloneqq \textup{dim}(\mathcal{S}^{1,cr}_0(\mathcal{T}_k^{\varepsilon}))$.
	
		The coarsest triangulation $\mathcal{T}_0$ consists of $32$ halved elements and $25$ vertices.~In~Figure~\ref{fig:Jumping_Triang}, for every $\varepsilon\in\{16,32,64\}$, the final triangulations
		 $\mathcal{T}_{40}^{\varepsilon}$  generated~by~Algorithm~\ref{alg:afem} are displayed.~In~it,
	a refinement~towards~$\partial B_1^2(\mathrm{e}_1)\cap \Omega$, \textit{i.e.}, the jump set $J_{A^\varepsilon}$ of the discontinuous coefficient matrix $A^\varepsilon \in L^\infty(\Omega;\mathbb{R}^{2\times 2})$  (\textit{cf}.\ \eqref{def:A_eps}) is reported.\vspace*{-2mm}
	
	 \begin{figure}[H]
		\centering
		\includegraphics[width=14.5cm]{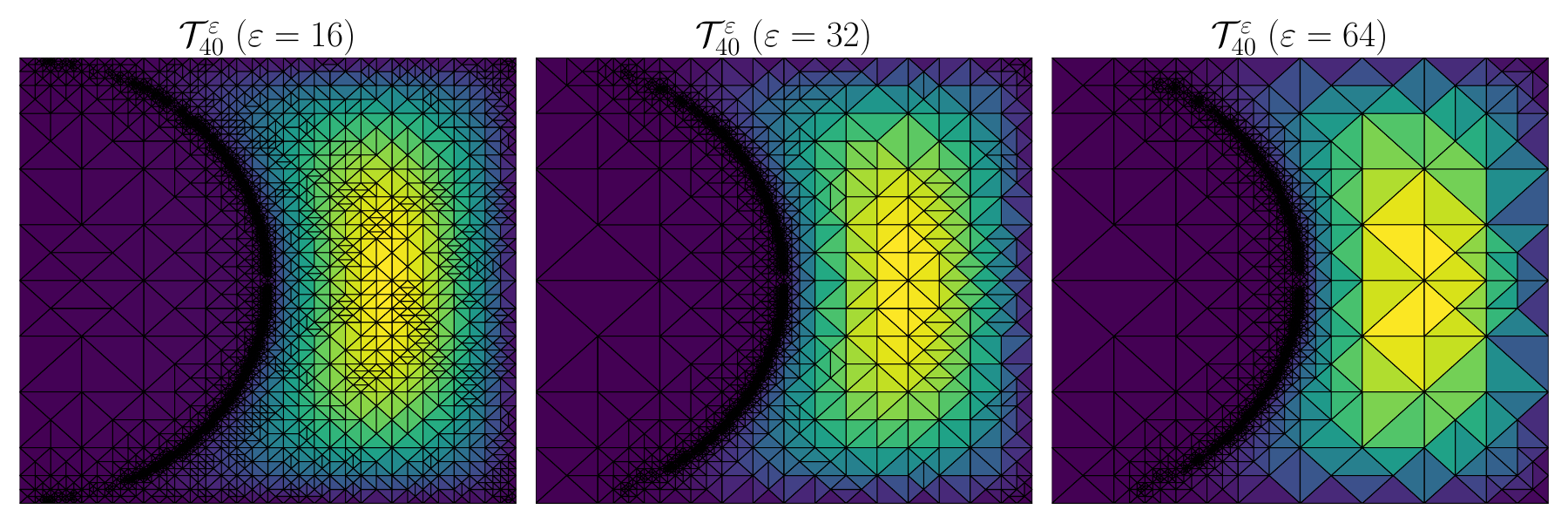}
		\caption{Final triangulations $\mathcal{T}_{40}^{\varepsilon}$, $\varepsilon\in \{16,32,64\}$, generated by Algorithm \ref{alg:afem} for $\theta=\frac{1}{2}$ in the jumping coefficient problem.}
		\label{fig:Jumping_Triang}
	\end{figure}
	
	This behavior is also seen in Figure \ref{fig:Jumping_Solution}, in which,  for every $\varepsilon\in \{16,32,64\}$, the discrete primal solution \hspace*{-0.1mm}$u_{30}^{cr,\varepsilon}\hspace*{-0.15em}\in\hspace*{-0.15em}   \mathcal{S}^{1,cr}_0(\mathcal{T}_{30}^{\varepsilon})$ \hspace*{-0.1mm}(\textit{i.e.}, \hspace*{-0.1mm}minimizer \hspace*{-0.1mm}of \hspace*{-0.1mm}\eqref{eq:jumping_discrete_primal} \hspace*{-0.1mm}with \hspace*{-0.1mm}$A_{h_k}\hspace*{-0.2em}=\hspace*{-0.15em}A_{h_k}^{\varepsilon}$)
	\hspace*{-0.1mm}and~\hspace*{-0.1mm}the~\hspace*{-0.1mm}\mbox{node-averaged}~\hspace*{-0.1mm}\mbox{discrete} primal solution $\overline{u}_{30}^{cr,\varepsilon}\hspace*{-0.15em}\in\hspace*{-0.15em}  \mathcal{S}^1_0(\mathcal{T}_{30}^{\varepsilon})$ are plotted.\
	In addition,  Figures~\ref{fig:Jumping_Triang},~\ref{fig:Jumping_Solution}~show~that~for~increasing value of $\varepsilon\in \{16,32,64\}$, the refinement is more concentrated~at~the~jump~set~${J_{A^\varepsilon}=\partial B_1^2(\mathrm{e}_1)\cap \Omega}$.\vspace*{-1mm}\enlargethispage{11mm}

	 \begin{figure}[H]
		\centering
		\includegraphics[width=14.5cm]{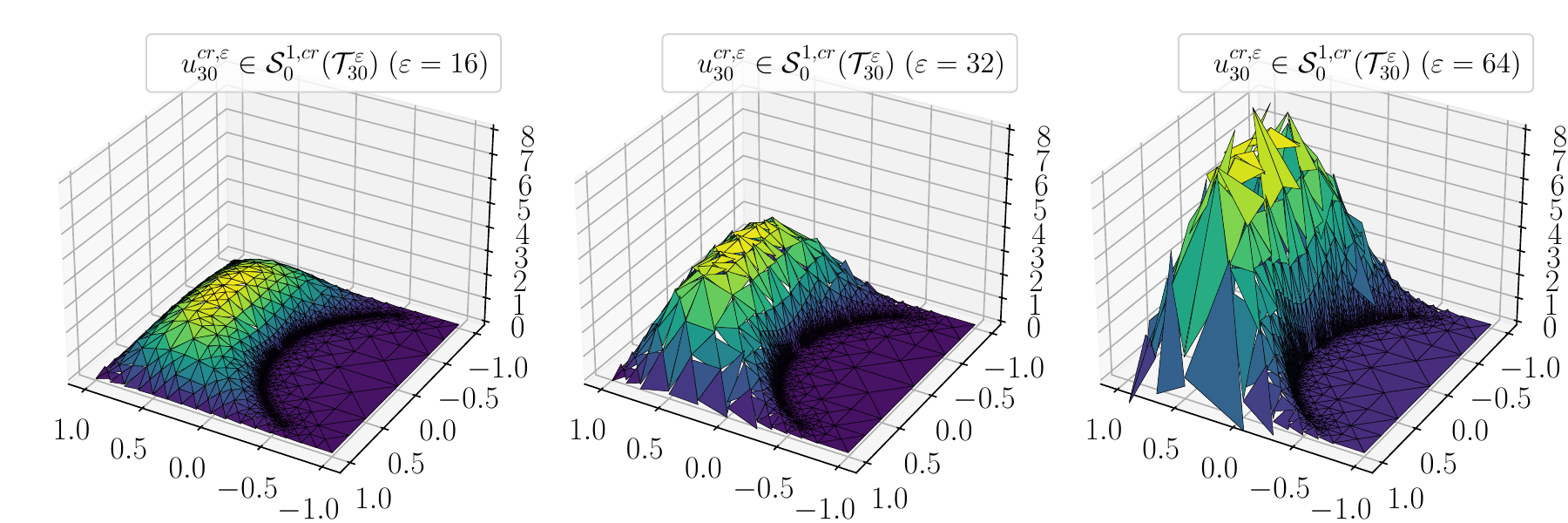}\\
		\includegraphics[width=14.5cm]{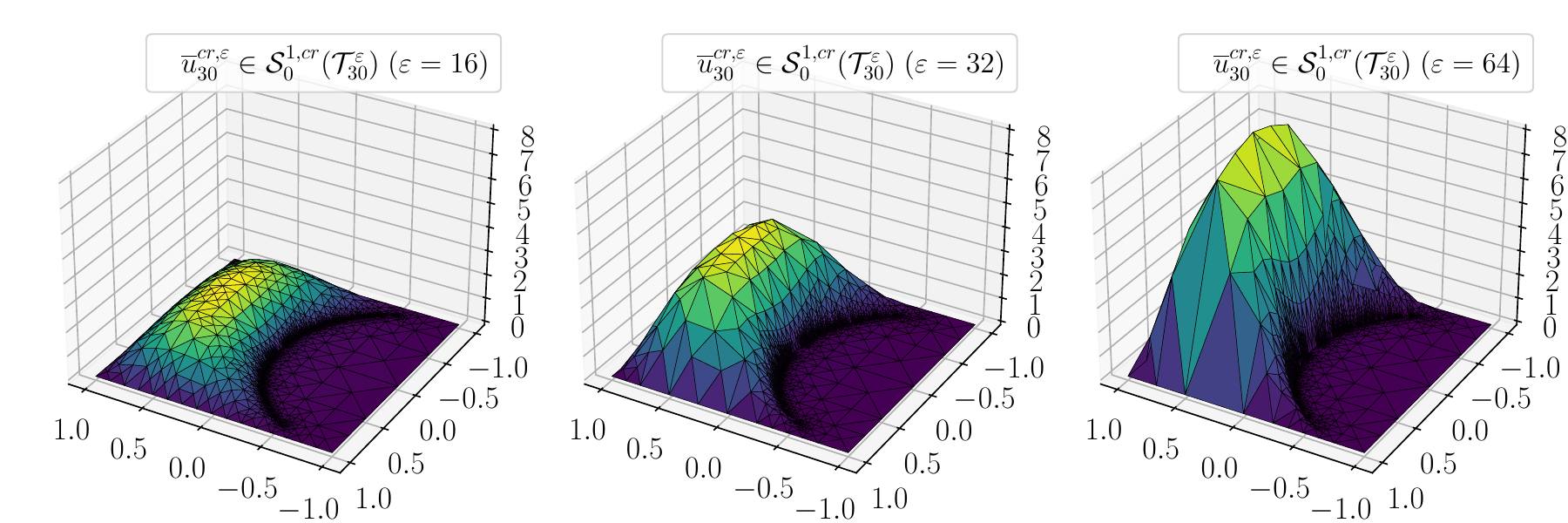}\vspace*{-1mm}
		\caption{TOP: discrete primal solutions $u_{30}^{cr,\varepsilon}\in \mathcal{S}^{1,cr}_0(\mathcal{T}_{30}^{\varepsilon})$, $\varepsilon\in \{16,32,64\}$; BOTTOM: 
		node-averaged	discrete primal solution $\overline{u}_{30}^{cr,\varepsilon}\in \mathcal{S}^{1,cr}_0(\mathcal{T}_{30}^{\varepsilon})$, $\varepsilon\in \{16,32,64\}$;
		each on the triangulation~$\mathcal{T}_{30}^{\varepsilon}$ in the jumping coefficient problem.}
		\label{fig:Jumping_Solution}
	\end{figure}
	
	In Figure \ref{fig:Jumping_Rate}, for  every $\varepsilon\in \{16,32,64\}$, one finds that uniform mesh refinement (\textit{i.e.}, $\theta=1$~in Algorithm~\ref{alg:afem}) yields the expected reduced convergence rate
	 $(h_k^{\varepsilon})^{0.7}\sim (N_k^{\varepsilon})^{\smash{-0.35}}$,~${k=0,\ldots,5}$, while adaptive mesh refinement (\textit{i.e.}, $\theta=\frac{1}{2}$ in Algorithm \ref{alg:afem}) yields the optimal convergence rate $(h_k^{\varepsilon})^2\sim (N_k^{\varepsilon})^{\smash{-1}}$,~$k=0,\ldots,40$.\vspace*{-1mm}

	\begin{figure}[H]
		\centering
		\hspace*{-2mm}\includegraphics[width=14.5cm]{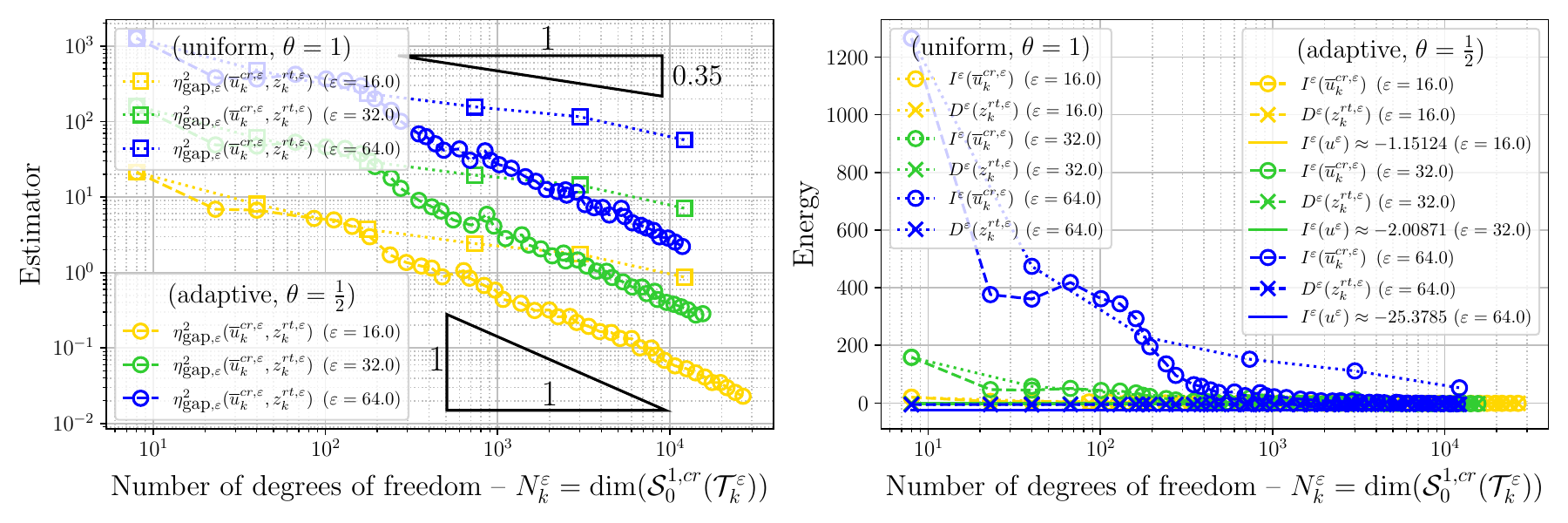}\vspace*{-1mm}
		\caption{LEFT: primal-dual gap estimator $\eta^2_{\textup{gap}}(\overline{u}_k^{cr,\varepsilon},z_k^{rt,\varepsilon})$; RIGHT: 
			primal energy $I^{\varepsilon}(\overline{u}_k^{cr,\varepsilon})$, dual energy $D^{\varepsilon}(z_k^{rt,\varepsilon})$, and primal energy $I^{\varepsilon}(u^{\varepsilon})$ approximated via Aitken's $\delta^2$-process (\textit{cf}. \cite{Ait26});
			each for $\varepsilon\in \{16,32,64\}$ and  $k=0,\dots,40$, when using adaptive mesh refinement  (\textit{i.e.}, $\theta=\frac{1}{2}$ in Algorithm \ref{alg:afem}), and for $k=0,\dots, 5$, when uniform mesh refinement  (\textit{i.e.}, $\theta=1$ in Algorithm~\ref{alg:afem}), in the jumping coefficient problem.}
		\label{fig:Jumping_Rate}
	\end{figure}

	\subsection{Anisotropic mesh refinement}\label{num:anisotropic}\vspace*{-0.75mm}\enlargethispage{13mm}

	\hspace*{5mm}In this subsection, we examine the behavior of the primal-dual gap estimator with respect to anisotropic mesh refinement. For an extensive examination of anisotropic mesh refinement, we refer the reader to \cite{Apel99,NicCre06}.
	We restrict to the Poisson problem \eqref{intro:poisson}, \textit{i.e.},
	we employ~the~implemen-tation of Subsection \ref{subsec:jumping}, but in the case $A(x)\coloneqq \textup{I}_{d\times d}$ for a.e.\ $x\in \Omega$.\vspace*{-0.5mm}
	
	\subsubsection{Example with unknown exact solution}\vspace*{-0.75mm}
	
	\hspace*{5mm}For our numerical experiments, we choose $\Omega \coloneqq \left(-1,1\right)^d\setminus([0,1]\times [-1,0]^{d-1})$, $d\in \{2,3\}$, $ \Gamma_D \coloneqq\partial\Omega$,~$ \Gamma_N\coloneqq \emptyset$,~${u_D = 0\in W^{\frac{1}{2},2}(\Gamma_D)}$, and $f=1\in L^2(\Omega)$. For different grading strengths $\beta\in \{\frac{1}{2},1,\frac{3}{2}\}$ and $k=0,\ldots,40$, we anisotropically refine the triangulations $\mathcal{T}_k^{\beta}$
		towards~the~origin, where the gradient of the unknown primal solution ${u\hspace*{-0.1em}\in\hspace*{-0.1em} W^{1,2}_0(\Omega)}$ (\textit{i.e.}, minimizer of \eqref{eq:jumping_primal} with $A\coloneqq I_{d\times d}$) is expected to have a singularity.  
		Note that the grading strength $\beta=1$  corresponds to uniform mesh refinement.
	
	\textit{2D Case.} For $\beta\in \{\frac{1}{2},1,\frac{3}{2}\}$, the coarsest triangulation $\mathcal{T}_0^{\beta}$ consists of $6$ element and $8$ vertices. Figure \ref{fig:Anisotropic2D_Triang} depicts 
	the anisotropically refined triangulation $\smash{\mathcal{T}_{20}^{\beta}}$  for  grading strengths $\beta\in \{\frac{1}{2},1,\frac{3}{2}\}$.
	 In it, one observes that for increasing grading strength $\beta\in \{\frac{1}{2},1,\frac{3}{2}\}$,
	the potential singularity of the gradient of the unknown primal solution $u\in W^{1,2}_0(\Omega)$ at the~origin~is~better~resolved.\linebreak The same behavior can be seen in Figure \ref{fig:Anisotropic2D_Solution}, in which 
	for different grading strengths $\beta\in \{\frac{1}{2},1,\frac{3}{2}\}$, the discrete primal solution $u_{20}^{cr,\beta}\in \mathcal{S}^{1,cr}_0(\mathcal{T}_{20}^{\beta})$ (\textit{i.e.}, minimizer of \eqref{eq:jumping_primal} with $A_{h_{20}}=\textup{I}_{2\times 2}$)
	and the node-averaged discrete primal solution $\overline{u}_{20}^{cr,\beta}\in \mathcal{S}^{1,cr}_0(\mathcal{T}_{20}^{\beta})$ are depicted.
	In Figure \ref{fig:Anisotropic2D_Rate}, one sees that the grading strengths $\beta=\frac{1}{2}$ and $\beta=1$ yield the reduced convergence rates $\smash{(h_k^{\beta})^{0.8}\sim (N_k^{\beta})^{-0.4}}$ (for $\smash{\beta=\frac{1}{2}}$) and $\smash{(h_k^{\beta})^{1.5}\sim (N_k^{\beta})^{-0.75}}$ (for $\beta=1$), respectively, where $\smash{N_k^{\beta}=\textup{dim}(\mathcal{S}^{1,cr}_0(\mathcal{T}_k^{\beta}))}$, 
	while
	\hspace*{-0.1mm}the \hspace*{-0.1mm}grading \hspace*{-0.1mm}strength \hspace*{-0.1mm}$\smash{\beta\hspace*{-0.15em}=\hspace*{-0.15em}\frac{3}{2}}$ \hspace*{-0.1mm}yields \hspace*{-0.1mm}the  \hspace*{-0.1mm}optimal \hspace*{-0.1mm}convergence~\hspace*{-0.1mm}rate~\hspace*{-0.1mm}$\smash{(h_k^{\beta})^2\hspace*{-0.15em}\sim\hspace*{-0.15em} (N_k^{\beta})^{-1}}$.~Overall, we find that the primal-dual gap estimator is robust with respect to the choice of grading strengths.\vspace*{-2mm}
	
	\begin{figure}[H]
	\centering
	\includegraphics[width=14.5cm]{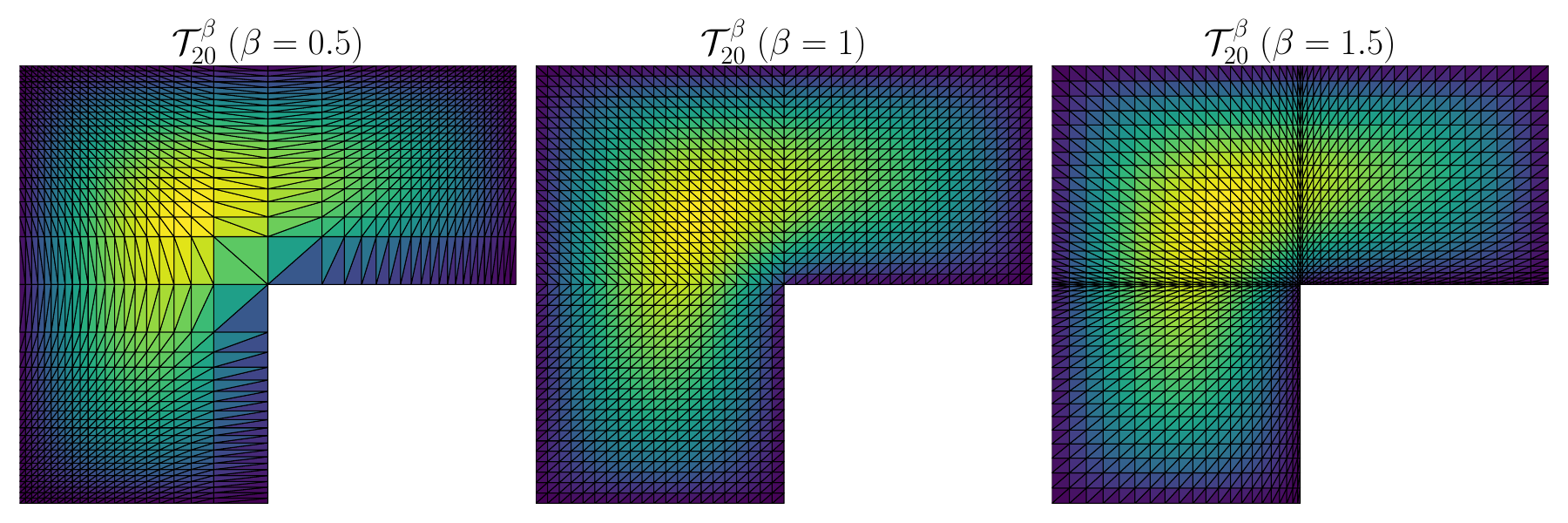}\vspace*{-1.5mm}
	\caption{Anisotropically refined triangulations $\mathcal{T}_{20}^{\beta}$ for different grading strengths $\beta\in \{\frac{1}{2},1,\frac{3}{2}\}$.}
	\label{fig:Anisotropic2D_Triang}
	\end{figure}\vspace*{-7.5mm}
	
		\begin{figure}[H]
		\centering
		\includegraphics[width=14.5cm]{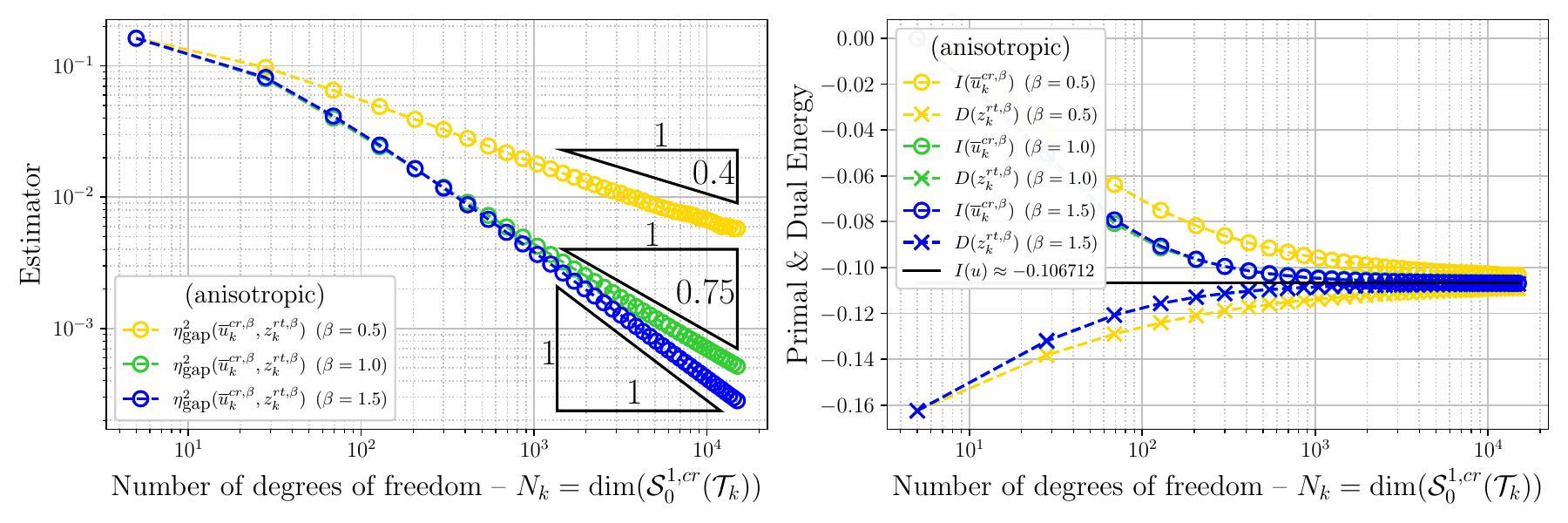}\vspace*{-2mm}
		\caption{LEFT: primal-dual gap estimator $\eta^2_{\textup{gap}}(\overline{u}_k^{cr,\varepsilon},z_k^{rt,\varepsilon})$; RIGHT: 
			primal energy $I(\overline{u}_k^{cr,\beta})$, dual energy $D (z_k^{rt,\beta})$, and primal energy $I(u)$ approximated via Aitken's $\delta^2$-process (\textit{cf}.\ \cite{Ait26});
			each for grading strengths $\beta\in \{\frac{1}{2},1,\frac{3}{2}\}$ and  $k=0,\dots,40$.}
		\label{fig:Anisotropic2D_Rate}
	\end{figure}
	
	\begin{figure}[H]
		\centering
		\includegraphics[width=14.5cm]{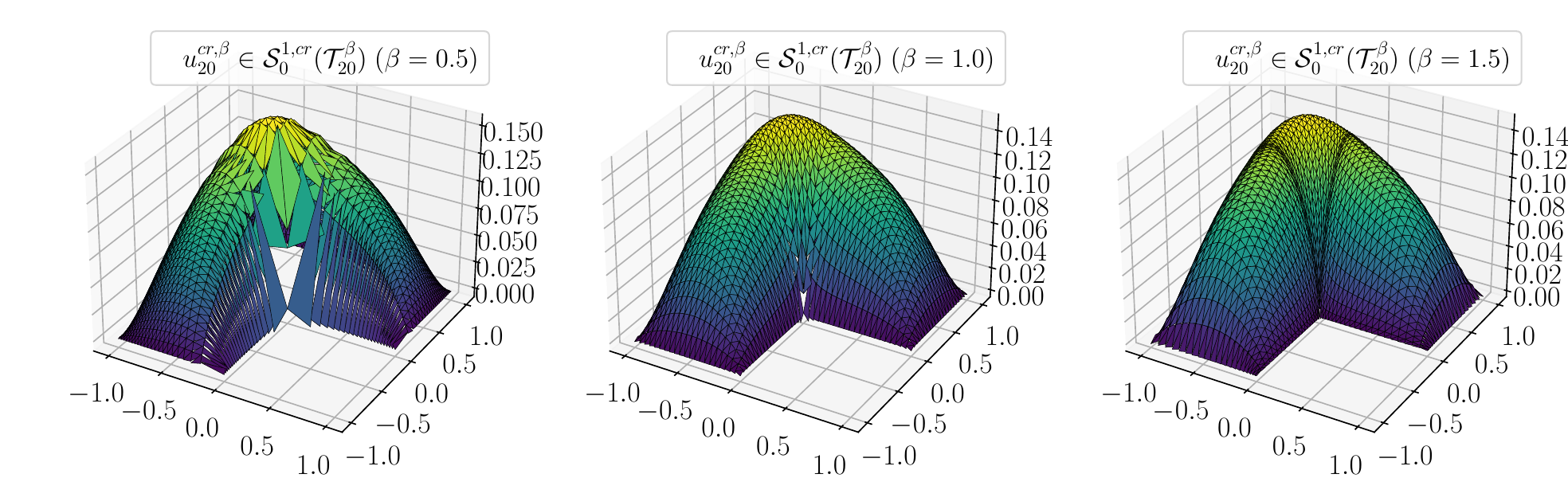}\\
		\includegraphics[width=14.5cm]{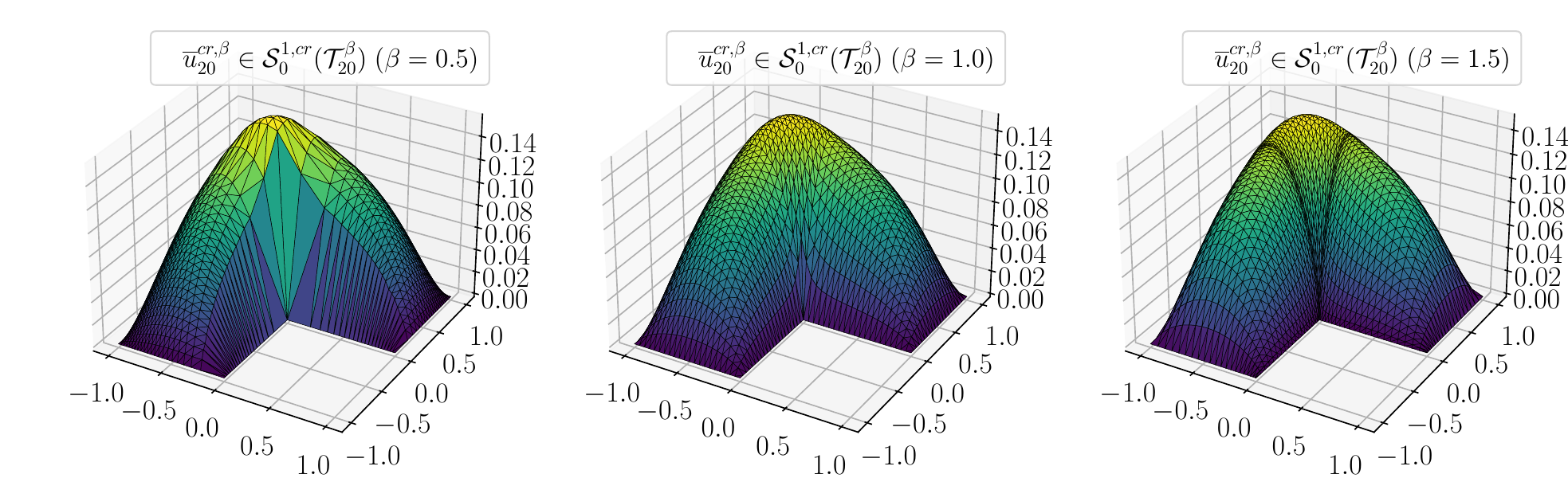}
		\caption{TOP: discrete primal solutions $u_{20}^{cr,\beta}\in \mathcal{S}^{1,cr}_0(\mathcal{T}_{20}^{\beta})$, $\beta\in \{\frac{1}{2},1,\frac{3}{2}\}$; BOTTOM: 
			node-averaged	discrete primal solution $\overline{u}_{20}^{cr,\beta}\in \mathcal{S}^{1,cr}_0(\mathcal{T}_{20}^{\beta})$, $\beta\in \{\frac{1}{2},1,\frac{3}{2}\}$;
			each on the anisotropically refined triangulation~$\mathcal{T}_{30}^{\beta}$ for  grading strengths $\beta\in \{\frac{1}{2},1,\frac{3}{2}\}$.}
		\label{fig:Anisotropic2D_Solution}
	\end{figure}
	
		\textit{3D Case.} For $\beta \in \{\frac{1}{2},1,\frac{3}{2}\}$, the coarsest triangulation $\mathcal{T}_0^{\beta}$ consists of $336$~elements~and~$117$~vertices. In Figure~\ref{fig:Anisotropic3D_Solution_uCR} and Figure  \ref{fig:Anisotropic3D_Solution_zRT}, for $k\in \{0,2,4\}$ and 
	different grading strengths $\beta\in \{\frac{1}{2},1,\frac{3}{2}\}$, the discrete primal solution $u_k^{cr,\beta}\hspace{-0.15em}\in\hspace{-0.15em} \mathcal{S}^{1,cr}_0(\mathcal{T}_k^{\beta})$ (\textit{i.e.}, minimizer of \eqref{eq:jumping_primal} with $A_{h_k}\hspace{-0.15em}=\hspace{-0.15em}\textup{I}_{2\times 2}$)~and~the~\mbox{(local)} $L^2$-projection (onto $\mathcal{L}^1(\mathcal{T}_k)$)  of the modulus of  discrete dual solution~${z_k^{rt,\beta}\hspace{-0.15em}\in\hspace{-0.15em} \mathcal{R}T^0(\mathcal{T}_k^{\beta})}$~are~\mbox{depicted}.
	In \hspace*{-0.1mm}Figure \hspace*{-0.1mm}\ref{fig:Anisotropic3D_Rate}, \hspace*{-0.1mm}one \hspace*{-0.1mm}finds \hspace*{-0.1mm}that \hspace*{-0.1mm}the \hspace*{-0.1mm}grading \hspace*{-0.1mm}strengths \hspace*{-0.1mm}$\beta\hspace*{-0.15em}=\hspace*{-0.15em}\frac{1}{2}$ \hspace*{-0.1mm}and $\beta\hspace*{-0.15em}=\hspace*{-0.15em}1$~\hspace*{-0.1mm}yield~\hspace*{-0.1mm}the~\hspace*{-0.1mm}reduced~\hspace*{-0.1mm}\mbox{convergence}~\hspace*{-0.1mm}rates $\smash{(h_k^{\beta})^{\frac{3}{4}}\sim (N_k^{\beta})^{-\frac{1}{4}}}$ (for $\beta =\frac{1}{2}$) and $\smash{(h_k^{\beta})^{\frac{3}{2}}\sim (N_k^{\beta})^{-\frac{1}{2}}}$ (for $\beta =1$),~where~$\smash{N_k^{\beta}=\textup{dim}(\mathcal{S}^{1,cr}_0(\mathcal{T}_k^{\beta}))}$, 
	while the grading strength $\beta=\frac{3}{2}$ yield the improved convergence rate $\smash{(h_k^{\beta})^2\sim (N_k^{\beta})^{-\frac{2}{3}}}$. 
	
	\begin{figure}[H]
		\centering
		\includegraphics[width=14.5cm]{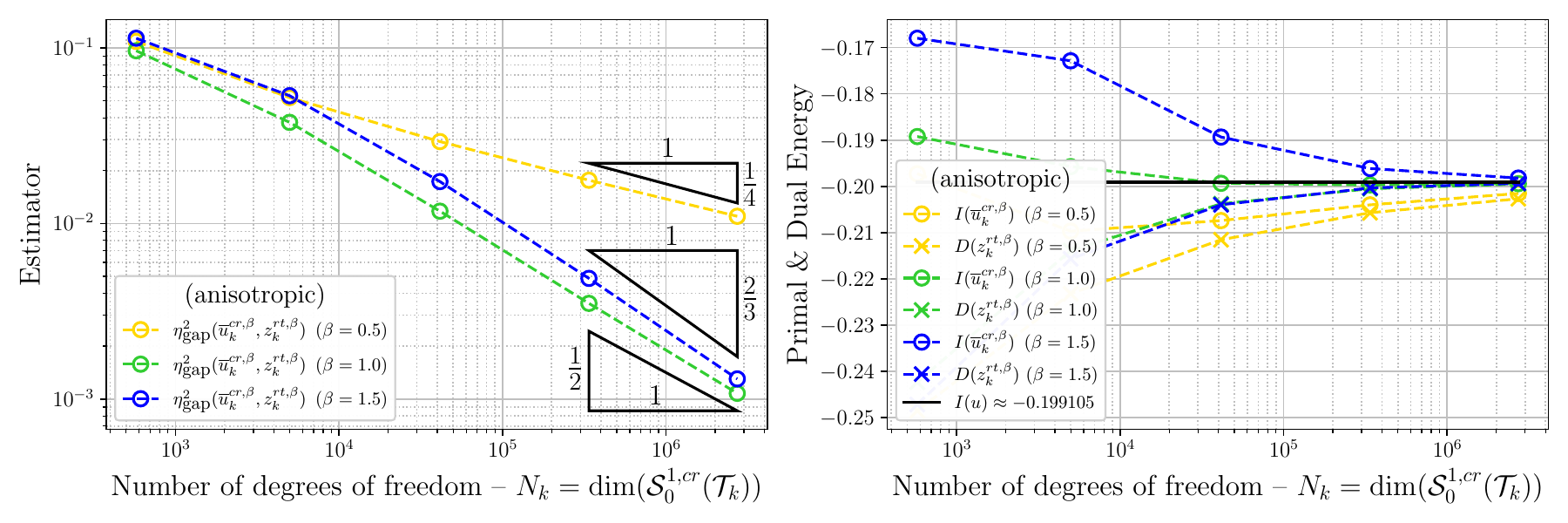}
		\caption{LEFT: primal-dual gap estimator $\eta^2_{\textup{gap}}(\overline{u}_k^{cr,\varepsilon},z_k^{rt,\varepsilon})$; RIGHT: 
			primal energy $I(\overline{u}_k^{cr,\beta})$, dual energy $D (z_k^{rt,\beta})$, and primal energy $I(u)$ approximated via Aitken's $\delta^2$-process (\textit{cf}.\ \cite{Ait26});
			each for grading strengths $\beta\in \{\frac{1}{2},1,\frac{3}{2}\}$ and  $k=0,\dots,3$.}
		\label{fig:Anisotropic3D_Rate}
	\end{figure}
	 \newpage
	 
	 \hphantom{.}
	 \vfill
	\begin{figure}[H]
		\centering
	\hspace*{-1mm}\includegraphics[width=4.75cm]{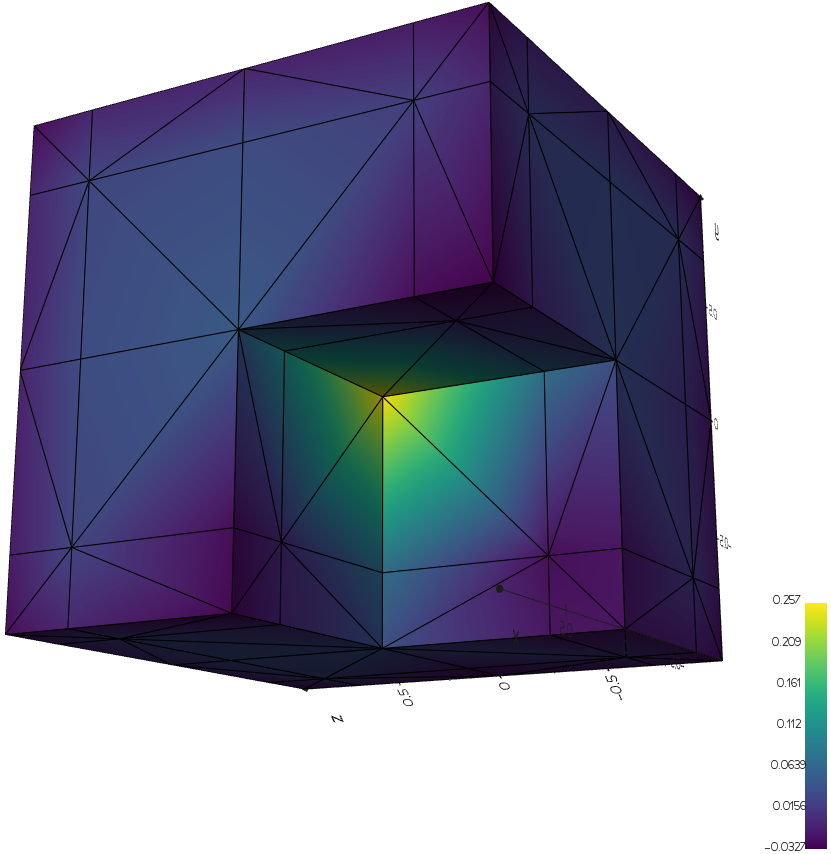}\hspace*{1mm}\includegraphics[width=4.75cm]{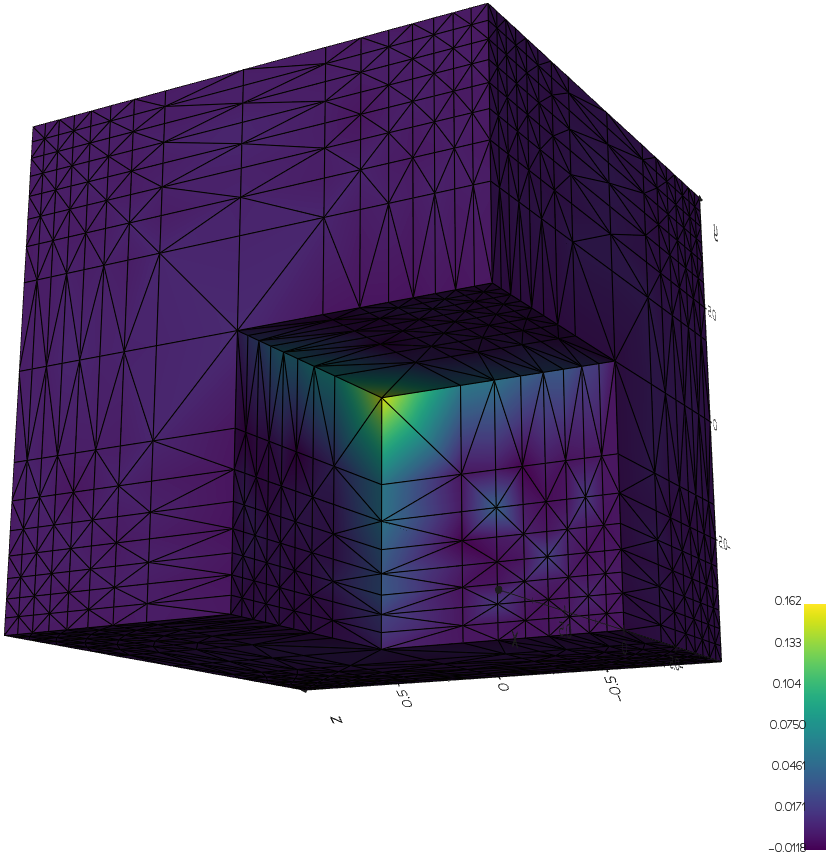}\hspace*{1mm}\includegraphics[width=4.75cm]{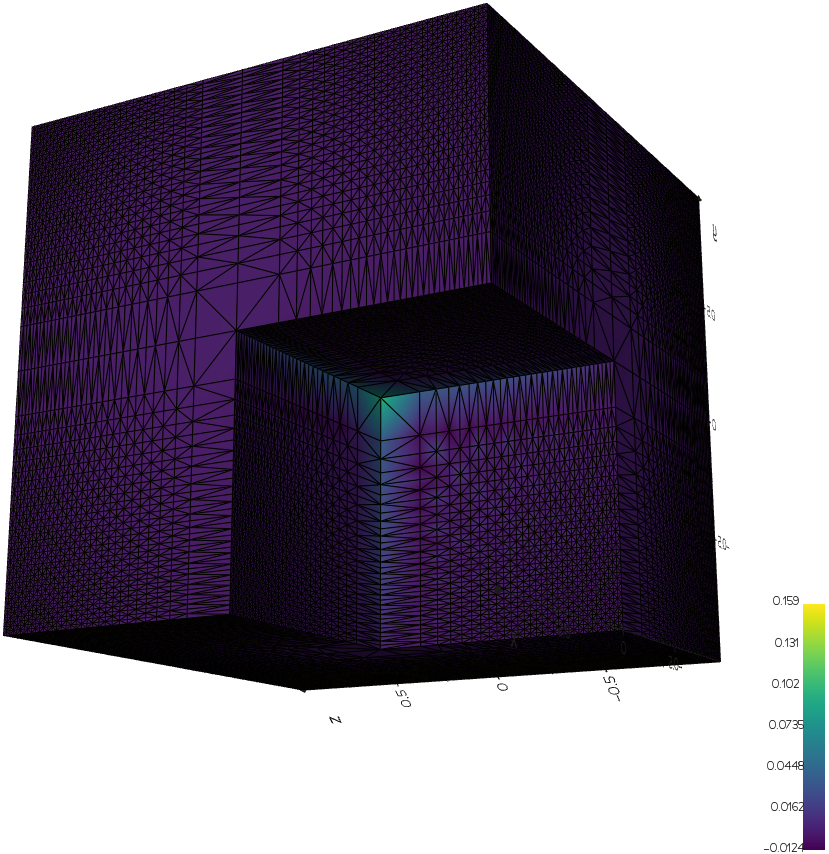}\\
	\hspace*{-1mm}\includegraphics[width=4.75cm]{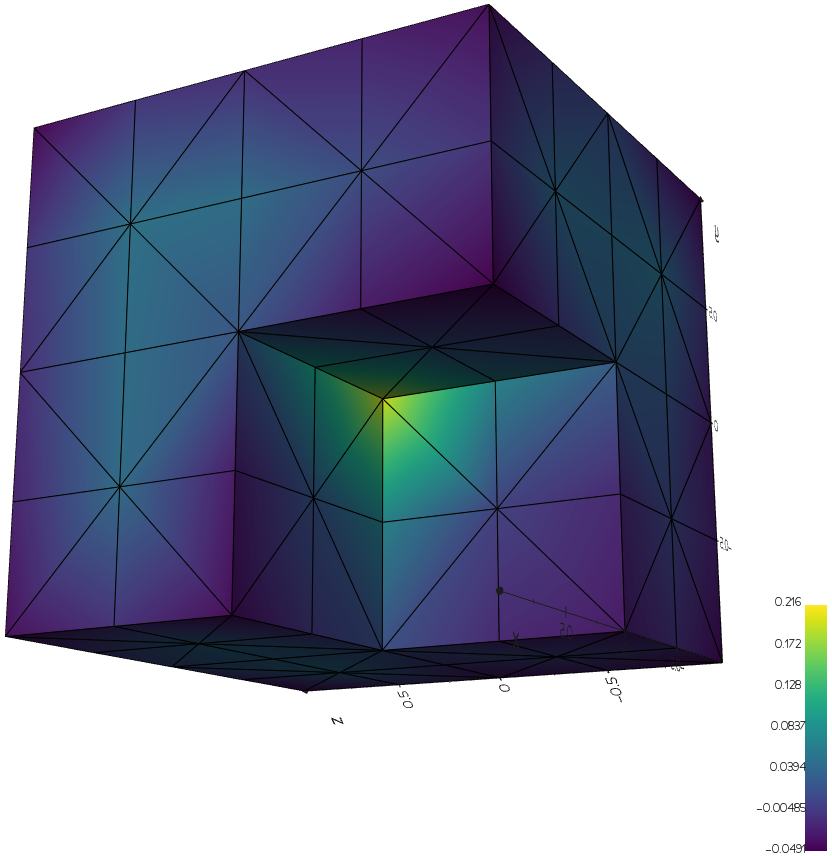}\hspace*{1mm}\includegraphics[width=4.75cm]{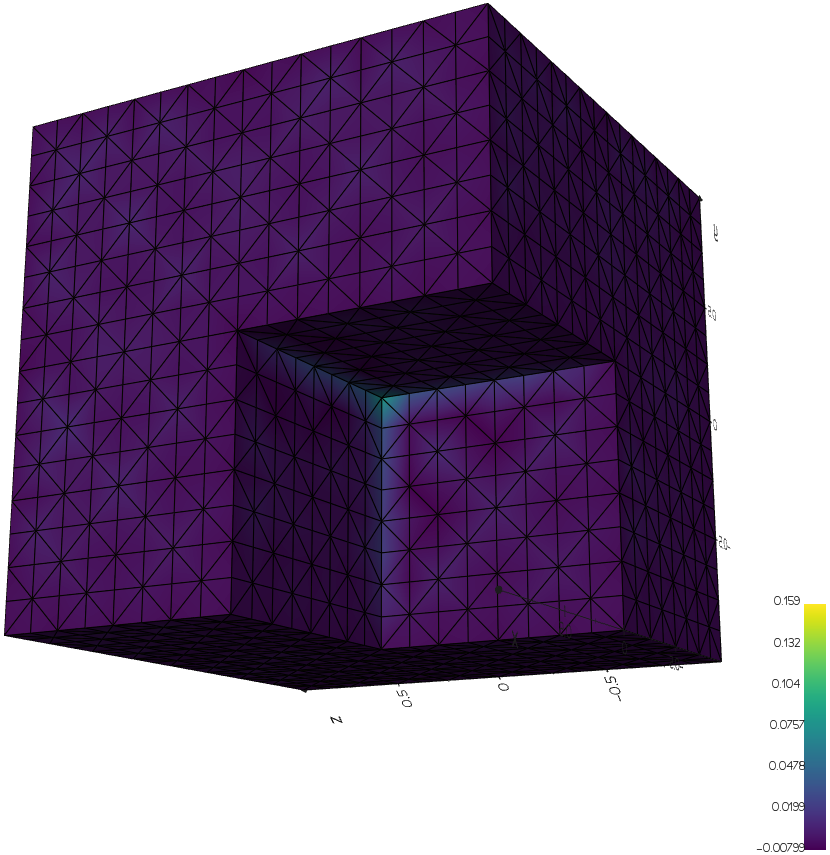}\hspace*{1mm}\includegraphics[width=4.75cm]{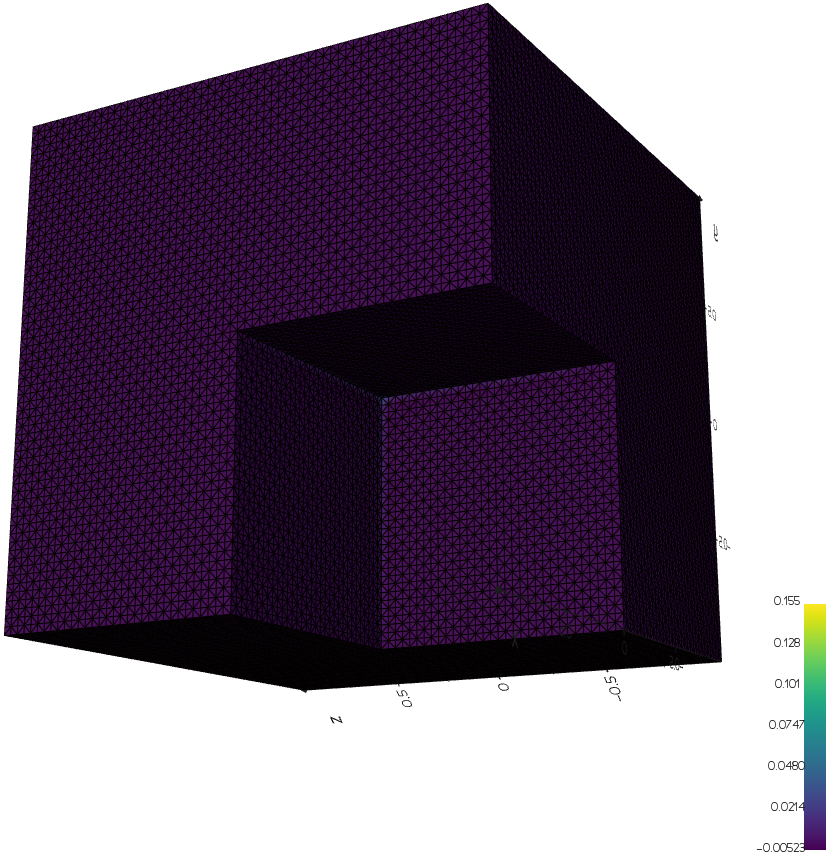}\\
		\hspace*{-1mm}\includegraphics[width=4.75cm]{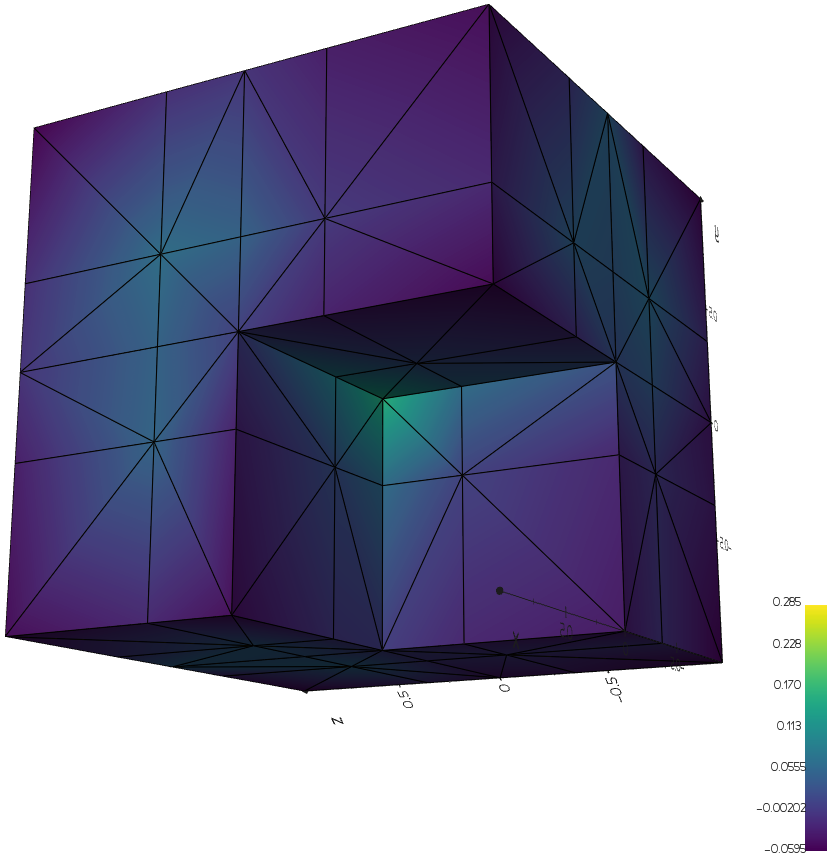}\hspace*{1mm}\includegraphics[width=4.75cm]{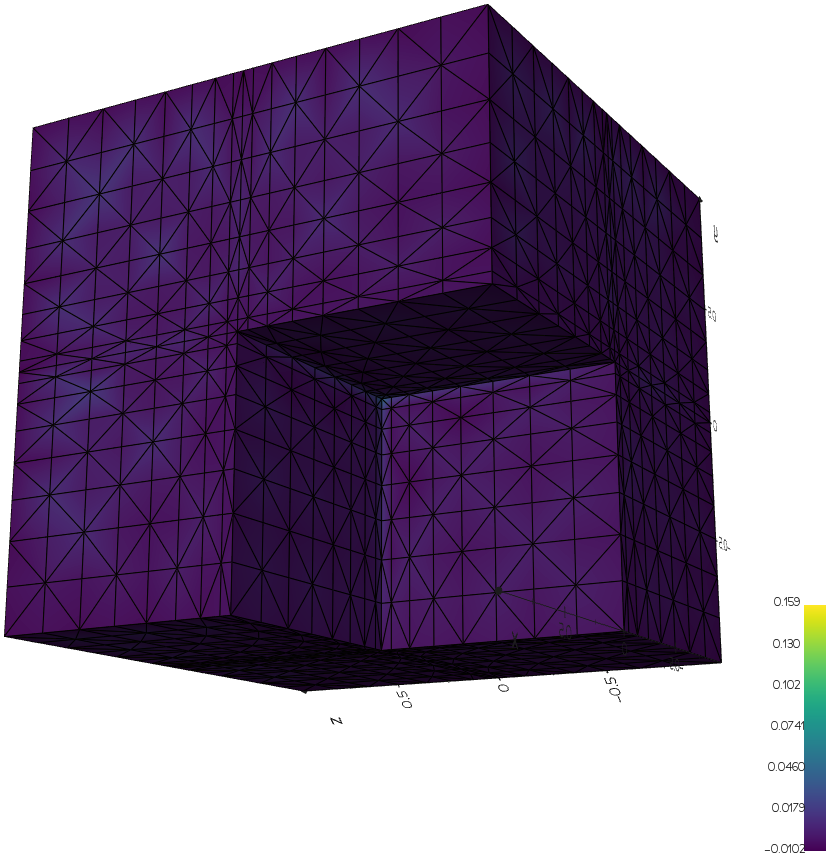}\hspace*{1mm}\includegraphics[width=4.75cm]{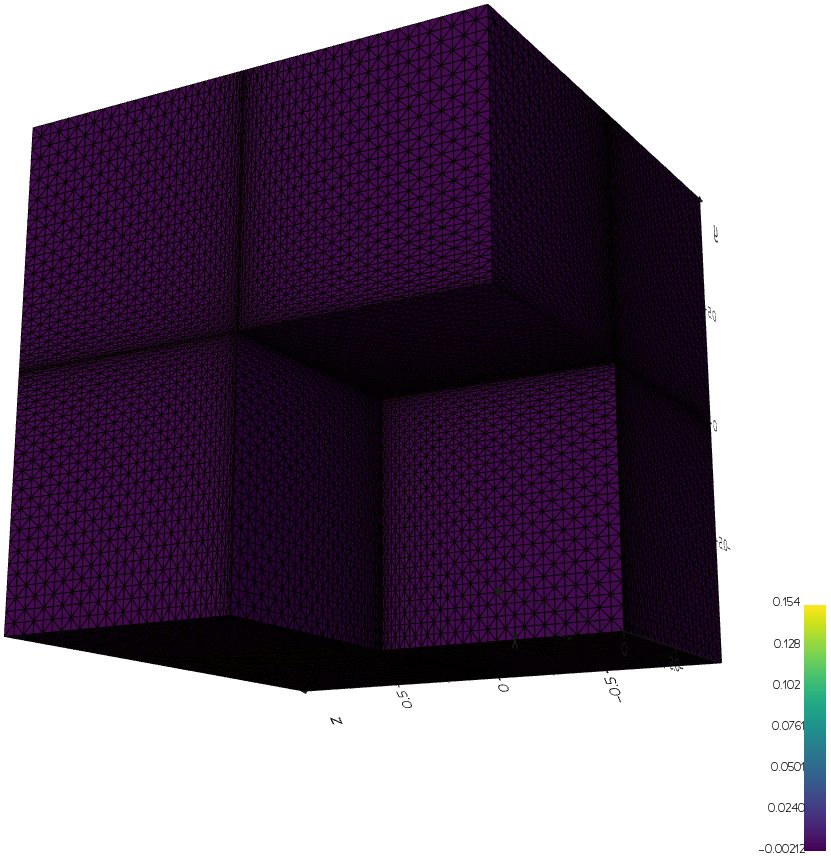}\\
\caption{Discrete primal solution $u_k^{cr,\beta}\in \mathcal{S}^{1,cr}_0(\mathcal{T}_k^{\beta})$ on triangulations$\mathcal{T}_k^{\beta}$ obtained using anisotropic mesh refinement  for $k\in\{0,2,4\}$ (from left to right) and 
	$\beta\in \{\frac{1}{2},1,\frac{3}{2}\}$ (from top to bottom).} 
		\label{fig:Anisotropic3D_Solution_uCR}
	\end{figure}
	\vfill
	\newpage
	
	\hphantom{.}
	\vfill
	\begin{figure}[H]
		\centering
		\hspace*{-1mm}\includegraphics[width=4.75cm]{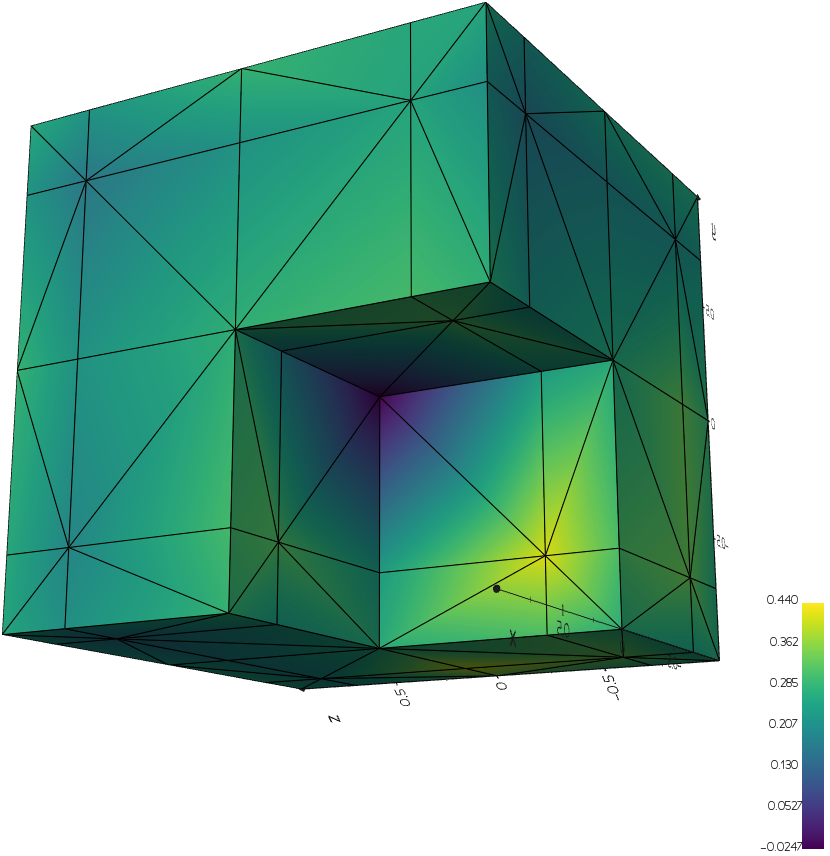}\hspace*{1mm}\includegraphics[width=4.75cm]{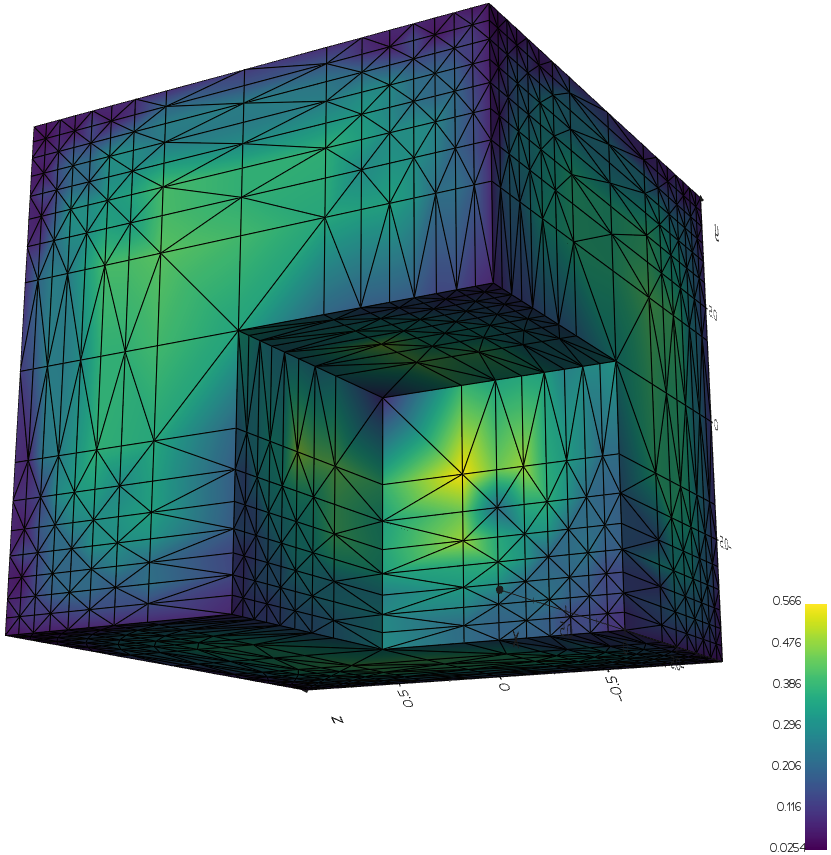}\hspace*{1mm}\includegraphics[width=4.75cm]{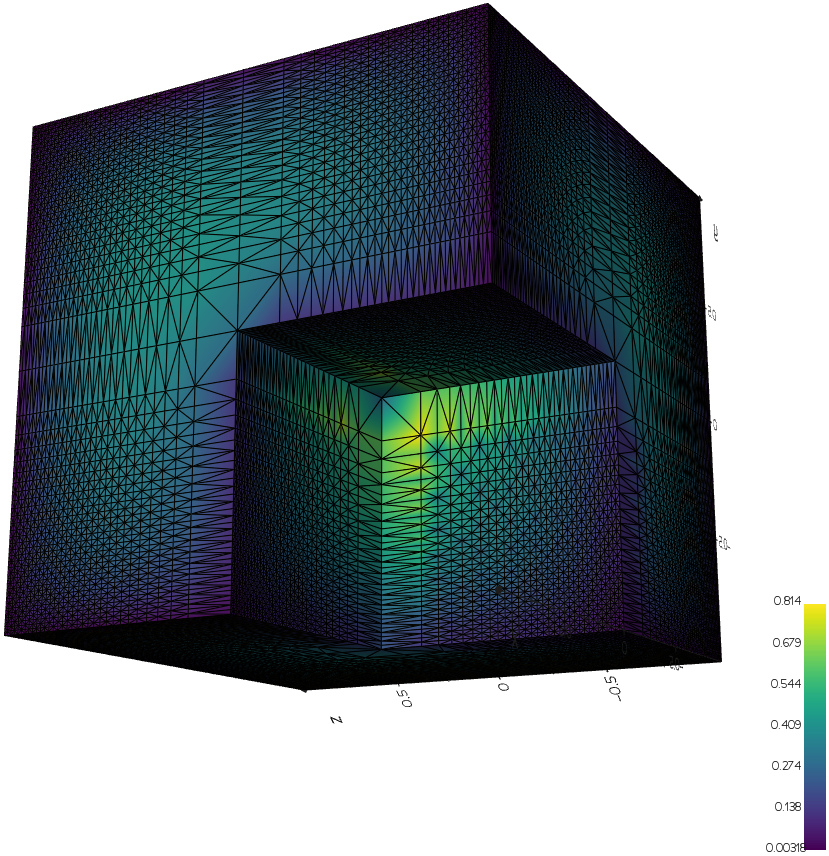}\\
		\hspace*{-1mm}\includegraphics[width=4.75cm]{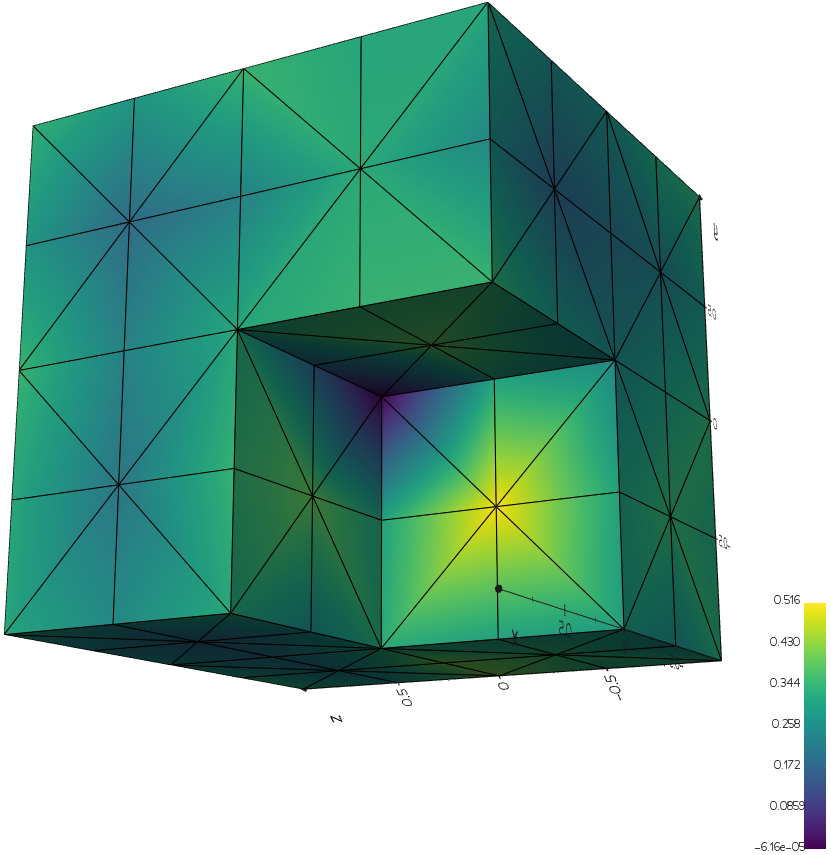}\hspace*{1mm}\includegraphics[width=4.75cm]{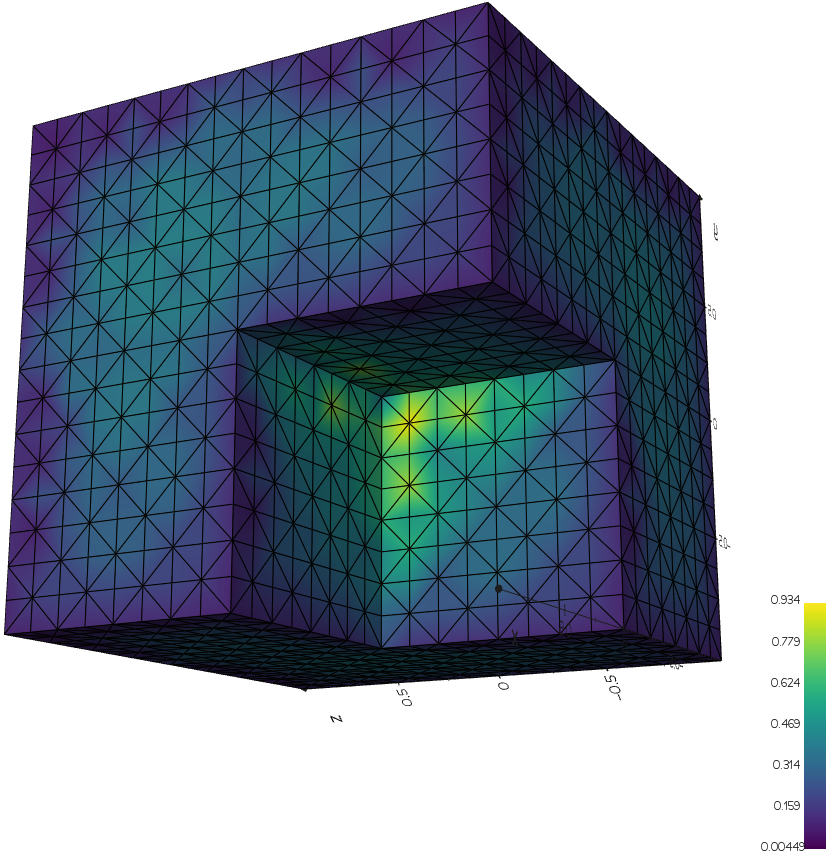}\hspace*{1mm}\includegraphics[width=4.75cm]{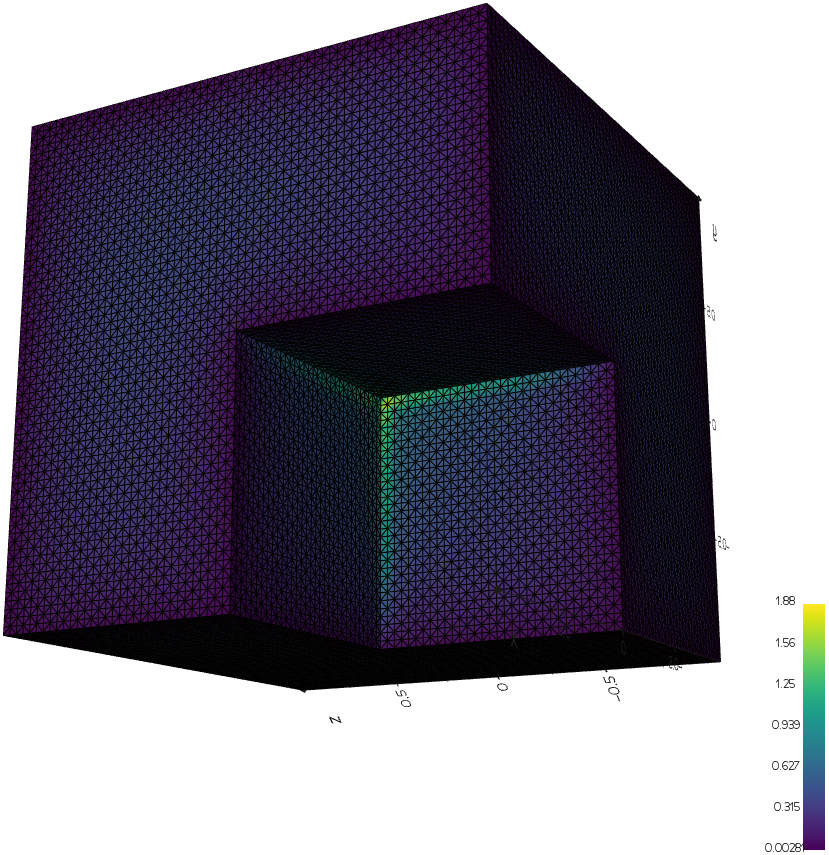}\\
		\hspace*{-1mm}\includegraphics[width=4.75cm]{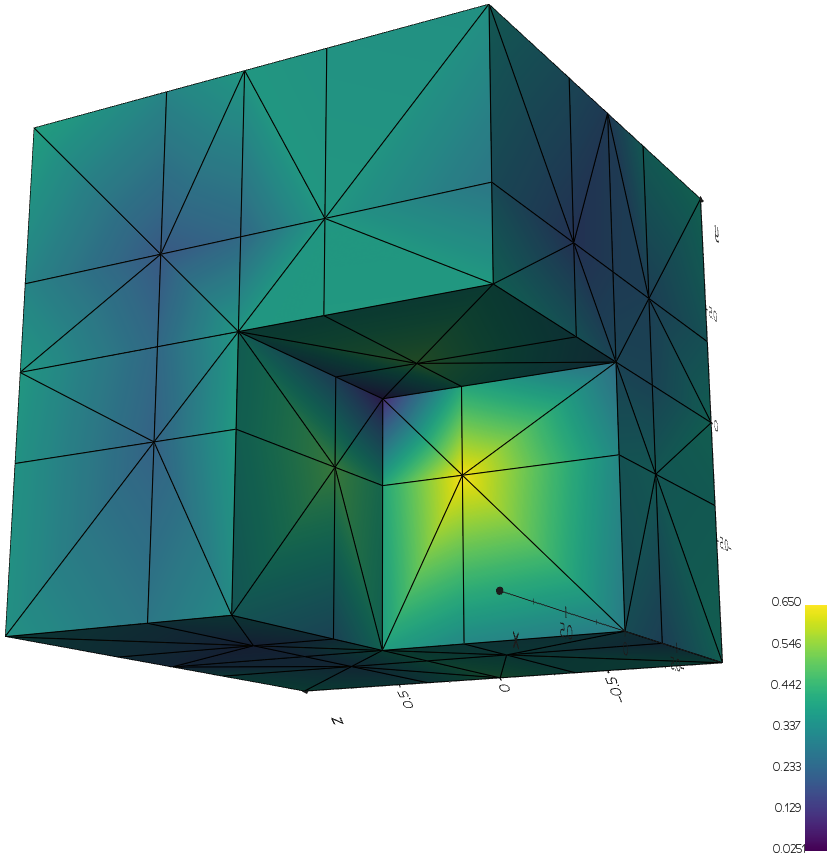}\hspace*{1mm}\includegraphics[width=4.75cm]{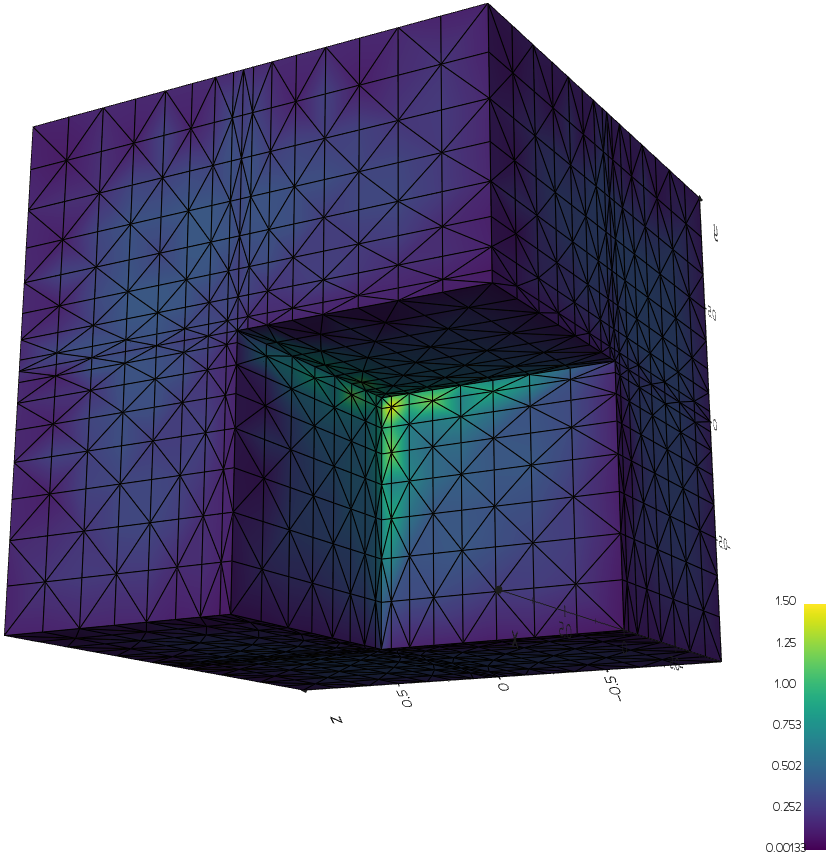}\hspace*{1mm}\includegraphics[width=4.75cm]{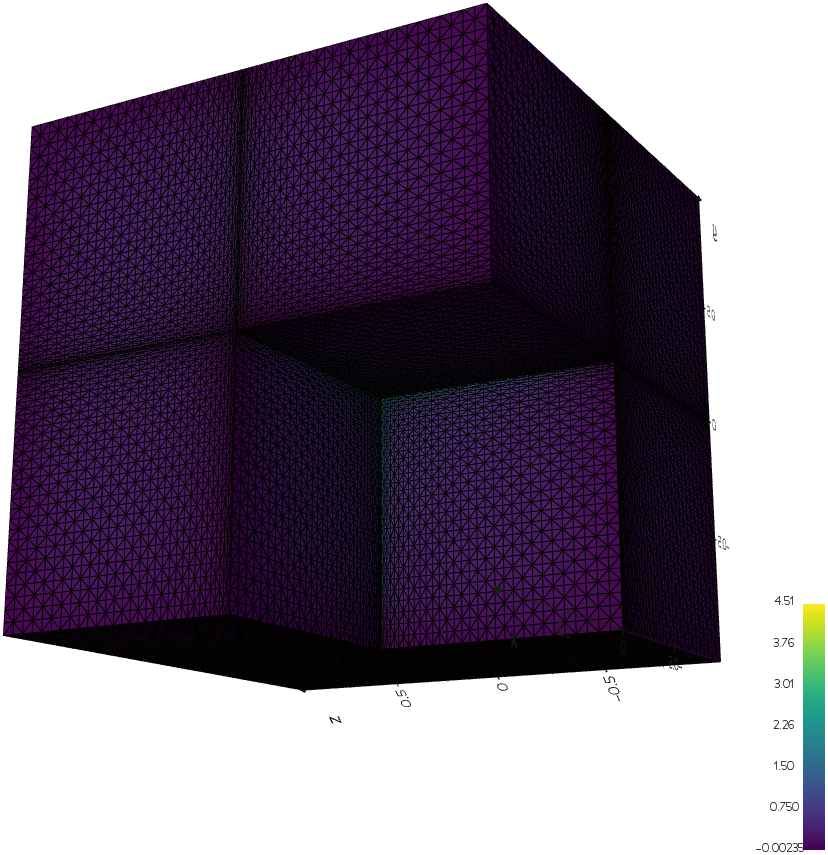}\\
		\caption{(Local) $L^2$-projection (onto $\mathcal{L}^1(\mathcal{T}_k^{\beta}))$ of modulus of discrete dual solution $\smash{\Pi_{h_k^{\beta}}^1\vert z_k^{cr,\beta}}\vert \in \mathcal{L}^1(\mathcal{T}_k^{\beta})$ on triangulation $\mathcal{T}_k^{\beta}$ obtained using anisotropic mesh refinement for $k\in\{0,2,4\}$ (from left to right) and 
			$\beta\in \{\frac{1}{2},1,\frac{3}{2}\}$ (from~top~to~bottom).} 
		\label{fig:Anisotropic3D_Solution_zRT}
	\end{figure}
	\vfill

	\newpage
	\section{Conclusion} 
	
	\hspace*{5mm}On the basis of convex duality arguments, explicit error representations, so-called \textit{generalized Prager--Synge type error identities}, have been derived.
	The minimal requirement for the validity of these identities is a so-called \textit{strong duality relation} between the primal problem and the dual problem, which is satisfied under mild assumptions.
	More precisely, the generalized Prager--Synge type error identities enable to compute the sum of the inaccessible primal and dual approximation errors by means of the so-called \textit{primal-dual gap estimator}.
	The primal-dual gap estimator has an integral representation with two integrands which both, by the Fenchel--Young inequality, are non-negative and vanish if and only if  so-called \textit{convex optimality relations} are satisfied.
	These convex duality relations do not require any regularity of the energy densities, apart from measurability together with convexity and lower semi-continuity with respect to the second argument.\linebreak This makes 
	the primal-dual gap estimator a predestined error estimator, in particular, for non-smooth, convex minimization problems.
	Since the generalized Prager--Synge type error identities treat the primal problem and the dual problem at the same time, it is necessary to insert both a primal approximation  and a dual approximation.
	This can be computationally~expensive,~\textit{e.g.},  if one is actually only interested in the primal approximation error.
	As a consequence, to be numerically practicable it is necessary to have a computationally inexpensive way to 
	obtain a dual approximation via post-processing a primal approximation or vice versa. To this end,   we formulated on the basis of the non-conforming approximation of the primal problem using the Crouzeix--Raviart element a discrete primal problem that generates convex duality relations like in the continuous case and, in particular, a discrete reconstruction formula for a maximizer of a corresponding discrete dual problem,  defined on the lowest-order Raviart--Thomas~element. This reconstruction formula gives us a computationally inexpensive way to approximate the primal and the dual problem at the same time using only the Crouzeix--Raviart element. 

	\appendix
	\section{Node-averaging quasi-interpolation operator and interpolation error estimate in terms of (shifted) $N$-functions}
	
	\hspace*{5mm}In this appendix, we  recall the definition of node-averaging operator $\Pi_h^{av}\colon\mathcal{L}^1(\mathcal{T}_h)\to \mathcal{S}^1_D(\mathcal{T}_h)$ and an interpolation error estimate in terms of (shifted) $N$-functions.
	
	\subsection{Node-averaging quasi-interpolation operator }\label{subsec:node_average} 
	
	\hspace*{5mm}The \textit{node-averaging quasi-interpolation operator} $\Pi_h^{av}\colon  \mathcal{L}^1(\mathcal{T}_h)\to \mathcal{S}^1_D(\mathcal{T}_h)$,  
	denoting~for~every ${\nu\in \mathcal{N}_h}$, 
	by ${\mathcal{T}_h(\nu)\coloneqq \{T\in \mathcal{T}_h\mid \nu \in T\}}$, the set of elements sharing $\nu$, for every $ v_h\in \mathcal{L}^1(\mathcal{T}_h)$, is defined by
	\begin{align*}
		\Pi_h^{av}v_h\coloneqq \sum_{\nu\in \smash{\mathcal{N}_h}}{\langle v_h\rangle_\nu \varphi_\nu}\,,\qquad \langle v_h\rangle_\nu\coloneqq \begin{cases}
			\frac{1}{\textup{card}(\mathcal{T}_h(\nu ))}\sum_{T\in \mathcal{T}_h(\nu)}{(v_h|_T)(\nu)}&\;\text{ if }\nu\in \Omega\cup \Gamma_N\,,\\
			0&\;\text{ if }\nu\in \Gamma_D\,,
		\end{cases}
	\end{align*}
	where we denote by $(\varphi_\nu )_{\smash{\nu\in \mathcal{N}_h}}$ the nodal basis of $\mathcal{S}^1(\mathcal{T}_h)$. 
	
	\subsection{Interpolation error estimate in terms of (shifted) $N$-functions}\label{subsec:auxiliary}
	
	\hspace*{5mm}A convex function $\varphi\colon \mathbb{R}_{\ge 0}\to \mathbb{R}_{\ge 0}$ is said to be an \textit{$N$-function} if and only if $\varphi(0)=0$,~$\varphi(t)>0$ for all $t>0$, $\lim_{t\to 0}{\varphi(t)/t}=0$, and $\lim_{t\to\infty}{\varphi(t)/t}=\infty$. Then, there exists a right-derivative $\varphi'\colon \mathbb{R}_{\ge 0}\to \mathbb{R}_{>0}$, which is non-decreasing and satisfies $\varphi'(0) =0$, $\varphi'(t)>0$ for all $t>0$, and $\lim_{t\to \infty}{\varphi'(t)}=\infty$.
	In addition, an $N$-function $\varphi\colon \mathbb{R}_{\ge 0}\to \mathbb{R}_{\ge 0}$ is said to satisfy the \textit{$\Delta_2$-condition}  (in short, $\varphi\in \Delta_2$) if and only if there exists a constant $c>0$ such that $\varphi(2t)\leq c\,\varphi(t)$~for~all~${t\ge 0}$. We denote the smallest such constant by $\Delta_2(\varphi)$. 
	An $N$-function $ \varphi\colon \mathbb{R}_{\ge 0}\to \mathbb{R}_{\ge 0}$ is said to satisfy the \textit{$\nabla_2$-condition} (in short, $\varphi\in \nabla_2$), if its Fenchel conjugate $\varphi^*\colon \mathbb{R}_{\ge 0}\to \mathbb{R}_{\ge 0}$ is an $N$-function satisfying the $\Delta_2$-condition. If $\varphi\colon \mathbb{R}_{\ge 0}\to \mathbb{R}_{\ge 0}$ satisfies the $\Delta_2$- and the $\nabla_2$-condition (in short, $\varphi\in \Delta_2\cap \nabla_2$), then 
	we define the corresponding  family of \textit{shifted $N$-functions} $\varphi_a\colon \mathbb{R}_{\ge 0}\to \mathbb{R}_{\ge 0}$, $a\ge 0$, for every $t\ge 0$ by
	\begin{align*}
			\varphi_a(t)\coloneqq\int_0^t{\varphi_a'(s)\,\textup{d}s}\,,\quad\text{ where }\varphi_a'(s)\coloneqq\frac{\varphi'(a+s)}{a+s}s\text{ for all }s\ge 0\,. 
	\end{align*}
	Appealing to \cite[Lem.\ 22]{DK08}, it holds that  $c_{\varphi}\coloneqq\sup_{a\ge 0}{\Delta_2(\varphi_a)}<\infty$.	In particular,~for~every~$\varepsilon>0$, there exists a constant 
	$c_\varepsilon>0$, not depending on $a\ge 0$, such that for every $t,s\ge 0$ and $a\ge 0$, there holds the following \textit{$\varepsilon$-Young inequality}:
	\begin{align}
		s\, t\leq  c_\varepsilon\,(\varphi_a)^*(s)+\varepsilon\,\varphi_a(t)\,.\label{eq:eps-young}
	\end{align}

	\begin{proposition}\label{cor:n-function}
		Let $\varphi\colon \mathbb{R}_{\ge 0}\to \mathbb{R}_{\ge 0}$ be an $N$-function  such that $\varphi\in\Delta_2\cap \nabla_2$. 
		Then,  for every $v_h\in  \mathcal{S}^{1,cr}(\mathcal{T}_h)$, $m\in\{0,1,2\}$, $a\ge 0$ and $T\in \mathcal{T}_h$, we have that 
		\begin{align*}
			\int_T{\varphi_a(h_T^m\vert \nabla_h^m(v_h-\Pi_h^{av}v_h)\vert)\,\textup{d}x}\leq  c_{av}
		\int_{\omega_T}{\varphi_a(h_T\vert\nabla_h v_h\vert)\,\textup{d}x}\,.
		\end{align*}
		where $c_{av}>0$ depends only  on $c_{\varphi}$  and $\omega_0$. 
	\end{proposition}
	
	\begin{proof}
		See \cite[Cor.\ A.2]{K22CR}.
	\end{proof}
	
	{\setlength{\bibsep}{0pt plus 0.0ex}\small
		
%	
%	\bibliographystyle{aomplain}
%	\bibliography{literatur}

	\providecommand{\bysame}{\leavevmode\hbox to3em{\hrulefill}\thinspace}
	\providecommand{\noopsort}[1]{}
	\providecommand{\mr}[1]{\href{http://www.ams.org/mathscinet-getitem?mr=#1}{MR~#1}}
	\providecommand{\zbl}[1]{\href{http://www.zentralblatt-math.org/zmath/en/search/?q=an:#1}{Zbl~#1}}
	\providecommand{\jfm}[1]{\href{http://www.emis.de/cgi-bin/JFM-item?#1}{JFM~#1}}
	\providecommand{\arxiv}[1]{\href{http://www.arxiv.org/abs/#1}{arXiv~#1}}
	\providecommand{\doi}[1]{\url{https://doi.org/#1}}
	\providecommand{\MR}{\relax\ifhmode\unskip\space\fi MR }
	\providecommand{\MRhref}[2]{%
		\href{http://www.ams.org/mathscinet-getitem?mr=#1}{#2}
	}
	\providecommand{\href}[2]{#2}

	}
	
\end{document}